\documentclass[10pt]{amsart}

\usepackage[utf8]{inputenc} 
\usepackage{geometry} 
\usepackage{booktabs} 
\usepackage{array} 
\usepackage{paralist} 
\usepackage{verbatim} 
\usepackage{tikz}
\usetikzlibrary{arrows}
\usetikzlibrary{decorations.markings}
\usepackage{subfig}
\usepackage{tabularx}
\usepackage{enumitem}
\usepackage{amsmath,amsthm,amscd}
\usepackage{amssymb,amsfonts}
\usepackage{fancyhdr} 
\usepackage{pgf,tikz}
\usepackage{caption}
\usepackage{tikz-cd}
\usepackage{tikzscale}
\usetikzlibrary{patterns}
\usepackage{rotating}
\usepackage{xcolor}
\usepackage{lscape}
\usepackage{xspace}
\usepackage{xfrac}
\usepackage{blindtext}
\usepackage[ruled,vlined]{algorithm2e}

\usetikzlibrary{decorations.markings}
\usetikzlibrary{patterns}
\usetikzlibrary{calc}
\usepackage{soul}

\usepackage{tikz}
        \usetikzlibrary{patterns,hobby}
        \usepackage{pgfplots}
        \pgfplotsset{compat=1.6}

\usepackage{faktor} 
\usepackage{pgf,tikz}
\usetikzlibrary{arrows.meta}
\usetikzlibrary{decorations.markings}
\usetikzlibrary{patterns}

\usepackage{color}   
\usepackage{hyperref}
\hypersetup{
    colorlinks=true, 
    linktoc=all,     
    linkcolor=blue,  
    citecolor=blue,
    filecolor=blue,
    urlcolor=blue
}

\usepackage[T1]{fontenc}
\usepackage{hyperref}
\usepackage{nameref}

\usepackage{amssymb,amsmath,latexsym,amsthm}
\usepackage{faktor} 		

\theoremstyle{plain}                    
\newtheorem{thm}{Theorem}[section]
\newtheorem*{hthm}{Haupt's Theorem}
\newtheorem{thma}{Theorem}
\newtheorem{propa}[thma]{Proposition}
\newtheorem{cora}[thma]{Corollary}

\newtheorem{thmaa}{Theorem}

\newtheorem{lem}[thm]{Lemma}
\newtheorem{prop}[thm]{Proposition}
\newtheorem{cor}[thm]{Corollary}
\newtheorem*{thmnn}{Theorem}

\theoremstyle{definition}
\newtheorem{defn}[thm]{Definition}

\newtheorem{ex}[thm]{Example}
\newtheorem*{qs}{Question}
\newtheorem{ques}[thm]{Question}

\newtheorem{py}[thm]{Properties}

\theoremstyle{remark}
\newtheorem{rmk}[thm]{Remark}

\numberwithin{equation}{section}

\newcommand{\C}{\mathbb{C}}
\newcommand{\Z}{\mathbb{Z}}

\newcommand{\Q}{\mathbb{Q}}

\newcommand{\cp}{\mathbb{C}\mathbb{P}^1}

\newcommand{\spz}{\mathrm{Sp}(2g,\Z)}

 {\left\lbrace\begin{array}{@{}l@{}}}%
 {\end{array}\right.}

{\begin{list}{}
         {\setlength{\leftmargin}{#1}}
         \item[]
}
{\end{list}}

\def\qr#1#2{%
      \raise1ex\hbox{$#1$}\Big/ \lower1ex\hbox{$#2$}%
}

\def\qrr#1#2{%
      \raise1ex\hbox{$#1$}\Big/\Big/ \lower1ex\hbox{$#2$}%
}

\def\ql#1#2{%
      \lower1ex\hbox{$#1$}\Big\backslash \raise1ex\hbox{$#2$}%
}

 {\left\lbrace\begin{array}{@{}l@{}}}%
 {\end{array}\right.}

\begin{document}
\title[Translation surfaces and periods of meromorphic differentials]{Translation surfaces and  periods of meromorphic differentials}

\author[S. Chenakkod]{Shabarish Chenakkod}
\address{\textnormal{\textbf{Shabarish Chenakkod}, current address:} Department of Mathematics, University of Michigan, Ann Arbor, USA}
\email{shabari@umich.edu}

\author[G. Faraco]{Gianluca Faraco}
\address{\textnormal{\textbf{Gianluca Faraco}, current institute:} Max Planck Institute for Mathematics, Bonn, Germany}
\email{gianlucafaraco@mpim-bonn.mpg.de}

\author[S. Gupta]{Subhojoy Gupta}

\address{Department of Mathematics, Indian Institute of Science, Bangalore 560012, India.}

\email{shabarishch@iisc.ac.in, gianlucaf@iisc.ac.in, subhojoy@iisc.ac.in}

\begin{abstract}
\noindent Let $S$ be an oriented surface of genus $g$ and $n$ punctures. The periods of any meromorphic differential on $S$, with respect to a choice of complex structure, determine a representation  $\chi:\Gamma_{g,n}   \to\mathbb C$ where $\Gamma_{g,n}$ is the first homology group of $S$.  We characterize  the representations that thus arise, that is, lie in the image of the period map   $\textsf{Per}:\Omega\mathcal{M}_{g,n}\to \textsf{Hom}(\Gamma_{g,n},\Bbb C)$. This generalizes a classical result of Haupt in the holomorphic case.  Moreover, we determine the image of this period map when restricted to any stratum  of  meromorphic differentials, having  prescribed orders of zeros and poles. Our proofs are geometric, as they aim to construct a translation structure on $S$ with the prescribed  holonomy $\chi$. Along the way, we describe  a connection with the Hurwitz problem concerning the existence of branched covers with prescribed branching data.
\end{abstract}

\maketitle
\tableofcontents

\section{Introduction} \label{intro}
\noindent Let $S_{g,n}$ be the connected and oriented surface of genus $g$ and with $n$ punctures, and let  $\mathcal{M}_{g,n}$ be the moduli space of punctured Riemann surfaces homeomorphic to $S_{g,n}$.  For  $X\in \mathcal{M}_{g,n}$, define  $\Omega(X)$ to be the space of 
 holomorphic (abelian) differentials on $X$  with at most finite-order poles at the punctures, that we refer to as \textit{meromorphic differentials} on $X$. For $\Gamma_{g,n} := H_1(S_{g,n};\Bbb Z)$, the \textit{period character} (or \textit{character} for simplicity) of such an abelian differential $\omega$ is the homomorphism
\begin{equation}\label{defchar}
\chi: \Gamma_{g,n} \to \Bbb C \,\,\text{ defined as } \,\,
\gamma\longmapsto \int_\gamma \omega.\\
\end{equation}

\noindent Let $\Omega\mathcal{M}_{g,n}$ be the space of pairs $(X,\omega)$ where $X\in\mathcal{M}_{g,n}$ and $\omega\in\Omega(X)$. In this paper, the \textit{period map} is the association 
\begin{equation}\label{permap}
\textsf{Per}:\Omega\mathcal{M}_{g,n}\to \textsf{Hom}(\Gamma_{g,n},\Bbb C)
\end{equation} mapping an abelian differential $\omega$ to its character. We shall  start with the problem of determining the image of the period map. 

\subsection{Translation structures on punctured surfaces with prescribed holonomy}\label{ssec:tswph} Our first main result can be stated as follows. 

\begin{thma}\label{mainthm} Let $n\geq 1$ and let $S_{g,n}$ be a surface of genus $g$ and $n$ punctures. Then the period map $\textsf{\emph{Per}}$ in \eqref{permap} is surjective, that is, every representation $\chi\in\textsf{\emph{Hom}}(\Gamma_{g,n},\,\Bbb C)$ appears as the character of a meromorphic differential $\omega\in\Omega(X)$ for some  $X\in \mathcal{M}_{g,n}$.
\end{thma}

\noindent  For closed surfaces, \emph{i.e} $n=0$, the image of the period map has originally been determined by Haupt in \cite{OH} and subsequently rediscovered by Kapovich in \cite{KM2} by using Ratner theory. It turns out that there are precisely two obstructions for a representation $\chi:\Gamma_g\to\C$ to be the character of an abelian differential, which we shall refer to as \textit{Haupt's conditions}. For a symplectic basis $\{\alpha_1,\beta_1,\dots,\alpha_g,\beta_g\}$  of $\Gamma_g$, we define the volume of $\chi$ as the quantity
\begin{equation}\label{eqvol}
 \textsf{vol}(\chi)=\displaystyle\sum_{i=1}^g \Im\big(\overline{\chi(\alpha_i)}\chi(\beta_i)\big).
\end{equation}
\noindent  The first requirement is that the  volume of $\chi$ is required to be positive with respect to some symplectic basis of $\Gamma_{g}$. Indeed, one can show that this equals the area of the surface $S$ endowed with the translation structure induced by $\omega$. The second obstruction applies when $g\geq 2$, in the case when the image of $\chi:\Gamma_g\to\C$ is a lattice $\Lambda$ in $\C$, since one can show then that the surface $S$ arises from a branched cover of the flat torus $\C/\Lambda$. We then have $\textsf{vol}(\chi)$ is an integer multiple of $\text{{Area}}(\Bbb C/\Lambda)$. Thus, we can state the Haupt's theorem for closed surfaces as follows.

\begin{hthm}[Complex-Analytic statement]\label{haupt:thm} If $g\geq 2$, then the image of $\textsf{\emph{Per}}:\Omega\mathcal{M}_g\to\textsf{\emph{Hom}}(\Gamma_g,\Bbb C)$ consists of those representations $\chi\in\text{\emph{Hom}}(\Gamma_{g},\Bbb C)$ such that
\begin{itemize}
\item $\textsf{\emph{vol}}(\chi)=\displaystyle\sum_{i=1}^g \Im\big(\overline{\chi(\alpha_i)}\chi(\beta_i)\big)>0$,
\item if $\chi(\Gamma)=\Lambda$ is a lattice in $\Bbb C$, then $\textsf{\emph{vol}}(\chi)\ge 2\,\text{\emph{Area}}(\Bbb C/\Lambda)$.\\
\end{itemize}
\end{hthm}

\begin{rmk} 
 In the case  that $g=1$ a representation $\chi:\Gamma_1\longrightarrow \Bbb C$ appears as the character of some holomorphic abelian differential if and only if its image is a lattice in $\Bbb C$ isomorphic to $\Bbb Z^2$ and $\textsf{vol}(\chi)>0$. In this case, the equality $\textsf{vol}(\chi)=\text{Area}(\Bbb C/\Lambda)$ automatically holds. 
 \end{rmk} 

\noindent 
The pair $(X,\omega)$ above determines a \textit{translation structure} on the underlying topological surface, that we already alluded to above. This comprises an atlas of charts to $\C$ that differ by translations on their overlaps (\textit{c.f.} section \ref{gts1}).  Conversely, any translation structure determines such a pair: the atlas equips the surface with a complex structure, and the differential $dz$ on $\mathbb{C}$ descends to an abelian differential $\omega$ on the resulting Riemann surface. From this point of view, Haupt's theorem provides necessary and sufficient conditions for a representation $\chi$ to be the holonomy of some translation structure on a closed surface.

\begin{hthm}[Geometric statement] A character $\chi\in\textsf{\emph{Hom}}(\Gamma_{g},\Bbb C)$ appears as the holonomy of a translation structure on $S_g$ if and only if it satisfies the Haupt's conditions.
\end{hthm}

\noindent 
In the case of a punctured surface $X$, a meromorphic differential with poles at the punctures defines a translation surface with infinite area. These are not new in the theory and they have already been studied by Boissy in his paper \cite{BC} where he calls them  \emph{translation surfaces with poles}. His definition is slightly different from our Definition \ref{tswop} since he requires a pole at every puncture (see  \cite[Convention 2.1]{BC}).  In this work, we shall make use of both the notions  (see Definition \ref{tswp}). \\

\noindent The holonomy representation alone is not sufficient to determine such a translation structure uniquely. Indeed, even for a puncture $P$ with trivial holonomy around it, the flat geometry in a neighborhood of $P$ depends on the prescribed order of the pole  (see section \ref{lgp}). However, this gives us more  freedom in our construction of translation structure with prescribed holonomy $\chi$. 
The geometric version of  Theorem \ref{mainthm} can be reformulated as follows.

\begin{thmaa}\label{main:thma2} Let $S_{g,n}$ be a punctured surface, $n\ge1$. Then any representation $\chi\in\textsf{\emph{Hom}}(\Gamma_{g,n},\,\Bbb C)$ is the holonomy of some translation structure on $S_{g,n}$. In particular, $\chi$ is the holonomy of a translation surface with poles.
\end{thmaa}

\noindent We shall adopt this geometric point of view for proving our Theorem \ref{mainthm}. Our strategy consists of first proving our Theorem \ref{main:thma2} for surfaces of type $(1,1)$ and $(0,n)$, see Proposition \ref{proppt} and Proposition \ref{propb} respectively. Subsequently, we shall make use of our understanding of these two particular cases for deriving the main Theorem in full generality. Recall that a topological surface of type $(g,n)$ splits as the connected sum of the closed surface $S_g$ and the $n$-punctured sphere $S_{0,n}$. As a consequence, any given representation $\chi:\Gamma_{g,n}\longrightarrow \Bbb C$ induces two representations
\begin{equation} \label{eqn:tworeps}
\chi_g:\Gamma_{g,0}\longrightarrow\Bbb C \text{ and }  \chi_n:\Gamma_{0,n}\longrightarrow\Bbb C \text{ .}
\end{equation} 
This follows because the target is abelian and any simple closed separating curve has trivial holonomy. If the surface $S_{g,n}$ is not of type $(1,1)$ or $(0,n)$, we make the use of such a splitting and two different scenarios appear depending on whether $\chi_g$ satisfies or not the Haupt's conditions. In section \ref{gos}, we provide an algorithm, in either case,  to construct a translation  structure with the prescribed holonomy.  See the preamble of Part \ref{part1} for a brief outline of this geometrizing process. The translation surface we obtain this way might correspond to an abelian differential that extends holomorphically over some punctures. However, we explain in subsection \ref{mtsgc} how to turn this to a metrically complete translation surface, that is, to a translation surface with poles.

\subsection{Translation structures on punctured surfaces with prescribed zeros and poles} 
For a closed surface of genus $g\geq 2$, the space of abelian differentials $\Omega\mathcal{M}_g$ forms a vector bundle over the moduli space $\mathcal{M}_g$ and the fiber over each Riemann surface $X$ is $\Omega(X)$, a complex vector space of dimension $g$ by Riemann-Roch. The total space $\Omega\mathcal{M}_g$ admits a natural stratification given by the strata $\mathcal{H}(k_1,\dots,k_n)$, enumerated by unordered partitions $k_1+\dots+k_n=2g-2$, where $\mathcal{H}(k_1,\dots,k_n)$ is the set of abelian differentials  having exactly $n$ zeros of order $k_1,\dots,k_n$. For results on the connectedness of such strata, see \cite{KZ}.

\begin{rmk} The fact that the sum of orders of zeros of an abelian differential on a Riemann surface of genus $g$ equals to $2g-2$, can be seen as the formula of Gauss--Bonnet for the singular Euclidean metric induced by $\omega$ over $X$. Here, recall that a zero of  order $k$ is of the form $z^kdz$, and corresponds to a singular (or branch) point of the associated translation structure with magnitude $2(k+1)\pi$. 
\end{rmk}

\noindent For any partition $\kappa$ of $2g-2$, the stratum $\mathcal{H}(\kappa)$ is a complex orbifold of dimension $2g+s-1$, where $s=|\kappa|$ is the length of the partition and hence the number of singular points. Local coordinates in a neighborhood of a generic point are provided by the periods of the  relative homology group $H_1(S,\Sigma;\Bbb Z)$ where $\Sigma=\{P_1,\dots,P_s\}$ is the set of singular points. Although  fibers of the absolute period map (called \textit{isoperiodic foliations}) have been intensively studied in the last years, see for instance \cite{CDF2}, \cite{UH}, \cite{McM}, \cite{McM2}, \cite{YF} and references therein, the image of $\textsf{Per}_{|\mathcal{H}(\kappa)}$ has been determined only recently by independent efforts of Bainbridge-Johnson-Judge-Park in \cite{BJJP} and Le Fils in \cite{LFT}. Their main result, that we now state,   can be seen a refinement of the Haupt's Theorem. 

\begin{thmnn}[Bainbridge-Johnson-Judge-Park, Le Fils]
Let $\kappa=(n_1,\dots,n_s)$ be a partition of $2g-2$. A character $\chi\in\textsf{\emph{Hom}}\big(\Gamma_g,\C\big)$ appears in the image of the mapping $\textsf{\emph{Per}}_{|\mathcal{H}(\kappa)}$ if and only if
\begin{itemize}
\item $\textsf{\emph{vol}}(\chi)>0$,
\item if $\chi(\Gamma)=\Lambda$ is a lattice in $\Bbb C$, then $\displaystyle\textsf{\emph{vol}}(\chi)\ge \big(\max\{n_1,\dots,n_s\}+1\,\big)\,\text{\emph{Area}}(\Bbb C/\Lambda)$.
\end{itemize}
\end{thmnn}

\medskip

\noindent In the case of meromorphic abelian differentials, we shall introduce the following stratification of $\Omega \mathcal{M}_{g,n}$. Let $\kappa=(d_1,\dots,d_k)$ and $\nu=(p_1,\dots,p_n)$ be two tuples of positive integers satisfying the degree condition
\begin{equation}\label{zpec}
\sum_{i=1}^k d_i-\sum_{j=1}^n p_j=2g-2.
\end{equation}
\noindent We shall denote by $\mathcal{H}(\kappa;\nu)=\mathcal{H}(d_1,\dots,d_k;p_1,\dots,p_n)$ the space of meromorphic abelian differentials with zeros of degrees $d_1,\dots,d_k$ and poles of degrees $p_1,\dots,p_n$. This is a complex orbifold of dimension $2g+s-2$ where $s=|\kappa|+|\nu|$, see \cite{BCGGM} and \cite{BC}.  Determining the image of $\textsf{Per}_{|\mathcal{H}(\kappa;\nu)}$, the period map restricted to such a stratum, turns out a more challenging and subtle problem, and is the goal of the second part of the paper. For recent related work concerning the \textit{residual map} on such strata, see \cite{GT2}.\\

\noindent The first case we shall consider is that of trivial holonomy, that is, when the period character is the trivial representation.  In this case our second result, handled by Propositions \ref{thm:mainthm} and \ref{thm:mainthm2}, provides a complete answer.

\begin{thma}[Trivial holonomy]\label{main:thmb} Let $g\geq 0$, and let $\nu = (p_1, \ldots, p_n)$ and $\kappa = ( d_1, \ldots, d_k )$ be positive integer tuples satisfying \eqref{zpec}.  
There exists a meromorphic differential in the stratum $\mathcal{H}(\kappa; \nu)$ having trivial holonomy, if and only if 
\begin{itemize}
\item[(i)] $p_i >1$ for  each $1 \leq i \leq n$, 
 \item[(ii)] $d_j \leq \displaystyle\sum_{i=1}^{n}p_i - n - 1$ for each $1 \leq j \leq k$, and 
 \item[(iii)] if $g>0$, $k>1$ whenever $n>1$.
 \end{itemize}
 \end{thma}

\noindent   Note that the holonomy around a pole of order one (also called a \textit{simple} pole) is proportional to the residue (see section \ref{lgp}); thus for trivial holonomy the poles must have order at least two, which is requirement (i) above. Moreover, for a translation structure with trivial holonomy, the developing map descends to a holomorphic function $S_{g,n}\longrightarrow \C$ which extends to a branched covering $S_g\longrightarrow \cp$ by mapping all the punctures (corresponding to the poles) to the point $\infty\in\cp$. The necessity of the rest of the requirements above are imposed by this branched covering. For instance, a meromorphic differential with trivial character has a single zero if and only if $(X,\omega)=(\C,\,z^d\,dz)$, see Proposition \ref{singlezero}. The proof of the Theorem above proceeds with first considering the case of the punctured sphere. In this special case, we consider the family $\mathcal{F}=\{(\C, z^{p_i-2}\,dz)\}$ of translation surfaces and the basic idea is to glue them in an inductive process using the slit construction (\textit{c.f.} subsection \ref{sec:seqslit}) followed by splitting a zero  (\textit{c.f.} subsection \ref{sec:zerosplit}). In the case of positive genus surfaces, we first construct a translation structure on a punctured sphere for an associated tuple $\kappa^\prime$, and add $g$ handles; our argument is specific to the trivial representation, though, and does not apply otherwise. \\

\noindent For non-trivial representations, our main result is Theorem \ref{thm:nthpgs}, stated here as follows:

\begin{thma}\label{main:thmc}
Let $\chi\in\textsf{\emph{Hom}}(\Gamma_{g,n}, \mathbb{C})$ be a non-trivial representation and let $\kappa = (p_1, p_2, \ldots, p_n)$ and $\nu = (d_1, d_2, \ldots d_k)$  be two tuples of positive integers satisfying \eqref{zpec},  such that $p_i \geq 2$ whenever $\chi_n (\gamma_i) =0$, and one of the following properties hold:
    \begin{itemize}
        \item[i.] the $\chi_n$ determined by $\chi$  -- see \eqref{eqn:tworeps} --  is trivial,
        \item[ii.] at least one of $p_1, p_2, \ldots p_n$ is different from 1,
        \item[iii.] $\textsf{\emph{Im}}(\chi_n)$ is not contained in the $\mathbb{Q}$-span of some $c \in \mathbb{C}$,
        \item[iv.] $\textsf{\emph{Im}}(\chi)$ is not contained in the $\mathbb{Q}$-span of some $c \in \mathbb{C}$.
    \end{itemize}
Then $\chi$ appears as the holonomy of a translation structure with poles on $S_{g,n}$ induced by a meromorphic differential in $\mathcal{H}(\kappa; \nu)$. 
\end{thma}

\noindent  This provides sufficient conditions for the realizability of a large class of non-trivial representations $\chi$. Define a representation $\chi \in \textsf{{Hom}}(\Gamma_{g,n}, \mathbb{C})$ to be \textit{rational} if it is non-trivial and its image is contained in the $\Q$-span of some complex number $c\in\C^*$. Then the only case not covered by Theorem \ref{main:thmc} is that of rational representations where all the poles are required to be simple. This shall be  handled by Theorem \ref{main:thme} that we shall discuss soon.

\noindent Once again, in the proof of Theorem \ref{main:thmc}, the case of punctured spheres is the first step: in section \ref{sec:nontrivholsphere} we prove it in this case, according to the items $(ii)-(iv)$. Note that when $g=0$, $\chi=\chi_n$ and item $(i)$ is subsumed by Theorem \ref{main:thmb}. In general, the cases corresponding to the items $(i)$ and $(ii)$ are easier, since the geometry around a pole of order at least two comprises Euclidean plane(s) glued together and therefore there is enough room to ``add handles" of the desired holonomy, by surgeries introduced  in subsection \ref{sec:seqslit}. The items $(iii)$ and $(iv)$ in the case when $g>0$ and  all the poles are simple, are the most difficult to treat, since the handles then need to ``fit" inside the corresponding cylindrical ends. For this, we find a suitable symplectic basis of $\Gamma_g<\Gamma_{g,n}$; this  motivates our Lemmata in subsection \ref{sec:mcgaction}.  Note that it suffices for our problem to find such a suitable basis of $\Gamma_{g,n}$; a change of basis is effected by an element of the mapping class group of the surface $S_{g,n}$, so by pulling back the translation structure we obtain one with the desired holonomy  (\textit{c.f.} the discussion at the beginning of section 11.1). \\

\noindent In the remaining case of rational representations and all simple poles,  we can first reduce to the case of an \textit{integral} representation, that is, we can assume without loss of generality that $\textsf{Im}(\chi)=\Z$. We can also assume that the holonomy around each puncture is non-trivial; a puncture is said to be \textit{positive} if the holonomy around it is  translation by a positive integer, and \textit{negative} otherwise.  Our result for this final case provides a succinct necessary and sufficient criterion for realizing such a representation:

\begin{thma}[Integral holonomy] \label{main:thme}  Suppose  $\chi:\Gamma_{g,n} \to \mathbb{Z}$ is a non-trivial surjective representation. Let the holonomies around the positive punctures be given by the integer-tuple $\lambda \in \mathbb{Z}_+^k$ and the holonomies around the negative punctures be given by $-\mu$ where $\mu \in \mathbb{Z}_+^l$.  
Then there exists a translation structure on $S_{g,n}$ with holonomy $\chi$, with simple poles at the punctures and a set of $r$ zeros with prescribed orders $(d_1,d_2,\ldots, d_r)$ that satisfies the degree condition 
\begin{equation} \label{degcond}
\displaystyle\sum\limits_{i=1}^r d_i = 2g-2 + n
\end{equation}
 if and only if 
   \begin{equation}\label{gcomb0} 
    \sum\limits_{i=1}^k \lambda_i=\sum\limits_{j=1}^l \mu_j> \max\{d_1,d_2,\ldots, d_r\}.
\end{equation}
\end{thma}

\smallskip 

\noindent Our proof of Theorem \ref{main:thme} uses the recent work of \cite{GT}, who proved the above theorem for the case when $g=0$.\\

\noindent As a corollary of Theorems \ref{main:thmb}, \ref{main:thmc} and \ref{main:thme}, in Appendix \ref{pmdpp} we prove a refinement of Theorem \ref{mainthm} where one prescribes, in addition, the orders of the poles at the punctures. For surfaces with at least $3$ punctures, this can be stated as follows.

\begin{cora}\label{cor:mdpp}   Let $S_{g,n}$ be a surface of genus $g$ and $n\geq 3$ punctures. Let $p_1,p_2,\ldots p_n$ be positive integers assigned to each puncture,  and let $\mathcal{I} = \{ i \ \vert\ p_i = 1\}$.  Then a representation  $\chi:\Gamma_{g,n} \to \Bbb C$  is  the character of a meromorphic differential  $\omega$ on some  $X\in \mathcal{M}_{g,n}$  with a pole of order $p_i$ at the $i$-th puncture for $1\leq i\leq n$,  if and only if the prescribed holonomy around any simple pole is non-trivial, that is, $\chi(\gamma_i)  \neq 0$ for any $i\in \mathcal{I}$, where $\gamma_i$ denotes a loop around the $i$-th puncture. 
\end{cora}

\noindent A more complete statement that also handles the cases $n=1$ and $2$ is provided in Appendix \ref{pmdpp}.

\subsection{Connection with the Hurwitz Existence Problem}\label{chep} The integral holonomy case is related to the longstanding problem known as \emph{Hurwitz Existence Problem}, that we have repeatedly run into during the course of this work.  Let us give a brief overview of such a problem; the reader can consult \cite{PePe} for a summary of known progress. Let $f:S\longrightarrow \Sigma$ be a branched cover of degree $\deg(f)=d\ge 2$. Let $n$ be the number of branch-points (or branch values) in $\Sigma$. There are $n$ partitions of $d$ given by sets of integers $B_i=\big\{d_{ij} \big\}_{1 \leq j \leq m_i}$ for $1 \leq i \leq n$ that record the local degrees of $f$ at the preimages of the $n$ branch-points. Let us call the collection $\mathcal{B}=\{B_1,\dots,B_n\}$. Note that $\sum_{ B_i} d_{ij} = d \quad \text{for any} \quad 1 \leq i \leq n$. In addition, for each $i$ there exists some $j$ such that $d_{ij} \neq 1$, since by definition, one of the points in the preimage of any branch-point is ramified. Let $\widetilde{n} = m_1 + \cdots + m_n$. Then the Riemann-Hurwitz theorem states that 
\begin{equation}\label{rhcon}
\chi(S) - \widetilde{n} = d \cdot \big(\chi(\Sigma)-n\big).
\end{equation}
\noindent The branched covering $f$ yields a tuple $\mathcal{D}(f)=\big(S,\Sigma,d,n,\mathcal{B}\big)$ which we call its \emph{branch datum}. Conversely, two topological surfaces $S$ and $\Sigma$, a positive integer $d\ge2$ and $n$ partitions $\mathcal{B}=\{B_1,\dots,B_n\}$ of $d$ yield an \emph{abstract branch datum}, namely a string $\mathcal{D}=\big(S,\Sigma,d,n,\mathcal{B}\big)$ as above if equation \eqref{rhcon} holds. An abstract branch datum $\mathcal{D}$ is said to be \emph{realizable} if there exists a branched covering $f:S\longrightarrow \Sigma$ such that $\mathcal{D}=\mathcal{D}(f)$.  The long-standing Hurwitz problem asks which abstract datum are realizable:
\begin{qs}
When does there exist a branched covering $f:S\longrightarrow \Sigma$ such that $\mathcal{D}=\mathcal{D}(f)$? 
\end{qs}
\noindent  In the previous subsection we have already seen a glimpse of the relationship between our problem and the Hurwitz Existence Problem. Indeed, for a translation surface with trivial holonomy, the developing map yields a branched covering $f:S_g\longrightarrow \cp\cong\Bbb S^2$ which yields in turn a realizable branch datum $\mathcal{D}(f)$ by construction. On the other hand, any branched covering map $f:S_g\longrightarrow \Bbb S^2$ restricts to a mapping $\overline{f}:S_{g,k}\longrightarrow \C$, where $k=|f^{-1}(\infty)|$, and the standard Euclidean structure $(\C,\,dz)$ pulls back via $\overline{f}$ to a translation structure on $S_{g,k}$ with poles at the punctures and trivial holonomy. \\

\noindent Theorem \ref{main:thmb} then implies that certain special classes of abstract branch datum is realizable, proving certain cases of the Hurwitz Existence Problem:

\begin{cora}[Corollary \ref{consthmb2}] \label{main:cord}
Let $\mathcal{D}=\big(S_g, \Bbb S^2,  d, n,  \mathcal{B}\big)$ be an abstract branch datum 
where $\mathcal{B} = \{ B_1,B_2,\ldots B_n\}$
such that 
\begin{itemize}
    \item \eqref{rhcon} holds, that is, $\widetilde{n}= 2-2g+ d \cdot (n-2)$,  and 
    \item $B_i\ni d_{ij} = 1$ whenever $i \neq 1$ and $j \neq 1$.
\end{itemize}
Then $\mathcal{D}$ is realizable.
\end{cora}

\medskip

\noindent  The question of realizing a non-trivial integral representation $\chi$ as the holonomy of some translation structure with all simple poles can also be shown to be equivalent to the problem of determining whether certain abstract branch data are realizable or not:

\begin{propa} \label{main:thmf}  Suppose  $\chi:\Gamma_{g,n} \to \mathbb{Z}$ is a non-trivial surjective representation. Let the holonomies around the positive punctures be given by the integer-tuple $\lambda \in \mathbb{Z}_+^k$ and the holonomies around the negative punctures be given by $-\mu$ where $\mu \in \mathbb{Z}_+^l$.  
Then the following are equivalent:
\begin{enumerate}
    \item There is a translation structure on $S_{g,n}$ with holonomy $\chi$, with simple poles at the punctures and a set of $r$ zeros with prescribed orders $(d_1,d_2,\ldots, d_r)$ that satisfies the degree condition \eqref{degcond}. 
    \item There exists realizable branching data $\big(S_g, \Bbb S^2, d, m, \mathcal{B}\big)$ for some $m\geq 1$ where  $d = \sum\limits_{i=1}^k \lambda_i$ and  $\mathcal{B}$  is a collection of $m$  partitions of $d$
    satisfying the following:
    \begin{itemize}
        \item $\lambda$ is part of the collection unless $\lambda = (1,1,\ldots, 1)$, 
        \item $\mu$ is part of the collection unless $\mu = (1,1,\ldots, 1)$,  and
        \item in all other partitions, the only integers that are different from $1$ are exactly $\{d_1+1,\ldots, d_r +1\}$.
    \end{itemize}
    \end{enumerate}
\end{propa}

\noindent In Section \ref{sec:comb}, this statement is a consequence of Proposition \ref{prop:bc} and its Corollary \ref{cor:cond3}.  The proof starts with the observation that a translation surface with simple poles at the punctures and integral holonomy, is always the pull-back of the translation structure on the Euclidean cylinder $(\C/\Z,\, z^{-1}dz)$  by some holomorphic map, that extends to a  branched covering $S_g\longrightarrow \Bbb S^2$. Note that the collection $\mathcal{B}$ of the branch datum naturally attached to this latter map contains two special partitions $\lambda$ and $\mu$ comprising the local degrees around the points that map to the two punctures at either end of the cylinder. \\

\noindent  The following corollary is then immediate from Theorem \ref{main:thme} and Proposition \ref{main:thmf}, and can be thought of as solving special cases of the Hurwitz existence problem:

\begin{cora}\label{main:corf}
Given integer tuples $\lambda\in \mathbb{Z}^k_+$, $\mu \in \mathbb{Z}^l_+$ such that $d = \sum_{i=1}^k \lambda_i = \sum_{j=1}^l \mu_j$ and positive integers $\{d_1,d_2,\ldots, d_r\}$ satisfying $d > \max\{d_1,d_2,\ldots, d_r\}$, there exists realizable branching data $\big(S_g, \Bbb S^2, d, m, \mathcal{B}\big)$  where   $\mathcal{B}$  is some  collection of $m$ partitions of $d$  satisfying the properties listed  in (2) of Proposition \ref{main:thmf}. 
\end{cora}

\subsection{Additional remarks and further questions}
Away from the zeros and poles, a translation structure is a special case of a \textit{(complex) affine structure} on a surface, which comprises an atlas of charts to $\C$ that differ by {affine} maps (of the form $z\mapsto az+b$)  on their overlaps. This in turn is a special case of a \textit{(complex) projective structure},  comprising charts to $\cp$ differing by M\"{o}bius transformations. Note that if one includes the zeros and poles, a translation structure  can in fact be considered as a \textit{branched} affine or projective structure, since the developing map around a pole or a zero of the corresponding abelian differential  is of the form $z \mapsto z^d$, and the puncture can thus be filled in as a branch point for the charts.

\noindent  In \cite{GKM}, Gallo-Kapovich-Marden characterized the holonomy representations of marked projective structures on a \textit{closed} surface of genus $g\geq 2$.  In particular, they showed that any non-elementary representation of the fundamental group of the surface to  \text{PSL}$(2,\C)$ appears as the holonomy of some projective structure on a closed surface,  if one allows one branch-point of degree two. The holonomy representations of branched {affine} structures have been studied in \cite{Ghaz}; note that such a representation has image in the affine group $\text{Aff}(\C)$ and is necessarily elementary.  However, that paper is not concerned about the number and order of branch-points.  Thus, it remains to address the following question and determine analogues of the results of this paper:

\begin{ques}
When is a representation $\rho:\pi_1 S \to \text{Aff}(\C)$ the holonomy of a branched affine structure,  when the number and order of branch-points are prescribed? 
\end{ques} 

\noindent   It turns out that for any such branched structure the Schwarzian derivative of the developing map has a pole of order two at each branch-point, and thus can also be thought of as a  \textit{meromorphic projective structure}, whose holonomy representations have been recently studied in \cite{Gup}, \cite{FarGup}. \\

\noindent This paper also does not address the problem of understanding a ``holonomy fiber" beyond whether it is empty or not. In particular, one can ask:

\begin{ques}
For a fixed character $\chi:\Gamma_{g,n} \to \C$, what is the structure of the set of meromorphic differentials in $\Omega\mathcal{M}_{g,n}$ that have periods given by $\chi$? For example, is it connected? 
\end{ques}

\noindent  For \textit{marked} projective structures on a closed surface, any holonomy fiber, if non-empty, is necessarily discrete, and has been studied in \cite{Baba} and \cite{BabGup}. In the case of translation structures on a closed surface, these fibers could have positive dimension, and define the isoperiodic foliations mentioned earlier in this Introduction. It would be interesting to study their analogues for translation structures on punctured surfaces.\\

\noindent Finally, our work seems to have common ground with that of Gendron-Tahar  in \cite{GT}; indeed, we use a result from their work in subsection \ref{pfthme}. For surfaces of positive genera, the work is concerned with the \textit{residue map} recording the residue of the pole at each puncture, and not the entire period map. On the other hand, they are interested in the image of each \textit{connected component} of a strata, and one can ask the same question for the period map.

\subsection{Organization of the paper} We begin with an introductory part comprising section \ref{gts} where we define translation structures on punctured surfaces and their geometry, and section \ref{surg} where we define some topological surgeries used throughout the present work.  Part \ref{part1} is entirely devoted to prove Theorem \ref{mainthm}, via its geometric reformulation Theorem \ref{main:thma2}. More precisely, after some general comments about geometrizing representations in section \ref{prelgeom}, we shall prove Theorem \ref{main:thma2} for the punctured torus and $n$-punctured spheres in sections \ref{htc} and \ref{fsps} respectively. We finally prove Theorem \ref{main:thma2} in full generality in section \ref{gos}.  Part \ref{part2}, the second part of this paper is devoted to prove the other main results. More precisely, in sections \ref{sec:necsuftriv} and \ref{sec:trivholposgen} we shall provide a proof of Theorem \ref{main:thmb}, and in sections \ref{sec:nontrivholsphere} and \ref{sec:gennontrivhol} we prove Theorem \ref{main:thmc}, for $n$-punctured spheres and positive genus surfaces respectively for both theorems. Corollary \ref{main:cord} is proved at the end of section  \ref{sec:trivholposgen}. Finally, in section \ref{sec:comb} we shall prove  Theorem \ref{main:thme}, and discuss the connection with branched covers, leading to Proposition \ref{main:thmf} and its Corollary \ref{main:corf}. In Section \ref{sec:nontrivholsphere} we make use of a technical lemma whose proof is deferred to Appendix \ref{app:irrmultproof}. In Appendix \ref{pmdpp} we prove a refinement of Theorem \ref{mainthm} that implies Corollary \ref{cor:mdpp}. 

\subsection{Acknowledgments}  SG is grateful for the support of the Department of Science and Technology (DST) MATRICS Grant no. MT/2017/000706. GF would like to thank Ursula Hamenst\"adt for the helpful discussion and comments on this work. Most of this work was done while GF was affiliated with the IISc Bangalore. Despite the author could not fully enjoy his period there because of the ongoing pandemic, he is grateful to the Department of Mathematics and everyone who helped him. SC is grateful to Kishore Vaigyanik Protsahan Yojana (KVPY) for the fellowship and contingency grant. The authors are grateful to Quentin Gendron and Guillaume Tahar for their comments on an earlier version of this paper, and pointing out their results in \cite{GT}.

\section{Geometry of translation surfaces}\label{gts}

\subsection{Translation surfaces and other definitions}\label{gts1} In this chapter we provide the definitions of the main objects involved. Let us start by recalling what a translation structure is on a surface, and its geometric and complex-analytic definitions.

\begin{defn}[Translation structures on punctured surfaces]\label{tswop}
A translation structure on a surface $S_{g,n}$ is the datum of a complex structure on $S_{g,n}$, which results in a Riemann surface $X$, together with a holomorphic (abelian) differential $\omega$ on $X$.
 Any such holomorphic differential $\omega$ defines a flat metric with isolated singularities corresponding to zeros of $\omega$. In a neighborhood of a point $P$ which is not a zero for $\omega$, a local coordinate is defined as
\[ z(Q)=\int_P^Q \omega 
\] in which  $\omega=dz$, and the coordinates of two overlapping neighborhoods differ by a translation $z\mapsto z+c$ some $c\in\Bbb C$, namely a translation. Around a zero of order $k\ge1$, say $P$, the differential $\omega$ locally can be written as $\omega=z^kdz$ with respect to some coordinate $z$ and so the surface around $P$ is locally a simple branched $k+1$ covering over $\Bbb E^2$. Therefore, a translation structure can be geometrically seen as a branched $(\Bbb C, \Bbb E^2)$-structure, \emph{i.e.} the datum of a maximal atlas where the local charts in the Euclidean plane $\Bbb E^2$ have the form $z\longmapsto z^k$, $k\ge1$, and transitions maps are restrictions of translations in $\Bbb C$.  Throughout, a surface equipped with a translation structure will be called a \textit{translation surface}. 
\end{defn}

\noindent By the analytic continuation property, any local chart of a translation surface extends to a local immersion $\textsf{dev}:\widetilde{S_{g,n}}\longrightarrow \Bbb E^2$, called \emph{developing map}, where $\widetilde{S_{g,n}}$ is the universal cover of $S_{g,n}$. The developing map is equivariant with respect to a representation $\chi: \Gamma_{g,n} \longrightarrow \Bbb C$ called the \emph{holonomy} of the translation structure. (Recall that $\Gamma_{g,n} = H_1(S_{g,n};\,\Bbb Z)$ is the first homology group of the surface.) The following lemma establishes the twofold nature of a representation:

\begin{lem}[Twofold nature of a representation]\label{tnr}
A representation $\chi:H_1(S_{g,n};\,\Bbb Z)\longrightarrow \Bbb C$ is the period of some abelian differential $\omega\in\Omega(X)$ with respect to some complex structure $X$ on $S_{g,n}$ if and only if it is the holonomy of the translation structure on $S_{g,n}$ determined by $\omega$.
\end{lem}

\noindent Let $X$ be a complex structure on $S_{g,n}$ and let $\omega\in\Omega(X)$ be a meromorphic differential as introduced in section \ref{intro}. Note that some punctures are allowed to be removable singularities for $\omega$. We introduce the following more stringent definition where the punctures are required to be poles:

\begin{defn}[Translation Surfaces with poles]\label{tswp}
Let $\omega$ be a meromorphic differential on a compact Riemann $\overline{X}$. We define a \emph{translation surface with poles} to be the translation structure induced by $\omega$ on $X = \overline{X}\setminus\Sigma$, where $\Sigma$ is the set of poles of $\omega$. 
\end{defn}


\subsection{Local geometry around a puncture}\label{lgp} Let $(X,\omega)$ be a translation structure on a punctured surface $S_{g,n}$. In this section we provide a geometric description of the local geometry around a puncture. Given any local coordinate, say $(U,z)$, around a puncture $P$ the abelian differential $\omega$ can be written as $\omega=f(z)dz$ and we consider the Laurent series expansion of the function $f(z) $ -- see \eqref{eq:wform}. We discuss the possible cases separately.
\SetLabelAlign{center}{\null\hfill\textbf{#1}\hfill\null}
\begin{itemize}[leftmargin=1.75em, labelwidth=1.5em, align=center, itemsep=\parskip]
\item[\bf 1.] In the case the Laurent series has no negative terms, the puncture is a removable singularity for $\omega$ and the function $f$ extends holomorphically on it. We may further distinguish two possible subcases.\\
\begin{itemize}[leftmargin=2.25em, labelwidth=1.75em, align=center, itemsep=\parskip]
\item[\bf 1.1.] The leading coefficient $a_0\neq0$. In this case, the point $P$ is not a zero for the abelian differential $\omega$ and, therefore, any local chart $U\setminus\{P\}\longrightarrow \Bbb E^2$ extends to a local homeomorphism $U\longrightarrow \Bbb E^2$. The geometry around the puncture is not complete.\\
\item[\bf 1.2.] The leading coefficient $a_0=0$. The point $P$ is a zero of order $k$ for $\omega$, where $k\in\Bbb Z^+$ is smallest index such that $a_k\neq0$. In this case, any local chart $U\setminus\{P\}\longrightarrow \Bbb E^2$ extends to a simple branched $k+1$ covering $U\longrightarrow \Bbb E^2$. The geometry around the puncture is not complete.\\
\end{itemize}
In both cases, the puncture can be filled by gluing a judicious neighborhood of the vertex of a $2\pi(k+1)$ Euclidean cone - note that $k$ corresponds to the index of the least coefficient $a_i$ in the Laurent series different to zero. More precisely, such a neighborhood can be taken as the image of the local extended chart $U\longrightarrow \Bbb E^2$.\\
\item[\bf 2.] Suppose the function $f$ has a pole of order one at $P$. The neighborhood of $P$ is then an infinite cylinder of order one. Possibly after a suitable change of coordinates, the pole in local coordinate is given by $\omega=\frac1zdz$. Writing $z=e^\zeta$, the abelian differential with respect to $\zeta$ is given by $\omega=d\zeta$ and $\zeta$ describes an infinite cylinder on which the geometry is metrically complete. The contour integral along a curve enclosing $P$, and no other punctures, yields the residue of $f$ at $P$ which is a non-zero complex number. We may notice that, the residue is proportional to the holonomy of the curve enclosing $P$. Indeed given a curve $\gamma$ enclosing the pole, the residue theorem implies that
\[ \gamma\longmapsto\int_\gamma\omega=2\pi i\, \text{Res}\big(f,\,P\big).
\]
\item[\bf 3.] We finally consider the flat geometry around poles of order at least two. Once again, we need to distinguish two subcases depending on whether the residue of the pole is null or not.\\
\begin{itemize}[leftmargin=2.25em, labelwidth=1.75em, align=center, itemsep=\parskip]
\item[\bf 3.1.] Let us consider first the case of $P$ is a pole of order $k+2\ge2$ and null residue. Let $U$ be an open neighborhood of $P$ and choose a local coordinate such that $\omega=\frac{dz}{z^{k+2}}$ around the puncture. Notice that $U$ is biholomorphic to the punctured disk $\Bbb D^*$. By applying the change of coordinate $\zeta=\frac1z$, the differential $\omega$ has a zero of order $k$ and the local chart $(U,\,\zeta) \longrightarrow \Bbb E^2$ is $k+1-$fold covering over the puncture disk. Equivalently, the coordinate neighborhood $(U,\,\zeta)$ is biholomorphic to a neighborhood of the vertex of the Euclidean cone of angle $2(k+1)\pi$ to which the conical singularity has been removed. We cannot the deduce the geometry around the pole from this model because the change of coordinate $\zeta=\frac1z$ is not a translation and so the geometry has been altered. However it gives a glimpse of what the geometry should be. The mapping $z\mapsto\frac1z$ is an inversion and hence the geometry around a pole is that of an Euclidean cone of angle $2(k+1)\pi$ to which a compact neighborhood of the conical singularity has been removed.\\
\item[\bf 3.2.] We finally consider the case of $P$ is a pole of order $k+2\ge2$ and non-zero residue $R$. Before going to describe the geometry around the pole $P$, we begin by introducing the following model of translation surface. Let $(\Bbb C,\,dw)$ be the complex plane equipped with the holomorphic $1$-form $dw$. The exponential mapping $w\longmapsto\exp{\big(\frac{k+1}{R}w\big)}$ yields an infinite-sheet covering onto a cylinder $C$.\\
\noindent By setting $\zeta=\exp{\big(\frac{k+1}{R}w\big)}$ we obtain
\begin{equation}
dw=\frac{R}{(k+1)\,\zeta}\,d\zeta.
\end{equation}
\noindent The cylinder $C$ is topologically a twice-punctured sphere endowed with a complex structure and an abelian differential that makes it a translation surface with two simple poles and no singularities because $dw$ has no zeros. The geometry around both poles is the one described above. \\
\noindent Let $Q$ be any point in $C$ and consider a straight line, say $l$, starting from $Q$ towards one of the ends of the cylinder. Such a line develops along an infinite ray, say $r$, on the complex plane starting from a fixed developed image of $Q$. Cutting $C$ along $l$ and then gluing a copy of $\Bbb C\setminus r$ by using the developing map, we obtain a surface still homeomorphic to a cylinder but equipped with a new translation structure $(C,\eta)$. Let us consider $C$ as a $2-$punctured sphere $\Bbb S^2\setminus\{P_1,\,P_2\}$. In a neighborhood of the uncut end the geometry remained unchanged and the puncture, say $P_1$, is a simple pole with residue $\frac{R}{k+1}$ for $\eta$. The other end, instead, is no longer cylindrical as the geometry is changed because we have glued a whole copy of $\Bbb E^2$! Such a surgery has also introduced a singular point of angle $4\pi$, that means that $\eta$ has a zero of order one. As the number of zeros (counted with multiplicity) minus the number of poles (counted with multiplicity) has to be equal to $-2$, we immediately deduce that $P_2$ is a pole of order $-2$ for $\eta$ and the residue theorem for Riemann surfaces implies that the residue around $P_2$ is not zero and equal to $\frac{R}{k+1}$. Let $V$ be a neighborhood of $P_2$, by choosing a judicious coordinate $z$ around $P_2$, the form $\eta$ can be written as 
\begin{equation} \eta=\Big(\,\frac{1}{z^2}+\frac{R}{(k+1)z}\,\Big)dz
\end{equation}
To determine the geometry around a pole of order greater than one with non-zero residue, let $U$ be an open neighborhood of $P$ and choose a local coordinate such that, possibly after re-scaling and rotations, $\omega$ is given by
\begin{equation}\Big(\,\frac{\,k+1\,}{z^{k+2}}+\frac{R}{z}\,\Big)dz
\end{equation} around the puncture. The form $\omega$ is the pullback of the form $\eta$ by the mapping $z\longmapsto z^{k+1}$. Therefore, the geometry around $P$ is that one of a cylinder with holonomy $R$ to which $k$ copies of the Euclidean plane $\Bbb E^2$ have been glued.\\
\end{itemize}
\noindent We can notice that, in both cases, the flat geometry around poles of order greater than one is metrically complete.\\
\end{itemize}

\noindent One consequence of the discussion in this section can be summarized as follows:

\begin{lem}\label{metlem}
Let $X$ be a punctured Riemann surface and let $\omega\in\Omega(X)$ be a meromorphic abelian differential. The geometry around a puncture $P$ is metrically complete if and only if $\omega$ does not extend holomorphically over any puncture.
\end{lem}

\section{Surgeries on translation surfaces} \label{surg}
\noindent We introduce in this section some surgeries we shall use in the sequel, which consists of  cutting translation surfaces along several geodesic segments and gluing them back along those segments in order to get new translation structures. Different gluings will provide different translation structures.

\subsection{Slit constructions}\label{sec:seqslit} We describe a procedure to glue two or more translation surfaces together. Consider $n$ translation surfaces $(X_1,\omega_1), \ldots, (X_{n},\omega_n)$ and oriented geodesic line segments $l_i \subset (X_i,\omega_i)$ with each $l_i$ having the same developed image $c \in \mathbb{C}\setminus\{0\}$. View the segments as vectors and label the base of these vectors as $P_i$ and the tip as $Q_i$. Making slits in the surfaces along these segments, we obtain two sides for each slit. We label the left side $l_i^+$ and the right side $l_i^-$ as in Figure \ref{fig:seqslit} below. Identifying $l_i^-$ with $l_{i+1}^+$, with indices being considered modulo $n$, we obtain a translation surface $(X,\omega)$. The genus of this latter is the sum of the genera of all the $(X_i,\omega_i)$.  We shall also use a modification of this procedure where we have multiple slits $l_1, \ldots, l_n$ on a single surface $(X_0,\omega_0)$. In this case, the effect of the slit construction is to add $n-1$ handles so that the genus of the resulting surface, say $(X,\omega)$, is $n-1$ higher than the genus of $(X_0,\omega_0)$.  In the surface $(X,\omega)$, all points $P_i$ get identified to the same point $P$, and all points $Q_i$ get identified to the same point $Q$. If none of the $P_i$'s are singular points in the respective translation surfaces $(X_i,\omega_i)$, then the point $P$ is a singular point in the translation surface $(X,\omega)$ with angle $2n\pi$. More generally, if the points $P_i$ are (possibly) singular points with angle $2\pi\,m_i$ in the singular flat metric of $(X_i,\omega_i)$, then the point $P$ is a singular point with magnitude $2\pi\big(m_1+\cdots+m_n\big)$. In the same fashion we can determine the singularity at the point $Q$.

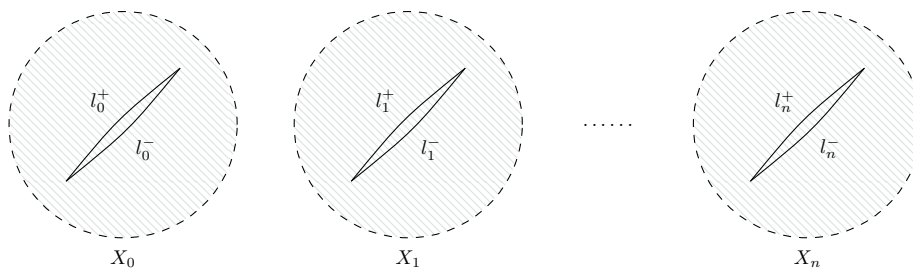
\begin{figure}[!h]
\centering
\begin{tikzpicture}[scale=0.75, every node/.style={scale=0.75}]
\definecolor{pallido}{RGB}{221,227,227}
\draw [dashed, black, pattern=north west lines, pattern color=pallido] (0,0) circle (2);
\draw (-1,-1) .. controls (0.1, -0.1) .. (1,1);
\draw (-1,-1) .. controls (-0.1, 0.1) .. (1,1);
\node[above left] at (-0.1, 0.1) {$l_0^+$};
\node[below right] at (0.1, -0.1) {$l_0^-$};
\node[below] at (0,-2.1) {$X_0$};
\draw [dashed, black, pattern=north west lines, pattern color=pallido] (5,0) circle (2);
\draw (4,-1) .. controls (5.1, -0.1) .. (6,1);
\draw (4,-1) .. controls (4.9, 0.1) .. (6,1);
\node[above left] at (4.9, 0.1) {$l_1^+$};
\node[below right] at (5.1, -0.1) {$l_1^-$}; 
\node at (8.5,0) {$\ldots \ldots$};
\node[below] at (5,-2.1) {$X_1$};
\draw [dashed, black, pattern=north west lines, pattern color=pallido] (12,0) circle (2);
\draw (11,-1) .. controls (12.1, -0.1) .. (13,1);
\draw (11,-1) .. controls (11.9, 0.1) .. (13,1);
\node[above left] at (11.9, 0.1) {$l_n^+$};
\node[below right] at (12.1, -0.1) {$l_n^-$}; 
\node[below] at (12,-2.1) {$X_n$};
\end{tikzpicture}
\caption{Sequential slit construction}
\label{fig:seqslit}
\end{figure}

\begin{rmk} Each slit $\overline{P_i\,Q_i}$ turns into a saddle connection joining $P$ and $Q$ in the resulting surface, so we have $n$ saddle connections between $P$ and $Q$ arising as a result of this process. If the angle at $P$ is $2\pi\,(d+1)$ we can locally find $d+1$ rays starting from $P$ that have the same developed image as one of the $n$ saddle connections joining $P$ and $Q$. Out of these $d+1$ rays, $n$ rays give us a saddle connection to $Q$.
\end{rmk}

\begin{rmk}\label{slitonepunct} The slit construction just described also applies when one, possibly both, of the extremal points is a puncture, namely a pole for the abelian differential. In this case, the segment $l$ may develop on an infinite ray on the complex plane. Suppose $(X_1,\omega_1)$ and $(X_2,\omega_2)$ are two translation surfaces, each with at least one pole. Let  $l_i \subset X_i$ for $i=1,2$ be an embedded straight-line ray that starts from a point $P_i$ and ends in a pole. Furthermore, assume that $l_1$ and $l_2$  develop onto infinite rays on $\C$ that are parallel. Similarly to the above, we can define a translation surface $(X,\omega)$ as follows: slit each ray $l_i$ and denote the resulting sides by $l_i^+$ and $l_i^-$; then identify $l_1^+$ with $l_2^-$ and $l_1^-$ and $l_2^+$ by a translation. The starting points of the rays, being identified, define a branch point on $(X,\omega)$ with magnitude that is the sum of the corresponding angles on $(X_1,\omega_1)$ and $(X_2,\omega_2)$, and the other endpoints (at infinity) are identified to a higher order pole with order given by the sum of the individual orders.
\end{rmk}

\subsection{Sequential slit construction with handle construction} \label{sec:seqslithandle} The sequential slit construction just described extends to a \emph{sequential slit construction with handle construction} by introducing one or more handle in the construction process. We shall see later on in section \ref{htc} procedures for adding a handle with prescribed holonomy, i.e.\ with a basis of $\Gamma_{1,1}$ with prescribed periods.  Here, we pick one set of parallel oriented sides of the parallelogram (which could possibly be degenerate). We label the side that has the surface on its left (before identification) as $l^+$ and the other side as $l^-$. We then identify $l_n^-$ with $l^+$ and $l^-$ with $l_0^+$ in the sequential slit construction with the other identifications of $l_i^-$ and $l_i^+$ remaining the same. The other pair of parallel sides of the parallelogram are then identified.

\begin{figure}[!h] \label{fig:seqslithandle}
\centering
\begin{tikzpicture}[scale=0.65, every node/.style={scale=0.75}]
\definecolor{pallido}{RGB}{221,227,227}
\draw[black, pattern=north west lines, pattern color=pallido] (-3,1) -- (-5, -1) -- (-9,-1) -- (-7, 1) -- (-3,1);
\node[above left] at (-8,0) {$l^-$};
\node[below right] at (-4,0) {$l^+$};
\draw[dashed, black, pattern=north west lines, pattern color=pallido] (0,0) circle (2);
\draw (-1,-1) .. controls (0.1, -0.1) .. (1,1);
\draw (-1,-1) .. controls (-0.1, 0.1) .. (1,1);
\node[above left] at (-0.1, 0.1) {$l_0^+$};
\node[below right] at (0.1, -0.1) {$l_0^-$}; 
\draw [dashed, black, pattern=north west lines, pattern color=pallido] (5,0) circle (2);
\draw (4,-1) .. controls (5.1, -0.1) .. (6,1);
\draw (4,-1) .. controls (4.9, 0.1) .. (6,1);
\node[above left] at (4.9, 0.1) {$l_1^+$};
\node[below right] at (5.1, -0.1) {$l_1^-$}; 
\node at (8.5,0) {$\ldots \ldots$};
\draw [dashed, black, pattern=north west lines, pattern color=pallido] (12,0) circle (2);
\draw (11,-1) .. controls (12.1, -0.1) .. (13,1);
\draw (11,-1) .. controls (11.9, 0.1) .. (13,1);
\node[above left] at (11.9, 0.1) {$l_n^+$};
\node[below right] at (12.1, -0.1) {$l_n^-$}; 
\end{tikzpicture}
\caption{Slit construction with one handle.}
\end{figure}
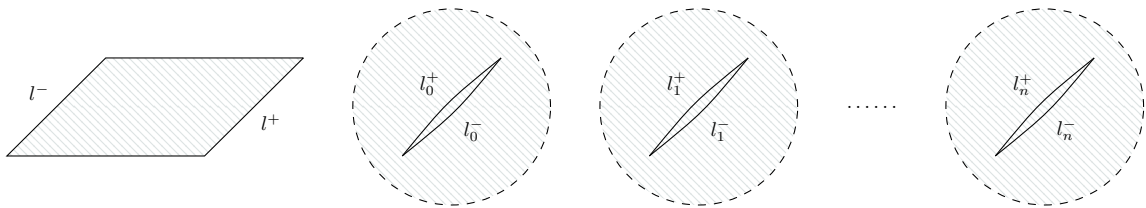

\begin{figure}[!h] \label{fig:seqslithandle}
\centering
\begin{tikzpicture}[scale=0.65, every node/.style={scale=0.75}]
\definecolor{pallido}{RGB}{221,227,227}
\begin{scope}[shift = {(0,-6)}]
\draw[dashed, black, pattern=north west lines, pattern color=pallido] (-6,0) circle (3);
\fill[white] (-4,1) -- (-6,-1) -- (-8,-1) -- (-6, 1) -- (-4,1);
\draw (-4,1) -- (-6,-1) -- (-8,-1) -- (-6, 1) -- (-4,1);
\node[above left] at (-7,0) {$l^+$};
\node[below right] at (-5,0) {$l^-$};
\draw[dashed, black, pattern=north west lines, pattern color=pallido] (0,0) circle (2);
\draw (-1,-1) .. controls (0.1, -0.1) .. (1,1);
\draw (-1,-1) .. controls (-0.1, 0.1) .. (1,1);
\node[above left] at (-0.1, 0.1) {$l_0^+$};
\node[below right] at (0.1, -0.1) {$l_0^-$}; 
\draw [dashed, black, pattern=north west lines, pattern color=pallido] (5,0) circle (2);
\draw (4,-1) .. controls (5.1, -0.1) .. (6,1);
\draw (4,-1) .. controls (4.9, 0.1) .. (6,1);
\node[above left] at (4.9, 0.1) {$l_1^+$};
\node[below right] at (5.1, -0.1) {$l_1^-$}; 
\node at (8.5,0) {$\ldots \ldots$};
\draw [dashed, black, pattern=north west lines, pattern color=pallido] (12,0) circle (2);
\draw (11,-1) .. controls (12.1, -0.1) .. (13,1);
\draw (11,-1) .. controls (11.9, 0.1) .. (13,1);
\node[above left] at (11.9, 0.1) {$l_n^+$};
\node[below right] at (12.1, -0.1) {$l_n^-$}; 
\end{scope}
\end{tikzpicture}
\caption{Slit construction with handles. The closure of the exterior of any parallelogram in $\Bbb C$ is still a parallelogram in the Riemann sphere $\cp$. Such a parallelogram is admissible in our construction. }
\end{figure}
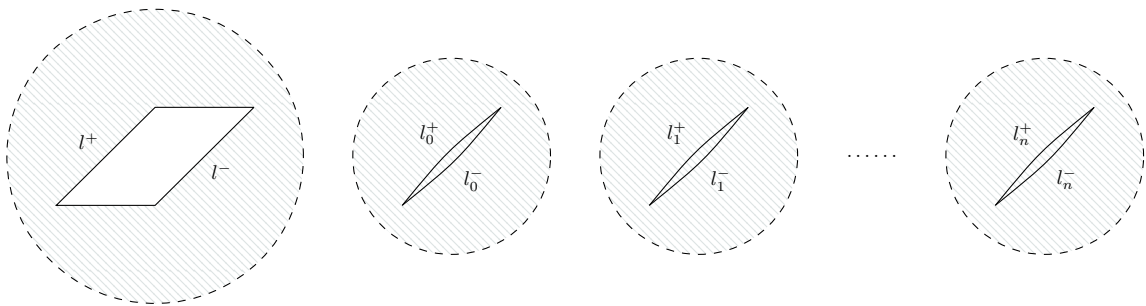

\subsection{Splitting a zero} \label{sec:zerosplit} The surgery we are going to introduce is a procedure of splitting a zero of an abelian differential. In the context of branched projective structures, such a surgery is commonly known as \emph{movement of branched points} and it has been originally introduced by Tan in \cite[Chapter 6]{TA} for showing the existence of a complex one-dimensional continuous family of deformations of a given structure. For translation surfaces, the same kind of surgery has been subsequently introduced by Kontsevich-Zorich as \emph{breaking a zero}, see \cite[Section 4.2]{KZ}. We shall adopt this latter point of view. What makes this surgery important for us is its fundamental property of preserving the topology of the underlying surface even if the overall geometry is changed. In particular, the new translation surface we obtain once the surgery is performed has the same holonomy of the original one.\\

\noindent Splitting a zero is a procedure that takes place at the $R$-neighbourhood of some zero of order $d$ of the differential on which it looks like the pull-back of the form $dz$ via a branched covering. The differential is then modified by a surgery inside this neighbourhood. Once this surgery is performed we obtain a new translation structure with two zeros of order $d_1$ and $d_2$ in place of the zero of order $d$, with $d_1+d_2 = d$. Furthermore, the translation structure remains unchanged outside the $R$-neighbourhood of the zero of order $d$ that we have considered earlier. We first view the $R$-neighbourhood of a zero of order $d$ as $d+1$ upper half discs and $d+1$ lower half discs, each one of radius $R$, having the diameters identified in a specified way. We label the left half of the diameter of the $i^{th}$ upper half disc as $ul_i$ and the right half as $ur_i$. For lower half discs, the corresponding labels are $ll_i$ and $lr_i$. We now identify $ul_i$ with $ll_i$ and $lr_i$ with $ur_{i+1}$ with the indices being considered modulo $d+1$. An illustration of the labelling and identification for a zero of order 4 is given in Figure \ref{fig:zerolocmodel}. All the centres of the half discs are identified and this point is the zero that we are looking at.

\begin{figure}[h]
    \centering
    \begin{tikzpicture}[scale=0.70, every node/.style={scale=0.8}]
    \definecolor{pallido}{RGB}{221,227,227}
    \foreach \x [evaluate=\x as \coord using 4 + 5*\x] in {0, 1, ..., 4} 
    {
    \draw [pattern=north west lines, pattern color=pallido] (\coord,0) arc [start angle = 0,end angle = 180,radius = 2];
    \draw [pattern=north west lines, pattern color=pallido] (\coord,-1) arc [start angle = 0,end angle = -180,radius = 2];
    }
    \foreach \x [evaluate=\x as \leftend using 5*\x] [evaluate=\x as \rightend using 2 + 5*\x] in {0, 1, ..., 4} 
    {
    \draw [thick] (\leftend, 0) -- (\rightend, 0);
    \draw [thick] (\leftend, -1) -- (\rightend, -1);
    }
    \foreach \x [evaluate=\x as \leftlabel using 1 + 5*\x] [evaluate=\x as \rightlabel using 3 + 5*\x] in {0, 1, ..., 4} 
    {
    \node [below] at (\leftlabel, 0) {$ul_{\x}$};
    \node [below] at (\rightlabel, 0) {$ur_{\x}$};
    \node [above] at (\leftlabel, -1) {$ll_{\x}$};
    \node [above] at (\rightlabel, -1) {$lr_{\x}$};
    }
    
    \foreach \botindex / \topindex / \colr [evaluate=\botindex as \botleftend using 5*\botindex + 2] [evaluate=\botindex as \botrightend using 5*\botindex + 4] [evaluate=\topindex as \topleftend using 5*\topindex + 2] [evaluate=\topindex as \toprightend using 5*\topindex + 4] in {0/1/blue, 1/2/pink, 2/3/orange, 3/4/green, 4/0/violet} 
    {
    \draw [thick, \colr] (\topleftend, 0) -- (\toprightend, 0);
    \draw [thick, \colr] (\botleftend, -1) -- (\botrightend, -1);
    \fill (\topleftend, 0) circle (1.5pt);
    \fill (\botleftend, -1) circle (1.5pt);
    }
    \end{tikzpicture}
    \caption{An $R$ neigbourhood of a zero of order 4}
    \label{fig:zerolocmodel}
\end{figure}
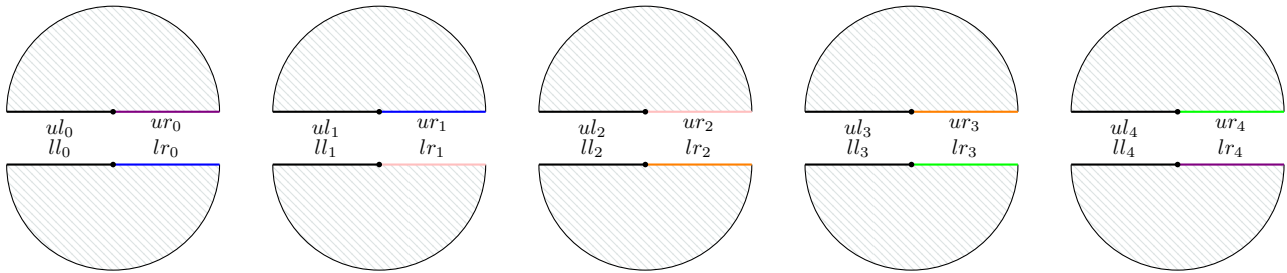

\noindent Now, to split this zero into zeros of order $d_1$ and $d_2$, we modify the labelling on the upper half disc indexed by $0$, the lower half disc indexed by $d_1$, and all upper and lower half discs with index more than $d_1$ accordingly. The modified labelling is shown below in Figure \ref{fig:splitlocmodel} for the case of splitting the zero of order 4 into two zeros of order 2. 

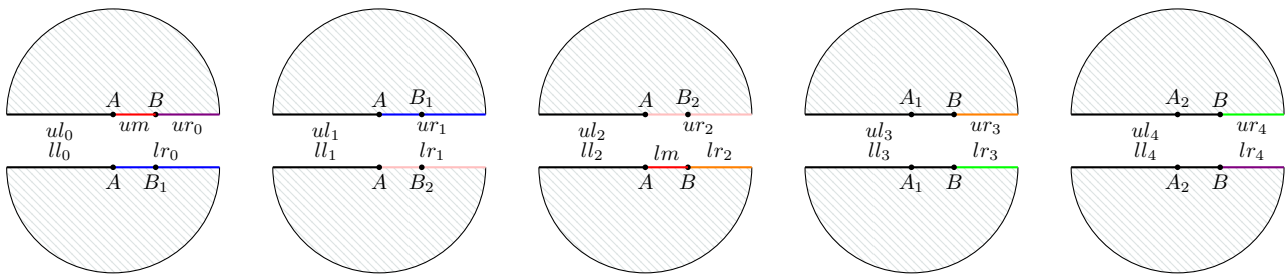
\begin{figure}[h]
    \centering
    \begin{tikzpicture}[scale=0.7, every node/.style={scale=0.8}]
    \definecolor{pallido}{RGB}{221,227,227}
    \foreach \x [evaluate=\x as \coord using 4 + 5*\x] in {0, 1, ..., 4} 
    {
    \draw [pattern=north west lines, pattern color=pallido] (\coord,0) arc [start angle = 0,end angle = 180,radius = 2];
    \draw [pattern=north west lines, pattern color=pallido] (\coord,-1) arc [start angle = 0,end angle = -180,radius = 2];
    }
    \foreach \x [evaluate=\x as \leftend using 5*\x] [evaluate=\x as \rightend using 2 + 5*\x] in {0, 1, 2} 
    {
    \draw [thick] (\leftend, 0) -- (\rightend, 0);
    \draw [thick] (\leftend, -1) -- (\rightend, -1);
    \node[above] at (\rightend, 0) {$A$};
    \node[below] at (\rightend, -1) {$A$};
    }
    \foreach \x / \y [evaluate=\x as \leftend using 5*\x] [evaluate=\x as \rightend using 2.8 + 5*\x] [evaluate=\x as \splpt using 2 + 5*\x] in {3/1, 4/2} 
    {
    \draw [thick] (\leftend, 0) -- (\rightend, 0);
    \draw [thick] (\leftend, -1) -- (\rightend, -1);
    \fill (\splpt, 0) circle (1.5pt);
    \fill (\splpt, -1) circle (1.5pt);
    \node[above] at (\splpt, 0) {$A_{\y}$};
    \node[below] at (\splpt, -1) {$A_{\y}$};
    }
    
    \foreach \botindex / \topindex / \colr [evaluate=\botindex as \botleftend using 5*\botindex + 2] [evaluate=\botindex as \botrightend using 5*\botindex + 4] [evaluate=\topindex as \topleftend using 5*\topindex + 2] [evaluate=\topindex as \toprightend using 5*\topindex + 4] [evaluate=\topindex as \topsplpt using 5*\topindex + 2.8] [evaluate=\botindex as \botsplpt using 5*\botindex + 2.8] in {0/1/blue, 1/2/pink} 
    {
    \draw [thick, \colr] (\topleftend, 0) -- (\toprightend, 0);
    \draw [thick, \colr] (\botleftend, -1) -- (\botrightend, -1);
    \fill (\topleftend, 0) circle (1.5pt);
    \fill (\botleftend, -1) circle (1.5pt);
    \fill (\topsplpt, 0) circle (1.5pt);
    \fill (\botsplpt, -1) circle (1.5pt);
    \node[above] at (\topsplpt, 0) {$B_{\topindex}$};
    \node[below] at (\botsplpt, -1) {$B_{\topindex}$};
    }
    
    \foreach \botindex / \topindex / \colr [evaluate=\botindex as \botleftend using 5*\botindex + 2.8] [evaluate=\botindex as \botrightend using 5*\botindex + 4] [evaluate=\topindex as \topleftend using 5*\topindex + 2.8] [evaluate=\topindex as \toprightend using 5*\topindex + 4] in {2/3/orange, 3/4/green, 4/0/violet} 
    {
    \draw [thick, \colr] (\topleftend, 0) -- (\toprightend, 0);
    \draw [thick, \colr] (\botleftend, -1) -- (\botrightend, -1);
    \fill (\topleftend, 0) circle (1.5pt);
    \fill (\botleftend, -1) circle (1.5pt);
    \node[above] at (\topleftend, 0) {$B$};
    \node[below] at (\botleftend, -1) {$B$};
    }
    
    \foreach \x [evaluate=\x as \leftlabel using 1 + 5*\x] in {0, 1, 2} 
    {
    \node [below] at (\leftlabel, 0) {$ul_{\x}$};
    \node [above] at (\leftlabel, -1) {$ll_{\x}$};
    }
    
    \foreach \x [evaluate=\x as \leftlabel using 1.4 + 5*\x] in {3,4} 
    {
    \node [below] at (\leftlabel, 0) {$ul_{\x}$};
    \node [above] at (\leftlabel, -1) {$ll_{\x}$};
    }
    
    \foreach \botindex / \topindex [evaluate=\botindex as \botlabel using 3+ 5*\botindex] [evaluate=\topindex as \toplabel using 3+ 5*\topindex] in {0/1, 1/2} 
    {
    \node [below] at (\toplabel, 0) {$ur_{\topindex}$};
    \node [above] at (\botlabel, -1) {$lr_{\botindex}$};
    }
    
    \foreach \botindex / \topindex [evaluate=\botindex as \botlabel using 3.4+ 5*\botindex] [evaluate=\topindex as \toplabel using 3.4+ 5*\topindex] in {2/3, 3/4, 4/0} 
    {
    \node [below] at (\toplabel, 0) {$ur_{\topindex}$};
    \node [above] at (\botlabel, -1) {$lr_{\botindex}$};
    }
    
    \draw [thick, red] (2,0) -- (2.8,0);
    \fill (2,0) circle (1.5pt);
    \draw [thick, red] (12,-1) -- (12.8,-1);
    \fill (12,-1) circle (1.5pt);
    \node [below] at (2.4, 0) {$um$};
    \node [above] at (12.4, -1) {$lm$};
    \end{tikzpicture}
    \caption{Modified labellings to split a zero of order 4 into two zeros of order 2}
    \label{fig:splitlocmodel}
\end{figure}

\noindent We now identify $ul_i$ with $ll_i$ and $lr_i$ with $ur_{i+1}$ as before with the added identification of $um$ with $lm$. This identification gives two singular points $A$ and $B$, $A$ is a zero of the differential of order $d_1$ and $B$ is a zero of order $d_2$. We also get a geodesic line segment joining $A$ and $B$.  Given $c \in \mathbb{C}\setminus\{0\}$ with length less than $2R$, we can perform the surgery in such a way that the line segment joining $A$ and $B$ is $c$. However, we shall work with only $c$ which have length less than $R$. It is also clear that the modification of the translation structure is only local. This procedure can be repeated multiple times to obtain zeros of orders $d_1, \dots, d_n$ from a single zero of order $d_1 + \cdots + d_n$. One way to do this would be to first split the zero into two zeros of order $d_1$ and $d_2 + \cdots + d_n$, and then split the latter to get a zero of order $d_2$ and so on. This is the procedure that we will use. We may further observe some additional features of the construction that we shall use later. In figure, \ref{fig:splitlocmodel}, there are two geodesic line segments $AB_1$ and $AB_2$ that have the same developed image as the segment $AB$. The same goes for $B$, there are two geodesic line segments $BA_1$ and $BA_2$ that have the same developed image as $BA$. In general, when we split a zero of order $d_1+d_2$ into zeros of order $d_1$ and $d_2$, we obtain $d_1$ geodesic line segments from the zero of order $d_1$ and $d_2$ geodesic line segments from the zero of order $d_2$ that have the same developed image as the saddle connection between the two zeros.\\ 

\noindent We shall  strongly rely on this procedure of splitting a zero in part \ref{part2} of this paper.

\part{Translation structures with prescribed holonomy}\label{part1}

\noindent In the first part of the paper, we shall realize any representation $\chi:\Gamma_{g,n}\longrightarrow \C$ as the holonomy of some translation structure on $S_{g,n}$, that is, we shall prove Theorem \ref{main:thma2}. The crucial point here is that we shall not care about to prescribe any data of $\omega$, namely the orders of zeros and orders of poles; we shall consider this more delicate problem in the second part of the paper.\\

\noindent We briefly describe the geometrizing process we use to prove Theorem \ref{main:thma2} and which will be fully developed later in section \ref{gos}. As mentioned in the Introduction, the given representation $\chi:\Gamma_{g,n}\to\C$ yields two further representations $\chi_g$ and $\chi_n$. By  Proposition \ref{propb},  the latter representation always appears as the holonomy of a translation structure on $S_{0,n}$, so we can focus the attention on $\chi_g$. If the representation satisfies Haupt's conditions, then it appears as the holonomy of some translation structure on $S_g$ by a direct application of the Haupt's Theorem. The translation structures on $S_g$ and $S_{0,n}$ can then be glued together along a geodesic slit and the resulting surface, homeomorphic to $S_{g,n}$, carries a translation structure having $\chi$ as the holonomy.  In the case $\chi_g$  does not satisfy Haupt's conditions, there are two cases:  it either has positive volume but its image is a lattice $\Lambda$  in $\C$, or  its volume is non-positive. In the first case, we shall glue the Euclidean torus $\C/\Lambda$ with the translation structure on $S_{0,n}$ having holonomy $\chi_n$. The resulting surface, homeomorphic to $S_{1,n}$, carries a translation structure and any branch-point can be used for attaching the remaining $g-1$ handles in the way we described in subsection \ref{trihan}.

\noindent The most interesting cases are given by representations with non-positive volume. Let $\{\alpha_1,\beta_1,\dots,\alpha_g,\beta_g\}$ be a symplectic basis of $\Gamma_g$, we can reorder the pairs $\{\alpha_i,\beta_i\}$ such that the first $k$ handles have negative volume, the following $h$ handles have null-volume and the (possibly) remaining $\overline{g}=g-(k+h)$ handles have positive volume. The reader can notice that, whenever $\chi_g$ has non-positive volume, then $k+h\ge1$. We shall apply Proposition \ref{proppt} to any representation $\langle\alpha_i,\beta_i\rangle\hookrightarrow \Gamma_g\to\C$, for any $i=1,\dots,k+h$, in order to have $k+h$ handles each one equipped with a translation structure having one (possibly two) zeros and one pole of order two. All these structures, properly glued together define a translation structure on a surface of type $S_{k+h,1}$. To this latter surface, we finally glue in a translation structure on $S_{0,n}$ with holonomy $\chi_n$ and the remaining $\overline{g}$ handles with positive volume. The final surface, homeomorphic to $S_{g,n}$ carries a translation structure having holonomy $\chi$. 

\section{Some preliminaries} \label{prelgeom}
\noindent Let $\gamma$ be a simple closed separating curve in $S_{g,n}$ that bounds a subsurface homeomorphic to $S_{g,1}$. There is a natural embedding $S_{g,1}\hookrightarrow S_{g,n}$ and hence an injective homomorphism $\imath:\Gamma_{g,1}\to \Gamma_{g,n}$. The mapping $\imath$ post-composed with any given representation $\chi:\Gamma_{g,n}\to\Bbb C$ yields a representation $\chi_g:\Gamma_{g,1}\to\Bbb C$. Since any representation satisfies the property of mapping simple closed separating curves to the identity element in $\Bbb C$, the equation $\chi_g(\gamma)=0$ holds and $\chi_g$ boils down to a representation $\chi_g:\Gamma_g\to\Bbb C$, where $\Gamma_g=H_1(S_g;\Bbb Z)$. We will see later the role that this latter representation will play in our proof of Theorem \ref{mainthm}. We focus on some considerations. As noticed in the introduction, given a closed surface $S_g$ and a representation $\chi_g:\Gamma_g\to\Bbb C$, there are two topological obstructions for $\chi_g$ to be the character of an abelian differential. The first obstruction is given by the volume of $\chi_g$ defined as the quantity
\begin{equation}
    \textsf{vol}(\chi_g)=\displaystyle\sum_{i=1}^g \Im\big(\overline{\chi(\alpha_i)}\chi(\beta_i)\big), 
\end{equation} 
 where $\{\alpha_1,\beta_1,\dots,\alpha_g,\beta_g\}$ is any standard symplectic basis of $\Gamma_g$. The second obstruction applies only to surfaces of genus at least $2$. If we assume the image of $\chi_g:\Gamma_g\to\C$ to be a lattice $\Lambda$ in $\C$, then we need to require that $\textsf{vol}(\chi_g)\ge 2\text{Area}(\C/\Lambda)$. This is equivalent to require that the mapping $f:S_g\to \C/\Lambda$, induced by the homomorphism $f_*:H_1(S_g;\Z)\to H_1(\C/\Lambda;\Bbb Z)$, has degree at least two. In \cite[Proposition 2.7]{CDF2}, Calsamiglia-Deroin-Francaviglia provided the following characterisation.

\begin{prop}\label{cdfprop}
Assume $\chi_g:\Gamma_g\to\C$ is a homomorphism such that $\textsf{\emph{vol}}(\chi_g)>0$ and $\chi(\Gamma_g)=\Lambda$ is a lattice in $\C$. The mapping $f:S_g\to \C/\Lambda$ factors by a collapse of $g-1$ handles iff $\textsf{\emph{vol}}(\chi_g)=\text{\emph{Area}}(\C/\Lambda)$ iff the second obstruction fails to be satisfied.
\end{prop}

\noindent Suppose a character $\chi:\Gamma_{g,n}\to\Bbb C$ induces a representation $\chi_g$ such that $\textsf{vol}(\chi_g)>0$, it is an easy matter to verify that $\chi$ arises as the holonomy of a translation structure on $S_{g,n}$ $-$ see Lemmata \ref{hte}, \ref{hnte} and \ref{nshc} below. However, we shall also deal with representations such that $\textsf{vol}(\chi_g)\le 0$ and to which Haupt's theorem does not apply. Any such a representation falls in one of the following categories:
\begin{enumerate}
\item[1.] $\textsf{vol}(\chi_g)=0$ and $\Im\big(\overline{\chi(\alpha_i)}\chi(\beta_i)\big)=0$ for every $i=1,\dots,g$;
\item[2.] $\textsf{vol}(\chi_g)=0$ and there exists $0<h<g$ such that $\textsf{vol}(\chi_h)>0$ and $\textsf{vol}(\chi_{g-h})=-\textsf{vol}(\chi_h)$;
\item[3.] $\textsf{vol}(\chi_g)<0$ and $\Im\big(\overline{\chi(\alpha_i)}\chi(\beta_i)\big)<0$ for every $i=1,\dots,g$; and finally
\item[4.]  $\textsf{vol}(\chi_g)<0$ and there exists $0<h<g$ such that $\textsf{vol}(\chi_h)\ge 0$ and $\textsf{vol}(\chi_{g-h})<-\textsf{vol}(\chi_h)$.\\
\end{enumerate}

\noindent In the second part of the present work, see part \ref{part2}, we shall consider the $\spz-$action on $\Gamma_{g,n}$ for finding a base such that all the handles have non-zero volume. In this part, however, we prefer to do not rely on such an action and exploit completely the flexibility of these structures.

\section{How to geometrize handles}\label{htc}

\noindent In order to prove Theorem \ref{mainthm} for punctured surfaces, we have to deal with representations $\chi$ such that $\textsf{vol}(\chi_g)\le0$. For this purpose, we need to know how to manage handles with non-positive volume. Let $\chi:\langle\alpha,\beta\rangle\cong\Bbb F_2\rightarrow \C$ be a representation with volume $\textsf{vol}(\chi)$. Let start considering representation with null volume.

\subsection{Handle with zero volume} We begin by noticing that the condition $\Im\big(\overline{\chi(\alpha)}\chi(\beta)\big)=0$ implies exactly one case of the following trichotomy.
\begin{enumerate}
\item[i.] $\chi(\alpha)$ and $\chi(\beta)$ are both zero,
\item[ii.] $\chi(\alpha)$ or $\chi(\beta)$ is zero (but not both),
\item[iii.] $\chi(\alpha)$ and $\chi(\beta)$ are (real) collinear: They translate along the same line in $\C$ but possibly in opposite directions.
\end{enumerate}

\noindent We now introduce three different constructions for handles with zero volume and prescribed holonomy according to the cases listed above. Notice that the second case can be subsumed in the third one by applying a suitable change of basis.

\subsubsection{Construction 1: trivial handles.}\label{trhan} Let $(X,\omega)$ be any translation surface. Assume $\omega$ has at least one zero, say $P$, on the underlying surface $S_{g,n}$. As the magnitude of $P$ is $2h\pi$, for some $h\ge2$, we can find two geodesic paths $\tau_1,\tau_2:[0,1]\longrightarrow S$ such that
\begin{itemize}
\item[1.] $\tau_1(0)=\tau_2(0)=P$,
\item[2.] $\tau_1$ and $\tau_2$ do not share any points other than $P$,
\item[3.] they are injectively developed and overlap once developed; in particular they have the same length with respect the singular Euclidean metric.
\end{itemize}
\noindent The segments $\tau_1,\tau_2$ are called \emph{geodesic twin paths}. These segments can be taken sufficiently short so that the both lie in some simply connected chart containing the singular point $P$. Let $Q_i$ be extremal point of $\tau_i$ different from $P$. Define $\delta_i$ as a proper sub-arc of $\tau_i$ starting from $Q_i$ of some length $l$. For $i=1,2$, cut along $\delta_i$ to get a surface with a piecewise geodesic boundary $\delta_i^1\cup\delta_i^2$ and two corner angles. Then glue $\delta_1^i$ with $\delta_2^i$, as shown in the figure below, producing an additional handle and two additional singular points both of magnitude $4\pi$ with trivial holonomy.  In the very special case that $(X,\omega)=(\Bbb C, z\,dz)$ we obtain the following result.

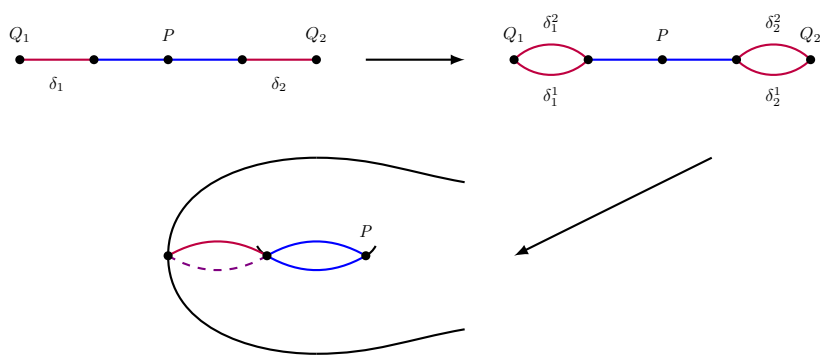
\begin{figure}[!h]
\centering
\begin{tikzpicture}[thick,scale=0.65, every node/.style={scale=0.65}]
\draw [thick, blue] (2.5,5) to (4,5);
\draw [thick, blue] (5.5,5) to (4,5);
\draw [thick, purple] (2.5,5) to (1,5);
\draw [thick, purple] (5.5,5) to (7,5);
\draw [thick] plot [mark=*, smooth] coordinates {(4,5)};
\draw [thick] plot [mark=*, smooth] coordinates {(1,5)};
\draw [thick] plot [mark=*, smooth] coordinates {(7,5)};
\draw [thick] plot [mark=*, smooth] coordinates {(2.5,5)};
\draw [thick] plot [mark=*, smooth] coordinates {(5.5,5)};
\node at (4,5.5) {\text{\emph{P}}};
\node at (1,5.5) {\text{\emph{$Q_1$}}};
\node at (7,5.5) {\text{\emph{$Q_2$}}};
\node at (1.75,4.5) {\text{\emph{$\delta_1$}}};
\node at (6.25,4.5) {\text{\emph{$\delta_2$}}};
\draw[thick, -latex] (8,5) to (10,5);
\draw [thick, blue] (12.5,5) to (14,5);
\draw [thick, blue] (15.5,5) to (14,5);
\draw[thick, purple] (12.5,5) to[out=135, in=45] (11,5);
\draw[thick, purple] (12.5,5) to[out=225, in=315] (11,5);
\draw[thick, purple] (17,5) to[out=135, in=45] (15.5,5);
\draw[thick, purple] (17,5) to[out=225, in=315] (15.5,5);
\draw [thick] plot [mark=*, smooth] coordinates {(14,5)};
\draw [thick] plot [mark=*, smooth] coordinates {(12.5,5)};
\draw [thick] plot [mark=*, smooth] coordinates {(11,5)};
\draw [thick] plot [mark=*, smooth] coordinates {(15.5,5)};
\draw [thick] plot [mark=*, smooth] coordinates {(17,5)};
\node at (14,5.5) {\text{\emph{P}}};
\node at (11,5.5) {\text{\emph{$Q_1$}}};
\node at (17,5.5) {\text{\emph{$Q_2$}}};
\node at (11.75,4.25) {\text{\emph{$\delta_1^1$}}};
\node at (16.25,4.25) {\text{\emph{$\delta_2^1$}}};
\node at (11.75,5.75) {\text{\emph{$\delta_1^2$}}};
\node at (16.25,5.75) {\text{\emph{$\delta_2^2$}}};
\draw[thick, -latex] (15,3) to (11,1);
\draw[ thick] (10,2.5) to[out=170, in=0] (7,3);
\draw[ thick] (7,3) to[out=180, in=90] (4,1);
\draw[ thick] (4,1) to[out=270, in=180] (7,-1);
\draw[ thick] (7,-1) to[out=0, in=190] (10,-0.5);
\draw [ thick, blue] (6,1) to [out=330, in=210] (8,1);
\draw [thick] (5.8,1.2) to [out=300, in=150] (6,1);
\draw [thick] (8,1) to [out=30, in=240] (8.2,1.2);
\draw [ thick, blue] (6,1) to [out=30, in=150] (8,1);
\draw [ thick, purple] (6,1) to [out=150, in=30] (4,1);
\draw [thick, dashed, violet] (6,1) to [out=210, in=330] (4,1);
\draw[thick] plot [mark=*, smooth] coordinates {(6,1)};
\draw[thick] plot [mark=*, smooth] coordinates {(8,1)};
\draw[thick] plot [mark=*, smooth] coordinates {(4,1)};
\node at (8,1.5) {\text{\emph{P}}};
\end{tikzpicture}
\caption[Creating an handle with trivial holonomy]{Creating an handle with trivial holonomy}
\end{figure}

\begin{lem}\label{trihan}
Let $\chi:\Bbb Z^2\longrightarrow \Bbb C$ be the trivial representation. Then $\chi$ appears as the holonomy of a translation structure on a punctured torus.
\end{lem}

\subsubsection{Construction 2: Elementary handles.}\label{elhan} Based on the purposes of the present work, we shall introduce a way for building up handles with elementary holonomy on the Euclidean plane $\Bbb E^2$. There are two possible cases according to the items $2$ and $3$ in the list above.\\

\noindent Let $l_1$ and $l_2$ be two geodesic segments in the complex plane such that: they have the same length and are parallel, \emph{i.e.} there is a complex number $a\in\Bbb C^*$ such that $l_2=l_1+a$. Slit $\Bbb E^2$ along these lines. The segment $l_1$ splits in two segments with the same end points: $l_1^+$ on the right and $l_1^-$ on the left, where left and right are taken with respect to some fixed orientation. In the same fashion, the segment $l_2$ splits two segments, $l_2^+$ on the right and $l_2^-$ on the left. Identify the segments $l_1^+$ with $l_2^-$ and the segments $l_1^-$ with $l_2^+$ by using the translation defined by $a$. Topologically we obtain a torus with one point removed, geometrically we obtain a translation surface $(X,\omega)$ with two singularities of order one on $S_{1,1}$ and elementary holonomy. In particular, we can easily find a basis $\{\alpha,\beta\}$ such that $\alpha $ has holonomy $a\in\Bbb C$ whereas $\beta$ has trivial holonomy. Let us briefly explain how to find such a basis. We can simply define $\beta$ to be a simple closed curve enclosing the line $l_1$, but not $l_2$. The curve $\beta$ remains simple and closed on the surface we obtain once the identification above is done, but it is no longer contractible and has trivial holonomy. Let us move on defining the curve $\alpha$. Let $\zeta$ be one of the extremal point of $l_1$ and consider the geodesic segment joining $\zeta$ with $\zeta+a$. Once the slits above are identified via the translation $z\mapsto z+a$, this segment closes up to a simple closed curve $\alpha$ with holonomy $a$ as desired. We notice that, this surgery takes place in a bounded region of the Euclidean plane. Let us make this point a bit more precise. Let $\zeta\in\Bbb E^2$ be one of the extremal points of $l_1$. Then there exists $M\in\Bbb R$ such that $B_M(\zeta)$ contains both the segments $l_1$ and $l_2$. Perform the surgery, and let $\overline{\zeta}\in(X,\omega)$ be the point we obtain by matching $\zeta$ to $\zeta+a$. The open region $X\setminus \overline{B_M(\overline{\zeta})}$ is isometric to the unbounded region $\Bbb E^2\setminus \overline{B_M(\zeta)}$ and the mapping realizing the isometry is the developing map for $(X,\omega)$. Equivalently, given a neighborhood $U\subset X$ of the puncture disjoint from $B_M(\zeta)$, there is compact set $K\subset \Bbb E^2$, homeomorphic to a closed disc, and a mapping $f:U\longrightarrow \Bbb E^2\setminus K$ such that $f$ is an isometry. This means that, although the topology is changed and hence the global geometry, this latter is remained unchanged outside a bounded region on which the surgery took place. The following lemma holds.

\begin{lem}\label{lelhan}
Let $\chi:\Bbb Z^2\longrightarrow \Bbb C$ be a representation such that $\chi(\alpha)=a\in\Bbb C^*$ and $\chi(\beta)=0$. Then $\chi$ is elementary and arises as the holonomy of a translation structure on a punctured torus.
\end{lem}

\noindent Let $a,b\in\Bbb C$ two complex number such that $b=\lambda a$ with $\lambda\in\Bbb R^*$, that is they determine the same direction. Let $l\in\Bbb E^2$ be a segment oriented along this direction of length $|a|+|b|$ and let $P_1$ and $P_2$ be the extremal points of $l$. By slitting $l$, as we have done above, we obtain two segments $l^+$ and $l^-$ both having $P_1,P_2$ as extremal points. Let $Q_1$ be the point on $l^+$ at distance $|a|$ from $P_1$ and let $Q_2$ be the point on $l^-$ at distance $|b|$ from $P_1$. Identify the segments $\overline{P_1\,Q_1}$ and $\overline{Q_2\,P_2}$ by using the translation $z\mapsto z+b$. Similarly, identify the segments $\overline{P_1\,Q_2}$ and $\overline{Q_1\,P_2}$ by using the translation $z\mapsto z+a$. The resulting surface is a torus with one point removed and geometrically a translation surface with elementary holonomy. In particular, there is a basis $\{\alpha,\beta\}$ such that $\alpha$ has holonomy $a\in\Bbb C$ whereas $\beta$ has holonomy $b$. Once again, this surgery takes place in a bounded region of the Euclidean plane in the sense described above. The following lemma holds.

\begin{lem}\label{lelhann}
Let $\chi:\Bbb Z^2\longrightarrow \Bbb C$ be a representation such that $\chi(\alpha)=a\in\Bbb C^*$ and $\chi(\beta)=\lambda a$ with $\lambda\in\Bbb R^*$. Then $\chi$ is elementary and arises as the holonomy of a translation structure on a punctured torus.\\
\end{lem}

\subsection{Handles with negative volume} In this section we show how to deal handles with negative volume. Let $\chi:\Bbb Z^2\cong\langle\alpha,\beta\rangle\to \Bbb C$ be a representation such that \textsf{vol}$(\chi)<0$ - notice that this implies  $\chi(\alpha),\chi(\beta)\in\Bbb C^*$. The vectors $\chi(\alpha)$ and $\chi(\beta)$ determine a non-degenerate parallelogram $\mathcal{P}$ on the complex plane; that is a fundamental parallelogram for the action of the discrete group $\chi(\Bbb Z^2)$. Let $\mathcal{E}$ denote the closure of the exterior of $\mathcal{P}$ in $\cp$. Topologically, the region $\mathcal{E}$ is a parallelogram. Identifying the opposite sides of $\mathcal{E}$ according to the holonomy $\chi$ we obtain a torus endowed with a translation structure with one branch point of magnitude $6\pi$ and one pole of two and having holonomy $\chi$. In particular, this structure has negative volume. While we need $\mathcal{P}$ to be non degenerate to have negative volume, the construction goes through even when $\mathcal{P}$ is degenerate, and this is the construction used in Lemma \ref{lelhann}. The following lemma is the consequence of this discussion. \\

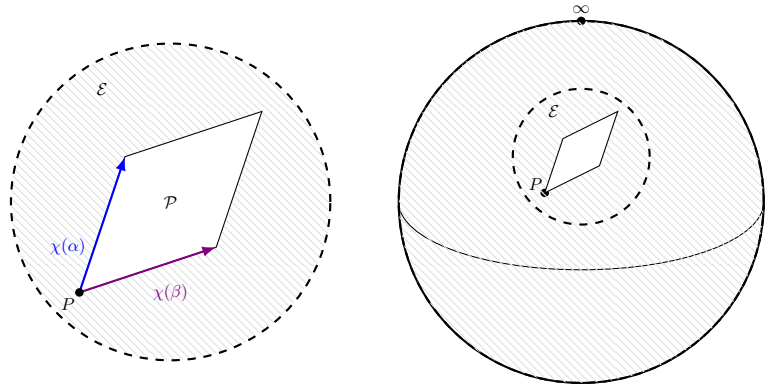
\begin{figure}[!h]
\centering
\begin{tikzpicture}[thick,scale=0.6, every node/.style={scale=0.65}]
\definecolor{pallido}{RGB}{221,227,227}
\draw [thin, dashed] (0,0) circle (35mm);
\draw [dashed, black, pattern=north west lines, pattern color=pallido] (0,0) circle (35mm);
\fill [white, thin, draw=black] (-2,-2) to (1,-1) to (2,2) to (-1,1) to (-2,-2);
\draw[thick, -latex, violet] (-2,-2) to (1,-1);
\draw[thick, -latex, blue] (-2,-2) to (-1,1);
\draw [black] plot [mark=*, smooth] coordinates {(-2,-2)};
\node at (-2.25,-2.25) {$P$};
\node at (-2.25,-1) {\textcolor{blue}{$\chi(\alpha)$}};
\node at (0,-2) {\textcolor{violet}{$\chi(\beta)$}};
\node at (-1.5,2.5) {$\mathcal{E}$};
\node at (0,0) {$\mathcal{P}$};
\draw [thin] (3+2,-3+3) arc (180:360:4 and 1.5);
\draw [thick] (3+6,-3+3) circle (40mm);
\draw plot [mark=*, smooth] coordinates {(9,4)};
\draw plot [mark=*, smooth] coordinates {(8.2,0.2)};
\draw [dashed, black, pattern=north west lines, pattern color=pallido] (3+6,-3+3) circle (40mm);
\draw [dashed, black] (3+6,1) circle (15mm);
\node at (3+6, 4.25) {$\infty$};
\node at (8, 0.4) {$P$};
\fill [white, thin, draw=black] (8.2,0.2) to (8.6,1.4) to (9.8,2) to (9.4,0.8) to (8.2,0.2);
\node at (8.4,2) {$\mathcal{E}$};
\end{tikzpicture}
\caption[]{Handle with negative volume. This can be seen as a complex structure on the torus and abelian differential with one zero of order two and one pole of order two.}
\end{figure}

\begin{lem}\label{nvpt}
Let $\chi:\Bbb Z^2\to \Bbb C$ be a representation such that $\textsf{\emph{vol}}(\chi)<0$. Then $\chi$ appears as the period of a meromorphic  differential on the torus with one pole of order two. Equivalently, $\chi$ appears as the holonomy of a translation structure on the once-punctured torus with a  single branch point of order two.\\
\end{lem}

\subsection{Handles with positive volume} This latter case is very well-known. Let $\chi:\Bbb Z^2\cong\langle\alpha,\beta\rangle\to \Bbb C$ be a representation such that \textsf{vol}$(\chi)>0$. As above,  $\chi(\alpha),\chi(\beta)\in\Bbb C^*$. The vectors $\chi(\alpha)$ and $\chi(\beta)$ determine a non-degenerate parallelogram $\mathcal{P}$ on the complex plane; that is a fundamental parallelogram for the action of the discrete group $\chi(\Bbb Z^2)$. Identifying the opposite sides of $\mathcal{P}$ according to the holonomy $\chi$ we obtain a torus endowed with a translation structure and no branch points. By deleting any point of this latter structure we obtain a translation structure on the one-puncture torus having holonomy $\chi$. The following lemma holds.

\begin{lem}\label{pvpt}
Let $\chi:\Bbb Z^2\to \Bbb C$ be a representation such that $\textsf{\emph{vol}}(\chi)>0$. Then $\chi$ appears as the holonomy of a translation structure on a  torus. Equivalently, $\chi$ appears as the period of a holomorphic abelian differential $\omega$ on the once-punctured torus and the puncture is a removable singularity for $\omega$. 
\end{lem}

\noindent If we identify one pair of opposite sides of $\mathcal{P}$ with the the sides of a slit made in $\mathbb{C}$ as shown in Figure \ref{fig:pvpt2}, we obtain the following lemma.

\begin{figure}[!h]
\centering
\begin{tikzpicture}[thick,scale=0.6, every node/.style={scale=0.65}]
\definecolor{pallido}{RGB}{221,227,227}

\fill [pattern=north west lines, pattern color=pallido, thin, draw=black] (-2,-2) to (1,-1) to (2,2) to (-1,1) to (-2,-2);
\draw[thick, -latex, violet] (-2,-2) to (1,-1);
\draw[thick, -latex, blue] (-2,-2) to (-1,1);
\draw [black] plot [mark=*, smooth] coordinates {(-2,-2)};
\node at (-2.25,-1) {\textcolor{blue}{$\chi(\beta)$}};
\node at (0,-2) {\textcolor{violet}{$\chi(\alpha)$}};
\node at (0,0) {$\mathcal{P}$};
\draw [thin] (3+2,-3+3) arc (180:360:4 and 1.5);
\draw [thick] (3+6,-3+3) circle (40mm);
\draw plot [mark=*, smooth] coordinates {(9,4)};
\draw [dashed, black, pattern=north west lines, pattern color=pallido] (3+6,-3+3) circle (40mm);
\node at (3+6, 4.25) {$\infty$};

\draw[thick] (8.2, 0.2) .. controls (8.9, 0.4) .. (9.4, 0.8);
\draw[thick] (8.2, 0.2) .. controls (8.7, 0.6) .. (9.4, 0.8);
\end{tikzpicture}
\caption[]{Handle with positive volume. The slit in $\mathbb{CP}^1$ is made across $\chi(\beta)$ and the two opposite sides of the paralleogram along $\chi(\beta)$ are identified to the two sides of the slit in the usual way. This results in a complex structure on the torus and abelian differential with one zero of order two and one pole of order two.}
 \label{fig:pvpt2}
\end{figure}
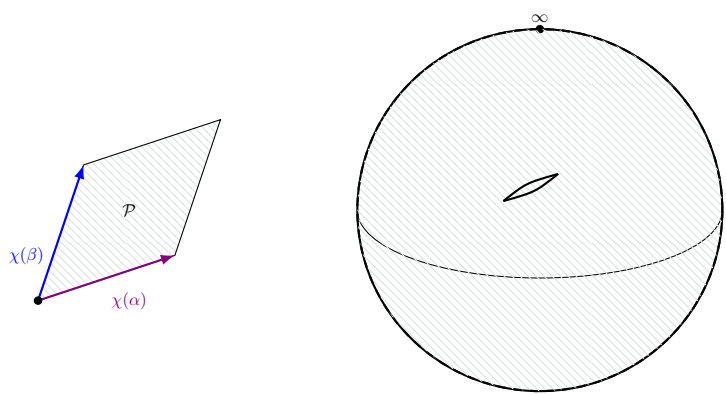

\begin{lem}\label{pvpt2}
Let $\chi:\Bbb Z^2\to \Bbb C$ be a representation such that $\textsf{\emph{vol}}(\chi)>0$. Then $\chi$ appears as the period character of a meromorphic differential on the torus with one pole of order two. Equivalently, $\chi$ appears as the holonomy of a translation structure on the once-punctured torus with one single branch point of order two.
\end{lem}

\noindent The series of Lemmata \ref{trihan}, \ref{lelhan}, \ref{lelhann}, \ref{nvpt} and \ref{pvpt} all together imply the following proposition, that is our main Theorem \ref{mainthm} in the particular case of the once-punctured torus.

\begin{prop}\label{proppt}
Any representation $\chi:\Bbb Z^2\to \Bbb C$ appears as the holonomy of a translation structure on the once-punctured torus. Equivalently, $\chi$ appears as the period character of an abelian differential on $S_{1,1}$.\\
\end{prop}

\noindent We are now masters of punctured torii, however we are not yet masters of punctured spheres. Let us move onto geometrize them.

\section{How to geometrize punctured spheres}\label{fsps} 
\noindent In this section we are going to prove Theorem \ref{mainthm} for punctured surfaces having genus $g=0$. 

\begin{prop}\label{propb}
Any character $\chi_n:\Gamma_{0,n}\longrightarrow \Bbb C$  arises as the holonomy of some translation structure on $S_{0,n}$.
\end{prop}

\begin{rmk}\label{psrmk}
Proposition \ref{propb} trivially holds when $n\le2$. The representation space of $\Gamma_{0,1}\cong\{1\}$ consists of the trivial representation only which arises as the holonomy of $(\Bbb C,\, dz)$. The representation space of $\Gamma_{0,2}$ is isomorphic to $\Bbb C$: Any non-trivial representation $\chi(1)=w\in\Bbb C^*$ arises as the holonomy of the translation structure on $\Bbb C/\langle z \to z + w\rangle$ induced by $dz$, whereas $\chi(1)=0$ appears as the holonomy of $(\Bbb C\setminus\{0\},\,dz)$. 
\end{rmk}

\noindent We assume there are $n\ge3$ punctures, say $\{P_1,\dots,P_n\}$. Let $\gamma_j$ be a simple closed curve enclosing $P_j$ and no other punctures. The curves $\gamma_1,\dots,\gamma_n$ satisfy the condition $\gamma_1\cdots\gamma_{n-1}\gamma_n^{-1}=1$. Let $\chi:\Gamma_{0,n}\longrightarrow \Bbb C$ be a representation and let $\chi(\gamma_j)=z_j$. The trivial representation appears as the holonomy of the translation structure on $\Bbb C\setminus\{P_1,\dots,P_{n-1}\}\cong\Bbb{CP}^1\setminus\{P_1,\dots,P_{n-1},\infty\}$ induced by $dz$. Therefore we may assume $\chi$ to be different from the trivial representation. We also assume $\chi(\gamma_j)=z_j\neq0$ for all $i=1,\dots,n$. Indeed, in the case some of the punctures have trivial holonomy, say $n-k$, we may regard $\chi$ as the a representation $\Gamma_{0,k}\to\Bbb C$ and therefore it is sufficient to determine a translation structure on the $k-$punctured sphere having the desired holonomy on which we eventually remove $n-k$ points. Therefore we may assume $z_j\neq0$ for all the $j$ without loss of generality.

\begin{proof}[Proof of Proposition \ref{propb}] 
\noindent Let $\chi:\Gamma_{0,n}\longrightarrow \Bbb C$ be a representation such that $\chi(\gamma_i)\neq0$ for every $i=1,\dots,n$. By using polar coordinates, the $z_j$ are of the form $r_je^{i\theta_j}$, where $(r_j,\theta_j)\neq(0,0)$. Up to rename all the generators, we may assume that $e^{i\theta_j}$'s are cyclically ordered on the unit circle $\Bbb S^1\subset \Bbb C$. Let $\zeta_1$ be any point in $\Bbb E^2\cong\Bbb C$ and consider the polygonal chain
\begin{equation}
\zeta_1\mapsto \chi(\gamma_1)+\zeta_1=\zeta_2\mapsto\chi(\gamma_2)+\zeta_2=\zeta_3\mapsto\cdots\mapsto \chi(\gamma_{n-1})+\zeta_{n-1}=\zeta_n\mapsto\chi(\gamma_n)+\zeta_n= \zeta_1
\end{equation}

\noindent Since the $e^{i\theta_j}$ are cyclically ordered and $\chi(\gamma_1)+\cdots+\chi(\gamma_n)=0$, such a polygon bounds a possibly degenerate convex $n$-polygon $\mathcal{P}$ on the complex plane and the sum of the inner angles is $(n-2)\pi$. 

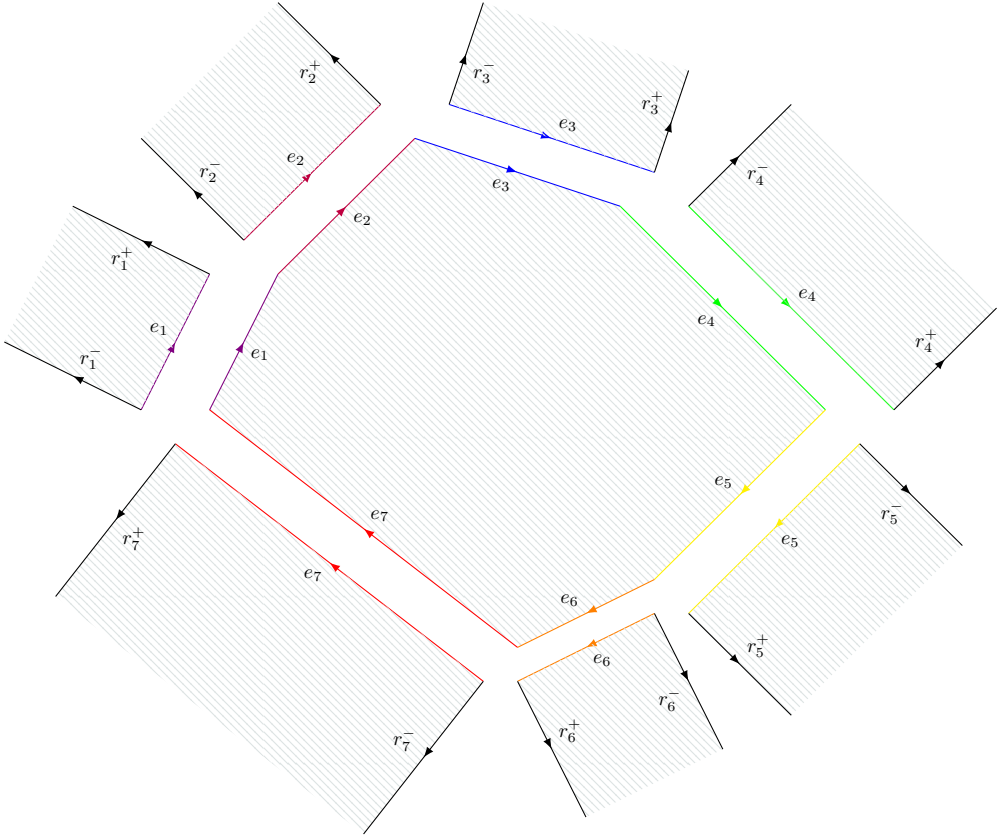
\begin{figure}[h!] \label{fig:nontrivholsphere1}
    \centering
    \begin{tikzpicture}[scale=0.9, every node/.style={scale=0.75}]
    \definecolor{pallido}{RGB}{221,227,227} 
    \path[pattern=north west lines, pattern color=pallido] (0,0) -- ++(1,2) -- ++ (2,2) -- ++ (3,-1) -- ++ (3, -3) -- ++(-2.5, -2.5) -- ++(-2, -1) -- ++ (-4.5, 3.5); 
    \coordinate (A) at (0,0);
    \coordinate (B) at ($(A) + (1,2)$);
    \coordinate (C) at ($(B)+(2,2)$);
    \coordinate (D) at ($(C) + (3,-1)$);
    \coordinate (E) at ($(D) + (3, -3)$);
    \coordinate (F) at ($(E) + (-2.5, -2.5)$);
    \coordinate (G) at ($(F) + (-2,-1)$);
    \foreach \stpt / \endpt / \colr in {(A)/(B)/violet, (B)/(C)/purple, (C)/(D)/blue, (D)/(E)/green, (E)/(F)/yellow, (F)/(G)/orange, (G)/(A)/red}
    {
        \draw[\colr, decoration={markings, mark=at position 0.5 with {\arrow{latex}}}, postaction={decorate}] \stpt -- \endpt;
    }
    
    \draw[violet, decoration={markings, mark=at position 0.5 with {\arrow{latex}}}, postaction={decorate}] ($(A)+(-1,0)$) -- ($(B)+(-1,0)$);
    \path[pattern=north west lines, pattern color=pallido] ($(A)+(-1,0)+(-2,1)$) -- ($(A)+(-1,0)$) -- ($(B)+(-1,0)$) --  ($(B)+(-1,0)+(-2,1)$);
    \draw[decoration={markings, mark=at position 0.5 with {\arrow{latex}}}, postaction={decorate}] ($(A)+(-1,0)$) -- +(-2,1);
    \draw[decoration={markings, mark=at position 0.5 with {\arrow{latex}}}, postaction={decorate}] ($(B)+(-1,0)$) -- +(-2,1);
    \draw[purple, decoration={markings, mark=at position 0.5 with {\arrow{latex}}}, postaction={decorate}] ($(B)+(-0.5,0.5)$) -- ($(C)+(-0.5,0.5)$);
    \path[pattern=north west lines, pattern color=pallido] ($(B)+(-0.5,0.5) + (-1.5,1.5)$) -- ($(B)+(-0.5,0.5)$) -- ($(C)+(-0.5,0.5)$) --  ($(C)+(-0.5,0.5)+(-1.5,1.5)$);
    \draw[decoration={markings, mark=at position 0.5 with {\arrow{latex}}}, postaction={decorate}] ($(C)+(-0.5,0.5)$) -- +(-1.5,1.5);
    \draw[decoration={markings, mark=at position 0.5 with {\arrow{latex}}}, postaction={decorate}] ($(B)+(-0.5,0.5)$) -- +(-1.5,1.5);
    \draw[blue, decoration={markings, mark=at position 0.5 with {\arrow{latex}}}, postaction={decorate}] ($(C)+(0.5,0.5)$) -- ($(D)+(0.5,0.5)$);
    \path[pattern=north west lines, pattern color=pallido] ($(C)+(0.5,0.5)+(0.5,1.5)$) -- ($(C)+(0.5,0.5)$) -- ($(D)+(0.5,0.5)$) --  ($(D)+(0.5,0.5)+(0.5,1.5)$);
    \draw[decoration={markings, mark=at position 0.5 with {\arrow{latex}}}, postaction={decorate}] ($(C)+(0.5,0.5)$) -- +(0.5,1.5);
    \draw[decoration={markings, mark=at position 0.5 with {\arrow{latex}}}, postaction={decorate}] ($(D)+(0.5,0.5)$) -- +(0.5,1.5);
    \draw[green, decoration={markings, mark=at position 0.5 with {\arrow{latex}}}, postaction={decorate}] ($(D)+(1,0)$) -- ($(E)+(1,0)$);
    \path[pattern=north west lines, pattern color=pallido] ($(D)+(1,0)+(1.5,1.5)$) -- ($(D)+(1,0)$) -- ($(E)+(1,0)$) --  ($(E)+(1,0)+(1.5,1.5)$);
    \draw[decoration={markings, mark=at position 0.5 with {\arrow{latex}}}, postaction={decorate}] ($(E)+(1,0)$) -- +(1.5,1.5);
    \draw[decoration={markings, mark=at position 0.5 with {\arrow{latex}}}, postaction={decorate}] ($(D)+(1,0)$) -- +(1.5,1.5);
    \draw[yellow, decoration={markings, mark=at position 0.5 with {\arrow{latex}}}, postaction={decorate}] ($(E)+(0.5,-0.5)$) -- ($(F)+(0.5,-0.5)$);
    \path[pattern=north west lines, pattern color=pallido] ($(E)+(0.5,-0.5)+(1.5,-1.5)$) -- ($(E)+(0.5,-0.5)$) -- ($(F)+(0.5,-0.5)$) --  ($(F)+(0.5,-0.5)+(1.5,-1.5)$);
    \draw[decoration={markings, mark=at position 0.5 with {\arrow{latex}}}, postaction={decorate}] ($(E)+(0.5,-0.5)$) -- +(1.5,-1.5);
    \draw[decoration={markings, mark=at position 0.5 with {\arrow{latex}}}, postaction={decorate}] ($(F)+(0.5,-0.5)$) -- +(1.5,-1.5);
    \draw[orange, decoration={markings, mark=at position 0.5 with {\arrow{latex}}}, postaction={decorate}] ($(F)+(0,-0.5)$) -- ($(G)+(0,-0.5)$);
    \path[pattern=north west lines, pattern color=pallido] ($(F)+(0,-0.5)+(1,-2)$) -- ($(F)+(0,-0.5)$) -- ($(G)+(0,-0.5)$) --  ($(G)+(0,-0.5)+(1,-2)$);
    \draw[decoration={markings, mark=at position 0.5 with {\arrow{latex}}}, postaction={decorate}] ($(F)+(0,-0.5)$) -- +(1,-2);
    \draw[decoration={markings, mark=at position 0.5 with {\arrow{latex}}}, postaction={decorate}] ($(G)+(0,-0.5)$) -- +(1,-2);
    \draw[red, decoration={markings, mark=at position 0.5 with {\arrow{latex}}}, postaction={decorate}] ($(G)+(-0.5,-0.5)$) -- ($(A)+(-0.5,-0.5)$);
    \path[pattern=north west lines, pattern color=pallido] ($(A)+(-0.5,-0.5)+(-1.75,-2.25)$) -- ($(A)+(-0.5,-0.5)$) -- ($(G)+(-0.5,-0.5)$) --  ($(G)+(-0.5,-0.5)+(-1.75,-2.25)$);
    \draw[decoration={markings, mark=at position 0.5 with {\arrow{latex}}}, postaction={decorate}] ($(A)+(-0.5,-0.5)$) -- +(-1.75,-2.25);
    \draw[decoration={markings, mark=at position 0.5 with {\arrow{latex}}}, postaction={decorate}] ($(G)+(-0.5,-0.5)$) -- +(-1.75,-2.25);
    
    \node[below right] at ($(A)!0.5!(B)$) {$e_1$};
    \node[above left] at ($(A)!0.5!(B)+(-1,0)$) {$e_1$};
    \node[below right] at ($(B)!0.5!(C)$) {$e_2$};
    \node[above left] at ($(B)!0.5!(C)+(-0.5,0.5)$) {$e_2$};
    \node[below left] at ($(C)!0.5!(D)$) {$e_3$};
    \node[above right] at ($(C)!0.5!(D)+(0.5,0.5)$) {$e_3$};
    \node[below left] at ($(D)!0.5!(E)$) {$e_4$};
    \node[above right] at ($(D)!0.5!(E)+(1,0)$) {$e_4$};
    \node[above left] at ($(E)!0.5!(F)$) {$e_5$};
    \node[below right] at ($(E)!0.5!(F)+(0.5,-0.5)$) {$e_5$};
    \node[above left] at ($(F)!0.5!(G)$) {$e_6$};
    \node[below right] at ($(F)!0.5!(G)+(0,-0.5)$) {$e_6$};
    \node[above right] at ($(A)!0.5!(G)$) {$e_7$};
    \node[below left] at ($(A)!0.5!(G)+(-0.5,-0.5)$) {$e_7$};
    \node[above right] at ($(A)+(-1,0)+ 0.5*(-2,1)$) {$r_1^-$};
    \node[below left] at ($(B)+(-1,0)+ 0.5*(-2,1)$) {$r_1^+$};
    \node[below left] at ($(C)+(-0.5,0.5)+ 0.5*(-1.5,1.5)$){$r_2^+$};
    \node[above right] at ($(B)+(-0.5,0.5)+ 0.5*(-1.5,1.5)$){$r_2^-$};
    \node[below right] at ($(C)+(0.5,0.5)+ 0.5*(0.5,1.5)$){$r_3^-$};
    \node[above left] at ($(D)+(0.5,0.5)+ 0.5*(0.5,1.5)$){$r_3^+$};
    \node[above left] at ($(E)+(1,0)+ 0.5*(1.5,1.5)$){$r_4^+$};
    \node[below right] at ($(D)+(1,0)+ 0.5*(1.5,1.5)$){$r_4^-$};
    \node[below left] at ($(E)+(0.5,-0.5)+ 0.5*(1.5,-1.5)$){$r_5^-$};
    \node[above right] at ($(F)+(0.5,-0.5)+ 0.5*(1.5,-1.5)$){$r_5^+$};
    \node[below left] at ($(F)+(0,-0.5)+ 0.5*(1,-2)$){$r_6^-$};
    \node[above right] at ($(G)+(0,-0.5)+ 0.5*(1,-2)$){$r_6^+$};
    \node[above left] at ($(G)+(-0.5,-0.5)+ 0.5*(-1.75,-2.25)$){$r_7^-$};
    \node[below right] at ($(A)+(-0.5,-0.5)+ 0.5*(-1.75,-2.25)$){$r_7^+$};
    \end{tikzpicture}
    \caption[]{A non-degenerate $7-$polygon with the corresponding half-strips.} 
\end{figure}

\noindent Let begin with assuming the polygon is non-degenerate. Then it is embedded in the complex plane. For each side $e_j$ of $\mathcal{P}$ consider the infinite half-strip $\mathcal{S}_j$ having base $e_j$ and such that $\mathcal{S}_j\,\cap\,\mathcal{P}=e_j$. Notice that each half-strip $\mathcal{S}_j$ is bounded by $e_j$ and two infinite rays $r_j^+$ and $r_j^-$. If two sides $e_j$ and $e_{j+1}$ are aligned, the rays $r_j^+$ and $r_{j+1}^-$ may overlap. We define the translation structure $(X,\omega)$ as the quotient of $\mathcal{P}\,\cup\,\bigcup_j \mathcal{S}_j$ obtained by\\
\begin{quote}
\begin{enumerate}
\item identifying all the vertices of $\mathcal{P}$, and\\
\item the identification of $r_j^+$ with $r_j^-$ for each $j=1,\dots,n$.\\
\end{enumerate}
\end{quote}
\noindent The resulting surface is homeomorphic to $S_{0,n}$. From the construction, the holonomy around the $i-$th puncture is $\chi(\gamma_i)$, as desired. Moreover, this translation surface has exactly one branch point of magnitude $(2n-2)\pi$.  Let us finally consider the degenerate case. In this case there exist $0<k-1<n$ and $\theta\in\Bbb S^1$ such that $\theta_i=\theta$ for $1\le i\le k-1$ and $\theta_i=-\theta$ for $k\le i\le n$. The polygon above degenerates to a segment in the complex plane joining $\zeta_1$ with $\chi(\gamma_{k-1}\cdots\gamma_1)\zeta_1=\zeta_{k}$. Let $\zeta'_{k+1}$ be another point in the complex plane, different from $\zeta_k$ and define $\zeta'_{i+1}=\chi(\gamma_{i-1}\cdots\gamma_k)\zeta'_{k+1}$ for $k+1 \leq i \leq n+1$. Since $\zeta_k\neq\zeta'_{k+1}$, the four points $\zeta_1,\,\zeta_{k},\,\zeta'_{k+1},\,\zeta'_{n+2}$ bounds an embedded parallelogram $\mathcal{P}$ in the complex plane. Notice that the segment $\overline{\zeta_{k}\,\zeta'_{k+1}}=\chi(\gamma_{k-1}\cdots\gamma_1)\overline{\zeta_{1}\,\zeta'_{n+2}}$. Divide the segment $\overline{\zeta_{1}\,\zeta_{k}}$ in sub-segments determined by the collection of points $\{\zeta_i\,|\, 1\le i\le k\}$. In the same fashion, divide the segment $\overline{\zeta'_{n+2}\,\zeta'_{k+1}}$ in sub-segments determined by the collection of points $\{\zeta'_i\,|\, k+1\le i\le n+2\}$.

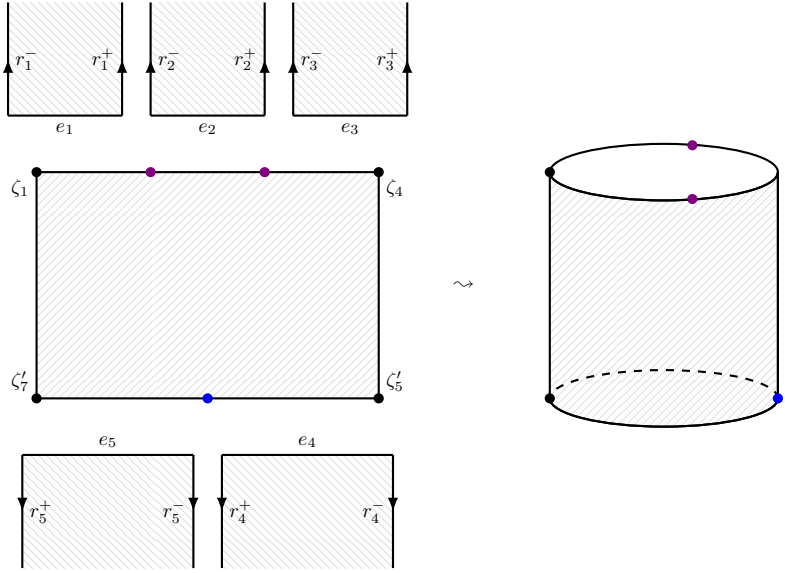
\begin{figure}[!h]
\centering
\begin{tikzpicture}[thick,scale=0.75, every node/.style={scale=0.75}]
\definecolor{pallido}{RGB}{221,227,227}
\draw [black, pattern=north east lines, pattern color=pallido] (-7,-2) -- (-1,-2) -- (-1,2) -- (-7,2) -- (-7,-2);
\draw [black, pattern=north east lines, pattern color=pallido] (2,2) arc (180:360:2 and 0.5) -- (6,-2)  -- (6,-2) arc (360:180:2 and 0.5) -- (2,-2);
\draw plot [mark=*, smooth] coordinates {(-7,-2)};
\draw plot [mark=*, smooth] coordinates {(-1,-2)};
\draw plot [mark=*, smooth] coordinates {(-1,2)};
\draw plot [mark=*, smooth] coordinates {(-7,2)};
\node at (0.5,0) {$\leadsto$};
\node[below left] at (-7,2) {$\zeta_1$};
\node[above left] at (-7,-2) {$\zeta'_7$};
\node[below right] at (-1,2) {$\zeta_4$};
\node[above right] at (-1,-2) {$\zeta'_5$};
\draw  (2,2) arc (180:360:2 and 0.5);
\draw  (6,2) arc (0:180:2 and 0.5);
\draw [thick, black] (2,2)-- (2,-2);
\draw [thick, black] (6,2)-- (6,-2);
\draw  (2,-2) arc (180:360:2 and 0.5);
\draw  [dashed] (6,-2) arc (0:180:2 and 0.5);
\draw plot [mark=*, smooth] coordinates {(2,-2)};
\draw plot [mark=*, smooth] coordinates {(2,2)};
\draw [violet] plot [mark=*, smooth] coordinates {(-5,2)};
\draw [violet] plot [mark=*, smooth] coordinates {(-3,2)};
\draw [blue] plot [mark=*, smooth] coordinates {(-4,-2)};
\draw [violet] plot [mark=*, smooth] coordinates {(4.5,2.475)};
\draw [violet] plot [mark=*, smooth] coordinates {(4.5,1.525)};
\draw [blue] plot [mark=*, smooth] coordinates {(6,-2)};

\begin{scope}[shift = {(-0.5, 0.5)}]
\path[pattern=north west lines, pattern color=pallido] (-7,4.5) -- (-7,2.5) -- (-5,2.5) -- (-5,4.5);
\draw[decoration={markings, mark=at position 0.5 with {\arrow{latex}}}, postaction={decorate}] (-7,2.5) -- (-7,4.5);
\draw[decoration={markings, mark=at position 0.5 with {\arrow{latex}}}, postaction={decorate}] (-5,2.5) -- (-5,4.5);
\draw (-7,2.5) -- (-5,2.5);
\node[right] at (-7, 3.5) {$r_1^-$};
\node[left] at (-5, 3.5) {$r_1^+$};
\node[below] at (-6,2.5) {$e_1$};
\end{scope}

\begin{scope}[shift = {(2, 0.5)}]
\path[pattern=north west lines, pattern color=pallido] (-7,4.5) -- (-7,2.5) -- (-5,2.5) -- (-5,4.5);
\draw[decoration={markings, mark=at position 0.5 with {\arrow{latex}}}, postaction={decorate}] (-7,2.5) -- (-7,4.5);
\draw[decoration={markings, mark=at position 0.5 with {\arrow{latex}}}, postaction={decorate}] (-5,2.5) -- (-5,4.5);
\draw (-7,2.5) -- (-5,2.5);
\node[right] at (-7, 3.5) {$r_2^-$};
\node[left] at (-5, 3.5) {$r_2^+$};
\node[below] at (-6,2.5) {$e_2$};
\end{scope}

\begin{scope}[shift = {(4.5 , 0.5)}]
\path[pattern=north west lines, pattern color=pallido] (-7,4.5) -- (-7,2.5) -- (-5,2.5) -- (-5,4.5);
\draw[decoration={markings, mark=at position 0.5 with {\arrow{latex}}}, postaction={decorate}] (-7,2.5) -- (-7,4.5);
\draw[decoration={markings, mark=at position 0.5 with {\arrow{latex}}}, postaction={decorate}] (-5,2.5) -- (-5,4.5);
\draw (-7,2.5) -- (-5,2.5);
\node[right] at (-7, 3.5) {$r_3^-$};
\node[left] at (-5, 3.5) {$r_3^+$};
\node[below] at (-6,2.5) {$e_3$};
\end{scope}

\begin{scope}[shift = {(-0.25, 0)}]
\path[pattern=north west lines, pattern color=pallido] (-7,-5) -- (-7,-3) -- (-4,-3) -- (-4,-5);
\draw[decoration={markings, mark=at position 0.5 with {\arrow{latex}}}, postaction={decorate}] (-7,-3) -- (-7,-5);
\draw[decoration={markings, mark=at position 0.5 with {\arrow{latex}}}, postaction={decorate}] (-4,-3) -- (-4,-5);
\draw (-7,-3) -- (-4,-3);
\node[right] at (-7, -4) {$r_5^+$};
\node[left] at (-4, -4) {$r_5^-$};
\node[above] at (-5.5,-3) {$e_5$};
\end{scope}

\begin{scope}[shift = {(3.25, 0)}]
\path[pattern=north west lines, pattern color=pallido] (-7,-5) -- (-7,-3) -- (-4,-3) -- (-4,-5);
\draw[decoration={markings, mark=at position 0.5 with {\arrow{latex}}}, postaction={decorate}] (-7,-3) -- (-7,-5);
\draw[decoration={markings, mark=at position 0.5 with {\arrow{latex}}}, postaction={decorate}] (-4,-3) -- (-4,-5);
\draw (-7,-3) -- (-4,-3);
\node[right] at (-7, -4) {$r_4^+$};
\node[left] at (-4, -4) {$r_4^-$};
\node[above] at (-5.5,-3) {$e_4$};
\end{scope}

\end{tikzpicture}
\caption[]{In the degenerate case we consider $\mathcal{P}$ as a parallelogram. By choosing $\zeta_{k+1}$ properly $P$ can be taken a rectangle.}
\end{figure} 

\noindent We now proceed as above. For each side $e_j$ of $\overline{\zeta_{1}\,\zeta_{k}}$ consider the infinite half-strip $\mathcal{S}_j$ having base $e_j$ and such that $\mathcal{S}_j\,\cap\,\mathcal{P}=e_j$. Notice that each half-strip $\mathcal{S}_j$ is bounded by $e_j$ and two infinite rays $r_j^+$ and $r_j^-$. Since the sides are aligned, the rays $r_j^+$ and $r_{j+1}^-$ overlap for each $1\le j\le k$. In the same fashion, for each side $e_{k+j}$ of $\overline{\zeta'_{n+2}\,\zeta_{k+1}}$ consider the infinite half-strip $\mathcal{S}_{k+j}$ having base $e_{k+j}$ and such that $\mathcal{S}_{k+j}\,\cap\,\mathcal{P}=e_{k+j}$. We define the translation structure $(X,\omega)$ as the quotient of $\mathcal{P}\,\cup\,\bigcup_j \mathcal{S}_j$ obtained by
\begin{enumerate}
\item identifying all the points $\zeta_i$ where $i=1,\dots,k$,
\item identifying all the points $\zeta'_i$ where $i=k+1,\dots,n+2$,
\item the identification of $r_j^+$ with $r_j^-$ for each $j=1,\dots,n$.
\end{enumerate}
\noindent The reader may notice that $(X,\omega)$ contains an open embedded annulus obtained from the identification $\overline{\zeta_{1}\,\zeta'_{n+2}}$ with $\overline{\zeta_{k}\,\zeta'_{k+1}}$ by the mapping $\chi(\gamma_{k-1}\cdots\gamma_1)$, see the picture above. The resulting surface is homeomorphic to $S_{0,n}$. From the construction, the holonomy around the $i-$th puncture is $\chi(\gamma_i)$, as desired. Moreover, this translation surface has exactly two branch points of magnitude $2(k-1)\pi$ and $2(n-k)\pi$.\qedhere
\end{proof}

\noindent Now that we are masters of handles and punctured spheres we can move on proving our main theorem \ref{mainthm}.\\

\section{Geometrizing open surfaces}\label{gos}

\noindent In this section we shall prove that each representation $\chi:\Gamma_{g,n}\longrightarrow \Bbb C$ arises as the holonomy representation of some translation structure on $S_{g,n}$ and indeed our Theorem \ref{mainthm} as a consequence of Lemma \ref{tnr}. In the previous chapter, we have seen that any representation $\chi:\Gamma_{g,n}\to \Bbb C$ determines a representation $\chi_g:\Gamma_{g,0}\to\Bbb C$. Since a punctured surface of finite type $(g,n)$ splits as the connected sum of the closed surface of genus $g$ with the $n$-punctured sphere, that is $S_{g,n}=S_{g,0}\,\sharp\, S_{0,n}$, each given representation $\chi$ also yields a representation $\chi_n: \Gamma_{0,n}\longrightarrow \Bbb C$. It is easy to check that $\chi$ is completely determined by the pair $(\chi_g,\,\chi_n)$. We introduce the following definition.

\begin{defn}
Let $\chi:\Gamma_{g,n}\to \Bbb C$ be a representation, where $g\ge1$. Then $\chi$ is said to be of \emph{trivial-volume type} if $\textsf{vol}(\chi_g)=0$, otherwise it is said to be of \emph{non trivial-volume type} representation. In the same fashion, the representation $\chi$ is said to be of \emph{trivial-ends type} if $\chi_n$ is the trivial representation, otherwise, it is said to be of \emph{non trivial-ends type}. 
\end{defn}

\noindent In what follows we shall treat the representation separately according to their properties. More precisely, we shall treat the representation according to the following partition of the representation space.

\begin{equation}
\textsf{Hom}(\Gamma_{g,n},\Bbb C)=
\left\lbrace 
\begin{array}{ccc}
\chi:\Gamma_{g,n}\to \Bbb C\\
\text{positive volume}\\
\textsf{vol}(\chi_g)>0\\
\text{}\\
(\text{Section }\ref{essvolpos})\\
\end{array} 
\right\rbrace \cup
\left\lbrace 
\begin{array}{ccc}
\chi:\Gamma_{g,n}\to \Bbb C\\
\text{non-positive volume}\\
\textsf{vol}(\chi_g)\le 0\\
\text{}\\
(\text{Section }\ref{essvolneg})\\
\end{array} 
\right\rbrace
\end{equation}
\text{}\\
\text{}\\
\noindent  Before proceeding, recall that  $\{\alpha_1,\beta_1,\dots,\alpha_g,\beta_g,\gamma_1,\dots,\gamma_n\}$ is a symplectic basis for $\Gamma_{g,n}=H_1(S_{g,n};\,\Bbb Z)$. 

\begin{rmk}
The case of once punctured surfaces is special in the sense that every representation is of trivial-ends type. Indeed, if $\gamma$ is the curve around the puncture in $S_{g,1}$, then $\chi(\gamma)=0$ necessarily.\\
\end{rmk}

\subsection{Representations with positive volume}\label{essvolpos} In this section we shall prove Theorem \ref{mainthm} for all the representations in the following set of the representation space:

\begin{equation}
\left\lbrace 
\begin{array}{ccc}
\chi:\Gamma_{g,n}\to \Bbb C\\
\text{positive volume type}\\
\end{array} 
\right\rbrace.
\end{equation}

\noindent Let $\chi:\Gamma_{g,n}\longrightarrow \Bbb C$ be a positive volume type, then three possible cases may appear according to the table below.

\begin{equation}
\begin{array}{rr}
\chi \text{ with }\textsf{vol}(\chi_g)>0
\end{array}\rightarrow
\begin{cases}
\chi_g \text{ satisfies the second Haupt condition }\quad\quad\quad\quad
\begin{cases}
\chi_n \text{ is trivial}\\
\text{See Lemma \ref{hte}.}
\text{}\\
\text{}\\
\chi_n \text{ is not trivial}\\
\text{See Lemma \ref{hnte}.}
\end{cases}\\
\text{}\\
\chi_g \text{ does not satisfy the second Haupt condition. } \quad \text{ See Lemma \ref{nshc}.}\\
\end{cases}
\end{equation}
\text{}\\

\begin{lem}\label{hte}
Let $\chi:\Gamma_{g,n}\longrightarrow \Bbb C$ be a representation of positive volume type and trivial-ends type such that $\chi_g$ satisfies the Haupt's conditions. Then $\chi$ appears as the holonomy of some translation structure on $S_{g,n}$.
\end{lem}

\begin{proof}
Let $\{\alpha_1,\beta_1,\dots,\alpha_g,\beta_g,\gamma_1,\dots,\gamma_n\}$ be a standard set of generators for $\Gamma_{g,n}$ as usual. As $\chi$ is assumed to be of trivial-ends type, the following holds $\chi(\gamma_i)=0$ for any $\chi$. As a consequence, there is an identification between the spaces $\textsf{Hom}(\Gamma_{g,n},\Bbb C)$ and $\textsf{Hom}(\Gamma_{g},\Bbb C)$. In other words, every representation $\chi$ can be seen as an element of both spaces. Since $\chi_g$ satisfies the Haupt's conditions, it arises as the holonomy of a translation structure on $S_{g,0}$ and hence as the holonomy of a translation structure on $S_{g,n}=S_{g}\setminus\{P_1,\dots,P_n\}$.
\end{proof}

\begin{lem}\label{hnte}
Let $\chi:\Gamma_{g,n}\longrightarrow \Bbb C$ be a representation of positive volume type and non trivial-ends type such that such that $\chi_g$ satisfies the Haupt's conditions. Then $\chi$ appears as the holonomy of some translation structure on $S_{g,n}$.
\end{lem}

\begin{proof}
The representation $\chi$ determines the pair $(\chi_g,\chi_n)$. The representation $\chi_g$ appears as the holonomy of a translation structure on $S_{g,0}$ because it satisfies both Haupt's conditions. The representation $\chi_n$, instead, appears as the holonomy of a translation structure on an $n$-punctured sphere according to Proposition \ref{propb}. These structures can be glued together and the resulting surface, homeomorphic to $S_{g,n}$, carries a translation structure with the desired holonomy.\qedhere
\end{proof}

\noindent Let $\chi:\Gamma_{g,n}\longrightarrow \Bbb C$ be a positive volume type representation and assume $\chi_g$ does not satisfy the second Haupt's condition. We recall that the second Haupt's condition applies only to a special class of representations, namely those whose image is a lattice in $\Bbb C$. The subset of representations for which the second Haupt's condition fails to be satisfied are completely characterized by Proposition \ref{cdfprop} above.

\begin{lem}\label{nshc}
Let $\chi:\Gamma_{g,n}\longrightarrow \Bbb C$ be a representation such that 
\begin{itemize}
    \item[1.] $\textsf{\emph{vol}}(\chi_g)>0$,
    \item[2.] $\text{\emph{Im}}(\chi_{g})=\Lambda$ is a lattice in $\Bbb C$,
    \item[3.] $\textsf{\emph{vol}}(\chi_g)=\text{\emph{Area}}\big(\Bbb C/\Lambda)$.
    \end{itemize}
Then $\chi$ appears as the holonomy of a translation structure on $S_{g,n}$.
\end{lem}

\begin{proof}
Let $f:S_g\to \C/\Lambda$ be the mapping induced by $f_*:\Gamma_g\to H^1(\C/\Lambda,\Bbb Z)$, where $\Lambda=\text{Im}\big(\chi_g\big)$. Since $\chi_g$ does not satisfy the second Haupt's condition, $f$ has degree one and factors through a collapse of $g-1$ handle by \cite[Proposition 2.7]{CDF2}. Hence, $\chi_g$ determines two further representations: a non-trivial $\chi_{1}:\Gamma_{1}\to\C$ which appears as the holonomy of a flat torus and the trivial representation $\chi_{g-1}:\Gamma_{g-1}\to \{0\}<\C$. Let us now consider the representation $\chi_{n}:\Gamma_{0,n}\to\Bbb C$. Proposition \ref{propb} applies and thence $\chi_n$ arises as the holonomy of a translation structure on $S_{0,n}$ where all the punctures have the prescribed holonomy. This latter structure can be glued with the flat torus determined by $\chi_{1,o}$ and the resulting surface, homeomorphic to $S_{1,n}$, carries a translation structure. Such a structure has at least one singularity, say $P$, and we can use it for attaching $g-1$ handles with trivial holonomy. Indeed, pick two embedded and geodesic twin paths starting from $P$. Divide each path in $2g-2$ sub-segments with the same length and labelled from $1$ to $2g-2$ in ascending order. By cutting the segments labelled by even integers and the gluing those with the same label we obtain $g-1$ handles with trivial holonomy. The final surface, homeomorphic to $S_{g,n}$, carries a translation structure with the desired holonomy.
\end{proof}

\subsection{Representations with non-positive volume}\label{essvolneg} In this section we focus on representations of non-positive volume. Let $\chi:\Gamma_{g,n}\longrightarrow \Bbb C$ be a representation such that $\textsf{vol}(\chi_g)\le0$ and let $\{\alpha_1,\beta_1,\dots,\alpha_g,\beta_g,\gamma_1,\dots,\gamma_n\}$ be a basis of $\Gamma_{g,n}$. As the volume of $\chi_g$ is assumed to be non positive, the representation $\chi$ falls in one the four cases listed above. Let us define $k$ be the number of handles with negative volume and let $h$ be the number of handles with zero volume. Set $\overline{g}=g-(k+h)$. Notice that $0\le k,\,h\le g$ and, as the volume is non-positive, the sum $k+h\ge1$. In particular, $k=0$ if and only if $h=g$. Indeed, if $h<g$ and $k=0$ then $\chi$ would have positive volume. We already know the representation $\chi$ splits in two representations, $\chi_g$ and $\chi_n$. In turns, the representation $\chi_g$ determines other three representations $\chi_k$, $\chi_h$ and $\chi_{\overline{g}}$. We shall now describe a way for constructing translation surfaces with non-positive volume $S_{g,n}$ by proceeding in steps by finding a translation structure for each of these representations which will be subsequently glued together. In particular, we describe a precise order of gluing.

\begin{rmk}\label{idenrepspaces}
Recall that the space $\textsf{Hom}(\Gamma_{\bullet,1},\Bbb C)$ naturally identifies with $\textsf{Hom}(\Gamma_{\bullet},\Bbb C)$ because the holonomy of any curve $\gamma$ enclosing the puncture has necessarily trivial holonomy. In what follows, we make the use of the representation $\chi_k$ to define some translation structure on $S_{k,1}$ having as the holonomy $\chi_{k,1}$, the representation naturally associated to $\chi_k$. In the same fashion, we use the representation $\chi_h$ to define a translation structure on $S_{h,1}$. 
\end{rmk}

\subsubsection{Step 1: Geometrizing $\chi_k$.}\label{s0} Assume there are $k>0$ handles with negative volume. Up to rename all the handles, we may assume without loss of generality that $\Im\big(\overline{\chi(\alpha_i)}\chi(\beta_i)\big)<0$ for every $i=1,\dots,k$. For this case we can simply extend the argumentation used to prove Lemma \ref{nvpt} above. Let $P_1$ be any point in the complex plane. The four points $P_1,\,\chi(\alpha_1)P_1,\,\chi(\beta_1)P_1,\,\chi(\alpha_1\beta_1)P_1$ bound a non-degenerate parallelogram $\mathcal{P}_1$, that is a fundamental domain for the action of $\langle\chi(\alpha_1),\,\chi(\beta_1)\rangle$. Remove the interior of $\mathcal{P}_1$. We can find a vertex of $\mathcal{P}_1$, say $P_2$, such that the points $P_2,\,\chi(\alpha_2)P_2,\,\chi(\beta_2)P_2,\,\chi(\alpha_2\beta_2)P_2$ bound a non-degenerate parallelogram $\mathcal{P}_2$. Remove the interior of $\mathcal{P}_2$. We proceed in the same fashion: For any $i=2,\dots,k$ there is a vertex $P_i$ of the parallelogram $\mathcal{P}_{i-1}$ such that $P_i,\,\chi(\alpha_i)P_i,\,\chi(\beta_i)P_i,\,\chi(\alpha_i\beta_i)P_i$ determine a non-degenerate parallelogram $\mathcal{P}_i$. Remove the interior of $\mathcal{P}_i$. We get a chain of parallelograms each one with the interior removed. Identifying the sides according to $\chi_k$ we obtain a surface homeomorphic to $S_{k,1}$ endowed with a translation structure $(X,\omega)$ with one branch point of magnitude $(4k+2)\pi$, one pole of order two and having holonomy $\chi_{k,1}$. The representations $\chi_k$ and $\chi_{k,1}$ are identified in the sense of Remark \ref{idenrepspaces}. Notice that, if $k=g$ and $n=1$ we are done and we do not need proceed any further. 
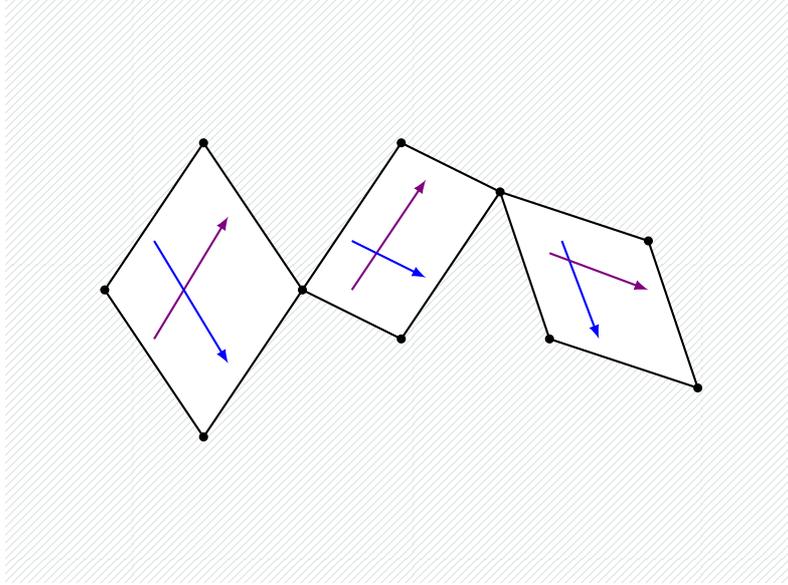
\begin{figure}[!h]
\centering
\begin{tikzpicture}[thick,scale=0.65, every node/.style={scale=0.5}]
\definecolor{pallido}{RGB}{221,227,227}
\draw [white, pattern=north east lines, pattern color=pallido] (-8,-6) -- (8,-6) -- (8,6) -- (-8,6);
\fill [white, thick, draw=black] (-6,0)-- (-4,-3)-- (-2,0)-- (-4,3)-- (-6,0);
\fill [white, thick, draw=black] (-2,0)-- (0,3)-- (2,2)-- (0,-1)-- (-2,0);
\fill [white, thick, draw=black] (2,2)-- (3,-1)-- (6,-2)-- (5,1)-- (2,2);
\draw plot [mark=*, smooth] coordinates {(-6,0)};
\draw plot [mark=*, smooth] coordinates {(-2,0)};
\draw plot [mark=*, smooth] coordinates {(-4,3)};
\draw plot [mark=*, smooth] coordinates {(-4,-3)};
\draw plot [mark=*, smooth] coordinates {(0,3)};
\draw plot [mark=*, smooth] coordinates {(2,2)};
\draw plot [mark=*, smooth] coordinates {(0,-1)};
\draw plot [mark=*, smooth] coordinates {(5,1)};
\draw plot [mark=*, smooth] coordinates {(6,-2)};
\draw plot [mark=*, smooth] coordinates {(3,-1)};
\draw[thick, -latex, violet] (-5,-1) to (-3.5,1.5);
\draw[thick, -latex, blue] (-5,1) to (-3.5,-1.5);
\draw[thick, -latex, violet] (-1,0) to (0.5,2.25);
\draw[thick, -latex, blue] (-1,1) to (0.5,0.25);
\draw[thick, latex-, blue] (4,-1) to (3.25,1);
\draw[thick, -latex, violet] (3,0.75) to (5,0);
\end{tikzpicture}
\caption[]{Chain of three parallelograms in the complex plane representing the situation when $k=3$.}
\end{figure}
\noindent Since each parallelogram is compact, their union lies in a bounded compact region, say $K_o$, of $\Bbb E^2$ whose complement is homeomorphic to a disk in the Riemann sphere. Therefore, there exists an open set $U_o$ and an isometry mapping $f_o:U_o\longrightarrow \Bbb E^2\setminus K_o$. We shall use this mapping for gluing $(X,\omega)$ to the translation structures arising from the other representations.

\subsubsection{Step 2: Geometrizing $\chi_h$.}\label{s1} Assume there are $h>0$ handles with zero volume and consider the representation $\chi_h$. We begin with by assuming $\chi_h$ to be different from the trivial representation; we shall treat this case separately later on. There are $1\le l\le h$ pairs $\{\alpha_i,\beta_i\}$ such that at least one between $\chi(\alpha_i)$ or $\chi(\beta_i)$ is not zero. We assume these pairs are labelled by $\{1,\dots,l\}$ without loss of generality. For any $i=1,\dots,l$, the inclusion mapping $\jmath_i:\langle\alpha_i,\beta_i\rangle\longrightarrow H_1(S_{g,n};\,\Bbb Z)$ yields a representation $\langle\alpha_i,\beta_i\rangle\longrightarrow\Bbb C$ by post-composition with the representation $\chi$. Lemmata \ref{lelhan} and \ref{lelhann} apply and so each one of these representations arises as the holonomy of a translation structure $(X_i,\omega_i)$ on a punctured torus. Let $U_i$ be an open neighborhood of the puncture on the translation structure $(X_i,\omega_i)$. For any $i=1,\dots,l$, there is a compact set $K_i\subset\Bbb E^2$, homeomorphic to a disk, and an isometry $f_i: U_i\longrightarrow \Bbb E^2\setminus K_i$. Since $l$ is finite and the set $K_i$ is compact for any $i$, there exists a compact region $K$, homeomorphic to a disc, containing all the $K_i$'s just defined. Let $V=\Bbb E^2\setminus K$. The preimage of $V$ via the mapping $f_i$ is an open neighborhood, say $V_i$, of the puncture on $(X_i,\omega_i)$ isometric to $V$ and properly contained in $U_i$. All the $V_i$'s are isometric and they can be identified. Let $P$ be any point in $V\subset \Bbb E^2$ and let $\overline{r}$ be an infinite ray starting from $P$ entirely contained in $V$. The mappings $f_i$, where $i=1,\dots,l$, can be used to define an infinite ray $r_i$ starting from $P_i=f^{-1}_i(P)\in V_i$ for every $i=1,\dots,l$. Slit all the structures $(X_i,\omega_i)$ along these rays and re-glue them accordingly by matching the side $r_i^+$ with $r_{i+1 \text{ mod }l}^-$. The surface we obtain from this surgery is homeomorphic to $S_{l,1}$ and equipped with a translation structure $(W,\eta)$. Once identified, the points $P_i$ form a singular point of magnitude $2\pi(l+1)$ and hence the differential $\eta$ has a zero of order $l$. We further notice that the infinity $\infty$ appears as a pole with trivial holonomy. If $h=l$ we are done, otherwise we have to glue the remaining $h-l$ handles with trivial holonomy. Since $(W,\eta)$ has at least one singular point, we can use it for attaching the handles with trivial holonomy in the way we have explained in subsection \ref{trhan}. The final surface, homeomorphic to $S_{h,1}$, is equipped with a translation structure with the desire holonomy $\chi_{h,1}$. The representations $\chi_h$ and $\chi_{h,1}$ are identified in the sense of Remark \ref{idenrepspaces}. We define this latter structure, with a small abuse of notation, as $(W,\eta)$.\\

\noindent In the case $\chi_h$ is the trivial representation, we need to proceed in a slightly different way. Let $(X_1,\omega_1)$ and $(X_2,\omega_2)$ be two copies of $(\Bbb C, dz)$. Slit both along an infinite ray $\overline{r}$ starting from some point $P$, and then re-glue by matching the side $\overline{r}_1^+$ with $\overline{r}_2^-$ and the side $\overline{r}_1^-$ with $\overline{r}_2^+$. The resulting surface is topologically a plane equipped with a translation structure with one branch point of magnitude $4\pi$ at $P$ and one pole of order three, namely $(\C, zdz)$. We then use the branch point for attaching $h$ handles with trivial holonomy as explained in subsection \ref{trhan}. Again, the final surface, homeomorphic to $S_{h,1}$, is equipped with a translation structure $(W,\eta)$ having the desire holonomy $\chi_h$.

\subsubsection{Step 3: The first gluing.}\label{s3} We now glue the structures obtained in the previous steps. If one of $k$ or $h$ is zero, this step can be skipped. Recall, however, that they cannot be both zero. The structure $(W,\eta)$ we have obtained in the previous step, in both cases, has a finite collection $\mathcal{C}$ of infinite rays pointing toward the puncture and that develop on $\overline{r}\subset \Bbb E^2$. Let $(X,\omega)$ be the final structure obtained in step \ref{s0}. The mapping $f_o$ can be used to define an infinite ray $r_o$ on $(X,\omega)$ pointing toward the puncture. If $\overline{r}$ does not intersect the collection of parallelograms $\mathcal{P}_1,\dots, \mathcal{P}_k$, then we can simply define $r_o\subset (X,\omega)$ as the preimage of $\overline{r}$. In particular, by slitting $(X,\omega)$ along $r_o$ and slitting $(W,\eta)$ along anyone of the rays in the collection $\mathcal{C}$, these structure can be glue to define a translation structure on a surface homeomorphic to $S_{k+h,1}$. The holonomy of the resulting structure depends only on $\chi_k$ and $\chi_h$.\\
\noindent Suppose $\overline{r}$ intersects the collection of parallelograms $\mathcal{P}_1,\dots, \mathcal{P}_k$. Let $K'$ be a compact region homeomorphic to a disc and containing both $K_o$ and $K$, the compact regions defined in the previous sections. Then define the open set $V'$ as $\Bbb E^2\setminus K'$. As $K\subseteq K'$, then $V'\subset V$. Since the ray $\overline{r}$ leaves any compact set, there is an infinite ray $\overline{r}'\subset \overline{r}\,\cap\,V'$. Since $K_o\subset K'$, the ray $\overline{r}'$ does no longer intersect the parallelograms $\mathcal{P}_1,\dots, \mathcal{P}_k$ and the preimage via the mapping $f_o$ defines an infinite ray $r_o$ on $(X,\omega)$. Similarly, the preimage of $\overline{r}'$ via any mapping $f_i$, defines an infinite sub-ray $r_i'\subset r_i$. Now the structures $(X,\omega)$ and $(W,\eta)$ can be glued by slitting these rays. As above, the resulting surface is homeomorphic to $S_{k+h,1}$ and carries a translation structures having holonomy uniquely determined by $\chi_k$ and $\chi_h$. In both cases, we shall denote the final structures as $(Y,\xi)$. 

\subsubsection{Step 4: Gluing the $n$-punctured sphere.}\label{s2} Let us now consider the representation $\chi_n:\Gamma_{0,n}\longrightarrow \Bbb C$ determined by $\chi$. If $n=1$, the representation is trivial and there is nothing interesting to say. Therefore assume $n\ge2$. Proposition \ref{propb} and Remark \ref{psrmk} say that $\chi_n$ arises as the holonomy of a translation structure on the $n$-punctured sphere. This latter structure can be glued to $(Y,\,\xi)$ in the following way.\\
\noindent If $\chi_n$ is trivial. It appears as the holonomy of the punctured plane $\Bbb C\,\setminus\,\{P_1,\dots,P_{n-1}\}\cong\mathbb{CP}^1\setminus\{P_1,\dots,P_{n-1},\infty\}$ equiped with the translation structure induced by the differential $dz$. Let $\jmath:\Bbb C\longrightarrow \Bbb E^2$ be the usual identification. Once again, consider the open set $V\subset \Bbb E^2$ already introduced in \ref{s1}. Recall that it contains the geodesic ray $\overline{r}$ based at $P$. There is no loss of generality in assuming that $\jmath\big(\{P_1,\dots,P_{n-1}\}\big)\subset \Bbb E^2\setminus V$. The preimage of the ray $\overline{r}$ via the mapping $\jmath$ is a geodesic line, say $r_\star$. Slit $\mathbb{CP}^1\setminus\{P_1,\dots,P_{n-1},\infty\}$ along the ray $r_\star$. Consider the translation surface $(Y,\,\xi)$ just defined in \ref{s3} and slit it along any infinite ray of the collection $\mathcal{C}$. Then glue the translation surfaces $\big(\mathbb{CP}^1\setminus\{P_1,\dots,P_{n-1},\infty\},\, dz\big)$ and $(Y,\,\xi)$ as usual by using a slit construction. The resulting surface is homeomorphic to $S_{k+h,n}$ and equipped with a translation structure $(Z,\tau)$.\\
\noindent Suppose $\chi_n$ is not trivial. Proposition \ref{propb} says that it appears as the holonomy of a translation structure on $S_{0,n}$ with at least one cylindrical ends. Let $U_o$ be a neighborhood of a puncture with non trivial holonomy and let $Q\in U_o$ be any point. Let $\rho$ be an infinite geodesic ray from $Q$ pointing towards the puncture. The developed image of $\rho$ is a geodesic ray $\overline{\rho}\subset \Bbb E^2$ leaving from the developed image of $Q$. Let $P\in V\subset \Bbb E^2$ be the point introduced in the second step, see \ref{s1}. Let $\overline{r}_\star$ be a geodesic ray starting from $P$ such that: $\overline{r}_\star$ is disjoint from $\overline{\rho}$ and, it is entirely contained in $V\subset \Bbb E^2$.  We think $(\Bbb E^2,\overline{r}_\star)$ as a plane with a marked infinite ray. Slit $U_o$ along $\rho$. In the same fashion, slit the marked plane $(\Bbb E^2,\overline{r}_\star)$ along $\overline{\rho}$. Then reglue $(\Bbb E^2, \overline{r}_\star)$ and $U_o$ along the slits as usual. We obtain a surface still homeomorphic to $S_{0,n}$ but equipped with a new translation structure, say $(Z',\tau')$. Geometrically, this new translation structure contains a whole copy of the marked plane $(\Bbb E^2, \overline{r}_\star)$ which we use for gluing $(Y,\,\xi)$, the translation structure obtained in the previous step, see \ref{s3}. Let $f_i:U_i\longrightarrow \Bbb E^2\setminus K_i$ be anyone of the isometries introduced in \ref{s1} and let $r_\star$ be the preimage of $\overline{r}_\star$ via $f_i$. Since $V\subset \Bbb E^2\setminus K_i$ for any $i$, the $r_\star$ is a infinite ray in $(X_i,\omega_i)$ and therefore an infinite ray pointing towards the (unique) puncture in $(Y,\,\xi)$. Slit $(Y,\,\xi)$ along $r_\star$ and slit $(Z',\tau')$ along $\overline{r}_\star$. Then re-glue as usual. The resulting surface is homeomorphic to $S_{k+h,n}$ and equipped with a singular Euclidean stucture $(Z,\tau)$. In both of cases, the Euclidean structure $(Z,\tau)$ has holonomy depending only on $\chi_k,\,\chi_h$ and $\chi_n$.

\subsubsection{Step 5: Gluing positive handles.} In this final step, we shall glue the remaining $\overline{g}$ handles, if any, with positive volume. The representation $\chi_{\overline{g}}$ has positive volume, Haupt's Theorem applies, and therefore appears as the holonomy of a translation structure on $S_{\overline{g}}$. This latter structure can be glued to $(Z,\tau)$, the structure obtained in \ref{s2}. The resulting surface is homeomorphic to $S_{g,n}$ and carries a singular translation structure with holonomy $\chi$ as desired. 

\subsection{Making translation surfaces metrically complete}\label{mtsgc} The translation surfaces we have constructed in the previous section may very well be not metrically complete. As we have noticed in section \ref{lgp}, this is due to the presence of removable singularities. Let us now described a way to make these structures metrically complete. Let $(X,\omega)$ be a translation structure on a punctured surface $S_{g,n}$. Assume this structure is not metrically complete. As a consequence of Lemma \ref{metlem}, there is a puncture $P$ of $X$ such that in any local coordinate $z$, the abelian differential $\omega$ can be written in the form
\begin{equation}\label{eq:wform}  \omega=f(z)dz, \quad \text{ where } f(z)=\sum_{i=n}^\infty a_i(z-z_o)^i  \text{ for some }n\ge0.
\end{equation}

\noindent We basically distinguish two cases according to the cases $1.1$ and $1.2$ in section \ref{lgp}.
\SetLabelAlign{center}{\null\hfill\textbf{#1}\hfill\null}
\begin{itemize}[leftmargin=1.75em, labelwidth=1.5em, align=center, itemsep=\parskip]
\item[\bf 1] The coefficient $a_0\neq0$. Then any local chart $\varphi:U\setminus\{P\}\longrightarrow \Bbb E^2$ extends to a local homeomorphism $U\longrightarrow \Bbb E^2$. Let $\varepsilon>0$ be a real number small enough that $B_\varepsilon(P)\subset U$. Let $D_\varepsilon=\varphi\big(B_\varepsilon(P)\big)\subset \Bbb E^2\cong \Bbb C$. Let $l$ be a geodesic segment in $B_\varepsilon(P)$ and let $\overline{l}$ the image of $l$ via $\varphi$. Slit $(X,\omega)$ along $l$ and slit $(\Bbb C,\,dz)$ along $\overline{l}$. Then reglue by matching $l^+$ with $\overline{l^-}$ and $l^-$ with $\overline{l^+}$. The final surface is homeomorphic to $S_{g,n}$. The puncture $P$ now appears as a pole of order $2$ and the geometry around it is metrically complete.

\item[\bf 2.] The coefficient $a_0=0$. The point $P$ is a zero of order $k$ for $\omega$, where $k\in\Bbb Z^+$ is smallest index such that $a_k\neq0$. Any local chart $U\setminus\{P\}\longrightarrow \Bbb E^2$ extends to a simple branched $k+1$ covering $U\longrightarrow \Bbb E^2$. Let $\varepsilon>0$ be a real number small enough that $B_\varepsilon(P)\subset U$. The open ball $B_\varepsilon(P)$ is isometric to the $\varepsilon-$neighborhood $D_\varepsilon$ of the vertex of the Euclidean cone of angle $2\pi(k+1)$. Let $g$ be the mapping realizing the isometry. As a translation surface, this singular Euclidean cone is determined by the pair $(\Bbb C,\, z^kdz)$. Let $l$ be a geodesic segment in $B_\varepsilon(P)$ and let $\overline{l}$ the image of $l$ via $g$. Slit $(X,\omega)$ along $l$ and slit $(\Bbb C,\, z^kdz)$ along $\overline{l}$. Then glue by matching $l^+$ with $\overline{l^-}$ and $l^-$ with $\overline{l^+}$. The final surface is homeomorphic to $S_{g,n}$. The puncture $P$ now appears as a pole of order $k+2$ and the geometry around it is metrically complete.
\end{itemize}

\noindent A recursive argument produce a translation surface with poles which is metrically complete. This concludes the proof of Theorem \ref{main:thma2} and thus of Theorem \ref{mainthm}.

\part{Translation structures in a  prescribed stratum}\label{part2}
\noindent The space $\Omega\mathcal{M}_{g,n}$ is naturally stratified by sub-spaces $\mathcal{H}(d_1,\dots,d_m;\,p_1,\dots,p_n)$ of those translation surfaces which have $m$ zeros of order exactly $d_1,\dots,d_m$ and $n$ poles of order exactly $p_1,\dots,p_n$. We may notice that these integers are subject to the well-known relation \eqref{zpec}
\[\sum_{i=1}^m d_i\,-\,\sum_{j=1}^n p_j=2g-2.
\]

\noindent In this second part we shall consider the more delicate problem of determining the image of the period map when restricted to any strata $\mathcal{H}(d_1,\dots,d_m,\,p_1,\dots,p_n)\subset \Omega\mathcal{M}_{g,n}$, and prove the rest of the results stated in the Introduction. Note that since we shall only consider translation structures with poles, our surfaces will automatically be metrically complete.

\section{Trivial holonomy I: Punctured Spheres}\label{sec:necsuftriv}

\noindent We begin by deriving necessary conditions for the number and orders of zeros for a meromorphic differential in $\Omega \mathcal{M}_{g,n}$ with trivial holonomy. Then, in this section,  we shall prove Theorem \ref{main:thmb} for the case when $g=0$ (see Proposition \label{thm:mainthm}). 

\subsection{Necessary Conditions}\label{necondtri} Let $(X,\omega)$ be a translation surface with trivial holonomy on $S_{g,n}$ and consider the developing map $\textsf{dev}:\widetilde{S}_{g,n}\longrightarrow \Bbb C$ for $(X,\omega)$. Since the holonomy is assumed to be trivial, such a mapping descends to a well-defined map $\overline{\textsf{dev}}: S_{g,n} \to \mathbb{C}$. By fixing an arbitrary point, say $Q_o$, on $S_{g,n}$, this map can be written explicitly as
\begin{equation}
    P \mapsto \int_{Q_o}^{P}\omega
\end{equation}

\noindent By viewing $\mathbb{C} \subset \cp$, where $\cp$ is the Riemann sphere, we can extend $\overline{\textsf{dev}}$ holomorphically over the punctures by mapping the punctures to $\infty \in \cp$, thus giving us a holomorphic branched covering $S_g\longrightarrow \cp$ of the sphere which we again denote by $\overline{\textsf{dev}}$. At points of $S_g$ where $\omega$ has a pole of order $k$, the multiplicity of the map $\overline{\textsf{dev}}$ is $k-1$. Thus, $\infty$ is a singular value of $\overline{\textsf{dev}}$ except when all poles of $\omega$ have order $2$. Note that all poles have to be of order greater than $1$ since the residue at each pole has to be zero for the holonomy to be trivial. \\

\noindent If $\{ p_1, \ldots, p_n \}$ are the order of the poles, the degree of the mapping $\overline{\textsf{dev}}$ can be computed in terms of $p_1,\dots,p_n$; \emph{i.e.}
\begin{equation} \label{eqdeg}
\deg\big(\,\overline{\textsf{dev}}\,\big)=\sum_{i=1}^{n}(p_i - 1)=d.
\end{equation}
\noindent We now assume the differential $\omega$ to have a single zero. Its order has to be 
\[
2g-2+\sum_{i=1}^np_i
\] as a consequence of equation \eqref{zpec}. Since $\omega$ has a single zero by assumption, the mapping $\overline{\textsf{dev}}$ has only one singular value other than (possibly) $\infty$. Let $\{ r_1, \dots, r_m \}$ be the partition of $d=\deg\big(\,\overline{\textsf{dev}}\,\big)$ determined by the multiplicities of the preimages of this singular value. Since $\omega$ has a single zero, exactly one of these $r_j$'s is greater than $1$ and all the others are exactly $1$. Assuming $r_1 > 1$, the order of the zero at this point is therefore $r_1 - 1$. This implies that 
\begin{equation}\label{mideq}
    \sum_{j=1}^m (r_j-1) = \sum_{j=1}^m r_j - m = 2g-2+\sum_{i=1}^np_i
\end{equation}
On the other hand the family $\{r_1,\dots,r_m\}$ is a partition of $d$ and therefore
\begin{equation}\label{mideq2}
    \sum_{j=1}^m r_j=\sum_{i=1}^{n}(p_i - 1) = \sum_{i=1}^{n}p_i - n.
\end{equation}
\noindent By replacing this latter in the equation \ref{mideq} we obtain
\begin{equation}
    \sum_{i=1}^{n}p_i - n - m = 2g-2+\sum_{i=1}^np_i 
\end{equation}
\begin{equation}\label{nmg}
    n+m = 2 - 2g .
\end{equation}
Since $n$ and $m$ are at least one, equation \eqref{nmg} holds only when $g=0$, and in this case we have $n=m=1$ and $g=0$. Thus $\omega$ is a holomorphic differential on the whole complex plane. The point $\infty$ appears as a pole of order $p$ at the puncture, and a single zero of order $p-2$. With respect to the global coordinate $z$ the abelian differential $\omega$ can be written as $z^{(p-2)}dz$ on $\mathbb{C}$. We have thus proved the following proposition.

\begin{prop}\label{singlezero}
Let $(X,\omega)$ be a translation surface with trivial holonomy on $S_{g,n}$. Then $\omega$ has a single zero if and only if $g=0$ and $n=1$.\\
\end{prop}

\noindent We now allow $\omega$ to have multiple zeros, and obtain an upper bound for the order of the zeros of $\omega$. If we assume the mapping $\overline{\textsf{dev}}$ to have $k$ singular values other than (possibly) $\infty$, we get $k$ partitions of $\deg\big(\,\overline{\textsf{dev}}\,\big)$, say $\{ r_1^1, \dots, r_{n_1}^1 \}, \dots , \{ r_1^k, \dots, r_{n_k}^k \}$, given by the preimages of these $k$ branch values. The zeros of $\omega$ are precisely those points where $r_j^i > 1$, and the order of the zero at such points is $r_j^i - 1$. We can notice that each one of the $r_j^i\,$'s cannot exceed $\deg\big(\,\overline{\textsf{dev}}\,\big)$ and therefore deduce that the order of each zero of $\omega$ is at most equal to 
\[\sum_{i=1}^{n}p_i - n - 1
\] by using equation \eqref{mideq2}. We shall prove that these necessary conditions turn out to be also sufficient. The rest of this chapter is devoted to prove the following proposition, which is a re-statement of Theorem \ref{main:thmb} for punctured spheres.

\begin{prop} \label{thm:mainthm}
Given $n$ integers $\{ p_1, \ldots, p_n\}$ with each $p_i \geq 2$ and $k$ positive integers $\{ d_1, \ldots, d_k \}$ in non-increasing order, and a non-negative integer $g$ satisfying the following conditions, 
\begin{enumerate}
    \item $k>1$ whenever $n>1$, \label{cond2mainthm}\\
    \item $d_j \leq \displaystyle\sum_{i=1}^{n}p_i - n - 1$ for $1 \leq j \leq k$, and \label{cond3mainthm}
    \item $\displaystyle\sum_{j=1}^k d_j = \displaystyle\sum_{i=1}^{n}p_i -2$ \label{cond4mainthm}\\
\end{enumerate}
there exists a holomorphic abelian differential on $S_{0,n}$ which extends to a meromorphic differential with poles of orders $p_1,p_2,\ldots, p_n$ at the $n$ punctures, zeros of orders $d_1, d_2,\ldots, d_k$ and has trivial holonomy.
\end{prop}

\noindent In what follows, we shall henceforth refer to condition (2) as the \emph{order condition}, and condition (3) as the \emph{degree condition}.

\subsection{Idea of the proof and motivating example}\label{ssth} Our strategy would be to repeatedly use the sequential slit construction as described in section \ref{sec:seqslit} and split the zeros that arise as described in section \ref{sec:zerosplit} to control the order of the zeros. Consider  the $n$ translation surfaces $(X_i,\omega_i)$ for $1\leq i\leq n$, where $X_i=\C$ and the differential $\omega_i$ has a pole of order $p_i\ge2$ at the infinity. We shall consider some (possibly all) of these translation surfaces and we glue them, by performing a slit construction, to define a sequence of translation surfaces $(Y_j,\eta_j)$. Here, each $(Y_j,\eta_j)$ is a  sphere with a meromorphic differential,  having poles of order $p_1, \ldots, p_t$ for some $1\leq t \leq n$ and zeros of order $d_1, \ldots, d_{s-1}, \widetilde{d}_s$ for some $1\leq s\leq k$ where $\widetilde{d}_s \geq d_s$ and $\widetilde{d_s} \leq d_s+ d_{s+1} + \cdots+d_k$. Since the degree of any meromorphic differential on the sphere is $-2$, we have
\begin{equation}\label{eq:87} 
    \sum_{l=1}^{s} d_l + (\widetilde{d}_s - d_s) = \sum_{i=1}^{t}p_i -2.
\end{equation}
\noindent The sequence $(Y_j,\eta_j)$ is constructed such that, as $j$ increases, $s$ and $t$ increase and the final differential has $t=n$. We then do a final splitting of zeros to get the required differential.
\noindent At each step, to go from  $(Y_j,\eta_j)$ to  $(Y_{j+1},\eta_{j+1})$, we often need to consider an intermediate translation surface $(Z_j, \xi_j)$ obtained by splitting a zero in $(Y_j, \eta_j)$. To this, we glue some of the remaining translation surfaces $(X_i,\omega_i)$ to get $(Y_{j+1},\eta_{j+1})$ and the process continues till all the $(X_i,\omega_i)$'s are exhausted. Let us give an example for motivating our strategy.

\begin{ex}
Let $n=4$, $p_1 = p_3 = p_4 = 3$ and $p_2 = 5$. Let $k =5$, and $d_1 = 4$, $d_2=d_3=3$ and $d_4 = d_5 =1$. We see that this data satisfies the conditions of Proposition \ref{thm:mainthm}. To start with, we construct $(Y_1, \eta_1)$ using a sequential slit construction involving $(\C, z^{p_1-2}dz), (W, \tau)$ and $(\C, z^{p_3-2}dz)$ where $(W, \tau)$ is the surface obtained by splitting the zero of order 3 in $(\C, z^{p_2-2}dz)$ into two zeroes of order 1 and 2. This is shown in Figure \ref{fig:exy1eta1}. The resulting $(Y_1, \eta_1)$ has poles of order 3, 3 and 5 and two zeroes of order 4 and 5 respectively. We then split the zero of order 5 into two zeroes of order 2 and 3 and call the resulting surface $(Z_1,\xi_1)$. In $(Z_1,\xi_1)$, we make a slit across the two zeroes of order 2 and 3 and glue it to $(\C, z^{p_4-2}dz)$ as shown in Figure \ref{fig:exy2eta2} to get $(Y_2, \eta_2)$ that has poles of order 3, 5, 3 and 3 and zeroes of order 4, 3 and 5. Finally, we split the zero of order 5 into three zeroes of order 3, 1 and 1 and complete the construction.

 \begin{figure}[!h]
\centering
\begin{tikzpicture}[scale=0.75, every node/.style={scale=0.75}]
\definecolor{pallido}{RGB}{221,227,227}
\draw [dashed, black, pattern=north west lines, pattern color=pallido] (0,0) circle (2.2);
\draw (-1,-1) .. controls (0.1, -0.1) .. (1,1);
\draw (-1,-1) .. controls (-0.1, 0.1) .. (1,1);
\node[above left] at (-0.1, 0.1) {$l_1^+$};
\node[below right] at (0.1, -0.1) {$l_1^-$};
\node[below ] at (-1,-1) {$P_1$};
\node[above ] at (1,1) {$Q_1$};
\fill (-1,-1) circle (1pt);
\fill (1,1) circle (1pt);
\node[below] at (0,-2.3) {$(\C,z^{p_1-2}dz)$};
\draw [dashed, black, pattern=north west lines, pattern color=pallido] (5,0) circle (2.2);
\draw (4,-1) .. controls (5.1, -0.1) .. (6,1);
\draw (4,-1) .. controls (4.9, 0.1) .. (6,1);
\node[above left] at (4.9, 0.1) {$l_2^+$};
\node[below right] at (5.1, -0.1) {$l_2^-$}; 
\node[below ] at (4,-1) {$P_2$};
\node[above ] at (6,1) {$Q_2$};
\fill (4,-1) circle (1pt);
\fill (6,1) circle (1pt);
\node[below] at (5,-2.3) {$(W, \tau)$};
\draw [dashed, black, pattern=north west lines, pattern color=pallido] (10,0) circle (2.2);
\draw (9,-1) .. controls (10.1, -0.1) .. (11,1);
\draw (9,-1) .. controls (9.9, 0.1) .. (11,1);
\node[above left] at (9.9, 0.1) {$l_3^+$};
\node[below right] at (10.1, -0.1) {$l_3^-$}; 
\node[below ] at (9,-1) {$P_3$};
\node[above ] at (11,1) {$Q_3$};
\fill (9,-1) circle (1pt);
\fill (11,1) circle (1pt);
\node[below] at (10,-2.3) {$(\C,z^{p_3-2}dz)$};
\end{tikzpicture}
\caption{Construction of $(Y_1, \eta_1)$ for the motivating example. In the leftmost surface, $P_1$ is the singular point. In the rightmost surface, $Q_3$ is the singular point. In $(W, \tau)$, $P_2$ is a zero of order 1 and $Q_2$ is a zero of order 2.} \label{fig:exy1eta1}
\end{figure}
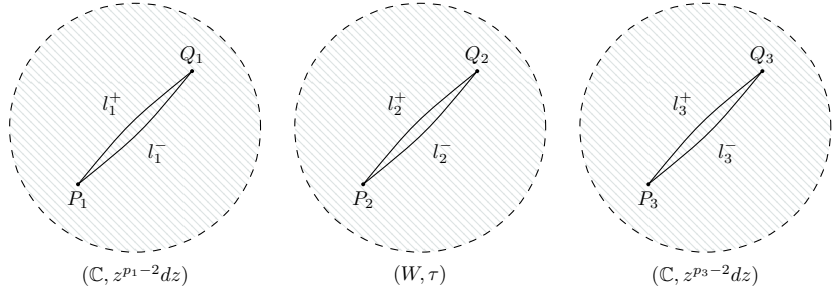

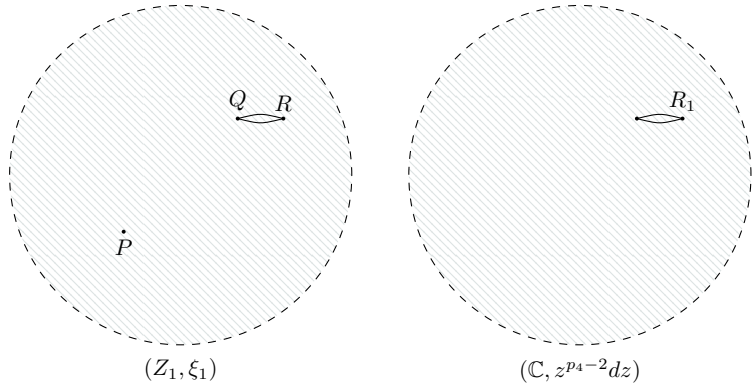
\begin{figure}[!h]
\centering
\begin{tikzpicture}[scale=0.75, every node/.style={scale=0.85}]
\definecolor{pallido}{RGB}{221,227,227}
\draw [dashed, black, pattern=north west lines, pattern color=pallido] (0,0) circle (3);
\draw (1,1) .. controls (1.4, 1.1) .. (1.8,1);
\draw (1,1) .. controls (1.4, 0.9) .. (1.8,1);
\node[below] at (-1,-1) {$P$};
\node[above] at (1,1) {$Q$};
\node[above] at (1.8,1) {$R$};
\fill (-1,-1) circle (1pt);
\fill (1,1) circle (1pt);
\fill (1.8,1) circle (1pt);
\node[below] at (0,-3.1) {$(Z_1,\xi_1)$};
\draw [dashed, black, pattern=north west lines, pattern color=pallido] (7,0) circle (3);
\draw (8,1) .. controls (8.4, 1.1) .. (8.8,1);
\draw (8,1) .. controls (8.4, 0.9) .. (8.8,1);
\node[above] at (8.8,1) {$R_1$};
\node[below] at (7,-3.1) {$(\mathbb{C},z^{p_4-2}dz)$};
\fill (8,1) circle (1pt);
\fill (8.8,1) circle (1pt);
\end{tikzpicture}
\caption{Construction of $(Y_2, \eta_2)$ for the motivating example. On the left, we have $(Z_1, \eta_1)$ where $P$ is the zero of order 4 and $Q$ and $R$ are the zeroes of order 2 and 3 obtained from splitting the zero of order 5 in $(Y_1, \eta_1)$. In the surface on the right, the right end of the slit, $R_1$ is the zero of order 1 in $(\mathbb{C},z^{p_4-2}dz)$.} \label{fig:exy2eta2}
\end{figure}
\end{ex}

\noindent The proof that follows is simply a generalisation of this procedure for arbitrary $p_i$'s and $d_j$'s subject to the conditions above.

\subsection{Proof of Proposition \ref{thm:mainthm}} According to our strategy, we divide the proof into steps, each one corresponding to a subsection.
\subsubsection{Construction of $(Y_1,\eta_1)$}\label{sec:y1eta1} Let $m$ be the smallest positive integer such that 
\begin{equation}\label{ineq1}
    \sum_{i=1}^m(p_i -1) \geq d_1 + 1.
\end{equation} 
Notice the existence of such an $m$ is guaranteed by the condition (2) of Proposition \ref{thm:mainthm}. If $m=1$, then define $(Y_1,\eta_1)$ as $(\C, z^{p_1-2}dz)$. The differential $\eta_1$ extends to a meromorphic differential on the Riemann sphere with one zero of order $\widetilde{d_1} = p_1 - 2 \geq d_1$ and one pole of order $p_1$. Clearly, 
\begin{equation}
    \widetilde{d_1} = p_1 - 2 \leq \sum_{j=1}^k d_j.
\end{equation} Thus, $(Y_1,\eta_1)$ is indeed of the desired form (as described in the strategy above). In the case $m>1$, we define $\widetilde{p}_m$ to be the integer for which the following equality holds, 
\begin{equation} \label{eq:orderd1}
    \sum_{i=1}^{m-1}(p_i-1) + (\widetilde{p}_m-1) = d_1 + 1.
\end{equation}


\noindent  Note that the minimality of $m$, and the inequality \eqref{ineq1} above implies that $2 \leq \widetilde{p}_m \leq p_m$. When $m>1$ the construction of $(Y_1,\eta_1)$ depends on the number $m + p_m - \widetilde{p}_m$. We shall distinguish two cases.\\

\paragraph{\textbf{Case 1 - $m + p_m - \widetilde{p}_m \geq d_2 + 1$}}\label{case1} This is the easiest case to deal with.  For every $i=1,\dots,m-1$, we consider the translation surfaces $(X_i,\omega_i)$ as $(\C, z^{p_i-2}dz)$. Consider the translation surface $(X_m,\omega_m)=(\C,z^{p_m-2}dz)$. Notice that the differential $\omega_m$ has a zero of order $p_m-2$. Let $(W,\tau)$ the translation surface obtained from $(X_m,\omega_m)$ by splitting the zero into two zeros of orders $\widetilde{p}_m-2$ and $p_m -\widetilde{p}_m$ at points $A$ and $B$ respectively. When $\widetilde{p}_m = 2$, or $\widetilde{p}_m= p_{m}$ there is no splitting involved and we have one zero and one marked point as $A$ and $B$ (or vice versa). For the sequential slit construction, we have to specify a geodesic line segment $l_i$ in each surface $(X_i,\omega_i)$, such that all of them have the same developed image $c \in \mathbb{C}-\{0\}$. We take $l_m$ to be the saddle connection joining $A$ to $B$. In the notation of section \ref{sec:seqslit}, $P_m = A$ and $Q_m = B$. Thus, $P_m$ is a singular point with angle $2\pi(\widetilde{p}_m-1)$ and $Q_m$ is a singular point with angle $2\pi(p_m - \widetilde{p}_m+1)$. Here, a singular point with angle $2\pi$ is just a regular point. For $1\leq i \leq m-1$, we take $l_i$ to be a geodesic line segment starting from the zero in $(X_i, \omega_i)$ to some point on the surface such that the line segment has the same developed image in $\mathbb{C}$ as $l_m$. Thus, for $1\leq i \leq m-1$, $P_i$ is a singular point with angle $2\pi(p_i-1)$ and $Q_i$ is a regular point. We glue all the surfaces $(X_i,\omega_i)$, for $i=1,\dots,m-1$, to $(W,\tau)$ by slit construction along the segments $l_i,$ just defined. We label the resulting surface as $(Y_1,\eta_1)$ and we show this is of the required form. The point $P$, which is the point obtained from the identification of all the $P_i$, is a singular point with magnitude
\begin{equation}
    2\pi\bigg(\sum_{i=1}^{m-1}(p_i-1) + (\widetilde{p}_m-1)\bigg) = 2\pi(d_1+1)
\end{equation}
\noindent Therefore, the differential $\eta_1$ has a zero of order $d_1$ at the point $P$. The point $Q$ is also singular and has magnitude equal to
 $ 2\pi(d+1) = 2\pi( m + p_m - \widetilde{p}_m) \geq 2\pi(d_2+1)$, where the equality defines $d$. In particular, $\eta_1$ has a zero of order $d \geq d_2$ at $Q$. We can finally deduce $(Y_1,\eta_1)$ is of the desired form because 
 \[d_1 + d = \sum_{i=1}^m p_i - 2 \leq \sum_{j=1}^k d_j \quad \text{ implies } \quad d \leq \sum_{j=2}^k d_j.
 \]
 \medskip
 
\paragraph{\textbf{Case 2 - $m + p_m - \widetilde{p}_m < d_2 + 1$}} In this case, if we were to carry out the construction as in case 1, then we would have obtained a differential with two zeros of order $d_1$ and $d$ respectively, with $d<d_2$ and poles of order respectively $p_1, \ldots, p_m$. 
This is not of the form desired of $(Y_1,\eta_1)$. So, we describe a construction that will involve more surfaces $(X_{m+1},\omega_{m+1}), \ldots, (X_{m'},\omega_{m'})$. Let $m'$ be the smallest integer such that the following inequality holds,
\begin{equation} \label{eq:defm'}
    m + p_m - \widetilde{p}_m + \sum_{i=m+1}^{m'}p_i \geq d_2 + 1
\end{equation}
Clearly $m' \geq m+1$ since $m + p_m - \widetilde{p}_m < d_2 + 1$. Using equation \eqref{eq:orderd1}, we see that inequality \eqref{eq:defm'} is the same as,
\begin{equation}\label{condmprime}
     \sum_{i=1}^{m-1}(p_i-1) + (\widetilde{p}_m-1) + m + p_m - \widetilde{p}_m + \sum_{i=m+1}^{m'}p_i \geq d_1 +1 + d_2 + 1 \quad
     \text{ or, }\quad \sum_{i=1}^{m'} p_i \geq d_1 + d_2 + 2
\end{equation}
In this form, it is easy to see that such an $m'$ exists because of the degree condition of the main theorem. We claim that $m'\leq d_2+1$. To see this, we first examine the case when $m'=m+1$. In this case, $m + p_m - \widetilde{p}_m < d_2 + 1$ implies $m' = m+1 \leq d_2 +1$. When $m' > m+1$, we have, 
\begin{equation}\label{eq:814}
     m + (m'-1-m) < m + p_m - \widetilde{p}_m + \sum_{i=m+1}^{m'-1}p_i < d_2 + 1
\end{equation}
\noindent Here, the second inequality comes from the minimality of $m'$ and the first inequality is a consequence of the fact that each $p_i > 1$ since that implies  $\sum\limits_{i=m+1}^{m'-1}p_i > m'-m$. Equation \eqref{eq:814} implies  $m' < d_2 + 2$, or, $m' \leq d_2 + 1$. Since $d_2 \leq d_1$, we also have $m' \leq d_1 +1$. We now define $m''\geq 1$ to be the smallest integer such that 
\begin{equation}\label{eq:815}
m' + \sum_{i=1}^{m''}(p_i-2) \geq d_1 + 1.
\end{equation} 
\noindent  In general, such an $m'' \leq m'$ exists because 
\[m' + \sum_{i=1}^{m'}(p_i-2) = \sum_{i=1}^{m'}p_i - m' \geq d_1+d_2+2-m' \geq d_1 + 1,
\] where the inequality in the middle follows from \eqref{condmprime} and the last inequality follows because $m' \leq d_2 + 1$. As before, we define $p'_{m''}$ as the integer for which the following equality holds,
\begin{equation}
    m' + \sum_{i=1}^{m''}(p_i-2) + (p'_{m''}-p_{m''}) = d_1 + 1
\end{equation}
\noindent  We note that $p_{m''} \geq p'_{m''} \geq 2$, where the first inequality comes from \eqref{eq:815}, and  the second inequality comes from the minimality of $m''$. In particular, $p'_{m''} = 2$ only when $m''=1$ and $m' = d_1 + 1$.\\ 

\noindent We are now in the right position to specify the surfaces $(X_1,\omega_1), \ldots, (X_{m'},\omega_{m'})$ which we shall glue via sequential slit construction to get $(Y_1,\eta_1)$. For $i \neq m''$, we take ($X_i,\omega_i)$ to be $(\C, z^{p_i-2}dz)$ as above. In this second case, $(W,\tau)$ will be the translation surface obtained by splitting the zero of order $p_{m''}-2$ in $(X_{m''},\omega_{m''})$ into two zeros of orders $p'_{m''}-2$ and $p_m - p'_{m''}$ at points $A$ and $B$ respectively. When $p'_{m''} = 2$, or $p'_{m''} = p_{m''}$, as before, there is no splitting involved and we have one zero and one marked point as $A$ and $B$ (or vice versa).\\
\noindent Recall that for the sequential slit construction, we have to specify a geodesic line segment $l_i$ for each translation surface $(X_i,\omega_i)$. We take $l_{m''}$ to be the saddle connection joining $A$ to $B$. By using the notation of section \ref{sec:seqslit}, we denote $P_{m''} = A$ and $Q_{m''} = B$. Thus, $P_{m''}$ is a singular point with angle $2\pi(p'_{m''}-1)$ and $Q_{m''}$ is a singular point with angle $2\pi(p_{m''} - p'_{m''}+1)$.  For $1\leq i \leq m''-1$, we take $l_i$ to be a geodesic line segment starting from the zero to some other point on the surface such that the line segment has the same developed image in $\mathbb{C}$ as $l_{m''}$. Thus, for $1\leq i \leq m''-1$, $P_i$ is a singular point with magnitude $2\pi(p_i-1)$ and $Q_i$ is a regular point. For $m''+1 \leq i \leq m'$, $l_i$ is chosen such that $P_i$ is a regular point and $Q_i$ is a singular point with magnitude $2\pi(p_i-1)$. This means that $P$, which is obtained from the identification of all the $P_i$, is a singular point with magnitude
\begin{equation} \label{eq:angleP}
    \sum_{i=1}^{m''-1}2\pi(p_i-1) + 2\pi(p'_{m''}-1) + \sum_{i=m''+1}^{m'}2\pi = 2\pi \bigg( m' + \sum_{i=1}^{m''}(p_i-2) + (p'_{m''}-p_{m''}) \bigg) = 2\pi( d_1 + 1).
\end{equation}
In the same fashion, the point $Q$, obtained from the identification of all the $Q_i$, is a singular point with magnitude
\begin{equation} \label{eq:angleQ}
    \sum_{i=1}^{m''-1}2\pi + 2\pi(p_{m''} - p'_{m''}+1) + \sum_{i=m''+1}^{m'}2\pi(p_i-1) =: 2\pi(d+1)
\end{equation}
The equations \eqref{eq:angleP} and \eqref{eq:angleQ} together yield
\begin{equation}
    d_1 + 1 + d + 1 = \sum_{i=1}^{m'}p_i \geq d_1 + d_2 + 2
\end{equation}
This means that the resulting abelian differential has a zero of order $d \geq d_2$ at $Q$. We also have,
\begin{equation}
    d_1 + 1 + d + 1 = \sum_{i=1}^{m'}p_i \leq \sum_{j=1}^k d_j + 2
\end{equation}
which implies that $d \leq d_2+\cdots +d_k$. Thus, we can label this translation surface as $(Y_1,\eta_1)$. 

\subsubsection{Getting $(Z_{j},\xi_j)$ from $(Y_j,\eta_j)$ and $(Y_{j+1},\eta_{j+1})$ from $(Z_j,\xi_j)$} \label{subsec:yj+1fromyj}

\noindent This is an intermediate step that we shall undertake only when $t<n$ in the surface $(Y_j,\eta_j)$. Recall that $t\le n$ is the number of poles on $(Y_j, \eta_j)$. If $t=n$, then we directly move to the next step. \\ 

\noindent When $(Y_j,\eta_j)$ has less than $n$ poles, suppose it has $s$ zeros of orders $d_1, \ldots, d_{s-1}, \widetilde{d}_s$, where $s < k$ and $\widetilde{d}_s < d_s+\cdots+d_k.$ Let define $r$ to be the largest integer in $s\le r\le k-1$ such that $\widetilde{d}_s \geq d_s+\cdots+d_r.$ Clearly, the following inequality holds $ \widetilde{d}_s < d_s+\cdots+d_{r+1}.$ 
Let us now define $t'$ to be the smallest index in $\{1,\dots,n-t\}$ for which 
\begin{gather}
    \widetilde{d}_s+ \sum_{i=t+1}^{t+t'} p_i \geq \sum_{j=s}^{r+1} d_j \label{eq:deft'} \\
    \Longleftrightarrow \quad \sum_{j=1}^{s-1}d_j + \widetilde{d}_s + \sum_{i=t+1}^{t+t'} p_i \geq \sum_{j=1}^{r+1} d_j \\
    \Longleftrightarrow \quad \sum_{i=1}^{t+t'} p_i - 2 \geq \sum_{j=1}^{r+1} d_j
\end{gather}
where the last implication uses \eqref{eq:87}.
In the last form, it is easy to see that such a $t'$ has to exist by the degree condition. Since each $p_i >1$, we have that in \eqref{eq:deft'} the term $ \sum\limits_{i=t+1}^{t+t'} p_i > t'$, and it easily follows that $t' \leq d_{r+1}$. Since $d_i$'s are non increasing, we also have $t' \leq d_r$. \\ 

\noindent We now have all that is necessary for constructing $(Z_j,\xi_j)$ from $(Y_j,\eta_j)$, namely,  split the zero of order $\widetilde{d}_s$ in $(Y_j,\eta_j)$ into zeros of order $d_s, \ldots, d_{r-1}, d_r - t'$ and $\widetilde{d}_s -
(d_s+\cdots+d_r-t')$, and call the resulting surface $(Z_j,\xi_j)$. 

\begin{rmk}
When $d_r = t'$ we interpret the ``zero of order $d_r - t'$" as a marked regular point on the surface. 
\end{rmk}

\noindent We now do a sequential slit construction involving $(Z_j,\xi_j)$ and $(X_{t+1},\omega_{t+1}), \ldots, (X_{t+t'},\omega_{t+t'})$, where we take $(X_i,\omega_i)=(\C,z^{p_i-2}dz)$. The slit in $(Z_j,\xi_j)$ is along the saddle connection joining the zero of order $d_r - t'$ and the zero of order $\widetilde{d}_s -
(d_s+\cdots+d_r-t')$. We label the former point $P_0$ and the latter point $Q_0$. In the surfaces $(X_i,\omega_i)$, the slit is made along the geodesic line segment joining the zero of order $p_i-2$, labelled $Q_i$ and some marked point $P_i$, such that the slit has the same developed image as the slit in $(Z_j,\xi_j)$. The resulting singularity $P$ is a zero of order $d_r$. The angle at $Q$ has magnitude 
\begin{equation}
    2\pi\bigg(\widetilde{d}_s - \sum_{j=s}^r d_j + t' + 1 + \sum_{i=t+1}^{t+t'}(p_i -1)\bigg) = 2\pi\bigg(1 + \widetilde{d}_s + \sum_{i=t+1}^{t+t'}p_i - \sum_{j=s}^r d_j \bigg) \geq 2\pi(1+d_{r+1})
\end{equation}
where the last inequality comes from \eqref{eq:deft'}. We now have the surface $(Y_{j+1},\eta_{j+1})$ with poles of order $p_1, \ldots, p_{t+t'}$ and zeros of order $d_1, \ldots, d_r, \widetilde{d}_{r+1}$. We iterate this step until we obtain a surface with $n$ poles, in which case, we move on to the next step.



\subsubsection{Final step} Suppose we reach this step with $(Y_1,\eta_1)$. This happens when $(Y_1,\eta_1)$ has $n$ poles $p_1, \ldots, p_n$. Then, 
\[d_1 + \widetilde{d}_2 = \sum_{i=1}^n p_i - 2 \quad\text{ which means that }\quad \widetilde{d}_2 = \sum_{j=2}^k d_j.
\] We can then split this zero of order $\widetilde{d}_2$ into zeros of order $d_2, \ldots, d_k$ using the procedure described in section \ref{sec:zerosplit}. The resulting surface is the surface that we had set out to construct. More generally, if $(Y_j,\eta_j)$ has $n$ poles $p_1, \ldots, p_n$, then we have 
\[\sum_{l=1}^{s} d_l + (\widetilde{d}_s - d_s) = \sum_{i=1}^{n}p_i -2.
\] This tells us that $\widetilde{d}_s= d_s+\cdots+d_k$. In this case, we split the zero of order $\widetilde{d}_s$ into zeros of order $d_s, \ldots, d_k$ and this completes the construction. 

\subsubsection{About the positioning of the zeros} \label{sec:zeropos}
For the purposes of this section, the locations of the zeros (in the sense of their image under the developing map) does not matter. However, they become important in section \ref{sec:trivholposgen} and Corollary \ref{consthmb2}. Let $d_1, \ldots d_k$ be the zeroes of a differential as obtained in this section. We desire that, for every $1 < i < k$, the distance between the zero of order $d_i$ and $d_j$, for $j>i$ be non-zero, and smaller than the distance between the zero of order $d_i$ and $d_{i-1}$. This also ensures that no two zeros have the same developed image. In addition, if the developed image of the $n^{th}$ zero is assumed to be $z_n \in \mathbb{C}$, we can ensure that $\vert z_n \vert > \vert z_{n-1} \vert$. \\
\noindent It is sufficient to enforce that the distance between the zeros of order $d_j$ and $d_{j+1}$ should be smaller than a third of the distance between the distance between the zeros of order $d_{j-1}$ and $d_j$. This is easy to do when splitting the zero of order.  $\widetilde{d}_2$ in $(Y_1,\eta_1)$ to obtain $(Z_1,\xi_1)$. If the saddle connection between the two zeros (the line segment $PQ$) has length $L$, then the points in $(Y_1,\eta_1)$ at a distance of less than $\frac{2L}{3}$ from $Q$ give us a neighbourhood of $Q$ of the form required in section \ref{sec:zerosplit}. We can then ensure that the first split of the zero at $Q$ gives us two zeros separated by less than $\frac{L}{3}$. Since we have freedom to choose both the length and the direction of the saddle connection separating these two zeros, we can even make sure that this saddle connection is along a direction different from $PQ$. For the next split, we look at a $\frac{2L}{9}$ neighbourhood of the zero and so on. An illustrative example is given in Figure \ref{fig:zeropos}. Such splitting automatically ensures that a slit made for the construction of $(Y_2,\eta_2)$ has length of the form $\frac{L}{3^m}$. We see that points at a distance less than $\frac{2L}{3^{m+1}}$ from the last zero of $(Y_2,\eta_2)$ still give us a neighbourhood of the form specified in section \ref{sec:zerosplit}, and we can continue splitting zeros separated by lengths $\frac{L}{3^{m+1}}, \frac{L}{3^{m+2}}, \ldots$.\\

\begin{figure}
    \centering
    \begin{tikzpicture}[scale=0.6, every node/.style={scale=0.8}]
    \definecolor{pallido}{RGB}{221,227,227}
    \draw [thin, dashed, pattern=north west lines, pattern color=pallido] (0,0) circle (5);
    \fill (210:2) circle (1pt);
    \node[below left] at (210:2) {$P$};
    \fill (30:2) circle (1pt);
    \node[above left] at (30:2) {$Q$};
    \draw [dashed] (30:2) circle (2.4);
    \draw [->] (30:2) -- +(2.4,0);
    \node[below] at ($(30:2)+(1,0)$) {$< \frac{2}{3}\vert PQ \vert$};
    \draw [thin, dashed, pattern=north west lines, pattern color=pallido] ($(30:2)+ (10,0)$) circle (4);
    \draw [thin, dashed, pattern=north west lines, pattern color=pallido] ($(30:2)+ (12,0)$) circle (1.2);
    \fill ($(30:2)+ (10,0)$) circle (1pt);
    \fill ($(30:2)+ (12,0)$) circle (1pt);
    \node[below] at ($(30:2)+ (10,0)$) {$Q_1$};
    \node[below] at ($(30:2)+ (12,0)$) {$Q_2$};
    \draw[->] (30:4.4) parabola[bend pos=0.5] bend +(0,1) ($(30:2)+(10,0)+(150:4)$);
    \node at (6,4) {after splitting};
    \draw[->] ($(30:2)+ (12,0)+(30:1.2)$) parabola[bend pos=0.5] bend +(0,1) ($(30:2)+ (16,0)+(30:1.2)$);
    \foreach \x in {1,2,..., 5}
    {
    \fill ($(30:2)+ (17,0) + 0.2*(\x,1)$) circle (1pt);
    }
    \end{tikzpicture}
    \caption{Specifying the location of the zeros. Here, the zero at $Q$ is split into two zeros at $Q_1$ and $Q_2$. The zero at $Q_2$ is split further and the resulting zeros are in the neighbourhood shown.}
    \label{fig:zeropos}
\end{figure}
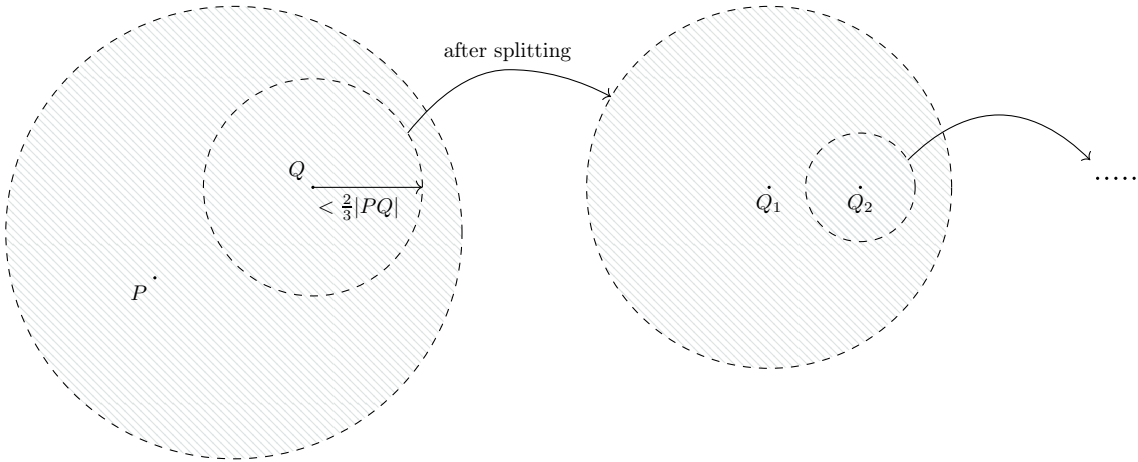

\section{Trivial holonomy II: Positive genus surfaces}\label{sec:trivholposgen}

\noindent  This section is devoted to prove Theorem \ref{main:thmb} for surfaces of positive genus, which is handled by the following statement.

\begin{prop} \label{thm:mainthm2}
Given $n$ integers $\{ p_1, \ldots, p_n\}$ with each $p_i \geq 2$ and $k\ge2$ positive integers $\{ d_1, \ldots, d_k \}$ in non-increasing order, and a non-negative integer $g$ satisfying the following conditions, 
\begin{enumerate}
    \item $d_j \leq \displaystyle\sum_{i=1}^{n}p_i - n - 1$ for $1 \leq j \leq k$, and \label{cond3mainthm}
    \item $\displaystyle\sum_{j=1}^k d_j = \displaystyle\sum_{i=1}^{n}p_i + 2g -2$ \label{cond4mainthm}\\
\end{enumerate}
there exists a meromorphic abelian differential on $S_{g,n}$  with poles of order $p_1,p_2,\ldots, p_n$ at the $n$ punctures, zeros of orders $d_1,d_2,\ldots, d_k$, and trivial holonomy.
\end{prop}

\noindent Notice that $k\ge2$ is an essential assumption because of Proposition \ref{singlezero}. We shall continue to call the two conditions (1) and (2) in the statement of the Proposition, as the \textit{order condition} and \textit{degree condition}, respectively.\\

\noindent The strategy of the proof is to first perform a series of $g$ reductions (described below) on the set $\{ d_1, \ldots, d_k \}$ to obtain another non-increasing sequence $\{ d_1, \ldots, d_{l-2}, d'_{l-1}, d'_l \}$ where $2\leq l\leq k$  that satisfies the conditions of Proposition \ref{thm:mainthm} of the $g=0$ case in the previous section. We  therefore obtain a meromorphic differential on the sphere with zeros of orders $\{ d_1, \ldots, d'_{l-1}, d'_l \}$ and poles of orders $\{ p_1, \ldots, p_n\}$; denote this translation structure by $(W_0,\tau_0)$. We then add $g$ handles, one by one, to obtain a sequence of translation surfaces $(W_i,\tau_i)$ for $1\leq i \leq g$. At each step, the set of orders of the zeros changes by undoing each reduction step, and the final $(W_g,\tau_g)$ is our desired translation surface.


\subsection{Reduction procedure and a motivating example} 
Before we describe a motivating example for the strategy we just outlined, we explain the reduction procedure mentioned in the strategy above. Each step of this process consists of applying one of the following moves which we are going to describe.

\noindent Given a set of positive integers $\{ e_1, \ldots, e_l \}$ indexed in non-increasing order with $l \geq 2$, 

\begin{itemize}
    \item if the last two integers are both greater than $1$, we reduce it to $\{ e_1, \ldots, e_{l-2}, e_{l-1}-1, e_l-1 \}$,
    \item if the last integer is $1$, we reduce it to $\{ e_1, \ldots, e_{l-2}, e_{l-1}-1 \}$. This is the same as the previous case with $e_l - 1$ not being included because it is $0$,
    \item if the last two integers are $1$, we reduce it to $\{ e_1, \ldots, e_{l-2} \}$.
\end{itemize}

\noindent We note that each reduction step only makes the last two integers of the resulting sequence different, smaller in particular, from the original sequence, if at all. This explains the form $\{ d_1, \ldots, d_{l-2}, d'_{l-1}, d'_l \}$. Each step of the process also maintains the non-increasing order of the integers.

\noindent Moreover, each reduction decreases the sum of the set of integers by $2$ and, after performing $g$ reductions, we can find a set of integers $\{ d_1, \ldots, d_{l-2}, d'_{l-1}, d'_l \}$  that sastifies the degree condition for the sphere ($g=0$). \\

\noindent Before proceeding, we verify that $g$ reductions can be performed. Indeed, there are two possible obstructions to applying the reduction process $g$ times, which we eliminate:
\begin{itemize}
    \item[\textbf{i.}]  We could end up with the empty set after $h<g$ reductions. However, if this happens, then it would mean that 
    \[ \sum_{j=1}^k d_j = 2h
    \] and this contradicts condition \ref{cond4mainthm} of Proposition \ref{thm:mainthm2}.\\
    \item[\textbf{ii.}] We could end up with $\{ e \}$, for some $e \geq 1$, after $h<g$ reductions. However, if this happens, we would have 
    \[ e + 2h = \displaystyle\sum_{j=1}^k d_j = \displaystyle\sum_{i=1}^n p_i + 2g - 2.
    \] Using the order condition, we then obtain 
    \[\sum_{i=1}^n p_i + 2(g-h) - 2 \leq \sum_{i=1}^n p_i - n - 1,
    \] or $0 < 2(g-h) \leq -n + 1 \leq 0$,  which gives us the desired contradiction.\\
\end{itemize}
\noindent Thus, applying the reduction process $g$ times on the set $\{ d_1, \ldots, d_k \}$ is always possible.  Finally, we need to check that the set of integers $\{ d_1, \ldots, d'_{l-1}, d'_l \}$ satisfies the conditions in Proposition \ref{thm:mainthm2} for $g=0$. The order condition is automatically satisfied and the degree condition holds by design. Suppose $l=2$ with $d'_l = 0$. We then have, 
\[d'_1 + 2g = \sum_{j=1}^k d_j = \sum_{i=1}^n p_i + 2g - 2,\quad \text{ or } \quad d'_1 = \sum_{i=1}^n p_i - 2.
\] But the order condition implies 
\[ d'_1 \leq \sum_{i=1}^n p_i - n - 1,
\] which means $n \leq 1$. Thus, condition \ref{cond2mainthm} of the Proposition holds as well and we have a translation surface $(W_0,\tau_0)$ as required. Let  $\{ d_1, \ldots, d_{l-2}, d'_{l-1}, d'_l \}$  denote the set obtained after performing $g$ reductions on $\{ d_1, \ldots, d_k \}$.\\
\noindent At this point, it is useful to consider an example as done in the previous chapter. 

\begin{ex} \label{posgex}
Let $g=3$, $p_1 =4$, $p_2 = 5$, $d_1 = 5$, $d_2=4$, $d_3 = 3$, $d_4 = 1$. Then, the reduction proceeds as follows - 
\begin{equation}
    \{ 5,4,3,1 \} \to \{ 5,4,2 \} \to \{ 5,3,1 \} \to \{ 5,2 \}
\end{equation}
Following Proposition \ref{thm:mainthm}, we obtain $(W_0, \tau_0)$ with poles of order 4 and 5 and zeroes of order 5 and 2 at points (say) $Q$ and $R$ respectively. Let $R'_1$ be a point different from $R$ such that $QR$ and $QR'_1$ have the same developed image. Such a point exists because we have 6 rays emanating from $Q$ that have the same developed image as the ray from $Q$ in the direction of $R$ and at most 2 such rays can have the point $R$ at a distance $\vert QR \vert$ from $Q$.\\
\noindent We now choose points $S'_1$ and $S$ sufficiently near $R'_1$ and $R$ resp. (ensuring that the requirements of \ref{sec:zeropos} hold) such that $RS$ and $R'_1S'_1$ have the same developed image. Making slits along $RS$ and $R'_1S'_1$ as shown in Figure \ref{fig:exw1tau1} and identifying appropriately adds a handle to $(W_0, \tau_0)$ and gives us $(W_1, \tau_1)$ with zeroes of order 5, 3 and 1 at points $Q, R$ and $S$ resp. \\
\noindent Now let $R'_2$ be a point different from $R$ and $R'_1$ (which are now identified) such that $QR$ and $QR'_2$ have the same developed image. Such a point lies in one of the remaining 4 gray rays in Figure \ref{fig:exw1tau1}.\\
Picking $S'_2$ similarly and proceeding as before gives us $(W_2, \tau_2)$ of genus 2 with zeroes of order 5, 4 and 2 at points $Q, R$ and $S$ resp.
\noindent For the next step, we look at $S'_3$ different from $S$ (and the points identified with $S$) near $R$ such that $RS$ and $RS'_3$ have the same developed image. As before, the orders of the zeros at $R$ and $S$ tell us that such an $S'_3$ exists. Considering points $T$ and $T'_3$ near $S$ and $S'_3$ resp. and proceeding as before gives us the final translation surface $(W_3, \tau_3)$ with poles of order 4 and 5 and zeroes of order 5, 4, 3 and 1.

\begin{figure}[!h]
    \centering
    \begin{tikzpicture}
    \definecolor{pallido}{RGB}{221,227,227}
    \draw [thin, dashed, pattern=north west lines, pattern color=pallido] (0,0) circle (35mm);
    \draw [red] (-1,0) -- (1,0);
    \node[below left] (p) at (-1.1,0) {$Q$};
    \fill (-1,0) circle (1pt);
    \node[below right] at (1,0) {$R$};
    \draw[thin] (1,0) .. controls +(37:0.25) .. +(45:0.5);
    \draw[thin] (1,0) .. controls +(53:0.25) .. +(45:0.5);
    \fill (1,0) circle (1pt);
    \fill ($(1,0)+ (45:0.5)$) circle (1pt);
    \node[right] at ($(1,0)+ (45:0.5)$) {$S$};
    \fill ($(-1,0)+(60:2)$) circle (1pt);
    \fill ($(-1,0)+(60:2)+(90:0.5)$) circle (1pt);
    \node[left] at ($(-1,0)+(60:2)$) {$R'_1$};
    \node[left] at ($(-1,0)+(60:2)+ (90:0.5)$) {$S'_1$};
    \draw[thin] ($(-1,0)+(60:2)$) .. controls +(82:0.25) .. +(90:0.5);
    \draw[thin] ($(-1,0)+(60:2)$) .. controls +(98:0.25) .. +(90:0.5);
    \draw[thin, gray][->] (-1, 0) -- +(60:2.5);
    \draw[thin, gray][->] (-1, 0) -- +(120:1);
    \draw[thin, gray][->] (-1, 0) -- +(180:1);
    \draw[thin, gray][->] (-1, 0) -- +(-60:1);
    \draw[thin, gray][->] (-1, 0) -- +(-120:1);
   
    \begin{scope}[shift = {(-1,0)}]
    \node[right] at (60:2.5) {$r_1$};
    \node[above] at (120:1) {$r_2$};
    \node[left] at (180:1) {$r_3$};
    \node[below] at (240:1) {$r_4$};
    \node[below] at (300:1) {$r_5$};
    \end{scope}
    \end{tikzpicture}
    \caption{Construction of $(W_1, \tau_1)$ in example \ref{posgex}. The gray rays $r_i$ have the same developed image as $QR$ and the slits $RS$ and $R'_1S'_1$ have the same developed image.} \label{fig:exw1tau1}
\end{figure}
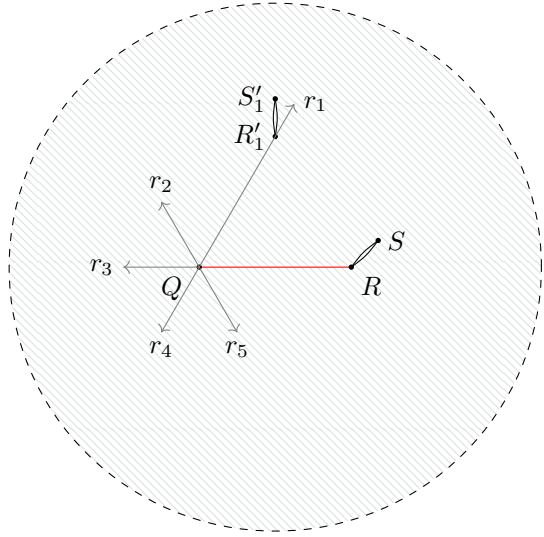
\end{ex}

\noindent The rest of the proof is a generalisation of the procedure carried out in example \ref{posgex}.

\subsection{Proof of Proposition \ref{thm:mainthm2}} As in the previous section, our proof is divided into steps, each  corresponding to a subsection.

\subsubsection{Finding suitable points for handle construction} \label{subsec:handleptchoice}
In this section, we single out suitable points in $(W_i,\tau_i)$ that we shall use for performing a handle construction and obtain a new translation surface $(W_{i+1},\tau_{i+1})$. The existence of such points will be expressed in terms of two properties we are going to introduce below. We assume that the orders of the zeros of $(W_i,\tau_i)$ are $\{ \ldots, a,b,c \}$ and denote the last three zeros of $(W_i,\tau_i)$ by $P$, $Q$ and $R$ respectively. In case of $(W_0,\tau_0)$, the point $P$ denotes the zero of order $d_{l-2}$, the point $Q$ denotes the zero of order $d'_{l-1}$ and finally $R$ denotes the zero of order $d'_l$. The point $P$ need not always exist, if there are only two zeros in $(W_i,\tau_i)$, we only look at $Q$ and $R$.

\begin{py}\label{propy} Let us consider the following properties.
\begin{enumerate}
    \item \label{itm:prop1} When $b - c > 0$, there exist $b-c$ many regular points $R'_1, \ldots , R'_{(b-c)}$ different from $R$ such that $QR'_i$ is a geodesic line segment having the same developed image as $QR$.\\
    \item \label{itm:prop2} When $a-b > 0$, there exist $a-b$ many regular points $Q_1, \ldots , Q_{(a-b)}$ different from $Q$ such that $PQ_i$ is a geodesic line segment with the same developed image as $PQ$, and near each $Q_i$ there exists a point $R_i$ such that $Q_iR_i$ is a geodesic line segment with same developed image as $QR$. 
\end{enumerate}
\end{py}

\noindent In the sequel we shall refer to these properties simply as property \ref{itm:prop1} and \ref{itm:prop2}. We shall proceed as follows. We begin with by showing that the properties just mentioned above hold for $(W_0,\tau_0)$. We then move on to describe how to attach a handle when the properties hold. Finally, we show that whenever the properties \ref{propy} hold for $(W_i,\tau_i)$, then they also hold for ($W_{i+1},\tau_{i+1})$. \\

\noindent It is sufficient to show the existence of $R'_i$'s, $Q_i$'s and $R_i$'s that have the required developing image. The fact that these points are regular follows from the positioning of the zeros enforced in section \ref{sec:zeropos}. Moreover, except for the point $P$ or $Q$, no other point in the geodesic line segments in the subsequent arguments shall be singular points.  

\subsubsection{Properties \ref{propy} hold for $(W_0,\tau_0)$} \label{subsec:propywo} This follows from the following Lemmata.

\begin{lem}
Property (\ref{itm:prop1}) of \ref{propy} holds for $(W_0,\tau_0)$.
\end{lem}

\begin{proof} It suffices to show that whenever $d'_{l-1} - d'_l > 0$, there exist $d'_{l-1} - d'_l$ many points $R'_1, \ldots , R'_{(d'_{l-1} - d'_l)}$ different from $R$ such that $Q\,R'_i$ is a geodesic segment with the same developed image as the geodesic segment $Q\,R$. To see this, we look at the $d'_{l-1} + 1$ many rays emanating from $Q$ in the direction of $QR$. At most $d'_l + 1$ many of these rays can have the point $R$ at a distance $\vert QR \vert$ from $Q$. We then choose the points $R'_i$ at a distance of $\vert QR \vert$ from $Q$ in the remaining $d'_{l-1} - d'_l$ rays. 
\end{proof}

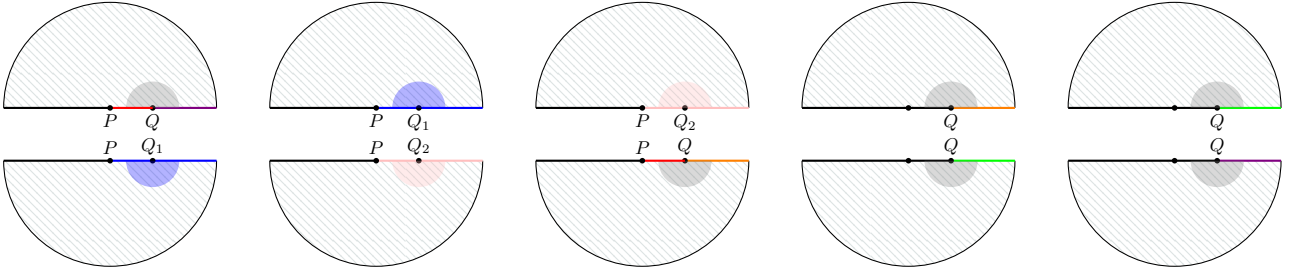
\begin{figure}
    \centering
    \begin{tikzpicture}[scale=0.7, every node/.style={scale=0.7}]
    \definecolor{pallido}{RGB}{221,227,227}
    \foreach \x [evaluate=\x as \coord using 4 + 5*\x] in {0, 1, ..., 4} 
    {
    \draw [pattern=north west lines, pattern color=pallido] (\coord,0) arc [start angle = 0,end angle = 180,radius = 2];
    \draw [pattern=north west lines, pattern color=pallido] (\coord,-1) arc [start angle = 0,end angle = -180,radius = 2];
    }
    \foreach \x [evaluate=\x as \leftend using 5*\x] [evaluate=\x as \rightend using 2 + 5*\x] in {0, 1, 2} 
    {
    \draw [thick] (\leftend, 0) -- (\rightend, 0);
    \draw [thick] (\leftend, -1) -- (\rightend, -1);
    \node[below] at (\rightend, 0) {$P$};
    \node[above] at (\rightend, -1) {$P$};
    }
    \foreach \x / \y [evaluate=\x as \leftend using 5*\x] [evaluate=\x as \rightend using 2.8 + 5*\x] [evaluate=\x as \splpt using 2 + 5*\x] in {3/1, 4/2} 
    {
    \draw [thick] (\leftend, 0) -- (\rightend, 0);
    \draw [thick] (\leftend, -1) -- (\rightend, -1);
    \fill (\splpt, 0) circle (1.5pt);
    \fill (\splpt, -1) circle (1.5pt);
    }
    
    \foreach \botindex / \topindex / \colr [evaluate=\botindex as \botleftend using 5*\botindex + 2] [evaluate=\botindex as \botrightend using 5*\botindex + 4] [evaluate=\topindex as \topleftend using 5*\topindex + 2] [evaluate=\topindex as \toprightend using 5*\topindex + 4] [evaluate=\topindex as \topsplpt using 5*\topindex + 2.8] [evaluate=\botindex as \botsplpt using 5*\botindex + 2.8] [evaluate=\topindex as \topcircle using 5*\topindex + 3.3] [evaluate=\botindex as \botcircle using 5*\botindex + 3.3] in {0/1/blue, 1/2/pink} 
    {
    \draw [thick, \colr] (\topleftend, 0) -- (\toprightend, 0);
    \draw [thick, \colr] (\botleftend, -1) -- (\botrightend, -1);
    \fill (\topleftend, 0) circle (1.5pt);
    \fill (\botleftend, -1) circle (1.5pt);
    \fill (\topsplpt, 0) circle (1.5pt);
    \fill (\botsplpt, -1) circle (1.5pt);
    \node[below] at (\topsplpt, 0) {$Q_{\topindex}$};
    \node[above] at (\botsplpt, -1) {$Q_{\topindex}$};
    \fill[\colr, opacity=0.3] (\topcircle,0) arc [start angle = 0,end angle = 180,radius = 0.5] -- cycle;
    \fill[\colr, opacity=0.3] (\botcircle,-1) arc [start angle = 0,end angle = -180,radius = 0.5] -- cycle;
    }
    
    \foreach \botindex / \topindex / \colr [evaluate=\botindex as \botleftend using 5*\botindex + 2.8] [evaluate=\botindex as \botrightend using 5*\botindex + 4] [evaluate=\topindex as \topleftend using 5*\topindex + 2.8] [evaluate=\topindex as \toprightend using 5*\topindex + 4] [evaluate=\topindex as \topcircle using 5*\topindex + 3.3] [evaluate=\botindex as \botcircle using 5*\botindex + 3.3] in {2/3/orange, 3/4/green, 4/0/violet} 
    {
    \draw [thick, \colr] (\topleftend, 0) -- (\toprightend, 0);
    \draw [thick, \colr] (\botleftend, -1) -- (\botrightend, -1);
    \fill (\topleftend, 0) circle (1.5pt);
    \fill (\botleftend, -1) circle (1.5pt);
    \node[below] at (\topleftend, 0) {$Q$};
    \node[above] at (\botleftend, -1) {$Q$};
    \fill[gray, opacity=0.3] (\topcircle,0) arc [start angle = 0,end angle = 180,radius = 0.5] -- cycle;
    \fill[gray, opacity=0.3] (\botcircle,-1) arc [start angle = 0,end angle = -180,radius = 0.5] -- cycle;
    }
    
    \draw [thick, red] (2,0) -- (2.8,0);
    \fill (2,0) circle (1.5pt);
    \draw [thick, red] (12,-1) -- (12.8,-1);
    \fill (12,-1) circle (1.5pt);
    \end{tikzpicture}
    \caption{In the case the points $P$ and $Q$, each of order $2$, are obtained after splitting a zero of order $4$, the points $Q_i$ are as shown above. 
    }
    \label{fig:prop2}
\end{figure}

\begin{lem}
Property (\ref{itm:prop2}) of \ref{propy} holds for $(W_0,\tau_0)$.
\end{lem}

\begin{proof} We need to show that when $d_{l-2} - d'_{l-1} > 0$, there exist $d_{l-2} - d'_{l-1}$ many points $Q_1, \ldots , Q_{(d_{l-2} - d'_{l-1})}$ such that $P\,Q_i$ is a geodesic line segment with the same developed image as $P\,Q$, and near each $Q_i$ there exists a point $R_i$ such that $Q_i\,R_i$ has the same developed image as $Q\,R$. As before we can find the points $Q_i$ on the $d'_{l-2} - d'_{l-1}$ rays emanating from $P$ along $PQ$ that do not lead to $Q$. Since we have positioned zeros according to the conditions mentioned in section \ref{sec:zeropos}, these $Q_i$'s are regular points that have neighborhoods isomorphic to disks of radius bigger than $\vert QR \vert $. In these neighbourhoods, we find the points $R_i$. For example, when $P$ and $Q$ each having order 2 are obtained after splitting a zero of order 4, the points $Q_i$ are as given in Figure \ref{fig:prop2}. Here, $R$ is located inside the gray neighborhood, and we obtain $R_i$ in the other shaded neighbourhoods of $Q_i$.
\end{proof}

\subsubsection{Adding the first handle} \label{sec:firsthandle} We wish to add a handle to the surface $(W_0,\tau_0)$, and get $(W_1,\tau_1)$ such that the zeros of the differential $\tau_1$ are the ones that we had just before the last reduction step. For instance, if the set of orders $\{ d_1, \ldots, d'_{l-1}, d'_l \}$ was obtained from $\{ d_1, \ldots, d'_{l-1}, d'_l, 1, 1 \}$, then we want the differential $\tau_1$ to have zeros of orders $\{ d_1, \ldots, d'_{l-1}, d'_l, 1, 1 \}$ with the orders of the poles remaining unchanged. We look at this process case by case.\\

\paragraph{\textbf{Case 1 - $\{ d_1, \ldots, d'_{l-1}, d'_l \}$ is empty}} This can happen only when the following condition holds:
\[\sum_{j=1}^k d_j = 2g.\] This, along with the degree condition implies that $n=1$ and $p_1 = 2$. But then, the order condition would imply that $d_j = 0$ for all $j$. This leads to a contradiction and so we conclude that $\{ d_1, \ldots, d'_{l-1}, d'_l \}$ can never be empty.\\

\paragraph{\textbf{Case 2 - $d'_l = 0$ in $\{ d_1, \ldots, d'_{l-1}, d'_l \}$ and $l=2$}} We can notice that $\{ d'_1 \}$ can only be obtained after a reduction from $\{ d'_1 + 1, 1 \}$ or $\{ d_1, 1,1 \}$, and in the latter case, $d'_1 = d_1$. We first rule out the former case. Recall that Proposition \ref{singlezero} says that when $(W_0,\tau_0)$ has only one zero then there is only one pole of order $p_1$ and $d'_1 = p_1 - 2$. On the other hand, the order condition  implies that $d'_1 + 1 \leq p_1 - 1 - 1$ and this leads to a contradiction. Thus, $\tau_0$ has to be the meromorphic differential on the sphere with a single zero of order $d_1 = p_1 - 2$ and a single pole of order $p_1$ is the one given by $z^{(p_1-2)}dz$ on $\Bbb C=\cp \backslash \{ \infty \}$. We use the singular point for attaching a handle with trivial holonomy by using the construction introduced in subsection \ref{trhan}. Once the handle is attached, we obtain a structure such that the resulting differential has two additional zeros near the zero of order $d_1$. We now have a meromorphic differential on the torus with three zeros of orders $\{ d_1, 1,1 \}$ respectively and a single pole of order $p_1$. It is easy to check that the positions of the new zeros can be made to satisfy the requirements of section \ref{sec:zeropos}. \\


\paragraph{\textbf{Case 3 - $d'_l \neq 0$ in $\{ d_1, \ldots, d'_{l-1}, d'_l \}$ and $l \geq 2$} }  When $d_l'\neq0$ we shall consider in turn different sub-cases.
\begin{enumerate}
    \item If $\{ d_1, \ldots, d'_{l-1}, d'_l \}$ is obtained by reduction from $\{ d_1, \ldots, d'_{l-1}, d'_l, 1, 1 \}$, we shall employ the construction as described in the previous case. 
    \item If $\{ d_1, \ldots, d'_{l-1}, d'_l \}$ is obtained by reduction from $\{ d_1, \ldots, d'_{l-1}, d'_l+1, 1 \}$, $d'_{l-1} > d'_l$. We pick a point $R'_1$ as in section \ref{subsec:handleptchoice}. Let $S'_1$ and $S$ be points near $R'_1$ and $R$ respectively such that $R'_1S'_1$ and $RS$ have the same developed image. We make slits across $R'_1S'_1$ and $RS$ and glue them to get the handle. Once again, we check that the choice of $S$ and $S'_1$ can be made to satisfy the requirements of section \ref{sec:zeropos}.
    \item If $\{ d_1, \ldots, d'_{l-1}, d'_l \}$ is obtained by reduction from $\{ d_1, \ldots, d'_{l-1} + 1, d'_l+1\}$ and $l>2$, that means that $d_{l-2} > d'_{l-1}$, we pick a points $Q_1$ and $R_1$ as in section \ref{subsec:handleptchoice}. We make slits across $Q_1R_1$ and $QR$ and glue them to get the handle.
    \item If $\{ d'_1, d'_2 \}$ is obtained by reduction from $\{ d'_1 +1, d'_2 + 1\}$, then it means that $d_1 \geq d'_1 + 1$ and $d_2 \geq d'_2 + 1$. Using the order condition, we obtain 
    \[ d'_1 \leq \sum_{i=1}^n p_i - n - 2
    \] which in turn implies $d'_2 \geq n$. Reversing the indices, we also have $d'_1 \geq n$. This tells us that, in the slits $P_iQ_i$ in the construction $(W_0,\tau_0)$ in section \ref{sec:necsuftriv} (see, for example, Figure \ref{fig:exy1eta1}), there is at least one singular point among the $P_i$'s and at least one singular point among the $Q_i$'s (If we have regular points at both the ends of the slits, then a slit construction involving $n$ slits produces zeros of order $n-1$ at the identified ends of the slits). If $P_{i_1}$ is a singular point, we can find $\widetilde{Q}_{i_1}$ such that $P_{i_1}\widetilde{Q}_{i_1}$ and $P_{i_1}Q_{i_1}$ have the same developed image. Similarly, if $Q_{i_2}$ is a singular point, we can find $\widetilde{P}_{i_2}$ such that $Q_{i_2}\widetilde{P}_{i_2}$ and $Q_{i_2}P_{i_2}$ have the same developed image. A picture of this situation can be seen in Figure \ref{fig:spltrivhandle}. If no slits were made for the construction of $(W_0,\tau_0)$, then the two zeros of this structure were obtained by splitting, in which case the geodesic line segments described in section \ref{sec:zerosplit} with the same developed image as the saddle connection joining the zeros gives the required $\widetilde{Q}_{i_1}$ and $\widetilde{P}_{i_2}$. In the sequential slit construction involved in obtaining $(W_0,\tau_0)$, the points $P_{i_1}$ gets identified to the point $P$ and the points $Q_{i_2}$ gets identified to the point $Q$. We make slits across $P\widetilde{Q}_{i_1}$ and $Q\widetilde{P}_{i_2}$ and glue them appropriately. Since both slits are made on the same surface, gluing them has the effect of adding a handle. Also since both ends of the slits are identified to an existing zero, this does not introduce any additional zeros. Instead, the effect of slit construction is to increase the orders of the zeros at the ends of the slits by $1$. 
\end{enumerate}

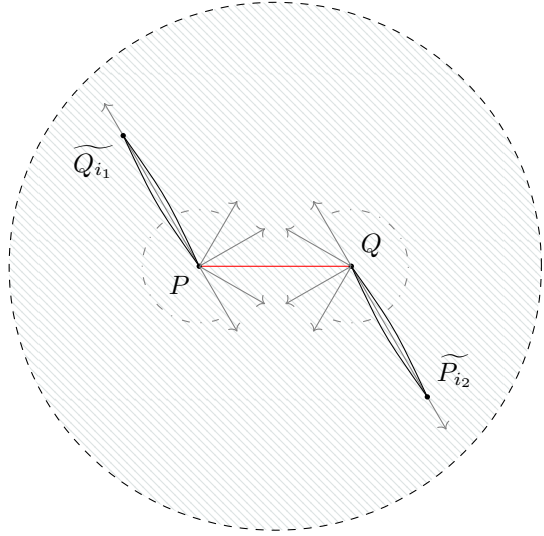
\begin{figure}
    \centering
    \begin{tikzpicture}
    \definecolor{pallido}{RGB}{221,227,227}
    \draw [thin, dashed, pattern=north west lines, pattern color=pallido] (0,0) circle (35mm);
    \draw [red] (-1,0) -- (1,0);
    \node[below left] (p) at (-1,0) {$P$};
    \fill (-1,0) circle (1pt);
    \node[above right] at (1,0) {$Q$};
    \fill (1,0) circle (1pt);
    \draw[thin, gray][->] (-1, 0) -- +(30:1);
    \draw[thin, gray][->] (-1, 0) -- +(60:1);
    \draw[thin, gray][->] (-1, 0) -- +(120:2.5) node[below left, black] at ++(120:2) {$\widetilde{Q_{i_1}}$};
    \draw[thin, gray][->] (-1, 0) -- +(-30:1);
    \draw[thin, gray][->] (-1, 0) -- +(-60:1);
    \draw (-1,0) .. controls +(115:1) .. +(120:2);
    \draw (-1,0) .. controls +(125:1) .. +(120:2) coordinate (tq);
    \fill (tq) circle (1pt);
    \draw[gray, loosely dash dot] (-1,0)+(130:0.75) arc [start angle = 130,end angle = 290,radius = 0.75];
    \draw[gray, loosely dash dot] (-1,0)+(70:0.75) arc [start angle = 70,end angle = 110,radius = 0.75];
    \draw[thin, gray][->] (1, 0) -- +(120:1);
    \draw[thin, gray][->] (1, 0) -- +(150:1);
    \draw[thin, gray][->] (1, 0) -- +(210:1);
    \draw[thin, gray][->] (1, 0) -- +(240:1);
    \draw[thin, gray][->] (1, 0) -- +(-60:2.5) node[above right, black] at +(-60:2) {$\widetilde{P_{i_2}}$};
    \draw (1,0) .. controls +(-55:1) .. +(-60:2);
    \draw (1,0) .. controls +(-65:1) .. +(-60:2) coordinate (tp);
    \fill (tp) circle (1pt);
    \draw[gray, loosely dash dot] (1,0)+(110:0.75) arc [start angle = 110,end angle = -50,radius = 0.75];
    \draw[gray, loosely dash dot] (1,0)+(-70:0.75) arc [start angle = -70,end angle = -110,radius = 0.75];
    \end{tikzpicture}
    \caption{Illustration for (4) of case 3. All the rays in gray have the same developed image}
    \label{fig:spltrivhandle}
\end{figure}

\subsubsection{Adding subsequent handles} To get $(W_2,\tau_2)$ from $(W_1,\tau_1)$, we need to add a handle to $(W_1,\tau_1)$ such that the zeros of the resulting differential are the zeros obtained by undoing the reduction step that gave us the zeros of $(W_1,\tau_1)$. We shall use the same idea to get $(W_{i+1}, \tau_{i+1})$ from $(W_i,\tau_i)$ for any $2\le i\le g$. We first get to a point where $(W_{i+1},\tau_{i+1})$ has at least three zeros. We then show that obtaining $(W_{i+1},\tau_{i+1})$ from $(W_i,\tau_i)$ is possible when $(W_i,\tau_{i})$ satisfies two properties mentioned in subsection \ref{subsec:handleptchoice}. Finally we show that if $(W_i,\tau_i)$ satisfies two properties mentioned in subsection \ref{subsec:handleptchoice}, then so does $(W_{i+1},\tau_{i+1})$. This will enable us to get to $(W_g,\tau_g)$ from $(W_1,\tau_1)$. This is required only when $g \geq 2$, for $g=1$ the previous section suffices. We distinguish two cases according to the number of zeros of $(W_{i+1},\tau_{i+1})$.\\

\paragraph{\textbf{Case A:} \textit{When $(W_{i+1},\tau_{i+1})$ has exactly two zeros}} We first note that $(W_1,\tau_1)$ has at least two zeros and that $(W_{i+1},\tau_{i+1})$ can have exactly two zeros only when $(W_i,\tau_i)$ has exactly two zeros. This means the structure $(W_0, \tau_0)$ we begin with has exactly two zeros of orders $\{ d'_1, d'_2 \}$ which necessarily satisfies the relation
\[ d'_1 + d'_2 = \sum_{i=1}^n p_i - 2.
\] Let $h$ be the integer satisfying $1 \leq h \leq g$ such that $(W_i,\tau_i)$ has zeros of orders $\{ d'_1 + i, d'_2 + i \}$ for $1 \leq i \leq h$, and $(W_{h+1},\tau_{h+1})$ (if any) has zeros of orders $\{ d'_1 + h, d'_2 + h + 1, 1 \}$ or $\{ d'_1 + h, d'_2 + h , 1, 1 \}$. Thus, once $(W_h,\tau_h)$ is constructed and we have shown that it satisfies the two properties mentioned in section \ref{subsec:handleptchoice}, we can move to the next section. Since we want $(W_h,\tau_h)$ to have zeros of orders $\{ d'_1 + h, d'_2 + h \}$, it means that $d'_1 + h \leq d_1$. The order condition then tells us that 
\[ d'_1 + h \leq \sum_{i=1}^n p_i - n - 1,
\] and since 
\[ d'_1 + d'_2 = \sum_{i=1}^n p_i - 2,
\] we have $d'_2 \geq n-1 + h$. Switching the indices, we derive the same lower bound for $d'_1$.  Following the notation of section \ref{ssth}, we have seem that, in the construction of $(W_0,\tau_0)$, the $n$ slits $P_i\,Q_i$ end up giving us $n$ saddle connections between the two zeros $P$ and $Q$ of $(W_0, \tau_0)$. Since the order of the zero at $P$ is greater than or equal to $n-1 + h$, we have at least $h$ geodesic segments $\big \{ P\,\widetilde{Q}_j \big \}_{1 \leq j \leq h}$ that have the same developed image as $P\,Q$. Similarly, we have at least $h$ geodesic segments $\big \{ \widetilde{P}_j\,Q \big \}_{1 \leq j \leq h}$, that have the same developed image as $P\,Q$. Such $\big \{ P\widetilde{Q}_j \big \}_{1 \leq j \leq h}$ and $\big \{ \widetilde{P}_j\,Q \big \}_{1 \leq j \leq h}$ can be found even when the zeros of $(W_0,\tau_0)$ are obtained by splitting a zero instead of sequential slit construction. In this case, $n=1$ and since we have zeros of order $d'_1, d'_2 \geq h$ obtained be splitting a zero of order $d'_1 + d'_2$, the the geodesic line segments described in section \ref{sec:zerosplit} are the required $\big \{ P\widetilde{Q}_j \big \}_{1 \leq j \leq h}$ and $\big \{ \widetilde{P}_j\,Q \big \}_{1 \leq j \leq h}$. For each $1 \leq j \leq h$, we make slits across $P\,\widetilde{Q}_j$ and $\widetilde{P}_j\,Q$ and glue them. This has the effect of adding $h$ handles. For each handle, one of the ends is $P$ and the other is $Q$, and each handle increases the degree of the zero at each end by 1. This gives us the required $(W_h,\tau_h)$ with zeros of orders $\{ d'_1 + h, d'_2 + h \}$.  The second property of subsection \ref{subsec:handleptchoice} is trivially true for $(W_h,\tau_h)$ since it has only two zeros. To check the first property, we note that $d'_2 \geq n-1 + h$ implies $d'_1 \geq n-1 + h + (d'_1 - d'_2)$. Thus, in addition to the $\widetilde{Q}_j$ of the previous paragraph, we can find $d'_1 - d'_2$ many points, that we denote by $Q''_j$ (for $1\leq j \leq d'_1-d'_2$) such that $PQ''_j$ has the same developed image as $P\,Q$. This shows that the first property is satisfied.\\

\paragraph{\textbf{Case B:} \textit{When $(W_{i+1},\tau_{i+1})$ has more than two zeros}} We assume that the orders of the zeros of $(W_i, \tau_i)$ are $\{ \ldots, \alpha, \beta, \gamma \}$ and that $(W_i,\tau_i)$ satisfies the two properties mentioned in subsection \ref{subsec:handleptchoice}. Following the notation of subsection \ref{subsec:handleptchoice}, we denote the last three zeros of $(W_i,\tau_i)$ by $P$, $Q$ and $R$ respectively. Recall that the points $R'_i$, $Q_i$, and $R_i$ satisfy properties 1 and 2 of section \ref{subsec:handleptchoice}. If we want $(W_{i+1},\tau_{i+1})$ to have zeros of orders $\{ \ldots, \alpha, \beta + 1, \gamma + 1 \}$, it can only happen when $\alpha > \beta >0$. Making slits across $Q_1\,R_1$ and $Q\,R$ and gluing, we obtain $(W_{i+1},\tau_{i+1})$. If we want $(W_{i+1},\tau_{i+1})$ to have zeros of orders $\{ \ldots, \alpha, \beta, \gamma + 1, 1 \}$, it can only happen when $\beta > \gamma$. In this case, we consider the point $R'_1$, and let $S'_1$ and $S$ be points near $R'_1$ and $R$ respectively such that $R'_1S'_1$ and $RS$ have the same developed image. We make slits across $R'_1S'_1$ and $RS$ and glue them to get the handle. Finally, if we want $(W_{i+1},\tau_{i+1})$ to have zeros of orders $\{ \ldots, \alpha, \beta, \gamma, 1, 1 \}$, we use the singular point of order $\gamma$ for attaching a handle with trivial holonomy by using the construction introduced in subsection \ref{trhan}.

\subsubsection{Verifying the two properties for $(W_{i+1},\tau_{i+1})$} In this section, we show that if $(W_i,\tau_i)$ satisfies the two properties mentioned in subsection \ref{subsec:handleptchoice}, then they also hold for the structure $(W_{i+1},\tau_{i+1})$, obtained in the way described previously. As before, we assume that the orders of the zeros of $(W_i,\tau_i)$ are $\{ \ldots, \alpha, \beta, \gamma \}$. While constructing $(W_{i+1},\tau_{i+1})$ which has zeros of orders $\{ \ldots, \gamma, 1, 1 \}$, the first property is vacuously true for $(W_{i+1},\tau_{i+1})$ since the second to last zero does not have higher order than the last zero. For the second property, we note that the two arcs $\delta_1$ and $\delta_2$ as in subsection \ref{trihan} were chosen out of $\gamma + 1$ many choices of arcs that have the same developed image. This means that even after handle construction, we can find arcs $\delta_3, \ldots, \delta_{\gamma+1}$ that have the same developed image as $\delta_1$ and $\delta_2$. The ends of these additional $\delta_i$'s serve as the choice of $Q_i$'s and $R_i$'s to satisfy the second property.\\
\noindent Next, we look at the case when $(W_{i+1},\tau_{i+1})$ has zeros of orders $\{ \ldots, \beta, \gamma + 1, 1 \}$. Since $R$ is a zero of order $\gamma$, there are $\gamma$ many points $S'_i$ such that $R\,S'_i$ has the same developed image as $R\,S$. Thus, the first property holds for $(W_{i+1},\tau_{i+1})$. To check the second property, we note that there are $\beta - \gamma$ many $R'_i$ for which $Q\,R'_i$ and $Q\,R$ have the same developed image. Since $R'_1$ is used for the handle construction, we have $\beta - \gamma - 1$ many $R'_i$ after the handle is constructed. For each of these $\beta - \gamma - 1$ many $R'_i$ we have $S'_i$ such that $R'_i\,S'_i$ has the same developed image as $R'_1\,S'_1$. This means that the second property holds for $(W_{i+1},\tau_{i+1})$. \\
\noindent The case when $(W_{i+1},\tau_{i+1})$ has zeros of orders $\{ \ldots, \alpha, \beta + 1, \gamma + 1 \}$ only occurs when $(W_{i+1},\tau_{i+1})$ has at least three zeros since we have separately dealt with the case where $W_{i+1}$ has exactly two zeros. Here, $\alpha > \beta$ and we have $\alpha - \beta$ many $Q_i$ and $R_i$ satisfying the second property. For the handle construction, one pair of $Q_i$ and $R_i$ is utilized, say, $Q_1$ and $R'_1$. The rest of the points show us that the second property is satisfied for $(W_{i+1},\tau_{i+1})$. The $\beta - \gamma$ many points $R'_i$ satisfying the first property for $(W_i,\tau_i)$ are not affected by this handle construction, and give us $(\beta + 1) - (\gamma +1)$ many points satisfying the first property for $(W_{i+1},\tau_{i+1})$.\\

\noindent This completes the inductive step and the proof of Proposition \ref{thm:mainthm2}; we obtain a sequence of translation surfaces $(W_i,\tau_i)$ for $1\leq i\leq g$ as described in the strategy of the proof at the beginning of the section, and $(W_g, \tau_g)$ is the desired translation surface. Together with Proposition \ref{thm:mainthm}  (that handles the $g=0$ case), this completes the proof of Theorem \ref{main:thmb} as stated in the Introduction.   

\subsection{Further consequences of Theorem \ref{main:thmb}} The techniques developed in the previous subsections can be used to prove a special case of the Hurwitz Existence Problem. For the reader's convenience, we briefly recall our notation. Let $S$ and $\Sigma$ be two surfaces, let $d\ge2$ be a positive integer and consider a collection $\mathcal{B}$ of $n$ partitions of $d$ denoted as $B_1,\dots,B_n$ where $B_i=\{d_{ij}\}_{1\le j\le m_i}$. The string $\mathcal{D}=\big(S,\Sigma,d,n,\mathcal{B}\big)$ yields an abstract branch datum whenever the equation
\medskip
\begin{equation} \label{eq:rh}
    \chi(S)-\widetilde{n}=d\cdot\big(\chi(\Sigma)-n\big)
\end{equation}
\medskip
holds, where $\widetilde{n}$ is the sum of the lengths of the partitions $B_i$. Recall that an abstract branch datum is \emph{realizable} if there exists a branched covering $f:S\longrightarrow \Sigma$ that realizes $\mathcal{D}$. The following statement holds.

\begin{cor}\label{consthmb2}
Let $\mathcal{D}=\big(S_g, \Bbb S^2, d, n, \mathcal{B}\big)$ be an abstract branch datum such that 
\begin{itemize}
    \item $\widetilde{n}= 2-2g+ d \cdot (n-2)$, which is nothing but equation \eqref{eq:rh}.
    \item $B_i\ni d_{ij} = 1$ whenever $i \neq 1$ and $j \neq 1$.
\end{itemize}
Then $\mathcal{D}$ is realizable.
\end{cor}

\begin{proof} We begin noticing that since $\widetilde{n}>0$, when $n=1$, we must have $g=0$, $d=1$ and $\widetilde{n}=1$. Thus, this case is degenerate and henceforth we assume $n\ge2$. Secondly, we note that if $d_{ij} = 1$ for all $i$ and $j$, then we must have $d=1$ and $g=0$, which means the identity map gives us the result. We may assume that for each $i$ there exists some $j$ for which $d_{ij} \neq 1$. To see this, we reduce the given datum $\big(S_g, \Bbb S^2, d, n, \mathcal{B}\big)$ to $\big(S_g, \Bbb S^2, d, m, \mathcal{B}'\big)$ for $m<n$ by removing those $i$ for which $d_{ij}=1$ for all $1 \leq j \leq m_i$. If the branch datum $\big(S_g, \Bbb S^2, d, m,\mathcal{B}'\big)$ is realizable, then $\big(S_g, \Bbb S^2, d, n, \mathcal{B}\big)$ is also realizable by adding the trivial branching data of some $n-m$ regular points.  From the branch datum $\big(S_g, \Bbb S^2, d, n, \mathcal{B}\big)$, we define the following datum that shall satisfy the hypotheses of Theorem \ref{main:thmb}:
\begin{itemize}
    \item for $1 \leq j \leq m_1$, define $p_j = d_{1j} + 1$,
    \item for $2 \leq i \leq n$, define $d_{i-1} = d_{i1}-1$.
\end{itemize}
Note that since we are looking only at non-trivial partitions, and $d_{ij}=1$ when $i \geq 2, j\neq1$, $d_{i1}>1$ and so $d_{i1}-1>0$. We now verify that $p_1, \ldots, p_{m_1}$ and $d_1, \ldots, d_{n-1}$ satisfy the requirements in the hypotheses of Theorem \ref{main:thmb}:
\begin{itemize}
\item First, since $\displaystyle\sum_{i=1}^{n-1}d_i = (n-1)d - (\widetilde{n}- m_1)$. Substituting $\widetilde{n} = 2-2g+ d \cdot (n-2)$, we have $$\displaystyle\sum_{i=1}^{n-1}d_i = d + m_1 + 2g-2 = \sum_{j=1}^{m_1}p_j + 2g -2$$
    and the degree condition is satisfied.
\item The requirement (i) that $p_i \geq 2$ is clear from the definition.
    
\item The requirement (ii) follows from the facts $\displaystyle \sum_{j=1}^{m_1} p_j - m_1 - 1 = d - 1$ and $d_{ij} \leq d$.

\item Finally, whenever $g=0$ and $n=2$, we have $m_1 = m_2 = 1$. This means that as soon as $m_1 > 1$, we must have $n>2$, namely $(n-1)>1$, which is requirement (iii).
\end{itemize}

\noindent We now use Theorem \ref{main:thmb} to obtain a meromorphic differential with trivial holonomy and consider the extended developing map of the corresponding translation structure. The branching data at $\infty$ gives us the partition $B_1$. All other branch values are images of zeros of the differential and the images of all the zeros are branch values. Going through the details of the construction of such a differential with trivial holonomy in the preceding sections, we note that the developed images of the zeros are all different (\textit{c.f.}  Section \ref{sec:zeropos}) . Thus the branching data at all branch values other than infinity are of the desired form $B_i$ for $i \geq 2$. 
\end{proof}

\noindent For similar results that provide sufficient criteria for the realizability of branch data, the reader can consult \cite{BK} and \cite{PePe}, and the references therein.

\section{Spheres with non-trivial holonomy} \label{sec:nontrivholsphere}

\noindent Let $\chi_n : \Gamma_{0,n} \to \mathbb{C}$ be any non-trivial representation. Notice that $n\ge2$ necessarily. This section is devoted to prove the following refinement of Proposition \ref{propb}.

\begin{prop}\label{hgc:mainprop}
Let $\chi_n : \Gamma_{0,n} \to \mathbb{C}$ be a non-trivial representation. Let $\{p_1, p_2, \ldots, p_n\}$ and $\{d_1, d_2, \ldots, d_k\}$ positive integers satisfying the following properties. 
\begin{itemize}
    \item Either \begin{itemize}
        \item[i.] $\textsf{\emph{Im}}(\chi_n)$ is not contained in the $\mathbb{R}$-span of some $c \in \mathbb{C}$,
        \item[ii.] $\textsf{\emph{Im}}(\chi_n)$ is not contained in the $\mathbb{Q}$-span of some $c \in \mathbb{C}$
        \item[iii.] at least one of $p_1, p_2, \ldots ,p_n$ is different from 1,
    \end{itemize}
    \item $p_i \geq 2$ whenever $\chi_n (\gamma_i) =0$,
    \item $\displaystyle\sum_{j=1}^k d_j = \sum_{i=1}^n p_i -2$.
\end{itemize}
Then $\chi_n$ appears as the holonomy of a translation structure on $S_{0,n}$ determined by a meromorphic differential on $\cp$ with zeros of orders $d_1, d_2,\ldots, d_k$ and a pole of order $p_i$ at the puncture enclosed by $\gamma_i$, for each $1\leq i\leq n$. 
\end{prop}

\begin{rmk}
The remaining case of "rational representations" when $\textsf{Im}(\chi_n)$ is contained in the $\mathbb{Q}$-span of some $c \in \mathbb{C}$, and  $p_1,p_2,\ldots , p_n$ are all $1$, is discussed in section \ref{sec:comb}; our Theorem \ref{main:thme} provides necessary and sufficient conditions in that case. 
\end{rmk}

\noindent We shall need the following technical lemma: 

\begin{lem} \label{lem:irrmult}
Let $s_1, \ldots, s_n$ and $t_1, \ldots, t_m$ be positive real numbers such that 
\begin{itemize}
    \item $\displaystyle\sum_{i=1}^n s_i = \sum_{j=1}^m t_j$,
    \item there exists a pair of numbers in $\{ s_1, \ldots, s_n, t_1, \ldots, t_m \}$ with irrational ratio.
\end{itemize}
Then, there exists a reordering of $s_i$'s and $t_j$'s such that 
\begin{equation}\label{combcond}
    \sum_{i=1}^k s_i = \sum_{j=1}^l t_j
\end{equation}
holds only for $k=n$ and $l=m$.
\end{lem}

\noindent For the sake of the exposition we shall postpone the proof of this lemma to Appendix \ref{app:irrmultproof}.  

\begin{proof}[Proof of Proposition \ref{hgc:mainprop}]
\noindent Let $J := \{1 \leq i \leq n : \chi_n(\gamma_i) \neq 0\}$, and let $\vert J \vert = m$. Note that $J$ is non empty because $\chi_n$ is assumed to be non-trivial. We may assume, by reordering, that $\{ 1,2, \ldots, m \} \in J$ and $\{ m+1,m+2, \ldots, n \} \notin J$. We shall proceed case by case according to the items $(i)$ and $(ii)$ in the list of Proposition \ref{hgc:mainprop}.\\

\noindent \textbf{Case 1.} We first assume that $\textsf{Im}(\chi_n)$ is not contained in the $\mathbb{R}$-span of some $c \in \mathbb{C}$ and proceed along the lines of the proof of Proposition \ref{propb}. We consider $\chi_n(\gamma_i)$ for $i \in J$ and construct the polygon as in the proof of \ref{propb}. The condition on the image implies that the polygon obtained in this way is not degenerate. Proceeding with the construction, we obtain a translation structure on the $m-$punctured sphere where the corresponding meromorphic differential has poles of order $1$ at the punctures and a single zero of order $m-2$. After gluing $p_i - 1$ many copies of $\mathbb{E}^2$ along $r_i^+ = r_i^-$ as described in subsection \ref{lgp} , we obtain a surface with $m$ poles $\{P_1,\dots,P_m\}$ such that each of the $P_i$ has order $p_i$ and holonomy given by $\chi_n(\gamma_i)$ for every $i \in J$. The resulting surface, call it $(Y,\eta)$, has a single zero of appropriate order determined by the degree condition. We then add poles of order $p_i \geq 2$ for $i \notin J$ by performing sequential slits construction. To $(Y,\eta)$, we glue surfaces $(X_i,\omega_i)$ which are determined by the meromorphic differential $z^{p_i-2}dz$ on the Riemann sphere $\cp$. The slit in $(Y,\eta)$ is made across one of the sides of the polygon and the slit in $(X_i,\omega_i)$ is made such that the zero of the differential is at one of the ends of the slit. Figure \ref{fig:nontrivholsphere1} below depicts an example where $\vert J \vert = 7$. The resulting surface, say $(Z,\xi)$, has a single zero of order $p_1+\cdots+p_n -2$. We finally split this zero locally to get zeros of order $d_1, d_2, \ldots, d_k$ and this completes the construction. 

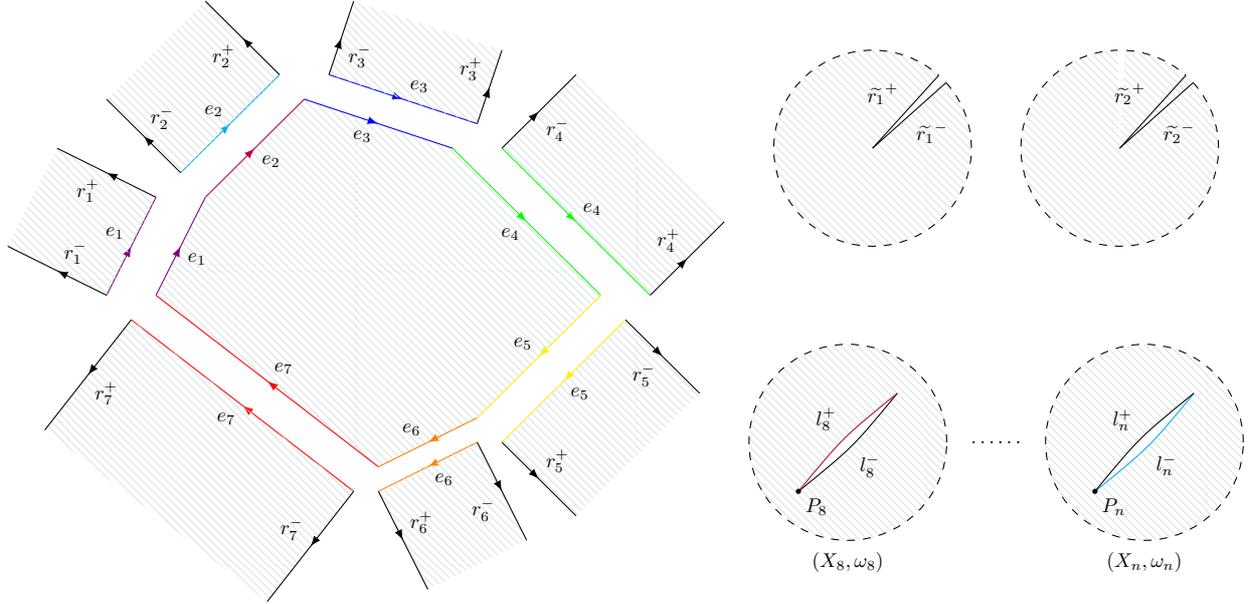
\begin{figure}[h!] 
    \centering
    \begin{tikzpicture}[scale=0.65, every node/.style={scale=0.75}]
    \definecolor{pallido}{RGB}{221,227,227} 
    \path[pattern=north west lines, pattern color=pallido] (0,0) -- ++(1,2) -- ++ (2,2) -- ++ (3,-1) -- ++ (3, -3) -- ++(-2.5, -2.5) -- ++(-2, -1) -- ++ (-4.5, 3.5); 
    \coordinate (A) at (0,0);
    \coordinate (B) at ($(A) + (1,2)$);
    \coordinate (C) at ($(B)+(2,2)$);
    \coordinate (D) at ($(C) + (3,-1)$);
    \coordinate (E) at ($(D) + (3, -3)$);
    \coordinate (F) at ($(E) + (-2.5, -2.5)$);
    \coordinate (G) at ($(F) + (-2,-1)$);
    \foreach \stpt / \endpt / \colr in {(A)/(B)/violet, (B)/(C)/purple, (C)/(D)/blue, (D)/(E)/green, (E)/(F)/yellow, (F)/(G)/orange, (G)/(A)/red}
    {
        \draw[\colr, decoration={markings, mark=at position 0.5 with {\arrow{latex}}}, postaction={decorate}] \stpt -- \endpt;
    }
    
    \draw[violet, decoration={markings, mark=at position 0.5 with {\arrow{latex}}}, postaction={decorate}] ($(A)+(-1,0)$) -- ($(B)+(-1,0)$);
    \path[pattern=north west lines, pattern color=pallido] ($(A)+(-1,0)+(-2,1)$) -- ($(A)+(-1,0)$) -- ($(B)+(-1,0)$) --  ($(B)+(-1,0)+(-2,1)$);
    \draw[decoration={markings, mark=at position 0.5 with {\arrow{latex}}}, postaction={decorate}] ($(A)+(-1,0)$) -- +(-2,1);
    \draw[decoration={markings, mark=at position 0.5 with {\arrow{latex}}}, postaction={decorate}] ($(B)+(-1,0)$) -- +(-2,1);
    \draw[cyan, decoration={markings, mark=at position 0.5 with {\arrow{latex}}}, postaction={decorate}] ($(B)+(-0.5,0.5)$) -- ($(C)+(-0.5,0.5)$);
    \path[pattern=north west lines, pattern color=pallido] ($(B)+(-0.5,0.5) + (-1.5,1.5)$) -- ($(B)+(-0.5,0.5)$) -- ($(C)+(-0.5,0.5)$) --  ($(C)+(-0.5,0.5)+(-1.5,1.5)$);
    \draw[decoration={markings, mark=at position 0.5 with {\arrow{latex}}}, postaction={decorate}] ($(C)+(-0.5,0.5)$) -- +(-1.5,1.5);
    \draw[decoration={markings, mark=at position 0.5 with {\arrow{latex}}}, postaction={decorate}] ($(B)+(-0.5,0.5)$) -- +(-1.5,1.5);
    \draw[blue, decoration={markings, mark=at position 0.5 with {\arrow{latex}}}, postaction={decorate}] ($(C)+(0.5,0.5)$) -- ($(D)+(0.5,0.5)$);
    \path[pattern=north west lines, pattern color=pallido] ($(C)+(0.5,0.5)+(0.5,1.5)$) -- ($(C)+(0.5,0.5)$) -- ($(D)+(0.5,0.5)$) --  ($(D)+(0.5,0.5)+(0.5,1.5)$);
    \draw[decoration={markings, mark=at position 0.5 with {\arrow{latex}}}, postaction={decorate}] ($(C)+(0.5,0.5)$) -- +(0.5,1.5);
    \draw[decoration={markings, mark=at position 0.5 with {\arrow{latex}}}, postaction={decorate}] ($(D)+(0.5,0.5)$) -- +(0.5,1.5);
    \draw[green, decoration={markings, mark=at position 0.5 with {\arrow{latex}}}, postaction={decorate}] ($(D)+(1,0)$) -- ($(E)+(1,0)$);
    \path[pattern=north west lines, pattern color=pallido] ($(D)+(1,0)+(1.5,1.5)$) -- ($(D)+(1,0)$) -- ($(E)+(1,0)$) --  ($(E)+(1,0)+(1.5,1.5)$);
    \draw[decoration={markings, mark=at position 0.5 with {\arrow{latex}}}, postaction={decorate}] ($(E)+(1,0)$) -- +(1.5,1.5);
    \draw[decoration={markings, mark=at position 0.5 with {\arrow{latex}}}, postaction={decorate}] ($(D)+(1,0)$) -- +(1.5,1.5);
    \draw[yellow, decoration={markings, mark=at position 0.5 with {\arrow{latex}}}, postaction={decorate}] ($(E)+(0.5,-0.5)$) -- ($(F)+(0.5,-0.5)$);
    \path[pattern=north west lines, pattern color=pallido] ($(E)+(0.5,-0.5)+(1.5,-1.5)$) -- ($(E)+(0.5,-0.5)$) -- ($(F)+(0.5,-0.5)$) --  ($(F)+(0.5,-0.5)+(1.5,-1.5)$);
    \draw[decoration={markings, mark=at position 0.5 with {\arrow{latex}}}, postaction={decorate}] ($(E)+(0.5,-0.5)$) -- +(1.5,-1.5);
    \draw[decoration={markings, mark=at position 0.5 with {\arrow{latex}}}, postaction={decorate}] ($(F)+(0.5,-0.5)$) -- +(1.5,-1.5);
    \draw[orange, decoration={markings, mark=at position 0.5 with {\arrow{latex}}}, postaction={decorate}] ($(F)+(0,-0.5)$) -- ($(G)+(0,-0.5)$);
    \path[pattern=north west lines, pattern color=pallido] ($(F)+(0,-0.5)+(1,-2)$) -- ($(F)+(0,-0.5)$) -- ($(G)+(0,-0.5)$) --  ($(G)+(0,-0.5)+(1,-2)$);
    \draw[decoration={markings, mark=at position 0.5 with {\arrow{latex}}}, postaction={decorate}] ($(F)+(0,-0.5)$) -- +(1,-2);
    \draw[decoration={markings, mark=at position 0.5 with {\arrow{latex}}}, postaction={decorate}] ($(G)+(0,-0.5)$) -- +(1,-2);
    \draw[red, decoration={markings, mark=at position 0.5 with {\arrow{latex}}}, postaction={decorate}] ($(G)+(-0.5,-0.5)$) -- ($(A)+(-0.5,-0.5)$);
    \path[pattern=north west lines, pattern color=pallido] ($(A)+(-0.5,-0.5)+(-1.75,-2.25)$) -- ($(A)+(-0.5,-0.5)$) -- ($(G)+(-0.5,-0.5)$) --  ($(G)+(-0.5,-0.5)+(-1.75,-2.25)$);
    \draw[decoration={markings, mark=at position 0.5 with {\arrow{latex}}}, postaction={decorate}] ($(A)+(-0.5,-0.5)$) -- +(-1.75,-2.25);
    \draw[decoration={markings, mark=at position 0.5 with {\arrow{latex}}}, postaction={decorate}] ($(G)+(-0.5,-0.5)$) -- +(-1.75,-2.25);
    
    \node[below right] at ($(A)!0.5!(B)$) {$e_1$};
    \node[above left] at ($(A)!0.5!(B)+(-1,0)$) {$e_1$};
    \node[below right] at ($(B)!0.5!(C)$) {$e_2$};
    \node[above left] at ($(B)!0.5!(C)+(-0.5,0.5)$) {$e_2$};
    \node[below left] at ($(C)!0.5!(D)$) {$e_3$};
    \node[above right] at ($(C)!0.5!(D)+(0.5,0.5)$) {$e_3$};
    \node[below left] at ($(D)!0.5!(E)$) {$e_4$};
    \node[above right] at ($(D)!0.5!(E)+(1,0)$) {$e_4$};
    \node[above left] at ($(E)!0.5!(F)$) {$e_5$};
    \node[below right] at ($(E)!0.5!(F)+(0.5,-0.5)$) {$e_5$};
    \node[above left] at ($(F)!0.5!(G)$) {$e_6$};
    \node[below right] at ($(F)!0.5!(G)+(0,-0.5)$) {$e_6$};
    \node[above right] at ($(A)!0.5!(G)$) {$e_7$};
    \node[below left] at ($(A)!0.5!(G)+(-0.5,-0.5)$) {$e_7$};
    \node[above right] at ($(A)+(-1,0)+ 0.5*(-2,1)$) {$r_1^-$};
    \node[below left] at ($(B)+(-1,0)+ 0.5*(-2,1)$) {$r_1^+$};
    \node[below left] at ($(C)+(-0.5,0.5)+ 0.5*(-1.5,1.5)$){$r_2^+$};
    \node[above right] at ($(B)+(-0.5,0.5)+ 0.5*(-1.5,1.5)$){$r_2^-$};
    \node[below right] at ($(C)+(0.5,0.5)+ 0.5*(0.5,1.5)$){$r_3^-$};
    \node[above left] at ($(D)+(0.5,0.5)+ 0.5*(0.5,1.5)$){$r_3^+$};
    \node[above left] at ($(E)+(1,0)+ 0.5*(1.5,1.5)$){$r_4^+$};
    \node[below right] at ($(D)+(1,0)+ 0.5*(1.5,1.5)$){$r_4^-$};
    \node[below left] at ($(E)+(0.5,-0.5)+ 0.5*(1.5,-1.5)$){$r_5^-$};
    \node[above right] at ($(F)+(0.5,-0.5)+ 0.5*(1.5,-1.5)$){$r_5^+$};
    \node[below left] at ($(F)+(0,-0.5)+ 0.5*(1,-2)$){$r_6^-$};
    \node[above right] at ($(G)+(0,-0.5)+ 0.5*(1,-2)$){$r_6^+$};
    \node[above left] at ($(G)+(-0.5,-0.5)+ 0.5*(-1.75,-2.25)$){$r_7^-$};
    \node[below right] at ($(A)+(-0.5,-0.5)+ 0.5*(-1.75,-2.25)$){$r_7^+$};
    
    \begin{scope}[shift={(2,3)}]
    \draw [dashed, pattern=north west lines, pattern color=pallido] ($(12.5,0) + (42:2)$) arc [start angle = 42,end angle = -312,radius = 2];
    \fill[white] (12.5,0) -- ($(12.5,0) + (42:2)$) arc [start angle = 42,end angle = 48,radius = 2] -- (12.5,0);
    \draw (12.5,0) -- +(42:2);
    \draw (12.5,0) -- +(48:2);
    \node[below right] at ($(12.5,0)+0.5*(42:2)$) {$\widetilde{r_1}^-$};
    \node[above left] at ($(12.5,0)+0.5*(48:2)$){$\widetilde{r_1}^+$};
    \draw [dashed, pattern=north west lines, pattern color=pallido] ($(17.5,0) + (42:2)$) arc [start angle = 42,end angle = -312,radius = 2];
    \fill[white] (17.5,0) -- ($(17.5,0) + (87:2)$) arc [start angle = 42,end angle = 48,radius = 2] -- (17.5,0);
    \draw (17.5,0) -- +(42:2);
    \draw (17.5,0) -- +(48:2);
    \node[below right] at ($(17.5,0)+0.5*(42:2)$) {$\widetilde{r_2}^-$};
    \node[above left] at ($(17.5,0)+0.5*(48:2)$){$\widetilde{r_2}^+$};
    \end{scope}
    
    \begin{scope}[shift={(2,-3)}]
    \draw [dashed, black, pattern=north west lines, pattern color=pallido] (12,0) circle (2);
    \fill[white] (11,-1) .. controls (12.1, -0.1) .. (13,1) .. controls (11.9, 0.1) .. (11,-1);
    \draw (11,-1) .. controls (12.1, -0.1) .. (13,1);
    \draw[purple] (11,-1) .. controls (11.9, 0.1) .. (13,1);
    \fill (11,-1) circle (1.5pt);
    \node[below right] at (11,-1) {$P_8$};
    \node[above left] at (11.9, 0.1) {$l_8^+$};
    \node[below right] at (12.1, -0.1) {$l_8^-$}; 
    \node[below] at (12,-2.1) {$(X_8, \omega_8)$};
    \node at (15,0) {$\ldots \ldots$};
    \draw [dashed, black, pattern=north west lines, pattern color=pallido] (18,0) circle (2);
    \fill[white] (17,-1) .. controls (18.1, -0.1) .. (19,1) .. controls (17.9, 0.1) .. (17,-1);
    \draw[cyan] (17,-1) .. controls (18.1, -0.1) .. (19,1);
    \fill (17,-1) circle (1.5pt);
    \node[below right] at (17,-1) {$P_n$};
    \draw (17,-1) .. controls (17.9, 0.1) .. (19,1);
    \node[above left] at (17.9, 0.1) {$l_n^+$};
    \node[below right] at (18.1, -0.1) {$l_n^-$};
    \node[below] at (18,-2.1) {$(X_n, \omega_n)$};
    \end{scope}
    \end{tikzpicture}
    \caption{An example to illustrate the steps involved in the construction of the translation structure when $\textsf{Im}(\chi_n)$ is not contained in the $\mathbb{R}$-span of some $c \in \mathbb{C}$.  The figure depicts an example where $\vert J \vert = 7$. With the only exception of $p_4=3$, $p_j = 1$ for all $j \in J$. In the top right, two copies of $\mathbb{E}^2$ are attached to the surface determined by the polygons and the infinite half strips by the identifications $r_4^+ \sim \widetilde{r_1}^-$, $\widetilde{r_1}^+ \sim \widetilde{r_2}^-$, $\widetilde{r_2}^+ \sim r_4^-$. In the bottom right, we have the surfaces $X_i$ as described with the zeros being located at $P_i$ and the identifications $l_i^- \sim l_{i+1}^+$ except for one pair which are identified with the pair of segments along $e_2$ as shown. The sides of the half strips are glued as done for proving Proposition \ref{propb} unless otherwise mentioned.}
    \label{fig:nontrivholsphere1}
\end{figure}

\noindent \textbf{Case 2.} We now consider the case when $\textsf{Im}(\chi_n)$ is  contained in the $\mathbb{R}$-span of some $c \in \mathbb{C}$, but there does not exist any $\widetilde{c}$ such that $\textsf{Im}(\chi_n)$ is contained in the $\mathbb{Q}$-span of $\widetilde{c}$. For convenience, we assume that $c$ is real. In this case, the polygon as described earlier is degenerate. As in the proof of Proposition \ref{propb}, we assume that $\{ \chi_n(\gamma_1), \ldots, \chi_n(\gamma_{k-1}) \}$ are positive and $\{ \chi_n(\gamma_{k}), \ldots, \chi_n(\gamma_m) \}$ are negative. Define $s_i := \chi_n(\gamma_i)$, for $1 \leq i \leq k-1$ and $t_j := -\chi_n(\gamma_{m-j+1})$, for $1 \leq j \leq m+1-k$. Then, $s_1, \ldots, s_{k-1}$ and $t_1, \ldots, t_{m+1-k}$ satisfy the hypothesis of Lemma \ref{lem:irrmult} and we reorder the indices of $\gamma_1, \ldots, \gamma_m$ such that the conclusion of Lemma \ref{lem:irrmult} is satisfied for $s_1, \ldots, s_{k-1}$ and $t_1, \ldots, t_{m+1-k}$.  Having reordered the indices of $\gamma_1, \ldots, \gamma_m$, let $\{ \zeta_i \}_{1 \leq i \leq m}$ be as in the proof of Proposition \ref{propb} with $\zeta_1 = 0$ and consider the infinite strip $\{ z \in \mathbb{C} \,\, \vert \,\, 0 < \Re(z) < \zeta_k \}$. In this infinite strip, we make half infinite vertical slits pointing upwards at the points $\zeta_2, \ldots, \zeta_{k-1}$ and half infinite vertical slits pointing downwards at the points $\zeta_{k+1}, \ldots, \zeta_{m}$ as in Figure \ref{fig:qspan}. Gluing $r_i^+$ and $r_i^-$ gives us the surface $(Z, \xi)$ in which we split the zero as in case-1 to complete the construction. 
\noindent A crucial requirement that ensures that $(Z, \xi)$ is a surface, is that $\zeta_{\lambda} \neq \zeta_{\mu}$ for $\lambda \neq \mu$, and this is ensured by the conclusion of Lemma \ref{lem:irrmult} (\textit{c.f.} Figure \ref{fig:29}). This is easy to see when $1 \leq \lambda, \mu \leq k$ or $k \leq \lambda, \mu \leq m$.  For the remaining cases, we may assume without loss of generality that $1 \leq \lambda \leq k$ and $k+1 \leq \mu \leq m$. Then, $\zeta_{\lambda} = \zeta_{\mu}$ implies
\begin{align}
    \sum_{i=1}^{\lambda-1} \chi_n(\gamma_i) &= \sum_{j=1}^{\mu-1} \chi_n(\gamma_j) \\
    &= - \bigg( \sum_{j=\mu }^{m} \chi_n(\gamma_j) \bigg) \\
    &= \sum_{j=1}^{m +1 - \mu}  - \chi_n(\gamma_{m+1-j})
\end{align}
which is not possible by our choice of ordering.

\begin{figure}[!h]
     \centering
     \begin{tikzpicture}[scale=1, every node/.style={scale=0.75}]
        \definecolor{pallido}{RGB}{221,227,227}
         \fill[pattern=north west lines, pattern color=pallido] (0,3) -- (0,-3) -- (3,-3) -- (3, 3);
         \fill[pattern=north west lines, pattern color=pallido] (5,3) -- (5,-3) -- (8.5,-3) -- (8.5, 3);
         \fill[white] ($(1,0) + (92:3)$)--(1,0)-- ++(88:3);
         \fill[white] ($(2,0) + (-92:3)$)--(2,0)-- ++(-88:3);
         \fill[white] ($(6,0) + (92:3)$)--(6,0)-- ++(88:3);
         \fill[white] ($(7,0) + (-92:3)$)--(7,0)-- ++(-88:3);
         \draw[decoration={markings, mark=at position 0.6 with {\arrow{latex}}}, postaction={decorate}] (0,0) -- (0,3);
         \draw[decoration={markings, mark=at position 0.6 with {\arrow{latex}}}, postaction={decorate}] (0,0) -- (0,-3);
         \draw[decoration={markings, mark=at position 0.6 with {\arrow{latex}}}, postaction={decorate}] (1,0) -- ++(92:3);
         \draw[decoration={markings, mark=at position 0.6 with {\arrow{latex}}}, postaction={decorate}] (1,0) -- ++(88:3);
         \draw[decoration={markings, mark=at position 0.6 with {\arrow{latex}}}, postaction={decorate}] (2,0) -- ++(-88:3);
         \draw[decoration={markings, mark=at position 0.6 with {\arrow{latex}}}, postaction={decorate}] (2,0) -- ++(-92:3);
         \node at (4,0) {$\ldots$};
         \draw[decoration={markings, mark=at position 0.6 with {\arrow{latex}}}, postaction={decorate}] (6,0) -- ++(92:3);
         \draw[decoration={markings, mark=at position 0.6 with {\arrow{latex}}}, postaction={decorate}] (6,0) -- ++(88:3);
         \draw[decoration={markings, mark=at position 0.6 with {\arrow{latex}}}, postaction={decorate}] (7,0) -- ++(-88:3);
         \draw[decoration={markings, mark=at position 0.6 with {\arrow{latex}}}, postaction={decorate}] (7,0) -- ++(-92:3);
         \draw[decoration={markings, mark=at position 0.6 with {\arrow{latex}}}, postaction={decorate}] (8.5,0) -- (8.5,-3);
        \draw[decoration={markings, mark=at position 0.6 with {\arrow{latex}}}, postaction={decorate}] (8.5,0) -- (8.5,3);
         \fill (0,0) circle (1pt);
         \fill (1,0) circle (1pt);
         \fill (2,0) circle (1pt);
         \fill (6,0) circle (1pt);
         \fill (7,0) circle (1pt);
         \fill (8.5,0) circle (1pt);
         \node[right] at (0, 1.5) {$r_1^-$};
         \node[left] at ($(1,0) + (92:1.5)$) {$r_1^+$};
         \node[right] at ($(1,0) + (88:1.5)$) {$r_2^-$};
         \node[left] at ($(6,0) + (92:1.5)$) {$r_{k-2}^+$};
         \node[right] at ($(6,0) + (88:1.5)$) {$r_{k-1}^-$};
        \node[left] at (8.5, 1.5) {$r_{k-1}^+$};
        \node[right] at (0, -1.5) {$r_m^+$};
         \node[left] at ($(2,0) + (-92:1.5)$) {$r_m^-$};
         \node[right] at ($(2,0) + (-88:1.5)$) {$r_{m-1}^+$};
         \node[left] at ($(7,0) + (-92:1.5)$) {$r_{k+1}^-$};
         \node[right] at ($(7,0) + (-88:1.5)$) {$r_{k}^+$};
        \node[left] at (8.5, -1.5) {$r_{k}^-$};
        \node[left] at (0,0) {$\zeta_1$};
        \node[below] at (1,0) {$\zeta_2$};
        \node[below] at (6,0) {$\zeta_{k-1}$};
        \node[right] at (8.5,0) {$\zeta_k$};
        \node[above] at (7,0) {$\zeta_{k+1}$};
        \node[above] at (2,0) {$\zeta_m$};
     \end{tikzpicture}
     \caption{When $\textsf{Im}(\chi_n)$ is not contained in the $\mathbb{Q}$-span of some $c \in \mathbb{C}$} \label{fig:qspan}
 \end{figure}
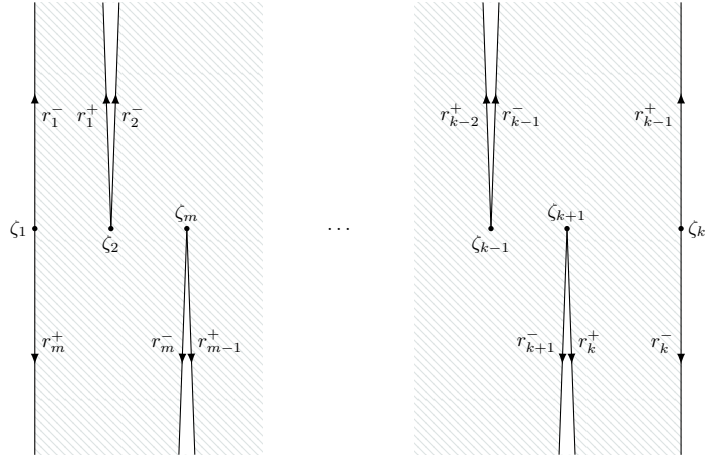

\text{}\\
\noindent \textbf{Case 3.} We now assume $\textsf{Im}(\chi_n)$ is contained in the $\mathbb{R}$-span of some $c \in \mathbb{C}$ and not all the $p_i$'s are 1. In this case, the polygon obtained previously turns out degenerate. We may again assume that $\theta = 0$, that is, the degenerate polygon is along the real axis. The construction now splits into two sub-cases.\\

\noindent The first sub-case occurs when $\vert J \vert = m \neq n$. We may note that the condition about the order of all poles not being $1$ is automatically satisfied. Consider the half strips $\mathcal{S}_j$, as already done in the proof of Proposition \ref{propb} in section \ref{fsps}, with base $e_j$ and the infinite rays $r_j^+$ and $r_j^-$ pointing upwards for all $j=1,\dots, n$ and pointing downwards for $k+1 \leq j \leq n$. Additionally, we consider the surfaces $(X_i,\omega_i)$, where $X_i=\cp$ and  $\omega_i=z^{p_i-2}dz$ for every $i \notin J$, with slits made on them along $\overline{\zeta_{1}\,\zeta_{k}}$, where $\zeta_i$ are as defined in section \ref{fsps}. As done in section \ref{fsps}, we label the two sides of the slit $l_i^+$ and $l_i^-$. For $1 \leq j \leq k$, we identify $e_j$ with adjacent sub-segments of $l_n^-$. For $k+1 \leq j \leq n$, we identify $e_j$ with adjacent sub-segments of $l_{m+1}^+$. For the other slits, we identify $l_i^-$ with $l_{i+1}^+$ for $m+1 \leq i \leq n-1$. $r_j^+$ and $r_j^-$ are identified as before. In Figure \ref{fig:nontrivholsphere2}, we see an example where $\vert J \vert = 5$, with $k=3$. We now have a differential with poles of order $p_i$ for $i \notin J$. For $i \in J$, we need to glue $p_i - 1$ many copies of $\mathbb{E}^2$ along $r_i^+ = r_i^-$ by cutting along infinite rays as done in subsection \ref{lgp}. The differential now has all poles with orders and residues as required and a single zero. Splitting this zero locally completes the construction.\\


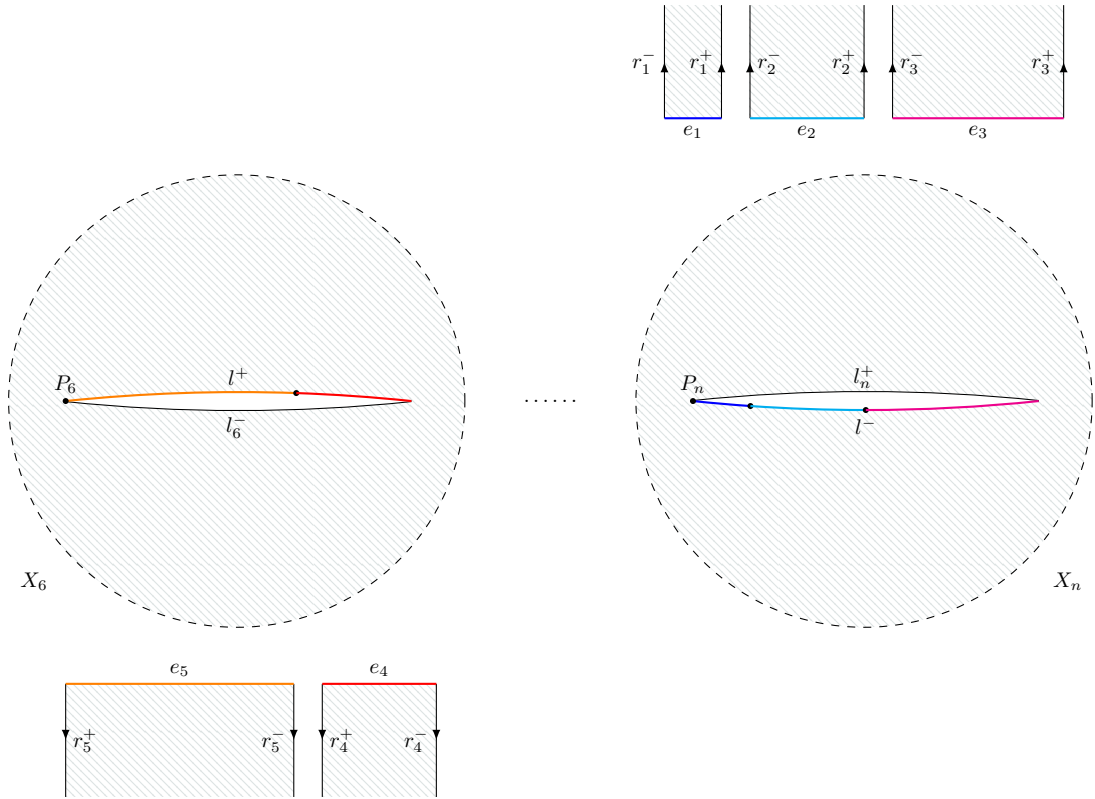
\begin{figure}[h!]
\centering
\begin{tikzpicture}[scale=0.75, every node/.style={scale=0.8}]
\definecolor{pallido}{RGB}{221,227,227}

\begin{scope}[shift = {(11,0)}]
\draw [dashed, black, pattern=north west lines, pattern color=pallido] (0,0) circle (4);
\draw[thick] (-3,0) arc [start angle = 96,end angle = 84,radius = 29];
\fill[white] (-3,0) arc [start angle = 96,end angle = 84,radius = 29] arc [start angle = -84,end angle = -96,radius = 29];
\draw[thick, blue] (-3,0) arc [start angle = -96,end angle = -94,radius = 29];
\coordinate (A) at ($(-3,0) + (84:29) + (-94:29)$);
\fill (A) circle (1.5pt);
\draw[thick, cyan] (A) arc [start angle = -94,end angle = -90,radius = 29];
\coordinate (B) at ($(A)+ (86:29) + (-90:29)$);
\fill (B) circle (1.5pt);
\node[below] at (B) {$l^-$};
\draw[thick, magenta] (B) arc [start angle = -90,end angle = -84,radius = 29];
\fill (-3,0) circle (1.5pt);
\node[above] at (-3,0) {$P_n$};
\node[below, right] at (-45:4.5) {$X_n$};
\node[above] at (0, 0.1) {$l_n^+$};

\path[pattern=north west lines, pattern color=pallido] (-3.5,7) -- (-3.5,5) -- (-2.5,5) -- (-2.5,7);
\draw[decoration={markings, mark=at position 0.5 with {\arrow{latex}}}, postaction={decorate}] (-3.5,5) -- (-3.5,7);
\draw[decoration={markings, mark=at position 0.5 with {\arrow{latex}}}, postaction={decorate}] (-2.5,5) -- (-2.5,7);
\draw[thick, blue] (-3.5,5) -- (-2.5,5);
\node[left] at (-3.5, 6) {$r_1^-$};
\node[left] at (-2.5, 6) {$r_1^+$};
\node[below] at (-3,5) {$e_1$};

\path[pattern=north west lines, pattern color=pallido] (-2,7) -- (-2,5) -- (0,5) -- (0,7);
\draw[decoration={markings, mark=at position 0.5 with {\arrow{latex}}}, postaction={decorate}] (-2,5) -- (-2,7);
\draw[decoration={markings, mark=at position 0.5 with {\arrow{latex}}}, postaction={decorate}] (-0,5) -- (-0,7);
\draw[thick, cyan] (-2,5) -- (0,5);
\node[right] at (-2, 6) {$r_2^-$};
\node[left] at (0, 6) {$r_2^+$};
\node[below] at (-1,5) {$e_2$};

\path[pattern=north west lines, pattern color=pallido] (0.5,7) -- (0.5,5) -- (3.5,5) -- (3.5,7);
\draw[decoration={markings, mark=at position 0.5 with {\arrow{latex}}}, postaction={decorate}] (0.5,5) -- (0.5,7);
\draw[decoration={markings, mark=at position 0.5 with {\arrow{latex}}}, postaction={decorate}] (3.5,5) -- (3.5,7);
\draw[thick, magenta] (0.5,5) -- (3.5,5);
\node[right] at (0.5, 6) {$r_3^-$};
\node[left] at (3.5, 6) {$r_3^+$};
\node[below] at (2,5) {$e_3$};
\end{scope}

\node at (5.5,0) {$\ldots \ldots$};

\begin{scope}[shift = {(-11,0)}]
\draw [dashed, black, pattern=north west lines, pattern color=pallido] (11,0) circle (4);
\draw[thick] (8,0) arc [start angle = -96,end angle = -84,radius = 29];
\fill[white] (8,0) arc [start angle = 96,end angle = 84,radius = 29] arc [start angle = -84,end angle = -96,radius = 29];
\draw[thick, orange] (8,0) arc [start angle = 96,end angle = 88,radius = 29];
\coordinate (C) at ($(8,0) + (-84:29) + (88:29)$);
\fill (C) circle (1.5pt);
\draw[thick, red] (C) arc [start angle = 88,end angle = 84,radius = 29];
\coordinate (D) at ($(8,0) + (-84:29) + (90:29)$);
\node[above] at (D) {$l^+$};
\node[below] at (11,-0.1) {$l_6^-$};
\fill (8,0) circle (1.5pt);
\node[above] at (8,0) {$P_6$};
\node[below, left] at ($(11,0)+(225:4.5)$) {$X_6$};

\path[pattern=north west lines, pattern color=pallido] (8,-7) -- (8,-5) -- (12,-5) -- (12,-7);
\draw[decoration={markings, mark=at position 0.5 with {\arrow{latex}}}, postaction={decorate}] (8,-5) -- (8,-7);
\draw[decoration={markings, mark=at position 0.5 with {\arrow{latex}}}, postaction={decorate}] (12,-5) -- (12,-7);
\draw[thick, orange] (8,-5) -- (12,-5);
\node[right] at (8, -6) {$r_5^+$};
\node[left] at (12, -6) {$r_5^-$};
\node[above] at (10,-5) {$e_5$};

\path[pattern=north west lines, pattern color=pallido] (14.5,-7) -- (14.5,-5) -- (12.5,-5) -- (12.5,-7);
\draw[decoration={markings, mark=at position 0.5 with {\arrow{latex}}}, postaction={decorate}] (14.5,-5) -- (14.5,-7);
\draw[decoration={markings, mark=at position 0.5 with {\arrow{latex}}}, postaction={decorate}] (12.5,-5) -- (12.5,-7);
\draw[thick, red] (14.5,-5) -- (12.5,-5);
\node[right] at (12.5, -6) {$r_4^+$};
\node[left] at (14.5, -6) {$r_4^-$};
\node[above] at (13.5,-5) {$e_4$};
\end{scope}


\end{tikzpicture}
\caption{An example to illustrate the construction of the translation structure when $\textsf{Im}(\chi_n)$ lies along the real axis and $\vert J \vert \neq n$. The zero of the differential in $X_i$ is located at $P_i$.}
\label{fig:nontrivholsphere2}
\end{figure}

\noindent The second sub-case occurs when $\vert J \vert = n$. Here, we take a copy of $\mathbb{E}^2$ and make a slit along the segment $\overline{\zeta_{1}\,\zeta_{k}}$. We label the two sides of the slit $l^+$ and $l^-$ and attach half strips as before. Now, suppose $m$ is an index such that $p_m \geq 2$. Then the rays $r_m^+$ and $r_m^-$ are both pointing upwards or both pointing downwards. Assume that they point upwards. We consider a ray $\widetilde{r}$ in the upward direction starting from some point on $l^+$ which is not identified with an end point of some $e_j$. Figure \ref{fig:nontrivholsphere3} shows a possible location for $\widetilde{r}$ when $\vert J \vert = n = 5$ and at least one of $p_1, p_2$ and $p_3$ is different from 1. After making a slit along $\widetilde{r}$, we obtain two rays $\widetilde{r}^+$ and $\widetilde{r}^-$. We identify $r_m^+$ with $\widetilde{r}^-$ and $r_m^-$ with $\widetilde{r}^+$. For all other indices $j$ except $m$, we identify $r_j^+$ with $r_j^-$. The bases $e_j$ are identified with the edges of the slit in an appropriate manner as before. We can perform a similar construction even when $r_m^+$ and $r_m^-$ point downwards.

\begin{figure}[h!] 
\centering
\begin{tikzpicture}[scale=0.85, every node/.style={scale=0.9}]
\definecolor{pallido}{RGB}{221,227,227}

\draw [dashed, pattern=north west lines, pattern color=pallido] (87:4) arc [start angle = 87,end angle = -267,radius = 4];
\coordinate (D1) at ($(-3,0) + (-84:29) + (90.1:29)$);
\coordinate (D2) at ($(-3,0) + (-84:29) + (89.9:29)$);
\fill[white] (D2) -- (87:4) arc [start angle = 87,end angle = 93,radius = 4] -- (D1);
\draw[decoration={markings, mark=at position 0.5 with {\arrow{latex}}}, postaction={decorate}] (D2) -- +(87:2);
\node[right] at ($(D2)+(87:1)$) {$\widetilde{r}^-$};
\node[left] at ($(D1)+(93:1)$) {$\widetilde{r}^+$};
\draw[decoration={markings, mark=at position 0.5 with {\arrow{latex}}}, postaction={decorate}] (D1) -- +(93:2);
\fill[white] (-3,0) arc [start angle = 96,end angle = 84,radius = 29] arc [start angle = -84,end angle = -96,radius = 29];
\draw[thick, orange] (-3,0) arc [start angle = 96,end angle = 90.1,radius = 29];
\draw[thick, orange] (D2) arc [start angle = 89.9,end angle = 88,radius = 29];
\coordinate (C) at ($(-3,0) + (-84:29) + (88:29)$);
\fill (C) circle (1.5pt);
\draw[thick, red] (C) arc [start angle = 88,end angle = 84,radius = 29];
\draw[thick, blue] (-3,0) arc [start angle = -96,end angle = -94,radius = 29];
\coordinate (A) at ($(-3,0) + (84:29) + (-94:29)$);
\fill (A) circle (1.5pt);
\draw[thick, cyan] (A) arc [start angle = -94,end angle = -90,radius = 29];
\coordinate (B) at ($(A)+ (86:29) + (-90:29)$);
\fill (B) circle (1.5pt);
\draw[thick, magenta] (B) arc [start angle = -90,end angle = -84,radius = 29];
\node[above right] at (-3, 0.1) {$l^+$};
\node[below right] at (-3, -0.1) {$l^-$}; 




\begin{scope}[shift = {(-8.5,-4)}]
\path[pattern=north west lines, pattern color=pallido] (-3.5,7) -- (-3.5,5) -- (-2.5,5) -- (-2.5,7);
\draw[decoration={markings, mark=at position 0.5 with {\arrow{latex}}}, postaction={decorate}] (-3.5,5) -- (-3.5,7);
\draw[decoration={markings, mark=at position 0.5 with {\arrow{latex}}}, postaction={decorate}] (-2.5,5) -- (-2.5,7);
\draw[thick, blue] (-3.5,5) -- (-2.5,5);
\node[left] at (-3.5, 6) {$r_1^-$};
\node[left] at (-2.5, 6) {$r_1^+$};
\node[below] at (-3,5) {$e_1$};

\path[pattern=north west lines, pattern color=pallido] (-2,7) -- (-2,5) -- (0,5) -- (0,7);
\draw[decoration={markings, mark=at position 0.5 with {\arrow{latex}}}, postaction={decorate}] (-2,5) -- (-2,7);
\draw[decoration={markings, mark=at position 0.5 with {\arrow{latex}}}, postaction={decorate}] (-0,5) -- (-0,7);
\draw[thick, cyan] (-2,5) -- (0,5);
\node[right] at (-2, 6) {$r_2^-$};
\node[left] at (0, 6) {$r_2^+$};
\node[below] at (-1,5) {$e_2$};

\path[pattern=north west lines, pattern color=pallido] (0.5,7) -- (0.5,5) -- (3.5,5) -- (3.5,7);
\draw[decoration={markings, mark=at position 0.5 with {\arrow{latex}}}, postaction={decorate}] (0.5,5) -- (0.5,7);
\draw[decoration={markings, mark=at position 0.5 with {\arrow{latex}}}, postaction={decorate}] (3.5,5) -- (3.5,7);
\draw[thick, magenta] (0.5,5) -- (3.5,5);
\node[right] at (0.5, 6) {$r_3^-$};
\node[left] at (3.5, 6) {$r_3^+$};
\node[below] at (2,5) {$e_3$};
\end{scope}

\begin{scope}[shift = {(-19.5,4)}]
\path[pattern=north west lines, pattern color=pallido] (8,-7) -- (8,-5) -- (12,-5) -- (12,-7);
\draw[decoration={markings, mark=at position 0.5 with {\arrow{latex}}}, postaction={decorate}] (8,-5) -- (8,-7);
\draw[decoration={markings, mark=at position 0.5 with {\arrow{latex}}}, postaction={decorate}] (12,-5) -- (12,-7);
\draw[thick, orange] (8,-5) -- (12,-5);
\node[right] at (8, -6) {$r_5^-$};
\node[left] at (12, -6) {$r_5^+$};
\node[above] at (10,-5) {$e_5$};

\path[pattern=north west lines, pattern color=pallido] (14.5,-7) -- (14.5,-5) -- (12.5,-5) -- (12.5,-7);
\draw[decoration={markings, mark=at position 0.5 with {\arrow{latex}}}, postaction={decorate}] (14.5,-5) -- (14.5,-7);
\draw[decoration={markings, mark=at position 0.5 with {\arrow{latex}}}, postaction={decorate}] (12.5,-5) -- (12.5,-7);
\draw[thick, red] (14.5,-5) -- (12.5,-5);
\node[right] at (12.5, -6) {$r_4^-$};
\node[left] at (14.5, -6) {$r_4^+$};
\node[above] at (13.5,-5) {$e_4$};
\end{scope}
\end{tikzpicture}
\caption{An example to illustrate the construction of the translation structure when $\textsf{Im}(\chi_n)$ lies along the real axis and $\vert J \vert = n$.}
\label{fig:nontrivholsphere3}
\end{figure}
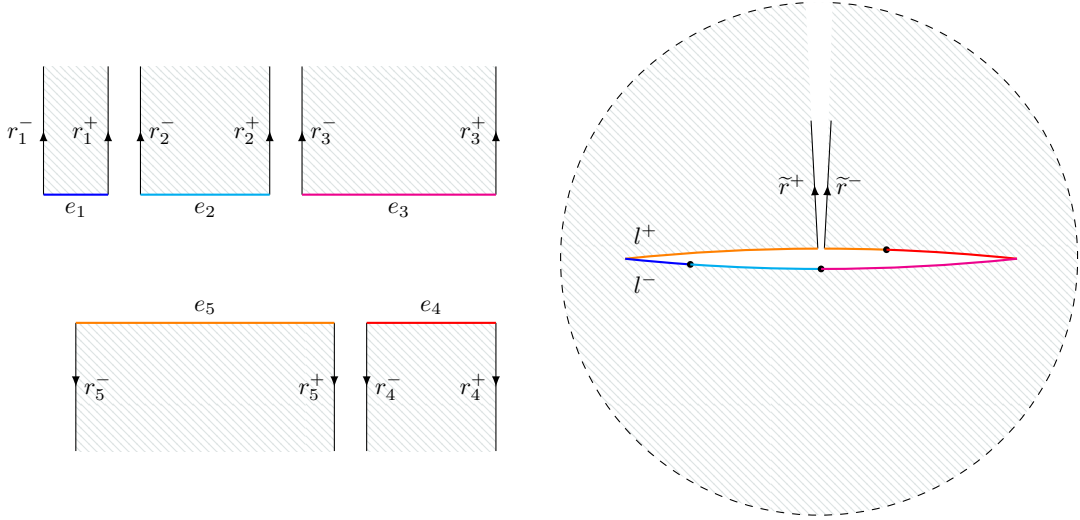

\noindent At this stage, we have a meromorphic differential with a single zero and the desired holonomy around the punctures. However,  all poles but one have order 1. To increase the order at the poles, we now glue sufficiently many copies of $\mathbb{E}^2$ along $r_j^+ = r_j^-$ (along $r_m^+ = \widetilde{r}^-$ when $j=m$) and denote the resulting surface by $(Z,\xi)$. We split the single zero of $(Z,\xi)$ to get zeros of required orders and finish the construction. 
\end{proof}

\begin{rmk} In fact, the above construction also works when $\textsf{Im}(\chi_n) = \mathbb{Z}$, and the prescribed integer holonomies around the punctures  satisfy the combinatorial condition  \eqref{combcond}, where $s_1,s_2,\ldots,s_n$ are the holonomies around the positive punctures, and $t_1,t_2,\ldots, t_m$ are the absolute values of the holonomies around the negative punctures.   However,  this is subsumed by the necessary and sufficient criterion provided by Theorem \ref{main:thme} for the case of integral holonomy. 
\end{rmk}

\section{Positive genus surfaces with non-trivial holonomy} \label{sec:gennontrivhol}
\noindent We now extend our main Theorem \ref{mainthm} for surfaces of positive genus. In particular, we shall provide conditions for a non-trivial representation $\chi$ to appear as the holonomy of some translation surface with prescribed zeros and poles.

\begin{thm}\label{thm:nthpgs}
Given a representation $\chi\in\textsf{\emph{Hom}}(\Gamma_{g,n}, \mathbb{C})$ and positive integers $(p_1, p_2, \ldots, p_n)$ and $(d_1, d_2, \ldots d_k)$ satisfying the following requirements
\begin{enumerate}
    \item One of the following holds: 
    \begin{itemize}
        \item[i.] The $\chi_n$ determined by $\chi$ is trivial.
        \item[ii.] At least one of $p_1, p_2, \ldots p_n$ is different from 1.
        \item[iii.] $\textsf{\emph{Im}}(\chi_n)$ is not contained in the $\mathbb{Q}$-span of some $c \in \mathbb{C}$.
        \item[iv.] $\textsf{\emph{Im}}(\chi)$ is not contained in the $\mathbb{Q}$-span of some $c \in \mathbb{C}$.
    \end{itemize}
    \item $p_i \geq 2$ whenever $\chi_n (\gamma_i) =0$, and
    \item $\displaystyle\sum_{j=1}^k d_j = \sum_{i=1}^n p_i + 2g-2$,
\end{enumerate}
then $\chi$ appears as the holonomy of a translation structure on $S_{g,n}$ induced by a meromorphic  differential on $S_{g,n}$  with zeros of orders $d_1,d_2,\ldots, d_k$ and a  pole of order $p_i$ at the puncture enclosed by the curve $\gamma_i$.
\end{thm}

\noindent Before we prove the theorem, we discuss an important tool used at various times in the proof. 

\subsection{Action of the mapping class group} \label{sec:mcgaction} 
Let $\text{Mod}(S_{g,n})$ denote the mapping class group of $S_{g,n}$. For every mapping class $\varphi \in \text{Mod}(S_{g,n})$, we have an induced automorphism $\varphi_*$ of $\Gamma_{g,n}$. For a representation $\chi$ induced by a holomorphic differential, $\chi \circ \varphi_*$ is the representation induced by the pullback of the differential via $\varphi$. Thus, to prove the theorem for a given $\chi$, it is sufficient to construct a holomorphic differential on $S_{g,n}$ for which the induced representation is $\chi \circ \varphi_*$ for some $\varphi \in \text{Mod}(S_{g,n})$. We use this idea to show that the image of every handle generator under $\chi$ can be assumed to be non-zero if $\chi_g$ is non-trivial. This assumption will be crucial in the construction that follows.


\begin{lem}\label{lem:allhandholnonzero}
Let $\chi \in\textsf{\emph{Hom}}(\Gamma_{g,n}, \mathbb{C})$ be a representation such that the corresponding $\chi_g$ is not trivial. Then, there exists $\varphi \in \text{\emph{Mod}}(S_{g,n})$ such that $\chi \circ \varphi_*(\alpha_i)$ and $\chi \circ \varphi_*(\beta_i)$ are non zero for all $1 \leq i \leq g$.
\end{lem}
\begin{proof}
We may assume that $\chi(\alpha_1) \neq 0$. If $i$ is such that $\chi(\alpha_i) = \chi(\beta_i) = 0$, consider $\varphi$ such that $\varphi_*(\alpha_i) = \alpha_i + \alpha_1$, $\varphi_*(\beta_1) = \beta_1 - \beta_n$ and $\varphi_*$ acts identically on all other handle generators. The existence of such $\varphi$ comes from the surjectivity of the map from $\text{Mod}(S_{g})$ to $\text{Sp}(2g, \mathbb{Z})$. Then $\chi \circ \varphi_*$ equals $\chi$ on all handle generators except $\alpha_i$ and $\chi \circ \varphi_* (\alpha_i) = \chi \circ \varphi_* (\alpha_1) \neq 0$. Thus, by composition of such maps, we now have $\varphi$ for which $\chi \circ \varphi_* (\alpha_i) \neq 0$ for all $1 \leq i \leq g$. Composing $\varphi$ with Dehn twists $\phi$ for which $\phi_*(\beta_i) = \beta_i + \alpha_i$, and denoting the resulting map by $\varphi$, we also have $\chi \circ \varphi_* (\beta_i) \neq 0$ for all $1 \leq i \leq g$.
\end{proof}

\noindent Furthermore, we can also assume that $\chi_g$ is non-trivial if $\chi_n$ is non-trivial. 

\begin{lem}\label{lem:handholnonzero}
Let $\chi \in\textsf{\emph{Hom}}(\Gamma_{g,n}, \mathbb{C})$ be a representation such that the corresponding $\chi_n$ is not trivial. Then, there exists $\varphi \in \text{\emph{Mod}}(S_{g,n})$ such that $(\chi \circ \varphi_*)_g$ is non-trivial.
\end{lem}
\begin{proof}
Assume $\chi(\alpha_i) = \chi(\beta_i) = 0$ for all $1 \leq i \leq g$. Let $\gamma_i$ be a small loop around a puncture such that $\chi(\gamma_i) \neq 0$. Then, we can find some handle generator $\alpha_j$ such that we have two curves $\alpha_j$ and $\widetilde{\alpha}_j$ satisfying $[\alpha_j]-[\widetilde{\alpha}_j] = [\gamma_i]$ in $\Gamma_{g,n}$ as shown in Figure \ref{fig:handlechoice}. This tells us that $\chi(\widetilde{\alpha}_j) \neq 0$. Thus, we can take $\varphi$ to be an element of $\text{Mod}(S_{g,n})$ which takes $\alpha_j$ to $\widetilde{\alpha}_j$ and leaves the other handle generators constant.
\end{proof}

\begin{figure}[!h]
    \centering
    \begin{tikzpicture}[scale=0.8, every node/.style={scale=1}]
         \draw[decoration={markings, mark=at position 0.5 with {\arrow{latex}}}, postaction={decorate}] (-110:4) .. controls (-85:3) .. (-60:4);
         \draw[decoration={markings, mark=at position 0.5 with {\arrow{latex}}}, postaction={decorate}] (-110:4) .. controls (-85:1) .. (-60:4);
         \draw[decoration={markings, mark=at position 0.5 with {\arrow{latex}}}, postaction={decorate}] (-60:4) .. controls (-35:3) .. (-10:4);
         \draw[decoration={markings, mark=at position 0.5 with {\arrow{latex}}}, postaction={decorate}] (-10:4) .. controls (15:3) .. (40:4);
         \draw[decoration={markings, mark=at position 0.5 with {\arrow{latex}}}, postaction={decorate}] (40:4) .. controls (65:3) .. (90:4);
         \draw[dashed] (-110:4) .. controls (-135:3) .. (-160:4);
         \fill (-85:2.5) circle (1.5pt);
         \fill (-0.5,1) circle (1.5pt);
         \fill (-1.5, -0.5) circle (1.5pt);
         \fill (1,0) circle (1.5pt);
         \node[below] at (-85:3.25) {$\alpha_1$};
         \node[above] at (-85:1.65) {$\widetilde{\alpha_1}$};
         \node[below right] at (-35:3.25) {$\beta_1$};
         \node[above right] at (15:3.25) {$\alpha_1^{-1}$};
         \node[above] at (65:3.25) {$\beta_1^{-1}$};
         \draw[decoration={markings, mark=at position 0.7 with {\arrow{latex}}}, postaction={decorate}] (-85:2.25) arc [start angle = 95,end angle = 455,radius = 0.25];
         \node[below right] at ($(-85:2.25)+ (-45:0.25)$) {$\gamma_i$};
         \begin{scope}[shift = {($(120:4)+(90:4)$)}]
         \draw[dashed] (-110:4) .. controls (-85:4.5) .. (-60:4);
         \end{scope}
         
     \end{tikzpicture}
    \caption{Choosing a handle generator whose image under $\chi$ is non-zero.}
    \label{fig:handlechoice}
\end{figure}
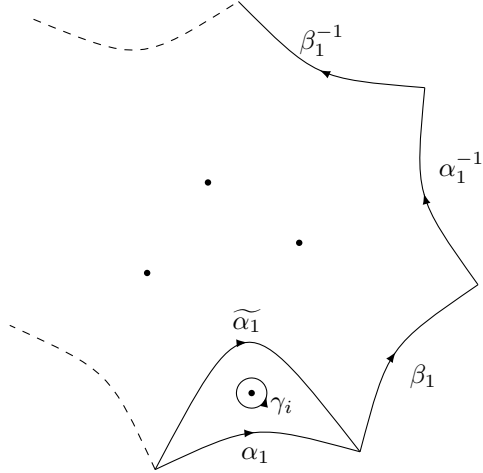

\noindent We shall also require the use of the following lemma in later sections. 

\begin{lem}\label{lem:irrsmallhol}
Let $\chi \in\textsf{\emph{Hom}}(\Gamma_{g,n}, \mathbb{C})$ be a representation such that the corresponding $\textsf{\emph{Im}}(\chi) \subset \mathbb{R}$. Let 
\[ J = \Big\{ 1 \leq i \leq g \,\,\vert\,\, \chi(\alpha_i),\, \chi(\beta_i) > 0 \text{ and } \frac{\chi(\alpha_i)}{\chi(\beta_i)}\notin \mathbb{Q}\Big\}.
\]Then, there exists $\varphi \in \text{\emph{Mod}}(S_{g,n})$ such that $\chi \circ \varphi_*(\alpha_i)$ and $\chi \circ \varphi_*(\beta_i)$ are arbitrarily small for all $i \in J$.
\end{lem}
\begin{proof}
Fix some $i \in J$ and assume $\chi(\alpha_i) > \chi(\beta_i)$. Define $\alpha_{i,0}:= \alpha_i$, $\beta_{i,0}:= \beta_i$. For $k \geq 1$, $\alpha_{i,k}$ and $\beta_{i,k}$ are defined recursively as $\alpha_{i,k}:= \alpha_{i,k-1} - n_k\beta_{i,k-1}$ and $\beta_{i,k}:= \beta_{i,k-1} - m_k\alpha_{i,k}$ for $n_k, m_k \in \mathbb{N}$ satisfying $0 < \chi(\alpha_{i,k-1} - n_k\beta_{i,k-1}) < \chi(\beta_{i,k-1})$ and $0 < \chi(\beta_{i,k-1} - m_k\alpha_{i,k}) < \chi(\alpha_{i,k})$. This is the Euclidean algorithm and since the ratio $\frac{\chi(\alpha_i)}{\chi(\beta_i)}$ is irrational, $\chi(\alpha_{i,k})$ and $\chi(\beta_{i,k})$ become arbitrarily small as $k$ becomes large. This holds even when $\chi(\alpha_i) < \chi(\beta_i)$, after appropriate modifications. Now, by a composition of Dehn twists, we can obtain $\varphi_i \in \text{Mod}(S_{g,n})$ for which $(\varphi_i)_*$ takes $\alpha_i$ to $\alpha_{i,k}$, $\beta_i$ to $\beta_{i,k}$ and acts trivially on other handle generators. Then, the required $\varphi$ is the composition of such $\varphi_i$ for all $i \in J$.
\end{proof}

\subsection{Proof of theorem} As usual, we shall prove Theorem \ref{thm:nthpgs} case by case according to the items $(i)-(iv)$ of requirement (1). It is easy to prove the necessity of the other requirements. 

\subsubsection{When the $\chi_n$ determined by $\chi$ is trivial.} Since $\textsf{Im}(\chi)$ is not trivial, there is a handle generator whose holonomy is not zero. By Lemma \ref{lem:allhandholnonzero}, we can assume that the holonomy of all handle generators is non-zero. We then pick one pair of handle generators. \\ 
\noindent Consider the translation surface $(\C,\, z^{p_1-2}dz)$. Note that each $p_i >1$ by requirement (2) in the hypotheses of the Theorem. By cutting out a slit or parallelogram, we can add a handle to this surface such that the holonomy of the handle generators is the desired holonomy. Moreover, by ensuring that one of the vertices of the slit of the parallelogram is at the zero of the differential $z^{p_1-2}dz$, the resulting surface, say $(Y,\eta)$, can be made to have only one singular point. We now consider, for $2 \leq i \leq n$, the translation surfaces $(X_i,\omega_i)$ induced by the holomorphic differential $z^{p_i-2}dz$ on $\C$, \emph{i.e} $(X_i,\omega_i)=(\C,\,z^{p_i-2}dz)$. We glue them via a sequential slit construction to $(Y,\eta)$, as described in section \ref{sec:seqslithandle}, by taking care that when $p_i > 2$, the slits are made in such a way that one of the ends of the slit is the zero of the differential $z^{p_i-2}dz$. This gives us differential on a genus $1$ surface with a single zero and $n$ poles of orders $p_1, p_2, \ldots, p_n$, all having zero residues. Call the resulting surface $(Z,\xi)$. When $g=1$, we split the zero of $(Z,\xi)$ into zeros of required orders and obtain the desired surface.\\
\noindent When $g>1$, we need to add other handles. We may note that there is an embedded copy of the first quadrant of $\mathbb{E}^2$ in $(Z,\xi)$ such that the origin in this embedded copy is the singular point of $(Z,\xi)$. Of the remaining $g-1$ handles, assume that $m \leq g-1$ handles have non-negative volume, which means $m$ slits that have to be made for adding these handles. Label the ends of these slits $P_i$ and $Q_i$, for $1 \leq i \leq m$, such that $\overline{P_iQ_i}$ makes an angle $\theta_i \in [0, \pi)$ with the positive real axis. For the remaining $g-1-m$ handles, we need to cut out $g-1-m$ many parallelograms from $(Z, \xi)$ and identify their edges. For such parallelograms, we identify two special points. For $m+1 \leq i \leq g-1$,  we let $P_i$ be the bottom most point on the parallelogram. When the parallelogram has a horizontal side, we let $P_i$ be the bottom-most and leftmost point. $Q_i$ is the point diagonally opposite to $P_i$. We then position the slits and parallelograms so that $P_1$ is at the singular point of $(Z, \xi)$ and $P_i = Q_{i-1}$ for $i \geq 2$. An example can be seen in Figure \ref{fig:gennontrivhol1} where we add one handle of each of the volume types. Thus, after constructing handles we obtain a surface with the required holonomy where the differential has a single zero. We split the zero as required to obtain the desired surface.

\begin{figure}[h!] 
    \centering
    \begin{tikzpicture}[scale=0.9, every node/.style={scale=0.9}]
    \definecolor{pallido}{RGB}{221,227,227}
    \path[pattern=north west lines, pattern color=pallido] (-1,-1) -- (11,-1) -- (11,5) -- (-1,5);
    \fill[white] (0,0) .. controls (0.9,0.6) .. (2,1) .. controls (1.1,0.4) .. (0,0);
    \fill[white] (2,1) .. controls (4,1.1) .. (6,1) .. controls (4,0.9) .. (2,1);
    \draw (0,0) .. controls (0.9,0.6) .. (2,1);
    \draw (0,0) .. controls (1.1,0.4) .. (2,1);
    \draw[blue] (2,1) .. controls (4,1.1) .. (6,1);
    \draw[red] (2,1) .. controls (4,0.9) .. (6,1);
    \fill[white] (6,1) -- ++ (3,1) -- ++ (1,2) -- ++ (-3,-1) -- ++ (-1,-2);
    \draw[thick, purple, ->] (7.5, 1.7) -- ++ (0.75, 1.5);
    \draw[thick, orange, ->] (6.7, 2) -- ++(2.4, 0.8);
    \draw (6,1) -- ++ (3,1) -- ++ (1,2) -- ++ (-3,-1) -- ++ (-1,-2);
    \node[below left] at (0,0) {$P_1$};
    \fill (0,0) circle (1.5pt);
    \node[above left] at (2,1) {$Q_1 = P_2$};
    \fill (2,1) circle (1.5pt);
    \node[below right] at (6,1) {$Q_2 = P_3$};
    \fill (6,1) circle (1.5pt);
    \node[above right] at (10,4) {$Q_3$};
    \fill (10,4) circle (1.5pt);
    \path[pattern=north west lines, pattern color=pallido] (13,1) -- ++(4,0) -- ++(-1,2) -- ++(-4,0) -- ++(1,-2);
    \draw[red] (13,1) -- ++(4,0);
    \draw[decoration={markings, mark=at position 0.5 with {\arrow{latex}}}, postaction={decorate}] (13,1) -- ++(-1,2);
    \draw[decoration={markings, mark=at position 0.5 with {\arrow{latex}}}, postaction={decorate}] (17,1) -- ++(-1,2);
    \draw[blue] (12,3) -- ++(4,0);
    \end{tikzpicture}
    \caption{Attaching handles to an embedded copy of $\mathbb{E}^2$. The edges of the slit determined by $P_1Q_1$ are identified as in section \ref{elhan}. A parallelogram is attached to the slit determined by $P_2Q_2$ as shown. The third handle is obtained by identifying the opposite edges of a parallelogram (having diagonal $P_3Q_3$) cut out.}
    \label{fig:gennontrivhol1}
\end{figure}
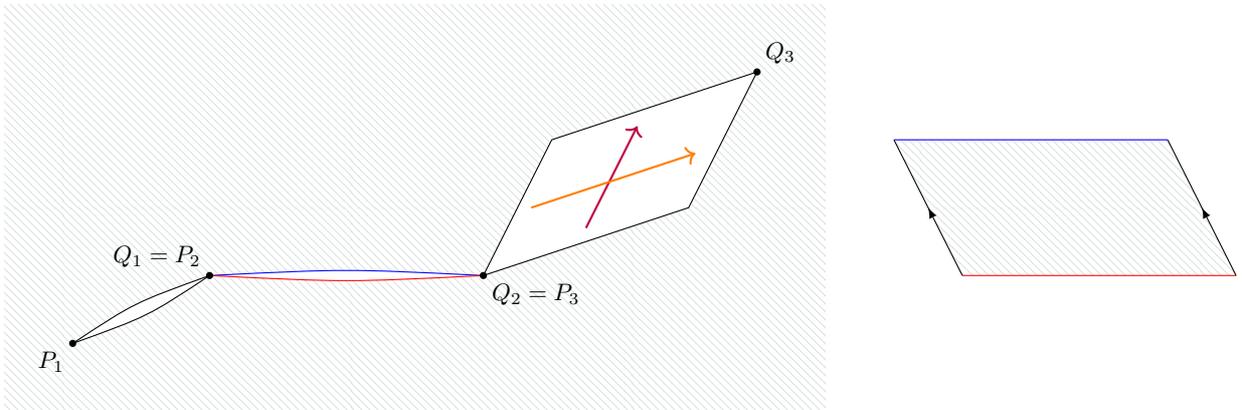

\subsubsection{At least one of $p_1, p_2, \ldots p_n$ is different from 1.} This is the simplest case to deal with. Now we assume that $\chi_n$ is not trivial and that at least one of $p_1, p_2, \ldots p_n$ is different from $1$. By Lemmata \ref{lem:handholnonzero} and \ref{lem:allhandholnonzero} we can also assume that $\chi(\alpha_i)$ and $\chi(\beta_i)$ are non zero for all $1 \leq i \leq g$. Following section \ref{sec:nontrivholsphere}, we can obtain a translation surface $(Z,\xi)$ determined by a differential on the $n$ punctured sphere with the required holonomy around the poles (located at the $n$ punctures) and only a single zero. Then, we add handles to $(Z,\xi)$ such that the resulting differential on the punctured genus $g$ surface still has a single zero. Now, since $(Z,\xi)$ is obtained via the construction in section \ref{sec:nontrivholsphere} when not all poles have order $1$, there is an embedded copy of the first quadrant of $\mathbb{E}^2$ such that the origin in this embedded copy is the singular point of $(Z,\xi)$. We can now add handles just as in the previous case when $\chi_n$ was trivial. This is followed by splitting the zero to complete the construction.

\subsubsection{When $\textsf{\emph{Im}}(\chi_n)$ is not contained in the $\mathbb{Q}$-span of some $c \in \mathbb{C}$} In this case, we assume that $\chi_n$ is not trivial and all of the $p_i$'s are one, but $\textsf{Im}(\chi_n)$ is not contained in the $\mathbb{Q}$-span of some $c \in \mathbb{C}$. This means that either 
\begin{itemize}
    \item $\textsf{Im}(\chi_n)$ is not contained in the $\mathbb{R}$-span of some $c \in \mathbb{C}$, or
    \item $\textsf{Im}(\chi_n)$ is contained in the $\mathbb{R}$-span of some $c \in \mathbb{C}$, but there does not exist any $\widetilde{c} \in \mathbb{C}$ such that $\textsf{Im}(\chi_n)$ is contained in the $\mathbb{Q}$-span of  $\widetilde{c} \in \mathbb{C}$.
\end{itemize}  We consider these two sub-cases separately.\\

\paragraph{\textbf{Case 1.}} We first deal with the case where $\textsf{Im}(\chi_n)$ is not contained in the $\mathbb{R}$-span of some $c \in \mathbb{C}$. In this case, we can proceed as in section \ref{sec:nontrivholsphere} and obtain a differential on the $n$ punctured sphere with the required properties. Here, we do not have an embedded copy of the first quadrant to work with, but we can modify the handles so that the parallelogram cut out or the slit required for them ``fits" inside one of the infinite half cylinders determined by the half strip $\mathcal{S}_j$. In what follows, we describe this in more detail.\\
\noindent We first deal with the slits. Since we are working in the case where all the $\chi_n(\gamma_i)$ are not collinear, we see that for a given slit, there exists $e_j$ which is not collinear with the slit. Then, by looking at a different fundamental domain for the infinite half cylinder formed by identifying the sides of the half strip with base $e_j$, we see that this slit can be cut out in the infinite half cylinder determined by $e_j$, as shown in Figure \ref{fig:gennontrivhol2}. In case multiple slits have to be cut out in the same cylinder, we place them adjacent to each other as before. 

\begin{figure}[h!] 
    \centering
    \begin{tikzpicture}
    \definecolor{pallido}{RGB}{221,227,227}
    \path[pattern=north west lines, pattern color=pallido] (0,4) -- (0,0) -- (1,0) -- (1,4);
    \draw[decoration={markings, mark=at position 0.5 with {\arrow{latex}}}, postaction={decorate}] (0,0) -- (0,4);
    \draw[decoration={markings, mark=at position 0.5 with {\arrow{latex}}}, postaction={decorate}] (1,0) -- (1,4);
    \draw (0,0) -- (1,0);
    \node[below] at (0.5, 0) {$e_j$};
    \node[left] at (0,2) {$r_j^-$};
    \node[right] at (1,2) {$r_j^+$};
    \draw[<->] (2, 2) -- (3, 2);
    \fill[pattern=north west lines, pattern color=pallido] (2.8,4) -- (4.2,2) -- (3,0) -- (4,0) -- (5.2,2) -- (3.8,4);
    \draw (3,0) -- (4,0);
    \node[below] at (3.5, 0) {$e_j$};
    \draw[decoration={markings, mark=at position 0.5 with {\arrow{latex}}}, postaction={decorate}] (3,0) -- (4.2, 2);
    \draw[decoration={markings, mark=at position 0.5 with {\arrow{latex}}}, postaction={decorate}] (4,0) -- (5.2, 2);
    \node[left] at (3.6,1) {$r_{j,1}^-$};
    \node[right] at (4.6,1) {$r_{j,1}^+$};
    \draw[decoration={markings, mark=at position 0.5 with {\arrow{latex}}}, postaction={decorate}] (4.2,2) -- (2.8,4);
    \node[left] at (3.5,3) {$r_{j,2}^-$};
    \node[right] at (4.5,3) {$r_{j,2}^+$};
    \draw[decoration={markings, mark=at position 0.5 with {\arrow{latex}}}, postaction={decorate}] (5.2, 2) -- (3.8, 4);
    \fill[white] (3,0) .. controls (3.95,1.05) .. (5, 2).. controls (4.05,0.95) .. (3, 0);
    \draw (3,0) .. controls (3.95,1.05) .. (5, 2).. controls (4.05,0.95) .. (3, 0);
    \fill[white] (3,4) .. controls (3.95,2.95) .. (5, 2).. controls (4.05,3.05) .. (3, 4);
    \draw (3,4) .. controls (3.95,2.95) .. (5, 2).. controls (4.05,3.05) .. (3, 4);
    \end{tikzpicture}
    \caption{Changing the fundamental domain for the infinite half cylinder.}
    \label{fig:gennontrivhol2}
\end{figure}
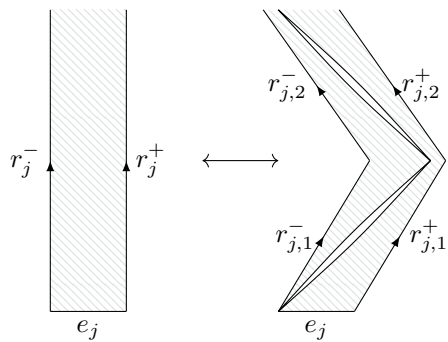

\noindent For the parallelograms, we show that after a sequence of Dehn twists as in section \ref{sec:mcgaction}, we can ensure that the parallelogram to be cut out stays inside a particular fundamental domain for the half strips. Given a non-degenerate parallelogram and some $e_j$, we can assume, after a Dehn twist if necessary, that neither side of the parallelogram is collinear with $e_j$. Assuming for simplicity that $e_j = 1$, we require that the horizontal translates of the parallelogram by 1 unit are disjoint. If the sides are labelled $a$ and $b$ as shown in Figure \ref{fig:gennontrivhol3}, initially the translates need not be disjoint. However, if we consider parallelograms with sides $a$ and $b + na$, for an integer $n$, then the parallelograms are disjoint when the point $B$ is to the right of the side $b$ based at $A$. Algebraically, this condition can be stated as
\begin{gather}
    \Im\bigg( \overline{(b+na)} (-a + 1 + b + na) \bigg) < 0 \\
    \text{i.e. } \Im \big ( -\bar{b}a + \bar{b} + n\bar{a} \big) < 0
\end{gather}
 which is true for $n$ sufficiently large.  Thus, we can place the slits and the parallelogram cut outs adjacent to each other and add the $g$ handles in such a way that the resulting differential has only one zero. Splitting the zero completes the construction.\\

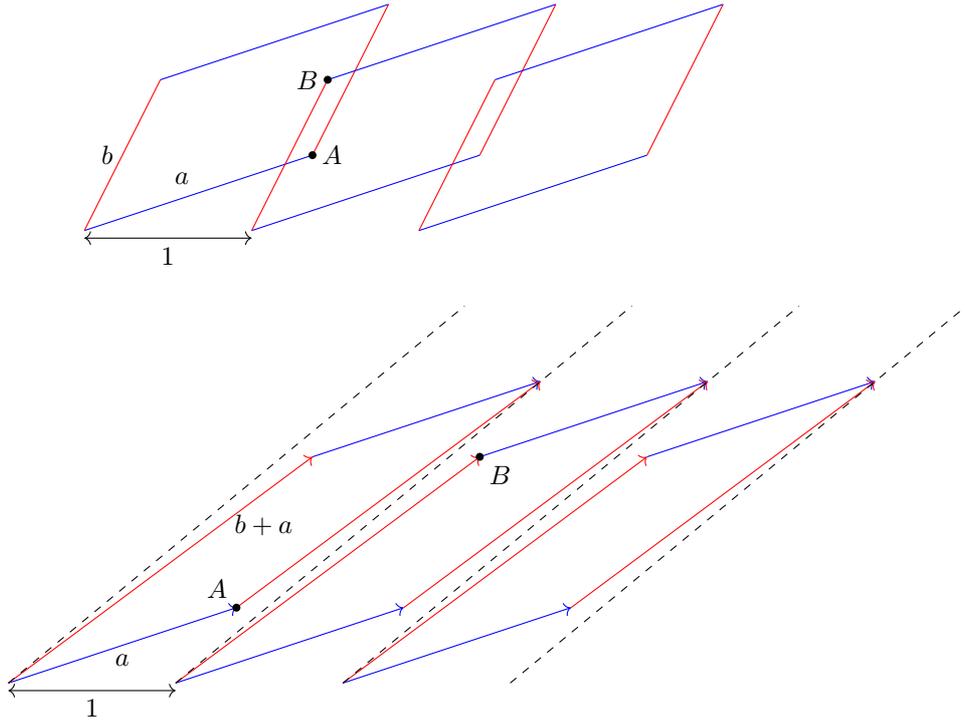
\begin{figure}[h!]
    \centering
    \begin{tikzpicture}
    \foreach \x [evaluate=\x as \bluetipx using \x + 3] [evaluate=\x as \redtipx using \x+ 1] in {-2.2,0,2.2}
    {
    \draw[blue] (\x, 0) -- ++(3,1);
    \draw[red] (\x, 0) -- ++(1,2);
    \draw[blue] (\redtipx, 2) -- ++(3, 1);
    \draw[red] (\bluetipx, 1) -- ++(1,2);
    }
    \node[right] at (0.8, 1) {$A$};
    \fill (0.8,1) circle (1.5pt);
    \node[left] at (1,2) {$B$};
    \fill (1,2) circle (1.5pt);
    \node[above left] at (-0.7, 0.5) {$a$};
    \node[left] at (-1.7, 1) {$b$};
    \draw[<->] (-2.2, -0.1) -- ++ (2.2,0);
    \node[below] at (-1.1, -0.1) {1};
    \begin{scope}[shift = {(-1,-6)}]
    \foreach \x [evaluate=\x as \bluetipx using \x + 3] [evaluate=\x as \redtipx using \x+ 4] in {-2.2,0,2.2}
    {
    \draw[->, blue] (\x, 0) -- ++(3,1);
    \draw[->, red] (\x, 0) -- ++(4,3);
    \draw[->, blue] (\redtipx, 3) -- ++(3, 1);
    \draw[->, red] (\bluetipx, 1) -- ++(4,3);
    \draw[dashed] (\x,0) -- ++(6,5);
    }
    \draw[dashed] (4.4,0) -- ++(6,5);
    \node[above left] at (0.8, 1) {$A$};
    \fill (0.8,1) circle (1.5pt);
    \node[below right] at (4,3) {$B$};
    \fill (4,3) circle (1.5pt);
    \node[below] at (-0.7, 0.5) {$a$};
    \node[right] at (0.65, 2.1) {$b+a$};
    \draw[<->] (-2.2, -0.1) -- ++ (2.2,0);
    \node[below] at (-1.1, -0.1) {1};
    \end{scope}
    \end{tikzpicture}
    \caption{For sufficiently large $n$, horizontal translates of the parallelogram formed by $a$ and $b+na$ by 1 unit are disjoint. In this example, $n=1$. The dotted line depicts a new fundamental domain for the infinite half cylinder with base $e_j$ from which a parallelogram with sides $a$ and $b+na$ can be cut out.} 
     \label{fig:gennontrivhol3}
\end{figure}

\paragraph{\textbf{Case 2.}} The remaining case is that of when $\textsf{Im}(\chi_n)$ is contained in the $\mathbb{R}$-span of some $c \in \mathbb{C}$, but there does not exist any $\widetilde{c} \in \mathbb{C}$ such that $\textsf{Im}(\chi_n)$ is contained in the $\mathbb{Q}$-span of  $\widetilde{c} \in \mathbb{C}$. We can obtain the differential on the $n$ punctured sphere with the required properties by using Proposition \ref{hgc:mainprop}, and add all handles for which the handle holonomy is not along $c$. Thus, we need to only look at the image of handle generators $\alpha_i, \beta_i$ such that $\chi(\alpha_i)= \lambda_i c$ and $\chi(\beta_i)=\mu_i c$ for real valued $\lambda_i$ and $\mu_i$ which may be assumed to be positive after Dehn twists. Now, if $\frac{\lambda}{\mu}$ is irrational, by Lemma \ref{lem:irrsmallhol} we may assume $\lambda$ and $\mu$ to be arbitrarily small. Otherwise, we know that there exists $i$ such that $\chi_n(\gamma_i)$ is not a rational multiple of $\lambda c$. Considering the handle generator $\widetilde{\alpha_i}$ as in the proof of Lemma \ref{lem:handholnonzero} in place of $\alpha_i$ now gives us irrational ratio and we can make the handle holonomy arbitrarily small. This allows us to construct these handles using arbitrarily small slits made adjacent to each other with one end of the sequence of slits being a zero of the differential as shown in Figure \ref{fig:handrspan}.

 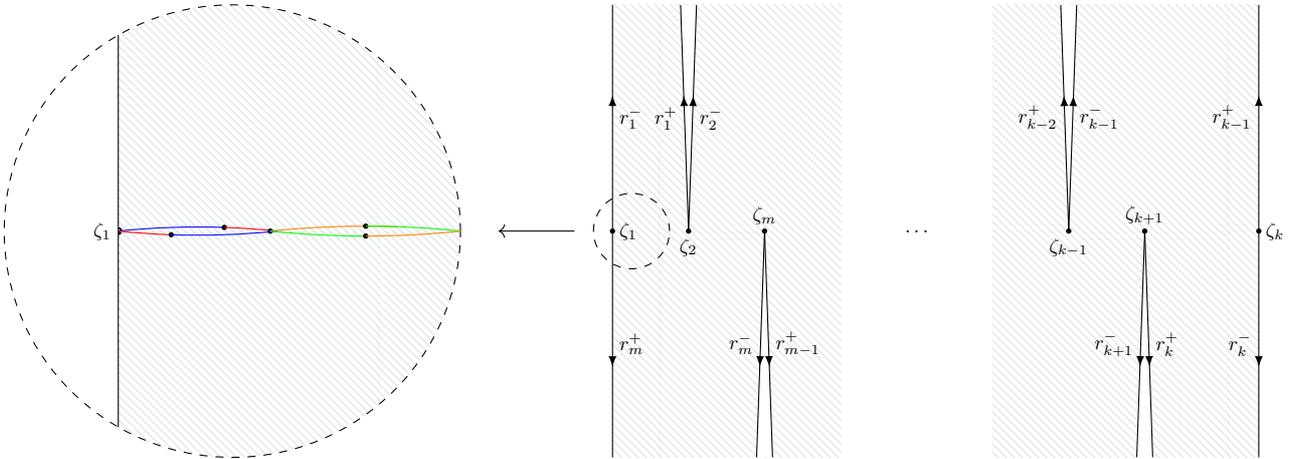
\begin{figure}[!h]
     \centering
     \begin{tikzpicture}[scale=1, every node/.style={scale=0.75}]
        \definecolor{pallido}{RGB}{221,227,227}
         \fill[pattern=north west lines, pattern color=pallido] (0,3) -- (0,-3) -- (3,-3) -- (3, 3);
         \fill[pattern=north west lines, pattern color=pallido] (5,3) -- (5,-3) -- (8.5,-3) -- (8.5, 3);
         \fill[white] ($(1,0) + (92:3)$)--(1,0)-- ++(88:3);
         \fill[white] ($(2,0) + (-92:3)$)--(2,0)-- ++(-88:3);
         \fill[white] ($(6,0) + (92:3)$)--(6,0)-- ++(88:3);
         \fill[white] ($(7,0) + (-92:3)$)--(7,0)-- ++(-88:3);
         \draw[decoration={markings, mark=at position 0.6 with {\arrow{latex}}}, postaction={decorate}] (0,0) -- (0,3);
         \draw[decoration={markings, mark=at position 0.6 with {\arrow{latex}}}, postaction={decorate}] (0,0) -- (0,-3);
         \draw[decoration={markings, mark=at position 0.6 with {\arrow{latex}}}, postaction={decorate}] (1,0) -- ++(92:3);
         \draw[decoration={markings, mark=at position 0.6 with {\arrow{latex}}}, postaction={decorate}] (1,0) -- ++(88:3);
         \draw[decoration={markings, mark=at position 0.6 with {\arrow{latex}}}, postaction={decorate}] (2,0) -- ++(-88:3);
         \draw[decoration={markings, mark=at position 0.6 with {\arrow{latex}}}, postaction={decorate}] (2,0) -- ++(-92:3);
         \node at (4,0) {$\ldots$};
         \draw[decoration={markings, mark=at position 0.6 with {\arrow{latex}}}, postaction={decorate}] (6,0) -- ++(92:3);
         \draw[decoration={markings, mark=at position 0.6 with {\arrow{latex}}}, postaction={decorate}] (6,0) -- ++(88:3);
         \draw[decoration={markings, mark=at position 0.6 with {\arrow{latex}}}, postaction={decorate}] (7,0) -- ++(-88:3);
         \draw[decoration={markings, mark=at position 0.6 with {\arrow{latex}}}, postaction={decorate}] (7,0) -- ++(-92:3);
         \draw[decoration={markings, mark=at position 0.6 with {\arrow{latex}}}, postaction={decorate}] (8.5,0) -- (8.5,-3);
        \draw[decoration={markings, mark=at position 0.6 with {\arrow{latex}}}, postaction={decorate}] (8.5,0) -- (8.5,3);
         \fill (0,0) circle (1pt);
         \fill (1,0) circle (1pt);
         \fill (2,0) circle (1pt);
         \fill (6,0) circle (1pt);
         \fill (7,0) circle (1pt);
         \fill (8.5,0) circle (1pt);
         \node[right] at (0, 1.5) {$r_1^-$};
         \node[left] at ($(1,0) + (92:1.5)$) {$r_1^+$};
         \node[right] at ($(1,0) + (88:1.5)$) {$r_2^-$};
         \node[left] at ($(6,0) + (92:1.5)$) {$r_{k-2}^+$};
         \node[right] at ($(6,0) + (88:1.5)$) {$r_{k-1}^-$};
        \node[left] at (8.5, 1.5) {$r_{k-1}^+$};
        \node[right] at (0, -1.5) {$r_m^+$};
         \node[left] at ($(2,0) + (-92:1.5)$) {$r_m^-$};
         \node[right] at ($(2,0) + (-88:1.5)$) {$r_{m-1}^+$};
         \node[left] at ($(7,0) + (-92:1.5)$) {$r_{k+1}^-$};
         \node[right] at ($(7,0) + (-88:1.5)$) {$r_{k}^+$};
        \node[left] at (8.5, -1.5) {$r_{k}^-$};
        \node[right] at (0,0) {$\zeta_1$};
        \node[below] at (1,0) {$\zeta_2$};
        \node[below] at (6,0) {$\zeta_{k-1}$};
        \node[right] at (8.5,0) {$\zeta_k$};
        \node[above] at (7,0) {$\zeta_{k+1}$};
        \node[above] at (2,0) {$\zeta_m$};
        \draw [dashed] (0.25,0) circle (0.5);
        \draw[->] (-0.5,0) -- (-1.5, 0);
        
        \begin{scope}[shift = {(-1, 0)}]
        \fill (-5.5, 0) circle (1.5pt);
        \fill[pattern=north west lines, pattern color=pallido] (-4,0) circle (3);
        \fill[white] (-5.5, -3) -- (-5.5, 3) -- (-7, 3) -- (-7, -3);
        \draw[dashed] (-4,0) circle (3);
        \draw (-5.5, 2.6) -- (-5.5, -2.6);
        \node[left] at (-5.5, 0) {$\zeta_1$};
        \draw[ blue] (-5.5,0) arc [start angle = 96,end angle = 88,radius = 10];
        \coordinate (C) at ($(-5.5,0) + (-84:10) + (88:10)$);
        \fill (C) circle (1pt);
        \draw[ red] (C) arc [start angle = 88,end angle = 84,radius = 10];
        \fill (-3.5, 0) circle (1 pt);
        \draw[ red] (-5.5,0) arc [start angle = -96,end angle = -92,radius = 10];
        \coordinate (D) at ($(-5.5,0) + (84:10) + (-92:10)$);
        \fill (D) circle (1pt);
        \draw[blue] (D) arc [start angle = -92,end angle = -84,radius = 10];
        \draw[orange] (-3.5,0) arc [start angle = 96,end angle = 88,radius = 12];
        \coordinate (C) at ($(-3.5,0) + (-84:12) + (90:12)$);
        \fill (C) circle (1pt);
        \draw[green] (C) arc [start angle = 90,end angle = 84,radius = 12];
        \draw[green] (-3.5,0) arc [start angle = -96,end angle = -90,radius = 12];
        \coordinate (D) at ($(-3.5,0) + (84:12) + (-90:12)$);
        \fill (D) circle (1pt);
        \draw[orange] (D) arc [start angle = -90,end angle = -84,radius = 12];
        \end{scope}
     \end{tikzpicture}
     \caption{Constructing handles with arbitrarily small real valued handle holonomy on the surface obtained from Proposition \ref{hgc:mainprop}} \label{fig:handrspan}
 \end{figure}

\subsubsection{When $\textsf{\emph{Im}}(\chi)$ is not contained in the $\mathbb{Q}$-span of some $c \in \mathbb{C}$} We now come to the last case where we assume that $\chi_n$ is not trivial, all of the $p_i$'s are one, and $\textsf{Im}(\chi_n)$ is contained in the $\mathbb{Q}$-span of some $c \in \mathbb{C}$ but $\textsf{Im}(\chi)$ is not contained in the $\mathbb{Q}$-span of $c$. This means that either 
\begin{itemize}
    \item $\textsf{Im}(\chi_g)$ is not contained in the $\mathbb{R}$-span of $c$, or
    \item $\textsf{Im}(\chi_g)$ is contained in the $\mathbb{R}$-span of $c$, but not in the $\mathbb{Q}$- span of $c$.
\end{itemize}  
We consider these two sub-cases separately. \\

\paragraph{\textbf{Case 1.}} We first consider the case when $\textsf{Im}(\chi_n)$ lies in the $\mathbb{Q}$-span, and therefore the $\mathbb{R}$-span of some $c \in \mathbb{C}$, but $\textsf{Im}(\chi_g)$ does not lie in the $\mathbb{R}$-span of $c$. We may assume that $\textsf{Im}(\chi_n)$ lies along $\mathbb{R}$. Then, following the second construction in the proof of Proposition \ref{propb}, we obtain a translation structure on the $m-$punctures sphere that has the required $\chi_n$, but with two singular points. We shall now attach $g$ handles to this surface, and in the process obtain only one singular point. Note that we have the flexibility to choose the vertical side of the parallelogram $\mathcal{P}$ in the previous construction. If at least one of the $g$ handles has positive volume, then we can attach one such handle by choosing the vertical side of the parallelogram so that the required slit is made along the diagonal of the parallelogram as shown in the bottom half of Figure \ref{fig:collapsezero} for adding a handle of positive volume. After this construction, the resulting translation structure has only one singular point. Otherwise, we have at least one handle with negative volume. Here too, we can ensure, by Dehn twists as in the previous section, that the horizontal translates of the parallelogram formed by the handle holonomy by $\overline{\zeta_{1}\,\zeta_{k}}$ are disjoint. Then we can choose $\mathcal{P}$ so that opposite vertices of the parallelogram to be cut-out coincide with opposite vertices of $\mathcal{P}$ as shown in the top half of Figure \ref{fig:collapsezero}. After adding this handle, the resulting structure has only one singular point. We can now attach the remaining $g-1$ handles in one of the half strips as done previously and complete the construction.\\

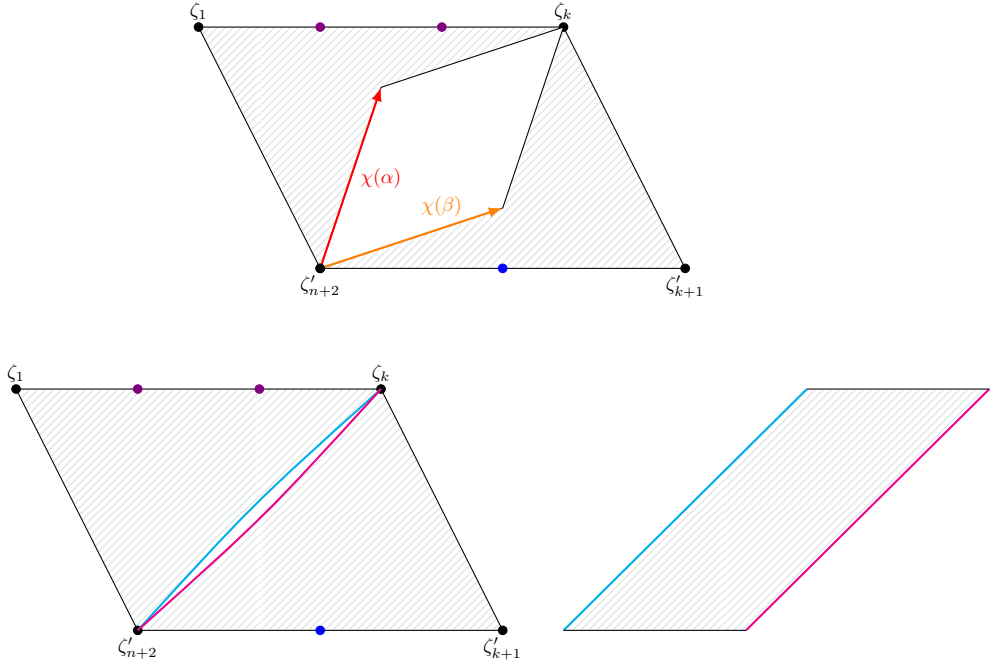
\begin{figure}[!h]
    \centering
    \begin{tikzpicture}[scale=0.8, every node/.style={scale=0.75}]
    \definecolor{pallido}{RGB}{221,227,227}
    \draw [black, pattern=north east lines, pattern color=pallido] (-7,-2) -- (-1,-2) -- (-3,2) -- (-9,2) -- (-7,-2);
    \draw plot [mark=*, smooth] coordinates {(-7,-2)};
    \draw plot [mark=*, smooth] coordinates {(-1,-2)};
    \draw plot [mark=*, smooth] coordinates {(-3,2)};
    \draw plot [mark=*, smooth] coordinates {(-9,2)};
    \node[above] at (-9,2) {$\zeta_1$};
    \node[below] at (-7,-2) {$\zeta'_{n+2}$};
    \node[above] at (-3,2) {$\zeta_{k}$};
    \node[below] at (-1,-2) {$\zeta'_{k+1}$};
    \draw [violet] plot [mark=*, smooth] coordinates {(-5,2)};
    \draw [violet] plot [mark=*, smooth] coordinates {(-7,2)};
    \draw [blue] plot [mark=*, smooth] coordinates {(-4,-2)};
    
    \begin{scope}[shift = {(-5, 0)}]
    \fill [white, thin, draw=black] (-2,-2) to (1,-1) to (2,2) to (-1,1) to (-2,-2);
    \draw[thick, -latex, orange] (-2,-2) to (1,-1);
    \draw[thick, -latex, red] (-2,-2) to (-1,1);
    \draw [black] plot [mark=*, smooth] coordinates {(-2,-2)};
    \node at (-1,-0.5) {\textcolor{red}{$\chi(\alpha)$}};
    \node[above] at (0,-1.25) {\textcolor{orange}{$\chi(\beta)$}};
    \end{scope}
    
    \begin{scope}[shift={(-3,-6)}]
    \draw [black, pattern=north east lines, pattern color=pallido] (-7,-2) -- (-1,-2) -- (-3,2) -- (-9,2) -- (-7,-2);
    \draw plot [mark=*, smooth] coordinates {(-7,-2)};
    \draw plot [mark=*, smooth] coordinates {(-1,-2)};
    \draw plot [mark=*, smooth] coordinates {(-3,2)};
    \draw plot [mark=*, smooth] coordinates {(-9,2)};
    \node[above] at (-9,2) {$\zeta_1$};
    \node[below] at (-7,-2) {$\zeta'_{n+2}$};
    \node[above] at (-3,2) {$\zeta_{k}$};
    \node[below] at (-1,-2) {$\zeta'_{k+1}$};
    \draw [violet] plot [mark=*, smooth] coordinates {(-5,2)};
    \draw [violet] plot [mark=*, smooth] coordinates {(-7,2)};
    \draw [blue] plot [mark=*, smooth] coordinates {(-4,-2)};
    \fill[white] (-7,-2) .. controls (-5.1, 0.1) .. (-3,2) .. controls (-4.9, -0.1) .. (-7,-2);
    \draw[thick, cyan] (-7,-2) .. controls (-5.1, 0.1) .. (-3,2);
    \draw[thick, magenta] (-3,2) .. controls (-4.9, -0.1) .. (-7,-2);

    \fill[pattern=north east lines, pattern color=pallido] (0,-2) -- (3, -2) -- (7, 2) -- (4, 2) -- (0,-2);
    \draw[thick, cyan] (0,-2) -- (4,2);
    \draw[thick, magenta] (3,-2) -- (7,2);
    \draw (0,-2) -- (3,-2);
    \draw (4,2) -- (7,2);
    \end{scope}
    \end{tikzpicture}
    \caption{Choosing $\mathcal{P}$ in the proof of Proposition \ref{propb} so that adding handles to $\mathcal{P}$ collapses the two zeros into one zero.} \label{fig:collapsezero}
\end{figure}

\paragraph{\textbf{Case 2.}} In what remains, we deal with the case when $\textsf{Im}(\chi_n)$ lies in the $\mathbb{Q}$-span of some $c \in \mathbb{C}$, and $\textsf{Im}(\chi_g)$ lies in the $\mathbb{R}$-span of $c$, but not in the $\mathbb{Q}$-span of $c$. This means that for all handle generators, we have $\chi(\alpha_j) = \lambda_j c$ and $\chi(\beta_j) = \mu_j c$ for non-rational values $\lambda_j$ and $\mu_j$. In what follows, we can assume without loss of generality, that $c = 1$ and make use of the following lemma.

\begin{lem}\label{lem:handholirr}
Let $\chi \in\textsf{\emph{Hom}}(\Gamma_{g,n}, \mathbb{C})$ be a representation such that $\textsf{\emph{Im}}(\chi_n) \subset \mathbb{Q}$, but $\textsf{\emph{Im}}(\chi) \not\subset \mathbb{Q}$. Then, there exists $\varphi \in \text{\emph{Mod}}(S_{g,n})$ such that $\chi \circ \varphi_*(\alpha_i)$ and $\chi \circ \varphi_*(\beta_i)$ are positive and the ratio $\frac{\chi \circ \varphi_*(\alpha_i)}{\chi \circ \varphi_*(\beta_i)}$ is irrational for all $1 \leq i \leq g$. Further, we may assume that $\chi \circ \varphi_*(\alpha_i)$ and $\chi \circ \varphi_*(\beta_i)$ are arbitrarily small.
\end{lem}
\begin{proof}
Since $\textsf{Im}(\chi) \not\subset \mathbb{Q}$, we have at least one handle generator whose image under $\chi$ is irrational. Without loss of generality, we assume that $\chi(\alpha_1)$ is irrational. First, we show the existence of $\widetilde{\varphi} \in \text{Mod}(S_{g,n})$ such that $\chi \circ \widetilde{\varphi}_*(\alpha_i)$ and $\chi \circ \widetilde{\varphi}_*(\beta_i)$ are not both rational. To see this, let $j \geq 2$ be such that $\chi(\alpha_j)$ and $\chi(\beta_j)$ are both rational. Let $\widetilde{\varphi}_j \in \text{Mod}(S_{g,n})$ be such that 
\begin{align}
    (\widetilde{\varphi}_j)_*(\alpha_j) &= \alpha_j + \alpha_1 \\
    (\widetilde{\varphi}_j)_*(\beta_1) &= \beta_1 - \beta_j \\
    (\widetilde{\varphi}_j)_*(\gamma) &= \gamma \text{ for } \gamma \neq \{ \alpha_j, \beta_1 \}
\end{align}
Then, the required $\widetilde{\varphi}$ is the composition of $\widetilde{\varphi}_j$ for all $j$ such that $\chi(\alpha_j)$ and $\chi(\beta_j)$ are both rational. By including Dehn twists in this composition, we may also assume that $\chi \circ \varphi_*(\alpha_i), \chi \circ \varphi_*(\beta_i) > 0$ for $1 \leq i \leq g$. 

\noindent $\widetilde{\varphi}$ need not satisfy the requirements of the lemma since the ratio $\frac{\chi \circ \varphi_*(\alpha_i)}{\chi \circ \varphi_*(\beta_i)}$ may be rational for some $i$ when $\chi \circ \varphi_*(\alpha_i)$ and $\chi \circ \varphi_*(\beta_i)$ are both irrational. Composing $\varphi$ with the diffeomorphisms that take $\alpha_i$ to $\widetilde{\alpha_i}$ for all such $i$ as in the proof of Lemma \ref{lem:handholnonzero} gives $\varphi$ which makes the ratios $\frac{\chi \circ \varphi_*(\alpha_i)}{\chi \circ \varphi_*(\beta_i)}$ irrational. Since Lemma \ref{lem:irrsmallhol} preserves the irrationality of the ratios $\frac{\chi \circ \varphi_*(\alpha_i)}{\chi \circ \varphi_*(\beta_i)}$, we may assume that $\chi \circ \varphi_*(\alpha_i)$ and $\chi \circ \varphi_*(\beta_i)$ are arbitrarily small.
\end{proof}

\noindent As a consequence of Lemma \ref{lem:handholirr} we may assume that, for $1 \leq j \leq g$, $\lambda_j$ and $\mu_j$ are arbitrarily small and positive such that the ratio $\frac{\lambda_j}{\mu_j}$ is irrational. Now let $n$ be such that $\text{max}(\zeta_{k-1}, \zeta_{k+1})< \lambda_1 + (n+1)\mu_1 < \zeta_k$.\\
\noindent Next, we consider the infinite strip $\{ z \in \mathbb{C} \,\vert\, 0 < \Re(z) < \zeta_k \}$ as in the proof of \ref{hgc:mainprop} and make a slit along the horizontal segment joining $\zeta_1$ and $\zeta_1 + \lambda_1 + (n+1)\mu_1$. Marking the points $\zeta_2, \ldots , \zeta_{k-1}$ on the top segment of the slit and the points $\zeta_{k+1}, \ldots, \zeta_{m}$ on the bottom segment as shown in Figure \ref{fig:irrhandhol}, we define $r_j^+$ and $r_j^-$ as before. By making further horizontal slits to the right of the slit already made, and identifying the edges of the slits appropriately, along with the identification of $r_j^+$ with $r_j^-$, we obtain a translation surface whose holonomy at the first handle differs from the required holonomy by $n$ Dehn twists. Reversing these twists gives us the required translation surface after splitting the single zero into zeroes of required orders.\\

\noindent This concludes the proof of Theorem \ref{thm:nthpgs}.

 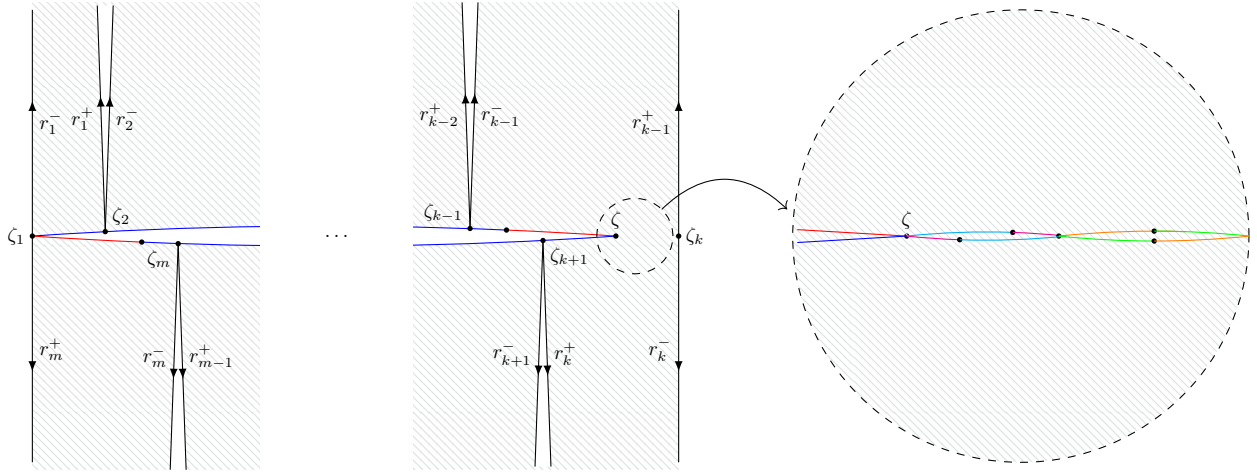
\begin{figure}[!h]
     \centering
     \begin{tikzpicture}[scale=1, every node/.style={scale=0.75}]
        \definecolor{pallido}{RGB}{221,227,227}
         \fill[pattern=north west lines, pattern color=pallido] (0,3.1) -- (0,-3.1) -- (8.5,-3.1) -- (8.5, 3.1);
         \fill[white] (0,0) arc [start angle = 94,end angle = 86,radius = 55] arc [start angle = -86,end angle = -94,radius = 55];
         \draw[blue] (0,0) arc [start angle = 94,end angle = 87.5,radius = 55];
         \coordinate (I) at ($(0,0) + (-86:55) + (87.5:55)$);
         \draw[red] (I) arc [start angle = 87.5,end angle = 86,radius = 55];
         \fill (I) circle (1pt);
         \draw[red] (I) [start angle = 87.5,end angle = 86,radius = 55];
         \draw[red] (0,0) arc [start angle = -94,end angle = -92.5,radius = 55];
         \coordinate (J) at ($(0,0) + (86:55) + (-92.5:55)$);
         \coordinate (K) at ($(0,0) + (86:55) + (-86:55)$);
         \draw[blue] (J) arc [start angle = -92.5,end angle = -86,radius = 55];
         \fill (J) circle (1pt);
         \fill (K) circle (1pt);
         \fill[white] (3,3.1) -- (3,-3.1) -- (5,-3.1) -- (5, 3.1);
         \coordinate (E) at ($(0,0) + (-86:55) + (93:55)$);
         \coordinate (F) at ($(0,0) + (86:55) + (-92:55)$);
         \coordinate (G) at ($(0,0) + (-86:55) + (88:55)$);
         \coordinate (H) at ($(0,0) + (86:55) + (-87:55)$);
         \fill[white] ($(E) + (92:3)$)--(E)-- ++(88:3);
         \fill[white] ($(F) + (-92:3)$)--(F)-- ++(-88:3);
         \fill[white] ($(G) + (92:3)$)--(G)-- ++(88:3);
         \fill[white] ($(H) + (-92:3)$)--(H)-- ++(-88:3);
         \draw[decoration={markings, mark=at position 0.6 with {\arrow{latex}}}, postaction={decorate}] (0,0) -- (0,3);
         \draw[decoration={markings, mark=at position 0.6 with {\arrow{latex}}}, postaction={decorate}] (0,0) -- (0,-3);
         \draw[decoration={markings, mark=at position 0.6 with {\arrow{latex}}}, postaction={decorate}] (E) -- ++(92:3);
         \draw[decoration={markings, mark=at position 0.6 with {\arrow{latex}}}, postaction={decorate}] (E) -- ++(88:3);
         \draw[decoration={markings, mark=at position 0.6 with {\arrow{latex}}}, postaction={decorate}] (F) -- ++(-88:3);
         \draw[decoration={markings, mark=at position 0.6 with {\arrow{latex}}}, postaction={decorate}] (F) -- ++(-92:3);
         \node at (4,0) {$\ldots$};
         \draw[decoration={markings, mark=at position 0.6 with {\arrow{latex}}}, postaction={decorate}] (G) -- ++(92:3);
         \draw[decoration={markings, mark=at position 0.6 with {\arrow{latex}}}, postaction={decorate}] (G) -- ++(88:3);
         \draw[decoration={markings, mark=at position 0.6 with {\arrow{latex}}}, postaction={decorate}] (H) -- ++(-88:3);
         \draw[decoration={markings, mark=at position 0.6 with {\arrow{latex}}}, postaction={decorate}] (H) -- ++(-92:3);
         \draw[decoration={markings, mark=at position 0.6 with {\arrow{latex}}}, postaction={decorate}] (8.5,0) -- (8.5,-3);
        \draw[decoration={markings, mark=at position 0.6 with {\arrow{latex}}}, postaction={decorate}] (8.5,0) -- (8.5,3);
         \fill (0,0) circle (1pt);
         \fill (E) circle (1pt);
         \fill (F) circle (1pt);
         \fill (G) circle (1pt);
         \fill (H) circle (1pt);
         \fill (8.5,0) circle (1pt);
         \node[right] at (0, 1.5) {$r_1^-$};
         \node[left] at ($(E) + (92:1.5)$) {$r_1^+$};
         \node[right] at ($(E) + (88:1.5)$) {$r_2^-$};
         \node[left] at ($(G) + (92:1.5)$) {$r_{k-2}^+$};
         \node[right] at ($(G) + (88:1.5)$) {$r_{k-1}^-$};
        \node[left] at (8.5, 1.5) {$r_{k-1}^+$};
        \node[right] at (0, -1.5) {$r_m^+$};
         \node[left] at ($(F) + (-92:1.5)$) {$r_m^-$};
         \node[right] at ($(F) + (-88:1.5)$) {$r_{m-1}^+$};
         \node[left] at ($(H) + (-92:1.5)$) {$r_{k+1}^-$};
         \node[right] at ($(H) + (-88:1.5)$) {$r_{k}^+$};
        \node[left] at (8.5, -1.5) {$r_{k}^-$};
        \node[left] at (0,0) {$\zeta_1$};
        \node[above right] at (E) {$\zeta_2$};
        \node[above left] at (G) {$\zeta_{k-1}$};
        \node[right] at (8.5,0) {$\zeta_k$};
        \node[below right] at (H) {$\zeta_{k+1}$};
        \node[below left] at (F) {$\zeta_m$};
        \node[above] at (K) {$\zeta$};
        \draw [dashed] ($(K)+(0.25,0)$) circle (0.5);
        \draw[->] ($(K)+(0.25,0)+(45:0.5)$) parabola[bend pos=0.5] bend +(0,0.4) ($(K)+(1.9,0)+(45:0.5)$);
        
        \begin{scope}[shift = {(17, 0)}]
        \fill (-5.5, 0) circle (1pt);
        \fill[pattern=north west lines, pattern color=pallido] (-4,0) circle (3);
        \draw[dashed] (-4,0) circle (3);
        \node[above] at (-5.5, 0) {$\zeta$};
        \draw[red] (-5.5,0) arc [start angle = 86,end angle = 86.25,radius = 330];
        \draw[blue] (-5.5,0) arc [start angle = -86,end angle = -86.25,radius = 330];
        \draw[cyan] (-5.5,0) arc [start angle = 96,end angle = 88,radius = 10];
        \coordinate (C) at ($(-5.5,0) + (-84:10) + (88:10)$);
        \fill (C) circle (1pt);
        \draw[magenta] (C) arc [start angle = 88,end angle = 84,radius = 10];
        \fill (-3.5, 0) circle (1 pt);
        \draw[magenta] (-5.5,0) arc [start angle = -96,end angle = -92,radius = 10];
        \coordinate (D) at ($(-5.5,0) + (84:10) + (-92:10)$);
        \fill (D) circle (1pt);
        \draw[cyan] (D) arc [start angle = -92,end angle = -84,radius = 10];
        \draw[orange] (-3.5,0) arc [start angle = 96,end angle = 88,radius = 12];
        \coordinate (C) at ($(-3.5,0) + (-84:12) + (90:12)$);
        \fill (C) circle (1pt);
        \draw[green] (C) arc [start angle = 90,end angle = 84,radius = 12];
        \draw[green] (-3.5,0) arc [start angle = -96,end angle = -90,radius = 12];
        \coordinate (D) at ($(-3.5,0) + (84:12) + (-90:12)$);
        \fill (D) circle (1pt);
        \draw[orange] (D) arc [start angle = -90,end angle = -84,radius = 12];
        \end{scope}
     \end{tikzpicture}
     \caption{ Marking the points $\zeta_2, \ldots , \zeta_{k-1}$ on the top segment and the points $\zeta_{k+1}, \ldots, \zeta_{m}$ on the bottom segment of the horizontal slit made across $\zeta_1$ and $\zeta := \zeta_1 + \lambda_1 + (n+1)\mu_1$ for the construction of the first handle. Further handles, which may be assumed to have arbitrarily small handle holonomy, are constructed to the right of $\zeta$.} \label{fig:irrhandhol}
 \end{figure}

\section{The case of rational holonomy} \label{sec:comb}

\noindent The remaining case is when all the poles are of order $1$, and $\textsf{Im}(\chi)$ is  contained in the $\mathbb{Q}$-span of some $c \in \mathbb{C}^\ast$. As before, we shall say that a representation $\chi:\Gamma_{g,n}\to \mathbb{C}$ is \textit{realizable} if there is a translation structure on $S_{g,n}$ with holonomy $\chi$. If we can prescribe, in addition, the zeros of the corresponding abelian differential and their orders, we shall say $\chi$ is \textit{realizable in a prescribed stratum}. Note that since the holonomy around a puncture equals the residue of the simple pole of the corresponding abelian differential, no puncture is an apparent singularity.  This necessitates the additional assumption that $\chi(\gamma) \neq 0$ for any $\gamma$ that is a loop around a puncture.



\subsection{Some reductions}\label{ssec:somered} We start by noting that it suffices to consider the case when the holonomy is \textit{integral}.

\begin{lem}
Let $\chi : \Gamma_{g,n} \to \mathbb{C}$ be a non-trivial representation, such that $\textsf{\emph{Im}}(\chi)$ is  contained in the $\mathbb{Q}$-span of some $c \in \mathbb{C}^\ast$. Then there is a surjective representation $\widehat{\chi}:\Gamma_{g,n}  \to \mathbb{Z}$ such that $\chi$ is realizable in a  prescribed stratum if and only if $\widehat{\chi}$ is realizable in the same stratum. 
\end{lem}

\begin{proof}
Let $\Gamma_{g,n} = H_1(S_{g,n}; \mathbb{Z})$ be generated by the handle-generators $\alpha_1,\beta_1, \alpha_2,\beta_2,\ldots, \alpha_g, \beta_g$ and loops around all except one puncture, namely $\gamma_1,\gamma_2,\ldots, \gamma_{n-1}$. From our assumption, the $\chi-$images of these generators are of the form $q_i\cdot c$ where $q_i \in \mathbb{Q}$ for $1\leq i\leq 2g+n-1$.  Let $N$ be the least common multiple of the denominators of the rational numbers $\{q_i\}$.\\
\noindent If $\chi$ is realizable in a given stratum, then scale the corresponding abelian differential by the constant $N/c$. The resulting translation structure will then have integral holonomy $\widehat{\chi}:\Gamma_{g,n} \to \mathbb{Z}$, where $\widehat{\chi}(\gamma) = (N/c) \cdot \chi(\gamma)$ for any $\gamma \in \Gamma_{g,n}$.  Conversely, suppose $\widehat{\chi}:\Gamma_{g,n} \to \mathbb{Z}$ is defined by sending the above generators of $\Gamma_{g,n}$ to $N \cdot q_i$. If $\widehat{\chi}$ is realizable, then scaling the corresponding abelian differential on $S_{g,n}$ by $c/N$ results in a translation structure with holonomy $\chi$. Note that such scaling preserves the numbers of zeros and poles, and do not change their orders. Hence if one of the representations is realizable in a prescribed stratum, so is the other.\\ 
\noindent We can further assume that the representation $\widehat{\chi}$ \textit{surjects} on to the integers $\mathbb{Z}$: any non-trivial proper subgroup of $\mathbb{Z}$ is of the form $d\mathbb{Z}$ for some integer $d\geq 2$, a  further scaling by $d$ would suffice to obtain a surjection.
\end{proof}

\medskip
\noindent We shall also need the following reduction. As before, given a representation $\chi:\Gamma_{g,n} \to \mathbb{Z}$ we can define the restrictions $\chi_g:\Gamma_g\to \mathbb{Z}$ and  $\chi_n:\Gamma_{0,n}\to \mathbb{Z}$. 

\begin{lem}\label{lem:red2} 
 Let $\chi:\Gamma_{g,n}\to \mathbb{Z}$ be a surjective representation. Then there exists a choice of handle-generators $\{\alpha_i,\beta_i\}_{1\leq i\leq g}$ such that $\chi_g(\eta) = 1$ for any handle-generator $\eta$.
\end{lem}

\begin{proof}
\noindent Since the holonomy around some (in fact, each) puncture is non-trivial, we can assume, by changing the handle-generators if need be, that the $\chi_g-$image around \textit{some} handle-generator is non-trivial.  It follows, as in section 11.1, that the $\chi_g-$image of each handle-generator is non-trivial.\\
\noindent Let $\vec{t} = (t_1,t_2,\ldots, t_{2g+n-1}) \in \mathbb{Z}^{2g+n-1}$ be the integer tuple that are the $\chi-$images of a generating set (i.e.\ the union of the handle-generators and the loops around all punctures except one.) Since these images generate $\mathbb{Z}$, this tuple is a primitive integer vector in $\mathbb{Z}^{2g+n-1}$, that is, the greatest common divisor of their entries in $1$. Consider the sub-tuple $v = (t_1,t_2,\ldots, t_{2g})^T\in \mathbb{Z}^{2g}$ that is the image of the handle-generators, and is thought of as a column vector.\\ 
\noindent \textit{Case 1: $v$ is a primitive integer vector.} There is an element $A\in \text{Sp}(2g, \mathbb{Z})$ such that $A\cdot v = (1,1,\ldots 1)$. This is because $\text{Sp}(2g, \mathbb{Z})$ acts transitively on primitive integer vectors in $\mathbb{Z}^{2g}$ (see, for example, section 5.1 of \cite{MS}). Such a transformation $A$ is induced by a mapping class of $S_{g,n}$ supported on the subsurface $S_{g,0}$ containing the handles; this induces a change of a generating basis of $\Gamma_{g,n}$ yielding the desired set of handle-generators. \\
\noindent \textit{Case 2: The entries of $v$ have greatest common divisor $d$.}  In this case there is an element $A\in \text{Sp}(2g, \mathbb{Z})$ such that $A\cdot v = (d,d,\ldots d)$.  Since $\vec{t}$ is a primitive integer vector, there is a puncture with holonomy $d^\prime$ that is not divisible by $d$.  We can change one of the handle-generators by a mapping class that adds a loop around that puncture; the resulting handle-generator then has holonomy $d + d^\prime$. The new holonomy-vector for the handles is now a primitive integer vector, so we are reduced to Case 1.
\end{proof}

\subsection{Branched covers} The main result of this subsection is to show that our problem of constructing a translation structure on $S_{g,n}$ with prescribed integral holonomy $\chi$ and prescribed zeros and poles, is equivalent to constructing a branched cover of $\Bbb S^2$ with prescribed branching data. The following terminology will be useful.

\begin{defn}\label{defn:ppunc}  For a representation $\chi:\Gamma_{g,n} \to \mathbb{Z}$
we say that a puncture is \textit{positive} if the holonomy around it is a translation by a positive integer, and is  \textit{negative} otherwise. 
\end{defn} 

\begin{prop}\label{prop:bc}  Suppose  $\chi:\Gamma_{g,n} \to \mathbb{Z}$ is a non-trivial representation. Let the holonomies around the positive punctures be given by the integer-tuple $\lambda \in \mathbb{Z}_+^k$ and the holonomies around the negative punctures be given by $-\mu$ where $\mu \in \mathbb{Z}_+^l$.  
Then the following are equivalent:
\begin{enumerate}
    \item There is a translation structure on $S_{g,n}$ with holonomy $\chi$, with simple poles at the punctures and a set of $r$ zeros with prescribed orders $(d_1,d_2,\ldots, d_r)$ that satisfies \[\displaystyle\sum\limits_{i=1}^r d_i = 2g+ n-2.\] 
    \item There is a branched covering $f:S_{g} \to \Bbb S^2$  with two special branch values $p_+$ and $p_-$ such that
    \begin{itemize}
        \item $f^{-1}(p_+)$ is a set of $k$ points with local degrees given by $\lambda$, 
        \item $f^{-1}(p_-)$ is a set of $l$ points with local degrees given by $\mu$, where $k+l = n$, and 
        \item apart from these, there are exactly $r$ other branch points on $S_g$ with local degrees $(d_1+1,\ldots, d_r+1)$.
    \end{itemize} 
    \end{enumerate}

\end{prop}

\begin{proof}
Notice that, by Lemma \ref{lem:red2} we can assume, without loss of generality, that $\chi(\gamma) =1$ for any handle-generator $\gamma \in \Gamma_{g,n}$.\\
\noindent We begin with (2) $\implies$ (1). Equip $\Bbb S^2 \setminus \{p_-,p_+\}$ with a translation structure in which it is isomorphic to an infinite Euclidean cylinder of circumference $1$, such that the holonomy around $p_-$ is $-1$ and $p_+$ is $+1$ respectively. (Recall that the holonomy around a puncture is the holonomy around a loop that is oriented anti-clockwise.) Note that the corresponding abelian differential has simple poles at $p_\pm$.\\ 
\noindent Pulling back via the branched covering $f$, we obtain a translation structure on $S_g \setminus \big(f^{-1}(p_-) \cup f^{-1}(p_+)\big)$. Note that the $n$  punctures correspond to simple poles of the corresponding abelian differential, since they are lifts of the simple poles at $p_\pm$ on $\Bbb S^2$. The holonomy around any puncture in $f^{-1}(p_+)$ would equal the local degree, i.e.\ the corresponding entry of $\lambda$, since 
\begin{equation}
    f^\ast \Big(\frac{dz}{z}\Big) = \text{ord}(f) \cdot \frac{dz}{z}
\end{equation} in local coordinates. Similarly, the holonomy around any puncture in $f^{-1}(p_-)$ would equal the corresponding entry of $-\mu$. Thus, the translation structure we obtained on $S_{g,n}$ has a holonomy $\chi^\prime:\Gamma_{g,n} \to \mathbb{Z}$ such that the holonomies around the punctures agree with those of $\chi$. By a rescaling, if necessary, we  can assume that $\chi^\prime$ is surjective. Now by Lemma \ref{lem:red2} there exists a choice of handle-generators such that $\chi^\prime = \chi$, and we are done.\\

\noindent For the reverse implication, namely (1)$\implies$ (2), consider the developing map $\widetilde{f}: \widetilde{S_{g,n}} \to \mathbb{C}$ of the given translation structure on $S_{g,n}$. This mapping is $\chi$-equivariant, namely $f(\gamma \cdot x) = \chi(\gamma) \cdot f(x)$ for any $x\in \widetilde{S_{g,n}}$ and any $\gamma \in \Gamma_{g,n}$. Since $\chi(\gamma)$ is some integer translation, the developing map descends to a map \begin{equation}
    \overline{f}:S_{g,n} \to \mathcal{C} = {\mathbb{C}}/{\langle z\mapsto z+1\rangle}.
\end{equation}
\noindent We can think of the target cylinder $\mathcal{C}$ as $\Bbb S^2$ with two punctures $p_+$ and $p_-$, corresponding to the two ends. Recall that the developing map $\widetilde{f}$ is an immersion except at any zero of the abelian differential (corresponding to the given translation structure), where it has local degree equal to one more than the order of the zero.  Thus, the map $\overline{f}$ has $r$ branch points with local degrees $(d_1+1,d_2+1,\ldots, d_r+1)$. It follows from the flat geometry of the abelian differential near simple poles, that the map $\overline{f}$ near any puncture maps to one of the two ends of the cylinder $\mathcal{C}$. In fact, the restriction of $\overline{f}$ to any cylindrical end of the translation structure on $S_{g,n}$  is a covering map to an end of $\mathcal{C}$.\\
\noindent Hence,  $\overline{f}$ extends continuously to a branched covering of the closed surface $f:S_g\to \Bbb S^2$, where any  positive puncture  (see Definition \ref{defn:ppunc}) is mapped to $p_+$ and any negative puncture is mapped to $p_-$. Indeed, if $\omega$ is the abelian differential corresponding to the translation structure on $S_{g,n}$, then its lift to the universal cover $\widetilde{\omega} = \widetilde{f}^\ast (dz)$; hence $\omega = f^\ast(\omega_o)$ where $\omega_o = dz/z$ is a meromorphic abelian differential on $\mathcal{C} = \mathbb{C}/\langle z\mapsto z+1\rangle$ with simple poles at $p_-$ and $p_+$ and residues $+1$ and $-1$ respectively. The positive (resp. negative) punctures are exactly those poles where $\omega$ has positive (resp. negative) residues. The local degree of $f$ at any point in $f^{-1}(p_+) \cup f^{-1}(p_-)$ is exactly the ratio of residues of $\omega$ and $\omega_o$ at the point, and hence the tuple of local degrees equals $\lambda$ at the positive punctures and $-\mu$ at the negative punctures. Thus, $f:S_g \to \Bbb S^2$ is our desired branched covering.
\end{proof}

\noindent We also note the following corollary, the converse of which will be proved in the next subsection:

\begin{cor}\label{cor:pgen}  Let $\chi:\Gamma_{g,n} \to \mathbb{Z}$ be a surjective representation.  Let the holonomies around the positive punctures be given by the integer-tuple $\lambda \in \mathbb{Z}_+^k$ and the holonomies around the negative punctures be given by $-\mu$ where $\mu \in \mathbb{Z}_+^l$.  Then if $\chi$ is realizable in the stratum where the punctures are simple poles and the zeros have orders  $(d_1,d_2,\ldots, d_r)$  then $d_i \leq d-1$ for each $1\leq i\leq r$ where $d = \sum\limits_{i=1}^k \lambda_i = \sum\limits_{i=1}^l \mu_i$. 
\end{cor}
\begin{proof}
By Proposition \ref{prop:bc} there is a branched cover $f:S_g \to \Bbb S^2$ of degree $d$ whose branch points include $r$ points with local degrees $d_i+1$ for $1\leq i\leq r$.  Since the local degree is at most the degree of the map, it follows that $d_i \leq d-1$ for each $i$.
\end{proof}

\begin{rmk}
The above corollary  implies that for  fixed integer vectors $\lambda$ and $\mu$, there is a \textit{finite} set of possible genera $g$ for which there is a translation structure on $S_{g,n}$ with integral holonomy having the holonomies around the punctures given by $\lambda$ and $-\mu$.  This is because 
\begin{equation}
    2g + n-2 = \sum\limits_{i=1}^r d_i \leq (d-1)\cdot r < (d-1)\cdot d,
\end{equation} which provides an upper bound on $g$. 
\end{rmk}

\subsection{Proof of Theorem \ref{main:thme}} \label{pfthme} 

Ideas from the recent work in \cite{GT}, also lead to the proof of the converse to Corollary \ref{cor:pgen}, that we describe in this section. \\

\noindent We shall use the following result for the case $g=0$, i.e.\ for punctured spheres, that Gendron-Tahar proved:

\begin{prop}\emph{\cite[Theorem 1.2(ii)]{GT}}\label{propgt2}
Suppose  $\chi:\Gamma_{0,n} \to \mathbb{Z}$ is a non-trivial surjective representation. Let the holonomies around the positive punctures be given by the integer-tuple $\lambda \in \mathbb{Z}_+^k$ and the holonomies around the negative punctures be given by $-\mu$ where $\mu \in \mathbb{Z}_+^l$.  
Then there is a translation structure on $S_{0,n}$ with holonomy $\chi$, with simple poles at the punctures and a set of $r$ zeros with prescribed orders $(d_1,d_2,\ldots, d_r)$ that satisfies the degree condition $\displaystyle\sum\limits_{i=1}^r d_i = 2g-2+n$ if and only if 
     \begin{equation}\label{eqn:comb3}
    \sum_{i} \lambda_i=\sum_j \mu_j> \max\{d_1,d_2,\ldots, d_r\}.
\end{equation}
    
\end{prop}

\medskip 

\noindent In the remainder of  this section, we shall fix a representation  $\chi:\Gamma_{g,n} \to \mathbb{Z}$  such that:
\begin{itemize}
\item $\chi$ is non-trivial, and surjective,
\item the $\chi$-images of the loops  around the positive punctures are given by the integer-tuple $\lambda \in \mathbb{Z}_+^k$ and the loops around the negative punctures be given by $-\mu$ where $\mu \in \mathbb{Z}_+^l$, and
\item  $\chi(\gamma) = 1$ for each handle-generator  (see Lemma \ref{lem:red2}). 
\end{itemize}

\medskip 
\noindent We shall  start with the case when the prescribed stratum has a single zero; first, we need the following definition:

\begin{defn}\label{combad} An integer tuple $\nu = (\lambda_1,\lambda_2,\ldots, \lambda_k, -\mu_1,-\mu_2,\ldots, - \mu_l)$ where $\lambda_i>0$ and $ \mu_j>0$ for each $i,j$  is said to be \textit{$g$-combinatorially admissible} for an integer $g\geq 0$ if 
\begin{equation}\label{eqn:gcomb}
    \sum_{i} \lambda_i=\sum_j  \mu_j  > 2g -2 + k + l. 
\end{equation}
\end{defn}

\noindent Note that in the single-zero case, being $g$-combinatorially admissible is equivalent to satisfying \eqref{eqn:comb3} since the order of the zero must equal $2g-2+n$.\\

\noindent We now prove:

\begin{prop}[Single-zero case] \label{prop:new1}
Let  $\chi:\Gamma_{g,n} \to \mathbb{Z}$  be as in the beginning of the section. Then there is a translation structure on $S_{g,n}$ with holonomy $\chi$, with simple poles at the punctures and a single zero of order $d$  (satisfying $d= 2g-2+n$)  if and only if  $\sum_{i}  \lambda_i=\sum_j \mu_j> d$,  that is, the integer tuple $\nu = (\lambda_1,\lambda_2,\ldots, \lambda_k, -\mu_1,-\mu_2,\ldots, - \mu_l)$  that records the residues at the punctures is $g$-combinatorially admissible. 
\end{prop}

\begin{proof}
The ``only if" part follows from Corollary \ref{cor:pgen}: namely, consider the branched covering $f : S_g \to \mathbb{S}^2$ given by Proposition \ref{prop:bc}. The single zero of the differential corresponding to the translation structure is a branch point  for the map $f$ of order $d + 1 = 2g+ n-1$. Since the order of this branch point cannot be more than the degree of $f$, we have, 
\begin{equation}
    \sum_{i} \lambda_i=\sum_j  \mu_j \geq 2g + n-1 = d + 1.
\end{equation}

\noindent Moving on to the other implication, we may assume, by Lemma \ref{lem:red2}, that $\chi(\rho) = 1$ for each handle generator $\rho$. We proceed as in Section \ref{sec:trivholposgen}, employing a reduction process consisting of $g$ reduction steps to obtain a surjective representation $\chi : \Gamma_{0,m} \to \mathbb{Z}$, for $m \leq n$ which shall be realised as the holonomy of a translation structure on $S_{0,m}$ induced by a meromorphic differential with simple poles  at the punctures, and a single zero. We then perform $g$ constructions on this translation surface, reversing the reduction steps, to obtain the final surface. To justify the reduction step, we first prove:\\

\noindent \textit{Claim.} Let  $\nu$ be a $g-$combinatorially admissible tuple with $g \geq 1$. Then one of the following must hold:\begin{enumerate}
    \item $\lambda_i \geq 3$ for some $i$ and  $\mu_j \geq 2$ for some $j$, or 
    \item $\mu_j \geq 3$ for some $j$ and  $\lambda_i \geq 2$ for some $i$.
\end{enumerate}

\noindent \textit{Proof of claim.}  Since $\nu$ is $g$-combinatorially admissible tuple with $g \geq 1$, we have 
\begin{equation}\label{eqn:gcomb2}
    \sum_{i} \lambda_i=\sum_j \mu_j > 2g + n-2 \geq n
\end{equation}
Assume, to the contrary, that $\lambda_i \leq 2$ for all $1 \leq i \leq k$ and $\mu_j \leq 2$ for all $1 \leq j \leq l$. This gives us, 
\begin{equation}
    2n < \sum_{i} \lambda_i + \sum_j \mu_j \leq 2k + 2l = 2n
\end{equation}
which is a contradiction. Without loss of generality, we may now assume that $\lambda_1 \geq 3$. It is now easy to see that a tuple $\nu$ cannot be $g$-combinatorially admissible if all the entries of the tuple $\mu$ are 1. $\qed$\\

\noindent We can now state the reduction process. Given a $g$-combinatorially admissible tuple $\nu$, with $g \geq 1$, we reduce it to $(\lambda_1, \ldots, \lambda_i -2, \ldots, \lambda_k, -\mu_1, \ldots, -\mu_j +2, \ldots, -\mu_l)$, with the indices $i$ and $j$ coming from the claim above, and the understanding that the reduced tuple is an $n-1$ tuple if either of $\lambda_i$ or $\mu_j$ is $2$. Since the reduction decreases both the sum of the positive periods and the absolute value of the negative periods by $2$, the reduced tuple (whether an $n$-tuple or an $(n-1)$-tuple) is $(g-1)$-combinatorially admissible and thus, non-trivial. So, starting with a $g$-combinatorially admissible tuple, it is possible to perform $g$ reductions to end up with a $0$-combinatorially admissible tuple. We let $\nu^{(i)} = (\lambda_1^{(i)}, \ldots, \lambda_{k_i}^{(i)}, -\mu_1^{(i)}, \ldots, -\mu_{l_i}^{(i)})$ denote the tuple obtained after the $i^{th}$ reduction step, for each $1\leq i\leq g$. Note that $k_i$ and $l_i$ are the number of positive and negative entries, respectively, of $\nu^{(i)}$.

\noindent By Proposition \ref{propgt2}, the tuple $\nu^{(g)}$ denotes the holonomy of a translation structure on $S_{0,n}$ induced by a meromorphic differential with simple poles  at the punctures, and a single zero. Following the notation of section \ref{sec:trivholposgen} we denote this translation surface by $(W_0, \tau_0)$ and proceed to construct translation surfaces $(W_i, \tau_i)$, for $ 1 \leq i \leq g$, given by a meromorphic differential on $S_{i, (k_i + l_i)}$ with a single zero, simple poles at the punctures and holonomy at the punctures given by the tuple $\nu^{(g-i)}$. The holonomy around each handle generator in all surfaces $(W_i, \tau_i)$ shall be the translation by $1$. We first construct $(W_1, \tau_1)$, by considering three cases.\\

\paragraph{\textbf{Case 1}- $k_g = k_{g-1}, l_g = l_{g-1}$} Let $i_g$ and $j_g$ be the indices of the periods changed in the $g^{th}$ reduction step. From the construction of $(W_0, \tau_0)$ as in the proof of Theorem 1.2 (ii) in \cite{GT}, it follows that there exists a  vertical geodesic ray, say $r_{up}$ going upward starting from the zero of $\tau_0$, towards the puncture with holonomy $\lambda_{i_g}^{(g)}$. Similarly, there is a vertical geodesic ray, say $r_{dn}$ going downward starting from the zero of $\tau_0$ towards the puncture with holonomy $-\mu_{j_g}^{(g)}$. We slit along these rays, and glue an infinite cylinders of circumference $2$ to $(W_0, \tau_0)$ using these two rays  as shown in Figure \ref{fig:onezerorathol1}. We then construct a handle with holonomy $1$ at each handle generator in the usual way along the saddle connection of length 2 on the freshly glued cylinder. \\

 \begin{figure}[!h]
     \centering
     \begin{tikzpicture}[scale=1, every node/.style={scale=0.65}]
        \definecolor{pallido}{RGB}{221,227,227}
         \fill[pattern=north west lines, pattern color=pallido] (0,3) -- (0,-3) -- (11,-3) -- (11, 3);
         \fill[white] (1,3.1) -- (1,-3.1) -- (2,-3.1) -- (2, 3.1);
         \fill[white] (5,3.1) -- (5,-3.1) -- (6,-3.1) -- (6, 3.1);
         \fill[white] (9,3.1) -- (9,-3.1) -- (10,-3.1) -- (10, 3.1);
         \node at (1.5,0) {$\ldots$};
         \node at (5.5,0) {$\ldots$};
         \node at (9.5,0) {$\ldots$};
         \fill[white] ($(3,0) + (92:3)$)--(3,0)-- ++(88:3);
         \fill[white] ($(4,0) + (92:3)$)--(4,0)-- ++(88:3);
         \fill[white] ($(7,0) + (-92:3)$)--(7,0)-- ++(88:3);
         \fill[white] ($(8,0) + (-92:3)$)--(8,0)-- ++(-88:3);
         \draw[<->] (3.1,3.1) -- (3.9,3.1);
         \draw[<->] (7.1,-3.1) -- (7.9,-3.1);
         \draw[<->] (12,3.1) -- (14,3.1);
         \draw[decoration={markings, mark=at position 0.6 with {\arrow{latex}}}, postaction={decorate}] (0,0) -- (0,3);
         \draw[decoration={markings, mark=at position 0.6 with {\arrow{latex}}}, postaction={decorate}] (0,0) -- (0,-3);
         \draw[decoration={markings, mark=at position 0.6 with {\arrow{latex}}}, postaction={decorate}] (3,0) -- ++(92:3);
         \draw[decoration={markings, mark=at position 0.6 with {\arrow{latex}}}, postaction={decorate}] (3,0) -- ++(88:3);
         \draw[decoration={markings, mark=at position 0.6 with {\arrow{latex}}}, postaction={decorate}] (4,0) -- ++(88:3);
         \draw[decoration={markings, mark=at position 0.6 with {\arrow{latex}}}, postaction={decorate}] (4,0) -- ++(92:3);
         \draw[decoration={markings, mark=at position 0.6 with {\arrow{latex}}}, postaction={decorate}] (7,0) -- ++(-92:3);
         \draw[decoration={markings, mark=at position 0.6 with {\arrow{latex}}}, postaction={decorate}] (7,0) -- ++(-88:3);
         \draw[decoration={markings, mark=at position 0.6 with {\arrow{latex}}}, postaction={decorate}] (8,0) -- ++(-88:3);
         \draw[decoration={markings, mark=at position 0.6 with {\arrow{latex}}}, postaction={decorate}] (8,0) -- ++(-92:3);
         \draw[decoration={markings, mark=at position 0.6 with {\arrow{latex}}}, postaction={decorate}] (11,0) -- (11,-3);
        \draw[decoration={markings, mark=at position 0.6 with {\arrow{latex}}}, postaction={decorate}] (11,0) -- (11,3);
         \fill (0,0) circle (1pt);
         \fill (3,0) circle (1pt);
         \fill (4,0) circle (1pt);
         \fill (7,0) circle (1pt);
         \fill (8,0) circle (1pt);
         \fill (11,0) circle (1pt);
         \node[right] at ($(3,0) + (88:1)$) {$r_{up}^-$};
         \node[left] at ($(4,0) + (92:1)$) {$r_{up}^+$};
         \node[left] at ($(8,0) + (-92:0.8)$) {$r_{dn}^-$};
         \node[right] at ($(7,0) + (-88:0.8)$) {$r_{dn}^+$};
         \node[above] at (3.5, 3.1) {$\lambda_{i_g}^{(g)}$};
         \node[below] at (7.5, -3.1) {$\mu_{j_g}^{(g)}$};
         \node[below] at (13,3) {$2$};
         \node[right] at (12, 1.5) {$s^-$};
         \node[left] at (14,1.5) {$s^+$};
         \node[right] at (12, -1.5) {$t^+$};
         \node[left] at (14,-1.5) {$t^-$};
        
        \path[pattern=north west lines, pattern color=pallido] (12,-3) -- (12,0) -- (14,0) -- (14,-3);
        \draw[decoration={markings, mark=at position 0.6 with {\arrow{latex}}}, postaction={decorate}] (12,0) -- (12,-3);
        \draw[decoration={markings, mark=at position 0.6 with {\arrow{latex}}}, postaction={decorate}] (14,0) -- (14,-3);
        \draw[blue] (12,0) arc [start angle = 96,end angle = 90,radius = 10];
        \coordinate (A) at ($(12,0) + (-84:10) + (90:10)$);
        \draw[red] (A) arc [start angle = 90,end angle = 84,radius = 10];
        \fill (A) circle (1pt);
        
        \path[pattern=north west lines, pattern color=pallido] (12,3) -- (12,0) -- (14,0) -- (14,3);
        \draw[decoration={markings, mark=at position 0.6 with {\arrow{latex}}}, postaction={decorate}] (12,0) -- (12,3);
        \draw[decoration={markings, mark=at position 0.6 with {\arrow{latex}}}, postaction={decorate}] (14,0) -- (14,3);
        \draw[red] (12,0) arc [start angle = -96,end angle = -90,radius = 10];
        \coordinate (B) at ($(12,0) + (84:10) + (-90:10)$);
        \draw[blue] (B) arc [start angle = -90,end angle = -84,radius = 10];
        \fill (B) circle (1pt);
     \end{tikzpicture} 
     \caption{Construction of $(W_1, \tau_1)$ in Case 1. Slitting along $r_{up}$ gives $r_{up}^+$ and $r_{up}^-$ which are identified with $s^-$ and $s^+$ respectively. Similarly, $r_{dn}^+$ and $r_{dn}^-$ are identified with $t^-$ and $t^+$ respectively.} \label{fig:onezerorathol1}
 \end{figure}

\paragraph{\textbf{Case 2}- $k_g = k_{g-1}, l_g = l_{g-1}-1$} We let $i_g$ and $j_g$ be as in the previous case and look only at the geodesic ray $r_g^+$. We glue an infinite cylinder of circumference $2$ along the ray $r_{up}^+$ as shown in Figure \ref{fig:onezerorathol2}. The handle is constructed in the same way as the previous case.\\

\begin{figure}[!h]
     \centering
     \begin{tikzpicture}[scale=1, every node/.style={scale=0.65}]
        \definecolor{pallido}{RGB}{221,227,227}
         \fill[pattern=north west lines, pattern color=pallido] (0,3) -- (0,-3) -- (8,-3) -- (8, 3);
         \fill[white] (1,3.1) -- (1,-3.1) -- (2,-3.1) -- (2, 3.1);
         \fill[white] (5,3.1) -- (5,-3.1) -- (6,-3.1) -- (6, 3.1);
         \node at (1.5,0) {$\ldots$};
         \node at (5.5,0) {$\ldots$};
         \fill[white] ($(3,0) + (92:3)$)--(3,0)-- ++(88:3);
         \fill[white] ($(4,0) + (92:3)$)--(4,0)-- ++(88:3);
         \draw[<->] (3.1,3.1) -- (3.9,3.1);
         \draw[decoration={markings, mark=at position 0.6 with {\arrow{latex}}}, postaction={decorate}] (0,0) -- (0,3);
         \draw[decoration={markings, mark=at position 0.6 with {\arrow{latex}}}, postaction={decorate}] (0,0) -- (0,-3);
         \draw[decoration={markings, mark=at position 0.6 with {\arrow{latex}}}, postaction={decorate}] (3,0) -- ++(92:3);
         \draw[decoration={markings, mark=at position 0.6 with {\arrow{latex}}}, postaction={decorate}] (3,0) -- ++(88:3);
         \draw[decoration={markings, mark=at position 0.6 with {\arrow{latex}}}, postaction={decorate}] (4,0) -- ++(88:3);
         \draw[decoration={markings, mark=at position 0.6 with {\arrow{latex}}}, postaction={decorate}] (4,0) -- ++(92:3);
         \draw[decoration={markings, mark=at position 0.6 with {\arrow{latex}}}, postaction={decorate}] (8,0) -- (8,-3);
        \draw[decoration={markings, mark=at position 0.6 with {\arrow{latex}}}, postaction={decorate}] (8,0) -- (8,3);
         \fill (0,0) circle (1pt);
         \fill (3,0) circle (1pt);
         \fill (4,0) circle (1pt);
         \fill (8,0) circle (1pt);
         \node[right] at ($(3,0) + (88:1)$) {$r_{up}^-$};
         \node[left] at ($(4,0) + (92:1)$) {$r_{up}^+$};
         \node[above] at (3.5, 3.1) {$\lambda_{i_g}^{(g)}$};
         
         \begin{scope}[shift = {(-3,0)}]
         \draw[<->] (12,3.1) -- (14,3.1);
         \node[below] at (13,3) {$2$};
         \node[right] at (12, 1.5) {$s^-$};
         \node[left] at (14,1.5) {$s^+$};
         \node[right] at (12, -1.5) {$t^+$};
         \node[left] at (14,-1.5) {$t^-$};
        
        \path[pattern=north west lines, pattern color=pallido] (12,-3) -- (12,0) -- (14,0) -- (14,-3);
        \draw[decoration={markings, mark=at position 0.6 with {\arrow{latex}}}, postaction={decorate}] (12,0) -- (12,-3);
        \draw[decoration={markings, mark=at position 0.6 with {\arrow{latex}}}, postaction={decorate}] (14,0) -- (14,-3);
        \draw[blue] (12,0) arc [start angle = 96,end angle = 90,radius = 10];
        \coordinate (A) at ($(12,0) + (-84:10) + (90:10)$);
        \draw[red] (A) arc [start angle = 90,end angle = 84,radius = 10];
        \fill (A) circle (1pt);
        
        \path[pattern=north west lines, pattern color=pallido] (12,3) -- (12,0) -- (14,0) -- (14,3);
        \draw[decoration={markings, mark=at position 0.6 with {\arrow{latex}}}, postaction={decorate}] (12,0) -- (12,3);
        \draw[decoration={markings, mark=at position 0.6 with {\arrow{latex}}}, postaction={decorate}] (14,0) -- (14,3);
        \draw[red] (12,0) arc [start angle = -96,end angle = -90,radius = 10];
        \coordinate (B) at ($(12,0) + (84:10) + (-90:10)$);
        \draw[blue] (B) arc [start angle = -90,end angle = -84,radius = 10];
        \fill (B) circle (1pt);
         \end{scope}
     \end{tikzpicture} 
     \caption{Construction of $(W_1, \tau_1)$ in Case 2. In this case, $r_{up}^+$ and $r_{up}^-$ are identified with $s^-$ and $s^+$ respectively, but $t^-$ and $t^+$ are identified with each other} \label{fig:onezerorathol2}
 \end{figure}
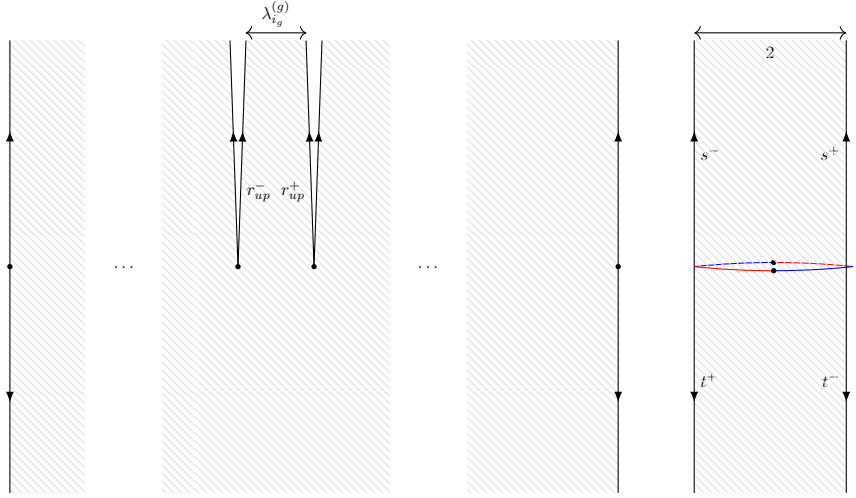

\paragraph{\textbf{Case 3}- $k_g = k_{g-1}-1, l_g = l_{g-1}$} Here, we do the same construction as in the previous case except that we glue along $r_{dn}^-$. \\

\noindent In all the three cases, the resulting translation surface $(W_1, \tau_1)$  is a surface of genus $1$  with punctures where the differential has simple poles, and has residues given by $\nu^{(g-1)}$, with the holonomy around each handle generator being the translation by $1$.  Note that in all three cases, we can still find geodesic rays going upward starting from the zero of $\tau_1$ towards punctures with holonomy $\lambda_{i_g}^{(g-1)}$ and going downward starting from the zero of $\tau_1$ towards punctures with holonomy $-\mu_{j_g}^{(g-1)}$.  We can thus repeat this construction, and in the same manner as described above, we obtain $(W_i, \tau_i)$ from $(W_{i-1}, \tau_{i-1})$, for  each $2 \leq g$. The final surface $(W_g, \tau_g)$ is our desired translation structure on $S_{g,n}$. 
\end{proof} 

\medskip

\noindent Using the previous Proposition as the base case, we shall now prove the general case, where the prescribed stratum has $r$ zeroes, by induction on $r$.  This can be thought of as generalizing Theorem 1.2 (ii) from \cite{GT} that had handled the $g=0$ case.

\begin{prop}[Theorem \ref{main:thme}] \label{prop:newr} 
Let  $\chi:\Gamma_{g,n} \to \mathbb{Z}$  be as in the beginning of the section. Then there is a translation structure on $S_{g,n}$ with holonomy $\chi$, with simple poles at the punctures and a set of $r$ zeros with prescribed orders $(d_1,d_2,\ldots, d_r)$ that satisfies the degree condition $\displaystyle\sum\limits_{i=1}^r d_i = 2g-2+n$ if and only if 
     \begin{equation}\label{gencomp}
    \sum\limits_{i=1}^k \lambda_i=\sum\limits_{j=1}^l \mu_j> \max\{d_1,d_2,\ldots, d_r\}
\end{equation}
where recall that $(\lambda_1,\lambda_2,\cdots, \lambda_k)$ and $(-\mu_1,-\mu_2,\ldots, - \mu_l)$  are the residues at the positive and negative punctures respectively. 
\end{prop} 

\begin{proof}
The proof will proceed by induction on $r$; as observed above, the base case when $r=1$ is handled by Proposition \ref{prop:new1}. For the inductive step, let us assume that we have proved the Proposition for $r-1$ zeroes. 
In what follows we shall divide the set of orders of zeroes $\{d_1,d_2,\ldots, d_r\}$ into $\{d_1\}$, and the rest. 
Let  $d_1^\prime := \sum\limits_{i=2}^r d_i $. \\

\noindent Assume without loss of generality that $k\geq l$; also note that $k+l = n$. We can also deduce that one of $d_1$ or $d_1^\prime$ must be strictly greater than $l$: this is because if both are at most $l \leq k$ then \eqref{dd} cannot hold. In what follows we shall assume that $d_1^\prime \geq l$;  if $d_1 \geq l$  then in what follows we can interchange their roles and the proof works \textit{mutatis mutandis}. We split our construction into two cases. \\

\noindent \textbf{Case 1:  $d_1^\prime  \leq  2g-2 + l + 1 $.} 
We shall first deal with the case when $d_1^\prime $ and $l$ have the opposite  parity, that is, either one is even and the other odd, or vice versa. Note that this happens if and only of $d_1$ and $k$ have the opposite parity,  since  we have
\begin{equation}\label{dd} 
d_ 1 + d_1^\prime  = 2g -2 +  k + l 
\end{equation}
from the degree condition. It is easy to derive (from our assumption on parity) that:
\begin{equation}\label{d1p}
d_1^\prime = 2h -2 + l+ 1 
\end{equation}
for some  integer $0\leq h \leq g$, and consequently from \eqref{dd} that 
\begin{equation}\label{d1} 
d_1 = 2(g-h) -2 + k+ 1 .
\end{equation}

\noindent From the base case of $r=1$ and \eqref{d1}, we know that there is a translation structure on $S_{g-h, k+1}$ with $k+1$ simple poles, having residues $\lambda_1,\ldots,\lambda_k,$ and $-\sum_{i} \lambda_i$, and exactly one zero of order $d_1$. Here, we assume as usual that the holonomy around each handle-generator (if $g-h>0$)  is the translation by $1$.  Note that the corresponding translation surface $(X_1,\omega_1)$ has a cylindrical end of circumference $ \sum_{i} \lambda_i$ that we can assume  (since the corresponding residue is negative) develops out to a half-infinite strip in $\C$ in the negative imaginary direction.
\noindent From our inductive hypothesis and \eqref{d1p}, we also have a a translation structure on $S_{h, l+1}$ with $l+1$ simple poles, having residues $-\mu_1,\ldots,-\mu_l,$ and $\sum_{j} \mu_j$, and $(r-1)$ zeroes of orders $d_2,d_3,\ldots,d_r$.  As before, the holonomy around each handle-generator (if $h>0$) is the translation by $1$.  The corresponding translation surface $(X_2,\omega_2)$ now has a cylindrical end of circumference $ \sum_{j} \mu_j$  that develops out to a half-infinite strip in $\C$ in the \textit{positive} imaginary direction.\\
\noindent By our assumption \eqref{gencomp} the circumferences of two cylindrical ends on $(X_1,\omega_1)$ on $(X_2,\omega_2)$ that we mentioned match, and we can truncate the cylindrical ends and identify the resulting boundary circles by an isometry to obtain a translation surface $(X,\omega)$ that is homeomorphic to $S_{g,n}$.  From our construction, $(X,\omega)$ is our desired translation surface with holonomy $\chi$ and has precisely $r$ zeroes of the prescribed orders $d_1,d_2,\ldots d_r$. \\

  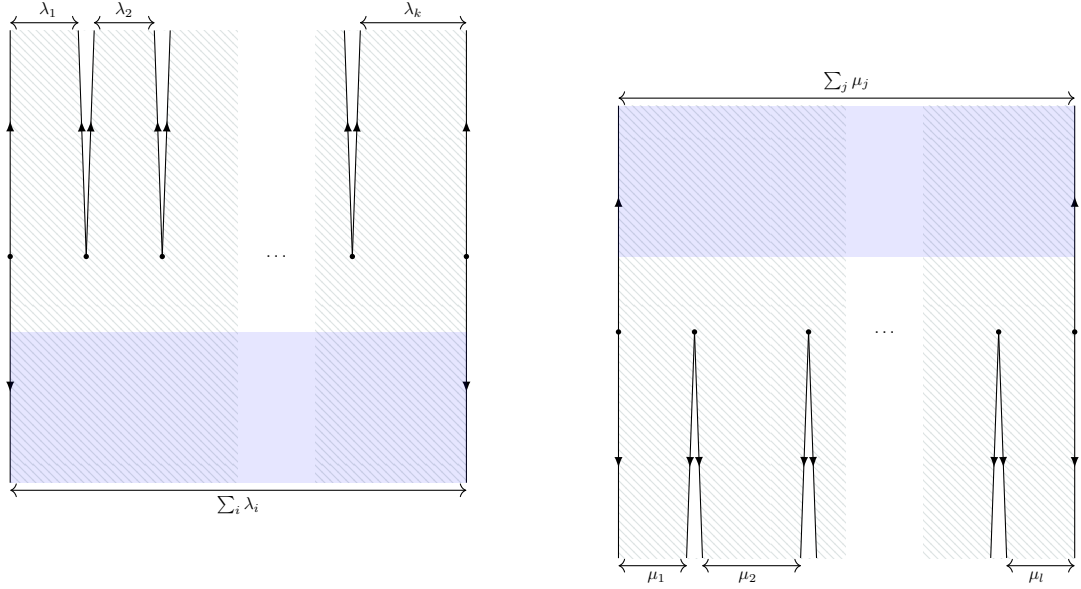
\begin{figure}[!h]
     \centering
     \begin{tikzpicture}[scale=1, every node/.style={scale=0.65}]
        \definecolor{pallido}{RGB}{221,227,227}
         \fill[pattern=north west lines, pattern color=pallido] (0,3) -- (0,-3) -- (6,-3) -- (6, 3);
         \fill[white] (3,3.1) -- (3,-3.1) -- (4,-3.1) -- (4, 3.1);
         \node at (3.5,0) {$\ldots$};

         \fill[white] ($(1,0) + (92:3)$)--(1,0)-- ++(88:3);
         \fill[white] ($(2,0) + (92:3)$)--(2,0)-- ++(88:3);
         \fill[white] ($(4.5,0) + (92:3)$)--(4.5,0)-- ++(88:3);
         \draw[decoration={markings, mark=at position 0.6 with {\arrow{latex}}}, postaction={decorate}] (1,0) -- ++(92:3);
         \draw[decoration={markings, mark=at position 0.6 with {\arrow{latex}}}, postaction={decorate}] (1,0) -- ++(88:3);
         \draw[decoration={markings, mark=at position 0.6 with {\arrow{latex}}}, postaction={decorate}] (2,0) -- ++(92:3);
         \draw[decoration={markings, mark=at position 0.6 with {\arrow{latex}}}, postaction={decorate}] (2,0) -- ++(88:3);
         \draw[decoration={markings, mark=at position 0.6 with {\arrow{latex}}}, postaction={decorate}] (4.5,0) -- ++(92:3);
         \draw[decoration={markings, mark=at position 0.6 with {\arrow{latex}}}, postaction={decorate}] (4.5,0) -- ++(88:3);
        \draw[decoration={markings, mark=at position 0.6 with {\arrow{latex}}}, postaction={decorate}] (6,0) -- (6,-3);
        \draw[decoration={markings, mark=at position 0.6 with {\arrow{latex}}}, postaction={decorate}] (6,0) -- (6,3);

         \draw[<->] (1.1,3.1) -- (1.9,3.1);
         \draw[<->] (0,3.1) -- (0.9,3.1);
         \draw[<->] (4.6,3.1) -- (6,3.1);
         \draw[<->] (0,-3.1) -- (6,-3.1);
         \draw[decoration={markings, mark=at position 0.6 with {\arrow{latex}}}, postaction={decorate}] (0,0) -- (0,3);
         \draw[decoration={markings, mark=at position 0.6 with {\arrow{latex}}}, postaction={decorate}] (0,0) -- (0,-3);

         \fill (0,0) circle (1pt);
         \fill (1,0) circle (1pt);
         \fill (2,0) circle (1pt);
         \fill (4.5,0) circle (1pt);
         \fill (6,0) circle (1pt);
         \node[above] at (0.5, 3.1) {$\lambda_{1}$};
         \node[above] at (1.5, 3.1) {$\lambda_{2}$};
         \node[above] at (5.3, 3.1) {$\lambda_{k}$};
         \node[below] at (3,-3.1) {$\sum_i \lambda_i$};
         \fill[color = blue, opacity = 0.1] (0,-1) -- (0,-3) -- (6,-3) -- (6, -1);
         
         \begin{scope}[shift = {(8,-1)}]
         \fill[pattern=north west lines, pattern color=pallido] (0,3) -- (0,-3) -- (6,-3) -- (6, 3);
         \fill[white] (3,3.1) -- (3,-3.1) -- (4,-3.1) -- (4, 3.1);
         \node at (3.5,0) {$\ldots$};

         \fill[white] ($(1,0) + (-92:3)$)--(1,0)-- ++(-88:3);
         \fill[white] ($(2.5,0) + (-92:3)$)--(2.5,0)-- ++(-88:3);
         \fill[white] ($(5,0) + (-92:3)$)--(5,0)-- ++(-88:3);
         \draw[decoration={markings, mark=at position 0.6 with {\arrow{latex}}}, postaction={decorate}] (1,0) -- ++(-92:3);
         \draw[decoration={markings, mark=at position 0.6 with {\arrow{latex}}}, postaction={decorate}] (1,0) -- ++(-88:3);
         \draw[decoration={markings, mark=at position 0.6 with {\arrow{latex}}}, postaction={decorate}] (2.5,0) -- ++(-92:3);
         \draw[decoration={markings, mark=at position 0.6 with {\arrow{latex}}}, postaction={decorate}] (2.5,0) -- ++(-88:3);
         \draw[decoration={markings, mark=at position 0.6 with {\arrow{latex}}}, postaction={decorate}] (5,0) -- ++(-92:3);
         \draw[decoration={markings, mark=at position 0.6 with {\arrow{latex}}}, postaction={decorate}] (5,0) -- ++(-88:3);
        \draw[decoration={markings, mark=at position 0.6 with {\arrow{latex}}}, postaction={decorate}] (6,0) -- (6,-3);
        \draw[decoration={markings, mark=at position 0.6 with {\arrow{latex}}}, postaction={decorate}] (6,0) -- (6,3);

         \draw[<->] (1.1,-3.1) -- (2.4,-3.1);
         \draw[<->] (0,-3.1) -- (0.9,-3.1);
         \draw[<->] (5.1,-3.1) -- (6,-3.1);
         \draw[<->] (0,3.1) -- (6,3.1);
         \draw[decoration={markings, mark=at position 0.6 with {\arrow{latex}}}, postaction={decorate}] (0,0) -- (0,3);
         \draw[decoration={markings, mark=at position 0.6 with {\arrow{latex}}}, postaction={decorate}] (0,0) -- (0,-3);

         \fill (0,0) circle (1pt);
         \fill (1,0) circle (1pt);
         \fill (2.5,0) circle (1pt);
         \fill (5,0) circle (1pt);
         \fill (6,0) circle (1pt);
         \node[below] at (0.5, -3.1) {$\mu_{1}$};
         \node[below] at (1.7, -3.1) {$\mu_{2}$};
         \node[below] at (5.5, -3.1) {$\mu_{l}$};
         \node[above] at (3, 3.1) {$\sum_j \mu_j$};
         \fill[color = blue, opacity = 0.1] (0,1) -- (0,3) -- (6,3) -- (6, 1);
         \end{scope}

     \end{tikzpicture}
     \caption{The surfaces $(X_1,\omega_1)$ (left) and $(X_2,\omega_2)$ (right) as described in Case-1. The shaded parts are truncated to glue $(X_1,\omega_1)$ and $(X_2,\omega_2)$.}
 \end{figure}

\noindent In the case the pairs $\{d_1^\prime, l\}$ and $\{d_1, k\}$ have the same parity, then consider a new set of $r$ zeroes with orders $\{d_1-1, d_2-1, d_3, d_4,\ldots, d_r\}$  where we have decreased the orders of $d_1$ and $d_2$ by $1$.   In the set of prescribed residues, we decrease $\lambda_k$ and $\mu_l$ by $1$ to modify the residues at the positive punctures to  $\{\lambda_1, \lambda_2,\ldots, \lambda_{k-1}, \lambda_k - 1\}$ and negative punctures to $\{-\mu_1, -\mu_2,\ldots, -\mu_{l-1}, -\mu_l + 1\}$.    Note that if $\lambda_l$   (respectively,  $\mu_l$)  was already equal to $1$, the number of positive  (respectively, negative) punctures reduces by one; the construction divides into some sub-cases depending on whether this happens. Note that if $\lambda_i$s are not all $1$, then we can order them such that   $\lambda_k>1$, and the same for the $\mu_j$s. \\

\noindent \textit{Sub-case (i):  Both $\lambda_k>1$ and $\mu_l>1$.} 
This new data  of residues (after $\lambda_k$ and $\mu_l$  have been decreased by $1$) satisfies the parity assumption, and \eqref{gencomp} is satisfied since both the left and right hand side of the inequality decreases by $1$ by our modification. So the previous construction holds and we obtain a translation surface $(X^\prime,\omega^\prime)$ homeomorphic to $S_{g-1,n}$  having holonomy with the prescribed residues at the  punctures (which are all simple poles) and zeroes of orders $d_1-1,d_2-1,d_3,\ldots, d_r$.   From our construction, the zeroes of orders $d_1-1$ and $d_2-1$ occur on the different intermediate translation surfaces with cylindrical ends that are "combined";  we can also arrange so that they lie at the base of the cylinders corresponding to the poles having residues $\lambda_k-1$ and $-\mu_l+1$ respectively. Moreover, in the gluing step we can arrange so that there is a vertical "saddle-connection" between them on $(X^\prime,\omega^\prime)$.  
\noindent From this discussion, it follows that  there is an isometrically embedded infinite vertical line $\gamma^\prime $ on $(X^\prime,\omega^\prime)$ that contains the two zeroes of orders $d_1-1$ and $d_2-1$ going out the cylindrical ends having circumferences $\lambda_k-1$ and $\mu_l -1$ respectively.  Now, consider an infinite Euclidean cylinder $(Y,\eta)$ of circumference $1$, and an embedded vertical line $\gamma$ there.  Cut $(X^\prime,\omega^\prime)$ along the two rays, say, $r_{up}$ and $r_{dn}$, that form the complement of the interior of the saddle connection $s$ in $\gamma^\prime$ and cut $(Y,\eta)$ along two disjoint rays in $\gamma$ at distance equal to the length of the saddle connection $s$.  Identify the resulting boundary lines by isometries so that the resulting translation surface $(X,\omega)$  is homeomorphic to $S_{g,n}$. The cylindrical ends of $(X^\prime,\omega^\prime)$ and $(Y,\eta)$ combine to produce two cylindrical ends in $(X,\omega)$ having circumference $\lambda_k$ and $\mu_l$; in the identification the zeroes of orders $d_1-1$ and $d_2-1$  get identified with two regular points in $(Y,\eta)$ to become zeroes of orders $d_1$ and $d_2$ respectively. The residues of the rest of the poles, and the holonomy of the handles, is unchanged. Moreover, the handle just added has trivial holonomy for one of the handle generators and holonomy a translation by $1$ for the other handle-generator; however we can easily change to both being a translation by $1$ by applying a Dehn-twist, as in the proof of Lemma \ref{lem:allhandholnonzero}. Thus, the surface $(X,\omega)$ is our desired translation surface with holonomy $\chi$ and zeroes of orders $d_1,d_2,\ldots,d_r$. \\

  \begin{figure}[!h]
     \centering
     \begin{tikzpicture}[scale=1, every node/.style={scale=0.65}]
        \definecolor{pallido}{RGB}{221,227,227}
         \fill[pattern=north west lines, pattern color=pallido] (0,3) -- (0,-4) -- (6,-4) -- (6, 3);
         \fill[pattern=north west lines, pattern color=pallido] (7,3) -- (7,-4) -- (8,-4) -- (8, 3);
         \fill[white] (3,3.1) -- (3,-4.1) -- (4,-4.1) -- (4, 3.1);
         \node at (3.5,0) {$\ldots$};
         \draw[red] (0,0) -- (0,-1); 
         \draw[red] (8,0) -- (8,-1);
         \draw[blue] (6,0) --(6,-1);
         \draw[blue] (7,0) -- (7,-1);
         \node[left] at (0, -0.5) {$s^-$};
         \node[right] at (6, -0.5) {$s^+$};
         \node[left] at (7,-0.5) {$t^-$};
         \node[right] at (8,-0.5) {$t^+$};

         \fill[white] ($(1,0) + (92:3)$)--(1,0)-- ++(88:3);
         \fill[white] ($(2,0) + (92:3)$)--(2,0)-- ++(88:3);
         \fill[white] ($(4.5,0) + (92:3)$)--(4.5,0)-- ++(88:3);
         \draw[decoration={markings, mark=at position 0.6 with {\arrow{latex}}}, postaction={decorate}] (1,0) -- ++(92:3);
         \draw[decoration={markings, mark=at position 0.6 with {\arrow{latex}}}, postaction={decorate}] (1,0) -- ++(88:3);
         \draw[decoration={markings, mark=at position 0.6 with {\arrow{latex}}}, postaction={decorate}] (2,0) -- ++(92:3);
         \draw[decoration={markings, mark=at position 0.6 with {\arrow{latex}}}, postaction={decorate}] (2,0) -- ++(88:3);
         \draw[decoration={markings, mark=at position 0.6 with {\arrow{latex}}}, postaction={decorate}] (4.5,0) -- ++(92:3);
         \draw[decoration={markings, mark=at position 0.6 with {\arrow{latex}}}, postaction={decorate}] (4.5,0) -- ++(88:3);
        \draw[decoration={markings, mark=at position 0.6 with {\arrow{latex}}}, postaction={decorate}] (6,0) -- (6,3);
        \draw[decoration={markings, mark=at position 0.6 with {\arrow{latex}}}, postaction={decorate}] (7,0) -- (7,3);
        \draw[decoration={markings, mark=at position 0.6 with {\arrow{latex}}}, postaction={decorate}] (8,0) -- (8,3);

         \draw[<->] (1.1,3.1) -- (1.9,3.1);
         \draw[<->] (0,3.1) -- (0.9,3.1);
         \draw[<->] (4.6,3.1) -- (6,3.1);
         \draw[decoration={markings, mark=at position 0.6 with {\arrow{latex}}}, postaction={decorate}] (0,0) -- (0,3);
         
         \fill (0,0) circle (1pt);
         \fill (1,0) circle (1pt);
         \fill (2,0) circle (1pt);
         \fill (4.5,0) circle (1pt);
         \fill (6,0) circle (1pt);
         \fill (7,0) circle (1pt);
         \fill (8,0) circle (1pt);
         \node[above] at (0.5, 3.1) {$\lambda_{1}$};
         \node[above] at (1.5, 3.1) {$\lambda_{2}$};
         \node[above] at (5.3, 3.1) {$\lambda_{k}$};
         \node[right] at ($(4.5,0) + (88:1)$) {$r_{up}^-$};
         \node[left] at (6,1) {$r_{up}^+$};
         \node[right] at (7,1) {$\gamma_{up}^-$};
         \node[left] at (8,1) {$\gamma_{up}^+$};
         
         \begin{scope}[shift = {(0,-1)}]
         \fill[white] ($(1,0) + (-92:3)$)--(1,0)-- ++(-88:3);
         \fill[white] ($(2.5,0) + (-92:3)$)--(2.5,0)-- ++(-88:3);
         \fill[white] ($(5,0) + (-92:3)$)--(5,0)-- ++(-88:3);
         \draw[decoration={markings, mark=at position 0.6 with {\arrow{latex}}}, postaction={decorate}] (1,0) -- ++(-92:3);
         \draw[decoration={markings, mark=at position 0.6 with {\arrow{latex}}}, postaction={decorate}] (1,0) -- ++(-88:3);
         \draw[decoration={markings, mark=at position 0.6 with {\arrow{latex}}}, postaction={decorate}] (2.5,0) -- ++(-92:3);
         \draw[decoration={markings, mark=at position 0.6 with {\arrow{latex}}}, postaction={decorate}] (2.5,0) -- ++(-88:3);
         \draw[decoration={markings, mark=at position 0.6 with {\arrow{latex}}}, postaction={decorate}] (5,0) -- ++(-92:3);
         \draw[decoration={markings, mark=at position 0.6 with {\arrow{latex}}}, postaction={decorate}] (5,0) -- ++(-88:3);
         \draw[decoration={markings, mark=at position 0.6 with {\arrow{latex}}}, postaction={decorate}] (6,0) -- (6,-3);
         \draw[decoration={markings, mark=at position 0.6 with {\arrow{latex}}}, postaction={decorate}] (0,0) -- (0,-3);
        \draw[decoration={markings, mark=at position 0.6 with {\arrow{latex}}}, postaction={decorate}] (7,0) -- (7,-3);
        \draw[decoration={markings, mark=at position 0.6 with {\arrow{latex}}}, postaction={decorate}] (8,0) -- (8,-3);

         \draw[<->] (1.1,-3.1) -- (2.4,-3.1);
         \draw[<->] (0,-3.1) -- (0.9,-3.1);
         \draw[<->] (5.1,-3.1) -- (6,-3.1);

         \fill (0,0) circle (1pt);
         \fill (1,0) circle (1pt);
         \fill (2.5,0) circle (1pt);
         \fill (5,0) circle (1pt);
         \fill (6,0) circle (1pt);
         \fill (7,0) circle (1pt);
         \fill (8,0) circle (1pt);
         \node[below] at (0.5, -3.1) {$\mu_{1}$};
         \node[below] at (1.7, -3.1) {$\mu_{2}$};
         \node[below] at (5.5, -3.1) {$\mu_{l}$};
         \node[right] at ($(5,0) + (-88:1)$) {$r_{dn}^+$};
         \node[left] at (6,-1) {$r_{dn}^-$};
         \node[right] at (7,-1) {$\gamma_{dn}^+$};
         \node[left] at (8,-1) {$\gamma_{dn}^-$};
         \end{scope}

     \end{tikzpicture}
     \caption{Illustrating sub-case (i) - The surface on the left is $(X',\omega^\prime)$ and the surface on the right is $Y$. $r_{up}^+$ is identified with $\gamma_{up}^-$, $r_{up}^-$ is identified with $\gamma_{up}^+$,  $r_{dn}^+$ is identified with $\gamma_{dn}^-$, $r_{dn}^-$ is identified with $\gamma_{dn}^+$, $s^+$ is identified with $t^-$ and $s^-$ is identified with $t^+$. The remaining rays are identified in the usual way.}
 \end{figure}
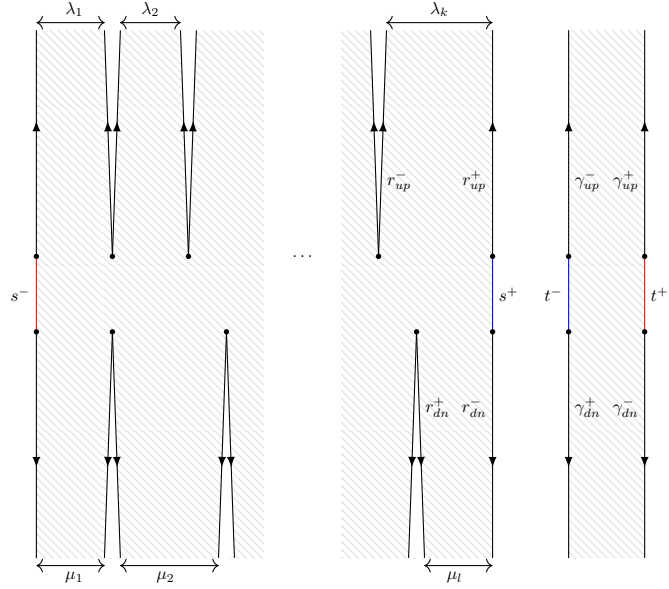

\noindent \textit{Sub-case (ii):  Both $\lambda_k=1$ and $\mu_l=1$.}  This is when $\lambda_i = 1$ and $\mu_j = 1$ for all $1\leq i\leq k$ and $1\leq j\leq l$: in this case the translation surface $X^\prime$ constructed as above has two less punctures, and is homeomorphic to $S_{g,n-2}$.   As before, we can construct $(X^\prime,\omega^\prime)$ in such a way that there is a vertical saddle-connection $s$ between the two zeroes of orders $d_1-1$ and $d_2-1$ respectively.  The combination with the infinite Euclidean cylinder $(Y,\eta)$ of circumference $1$ will now be different:  choose a vertical arc $s^\prime$ on $(Y,\eta)$ of the same length as $s$, slit along $s$ and $s^\prime$ and identify the resulting boundary segments by isometries so that we obtain a translation  surface $(X,\omega)$  homeomorphic to $S_{g,n}$.  (Topologically, $X$ is the connect-sum of $X^\prime$ and $Y$.) The two ends of $(Y,\eta)$ add two cylindrical ends to those already in $(X^\prime,\omega^\prime)$, and correspond to poles of residues $\lambda_k = 1$ and $-\mu_l = -1$ respectively; as before, the orders of the zeroes at the endpoints of $s$ increase by $1$ each after the identification with the two endpoints of $s^\prime$.  The surface $(X,\omega)$ is thus our desired translation surface. \\

   \begin{figure}[!h]
     \centering
          \begin{tikzpicture}[scale=1, every node/.style={scale=0.65}]
        \definecolor{pallido}{RGB}{221,227,227}
         \fill[pattern=north west lines, pattern color=pallido] (0,3) -- (0,-4) -- (6,-4) -- (6, 3);
         \fill[pattern=north west lines, pattern color=pallido] (7,3) -- (7,-4) -- (8,-4) -- (8, 3);
         \fill[white] (3,3.1) -- (3,-4.1) -- (4,-4.1) -- (4, 3.1);
         \node at (3.5,0) {$\ldots$};
         \draw[red] (0,0) -- (0,-1); 
         \draw[red] (8,0) -- (8,-1);
         \draw[blue] (6,0) --(6,-1);
         \draw[blue] (7,0) -- (7,-1);
         \node[left] at (0, -0.5) {$s^-$};
         \node[right] at (6, -0.5) {$s^+$};
         \node[left] at (7,-0.5) {$t^-$};
         \node[right] at (8,-0.5) {$t^+$};

         \fill[white] ($(1,0) + (92:3)$)--(1,0)-- ++(88:3);
         \fill[white] ($(2,0) + (92:3)$)--(2,0)-- ++(88:3);
         \fill[white] ($(4.5,0) + (92:3)$)--(4.5,0)-- ++(88:3);
         \draw[decoration={markings, mark=at position 0.6 with {\arrow{latex}}}, postaction={decorate}] (1,0) -- ++(92:3);
         \draw[decoration={markings, mark=at position 0.6 with {\arrow{latex}}}, postaction={decorate}] (1,0) -- ++(88:3);
         \draw[decoration={markings, mark=at position 0.6 with {\arrow{latex}}}, postaction={decorate}] (2,0) -- ++(92:3);
         \draw[decoration={markings, mark=at position 0.6 with {\arrow{latex}}}, postaction={decorate}] (2,0) -- ++(88:3);
         \draw[decoration={markings, mark=at position 0.6 with {\arrow{latex}}}, postaction={decorate}] (4.5,0) -- ++(92:3);
         \draw[decoration={markings, mark=at position 0.6 with {\arrow{latex}}}, postaction={decorate}] (4.5,0) -- ++(88:3);
        \draw[decoration={markings, mark=at position 0.6 with {\arrow{latex}}}, postaction={decorate}] (6,0) -- (6,3);
        \draw[decoration={markings, mark=at position 0.6 with {\arrow{latex}}}, postaction={decorate}] (7,0) -- (7,3);
        \draw[decoration={markings, mark=at position 0.6 with {\arrow{latex}}}, postaction={decorate}] (8,0) -- (8,3);

         \draw[<->] (1.1,3.1) -- (1.9,3.1);
         \draw[<->] (0,3.1) -- (0.9,3.1);
         \draw[<->] (4.6,3.1) -- (6,3.1);
         \draw[decoration={markings, mark=at position 0.6 with {\arrow{latex}}}, postaction={decorate}] (0,0) -- (0,3);
         
         \fill (0,0) circle (1pt);
         \fill (1,0) circle (1pt);
         \fill (2,0) circle (1pt);
         \fill (4.5,0) circle (1pt);
         \fill (6,0) circle (1pt);
         \fill (7,0) circle (1pt);
         \fill (8,0) circle (1pt);
         \node[above] at (0.5, 3.1) {$\lambda_{1}$};
         \node[above] at (1.5, 3.1) {$\lambda_{2}$};
         \node[above] at (5.3, 3.1) {$\lambda_{k}$};

         \begin{scope}[shift = {(0,-1)}]
         \fill[white] ($(1,0) + (-92:3)$)--(1,0)-- ++(-88:3);
         \fill[white] ($(2.5,0) + (-92:3)$)--(2.5,0)-- ++(-88:3);
         \fill[white] ($(5,0) + (-92:3)$)--(5,0)-- ++(-88:3);
         \draw[decoration={markings, mark=at position 0.6 with {\arrow{latex}}}, postaction={decorate}] (1,0) -- ++(-92:3);
         \draw[decoration={markings, mark=at position 0.6 with {\arrow{latex}}}, postaction={decorate}] (1,0) -- ++(-88:3);
         \draw[decoration={markings, mark=at position 0.6 with {\arrow{latex}}}, postaction={decorate}] (2.5,0) -- ++(-92:3);
         \draw[decoration={markings, mark=at position 0.6 with {\arrow{latex}}}, postaction={decorate}] (2.5,0) -- ++(-88:3);
         \draw[decoration={markings, mark=at position 0.6 with {\arrow{latex}}}, postaction={decorate}] (5,0) -- ++(-92:3);
         \draw[decoration={markings, mark=at position 0.6 with {\arrow{latex}}}, postaction={decorate}] (5,0) -- ++(-88:3);
         \draw[decoration={markings, mark=at position 0.6 with {\arrow{latex}}}, postaction={decorate}] (6,0) -- (6,-3);
         \draw[decoration={markings, mark=at position 0.6 with {\arrow{latex}}}, postaction={decorate}] (0,0) -- (0,-3);
        \draw[decoration={markings, mark=at position 0.6 with {\arrow{latex}}}, postaction={decorate}] (7,0) -- (7,-3);
        \draw[decoration={markings, mark=at position 0.6 with {\arrow{latex}}}, postaction={decorate}] (8,0) -- (8,-3);

         \draw[<->] (1.1,-3.1) -- (2.4,-3.1);
         \draw[<->] (0,-3.1) -- (0.9,-3.1);
         \draw[<->] (5.1,-3.1) -- (6,-3.1);

         \fill (0,0) circle (1pt);
         \fill (1,0) circle (1pt);
         \fill (2.5,0) circle (1pt);
         \fill (5,0) circle (1pt);
         \fill (6,0) circle (1pt);
         \fill (7,0) circle (1pt);
         \fill (8,0) circle (1pt);
         \node[below] at (0.5, -3.1) {$\mu_{1}$};
         \node[below] at (1.7, -3.1) {$\mu_{2}$};
         \node[below] at (5.5, -3.1) {$\mu_{l}$};
         \end{scope}

     \end{tikzpicture}
    \caption{Illustrating sub-case (ii) - The surface on the left is $(X',\omega^\prime)$ and the surface on the right is $(Y,\eta)$. In this case, we only identify the segments - $s^+$ is identified with $t^-$ and $s^-$ is identified with $t^+$. The remaining rays are identified in the usual way.}
 \end{figure}
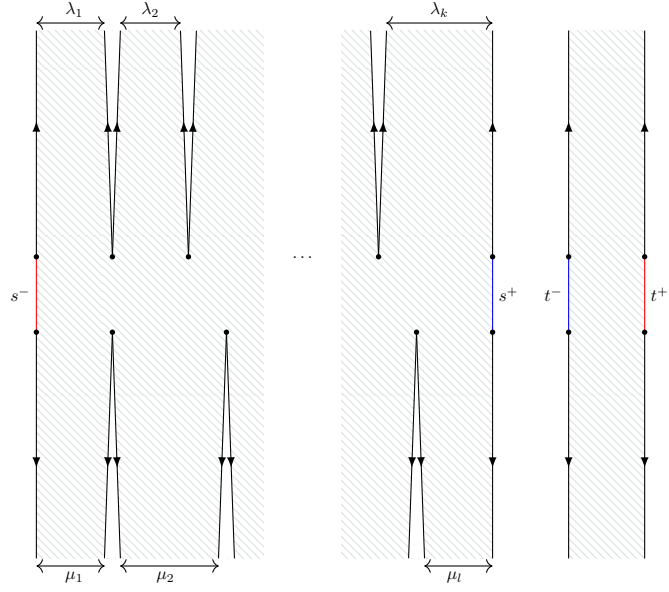

\noindent \textit{Sub-case (iii):  $\lambda_k=1$ and $\mu_l >1$.}  In this case we only decrease the order of $d_1$ by $1$; the  translation surface $(X^\prime,\omega^\prime)$ constructed as above is homeomorphic to $S_{g,n-1}$, realizes the new tuple of residues and has  zeros of orders $d_1-1,d_2,\ldots, d_r$. Moreover, from our construction we can assume that there is a vertical ray $\gamma^\prime$ going ``downward"  from the zero of order $d_1-1$  to the puncture corresponding to the pole of residue $-\mu_l+ 1$.
As before, let $(Y,\eta)$ be an infinite Euclidean cylinder of circumference $1$, and now  let $\gamma$ be an infinite vertical ray in the "downward" direction.  Slit along $\gamma$ and $\gamma^\prime$ and identify the resulting boundary rays such that the resulting translation surface $(X,\omega)$ is homeomorphic to $S_{g,n}$. The ``downward"  end of $(Y,\eta)$ merges with the cylindrical end of circumference $\mu_l-1$ to form a cylindrical end of $(X,\omega)$ of circumference $\mu_l$ that corresponds to the pole of of residue $-\mu_l$. The other end of $(Y,\eta)$ is an additional cylindrical end of $(X,\omega)$ of circumference $1$; this corresponds to the pole of residue $\lambda_k =1$. Moreover, in this combination of $(X^\prime,\omega^\prime)$ and $(Y,\eta)$ the zero of order $d_1-1$ is identified with the endpoint of the vertical ray $\gamma$ on $(Y,\eta)$ and its order thus increases by $1$. The orders of the other zeroes, and the holonomies of the handle-generators  and loops around the other poles on $(X^\prime,\omega^\prime)$ remain unchanged. Thus, $(X,\omega)$ is our desired translation surface.

\begin{rmk}
We can note that $\lambda_l>1$ and all $\mu_l=1$ is not possible, from the assumption that $k\geq l$ and $\sum_i \lambda_i = \sum_j \mu_j$.
\end{rmk}

  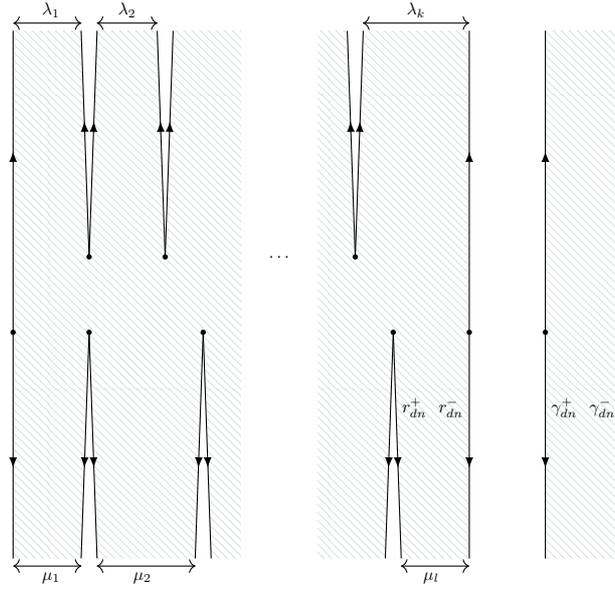
\begin{figure}[!h]
     \centering
     \begin{tikzpicture}[scale=1, every node/.style={scale=0.65}]
        \definecolor{pallido}{RGB}{221,227,227}
         \fill[pattern=north west lines, pattern color=pallido] (0,3) -- (0,-4) -- (6,-4) -- (6, 3);
         \fill[pattern=north west lines, pattern color=pallido] (7,3) -- (7,-4) -- (8,-4) -- (8, 3);
         \fill[white] (3,3.1) -- (3,-4.1) -- (4,-4.1) -- (4, 3.1);
         \node at (3.5,0) {$\ldots$};

         \fill[white] ($(1,0) + (92:3)$)--(1,0)-- ++(88:3);
         \fill[white] ($(2,0) + (92:3)$)--(2,0)-- ++(88:3);
         \fill[white] ($(4.5,0) + (92:3)$)--(4.5,0)-- ++(88:3);
         \draw[decoration={markings, mark=at position 0.6 with {\arrow{latex}}}, postaction={decorate}] (1,0) -- ++(92:3);
         \draw[decoration={markings, mark=at position 0.6 with {\arrow{latex}}}, postaction={decorate}] (1,0) -- ++(88:3);
         \draw[decoration={markings, mark=at position 0.6 with {\arrow{latex}}}, postaction={decorate}] (2,0) -- ++(92:3);
         \draw[decoration={markings, mark=at position 0.6 with {\arrow{latex}}}, postaction={decorate}] (2,0) -- ++(88:3);
         \draw[decoration={markings, mark=at position 0.6 with {\arrow{latex}}}, postaction={decorate}] (4.5,0) -- ++(92:3);
         \draw[decoration={markings, mark=at position 0.6 with {\arrow{latex}}}, postaction={decorate}] (4.5,0) -- ++(88:3);
          \draw[decoration={markings, mark=at position 0.6 with {\arrow{latex}}}, postaction={decorate}] (0,-1) -- (0,3);
        \draw[decoration={markings, mark=at position 0.6 with {\arrow{latex}}}, postaction={decorate}] (6,-1) -- (6,3);
        \draw[decoration={markings, mark=at position 0.6 with {\arrow{latex}}}, postaction={decorate}] (7,-1) -- (7,3);
        \draw[decoration={markings, mark=at position 0.6 with {\arrow{latex}}}, postaction={decorate}] (8,-1) -- (8,3);

         \draw[<->] (1.1,3.1) -- (1.9,3.1);
         \draw[<->] (0,3.1) -- (0.9,3.1);
         \draw[<->] (4.6,3.1) -- (6,3.1);
        
         \fill (1,0) circle (1pt);
         \fill (2,0) circle (1pt);
         \fill (4.5,0) circle (1pt);

         \node[above] at (0.5, 3.1) {$\lambda_{1}$};
         \node[above] at (1.5, 3.1) {$\lambda_{2}$};
         \node[above] at (5.3, 3.1) {$\lambda_{k}$};
         
         \begin{scope}[shift = {(0,-1)}]
         \fill[white] ($(1,0) + (-92:3)$)--(1,0)-- ++(-88:3);
         \fill[white] ($(2.5,0) + (-92:3)$)--(2.5,0)-- ++(-88:3);
         \fill[white] ($(5,0) + (-92:3)$)--(5,0)-- ++(-88:3);
         \draw[decoration={markings, mark=at position 0.6 with {\arrow{latex}}}, postaction={decorate}] (1,0) -- ++(-92:3);
         \draw[decoration={markings, mark=at position 0.6 with {\arrow{latex}}}, postaction={decorate}] (1,0) -- ++(-88:3);
         \draw[decoration={markings, mark=at position 0.6 with {\arrow{latex}}}, postaction={decorate}] (2.5,0) -- ++(-92:3);
         \draw[decoration={markings, mark=at position 0.6 with {\arrow{latex}}}, postaction={decorate}] (2.5,0) -- ++(-88:3);
         \draw[decoration={markings, mark=at position 0.6 with {\arrow{latex}}}, postaction={decorate}] (5,0) -- ++(-92:3);
         \draw[decoration={markings, mark=at position 0.6 with {\arrow{latex}}}, postaction={decorate}] (5,0) -- ++(-88:3);
         \draw[decoration={markings, mark=at position 0.6 with {\arrow{latex}}}, postaction={decorate}] (6,0) -- (6,-3);
         \draw[decoration={markings, mark=at position 0.6 with {\arrow{latex}}}, postaction={decorate}] (0,0) -- (0,-3);
        \draw[decoration={markings, mark=at position 0.6 with {\arrow{latex}}}, postaction={decorate}] (7,0) -- (7,-3);
        \draw[decoration={markings, mark=at position 0.6 with {\arrow{latex}}}, postaction={decorate}] (8,0) -- (8,-3);

         \draw[<->] (1.1,-3.1) -- (2.4,-3.1);
         \draw[<->] (0,-3.1) -- (0.9,-3.1);
         \draw[<->] (5.1,-3.1) -- (6,-3.1);

         \fill (0,0) circle (1pt);
         \fill (1,0) circle (1pt);
         \fill (2.5,0) circle (1pt);
         \fill (5,0) circle (1pt);
         \fill (6,0) circle (1pt);
         \fill (7,0) circle (1pt);
         \fill (8,0) circle (1pt);
         \node[below] at (0.5, -3.1) {$\mu_{1}$};
         \node[below] at (1.7, -3.1) {$\mu_{2}$};
        \node[below] at (5.5, -3.1) {$\mu_{l}$};
        \node[right] at ($(5,0) + (-88:1)$) {$r_{dn}^+$};
         \node[left] at (6,-1) {$r_{dn}^-$};
         \node[right] at (7,-1) {$\gamma_{dn}^+$};
         \node[left] at (8,-1) {$\gamma_{dn}^-$};
         \end{scope}

     \end{tikzpicture}
     \caption{Illustrating sub-case (iii) - The surface on the left is $(X',\omega^\prime)$ and the surface on the right is $(Y,\eta)$. In this case, we only identify the "downward`` rays - $r_{dn}^+$ is identified with $\gamma_{dn}^-$, $r_{dn}^-$ is identified with $\gamma_{dn}^+$. The remaining rays are identified in the usual way.}
 \end{figure}

\noindent \textbf{Case 2:   $d_1^\prime  >  2g-2 + l + 1 $.}   In this case we rewrite 
\begin{equation}\label{d1p-2}
d_1^\prime = 2g -2 + l+ t 
\end{equation}
where $1 <  t  < k$, and consequently 
\begin{equation}\label{d1-2} 
d_1 = k-t.
\end{equation}

\noindent From our inductive hypothesis and \eqref{d1p-2}, we have a translation structure on $S_{g, l+t}$ with $l+t$ simple poles, having residues $-\mu_1,\ldots,-\mu_l,$ and $\lambda_1, \lambda_2,\ldots, \lambda_{t-1}, \sum_{i=t}^k \lambda_i$, and $(r-1)$ zeroes of orders $d_2,d_3,\ldots,d_r$. As before, the holonomy around each handle-generator is the translation by $1$.  One of the cylindrical ends of the corresponding translation surface $(X_1^\prime,\omega_1^\prime)$ has circumference $\sum_{i=t}^k \lambda_i$  and develops out to a half-infinite strip in $\C$ in the {positive} imaginary direction.\\
\noindent From the base case of $r=1$ and \eqref{d1-2}, we know that there is a translation structure on $S_{0, k-t+2}$ with $k-t+2$ simple poles, having residues $\lambda_t,\ldots,\lambda_k$ and $-\sum_{i=t}^k \lambda_i$, and exactly one zero of order $d_1$.   (Note that the inequality  in the hypothesis of Proposition \ref{prop:new1} is satisfied since $\sum_{i=t}^k \lambda_i \geq k-t+1$ which is greater than $d_1$ because of \eqref{d1-2}. ) The corresponding translation surface $(X_2^\prime,\omega_2^\prime)$ now has a cylindrical end of circumference $\sum_{i=t}^k \lambda_i$ that develops out to a half-infinite strip in $\C$ in the negative imaginary direction. We can now "combine" the translation surfaces $(X_1^\prime,\omega_1^\prime)$ and $(X_2^\prime,\omega_2^\prime)$ by truncating the cylindrical ends mentioned of matching circumferences on each, and identifying the resulting boundaries to obtain a translation surface $(X,\omega)$ homeomorphic to $S_{g, n}$ having holonomy $\chi$ and zeroes of the prescribed orders $d_1,d_2,\ldots, d_r$. 
\end{proof}

  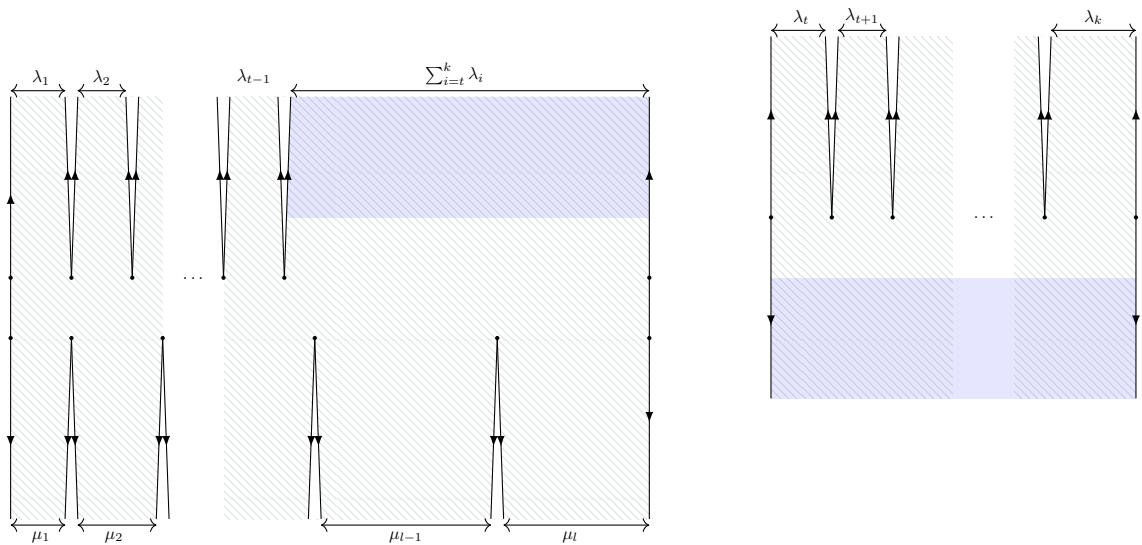
\begin{figure}[!h]
     \centering
     \begin{tikzpicture}[scale=0.8, every node/.style={scale=0.65}]
        \definecolor{pallido}{RGB}{221,227,227}
         \fill[pattern=north west lines, pattern color=pallido] (0,3) -- (0,-4) -- (10.5,-4) -- (10.5, 3);
         \fill[white] (2.5,3.1) -- (2.5,-4.1) -- (3.5,-4.1) -- (3.5, 3.1);
         \node at (3,0) {$\ldots$};
         \fill[color = blue, opacity = 0.1] (4.5,3) -- (4.5,1) -- (10.5,1) -- (10.5, 3);

         \fill[white] ($(1,0) + (92:3)$)--(1,0)-- ++(88:3);
         \fill[white] ($(2,0) + (92:3)$)--(2,0)-- ++(88:3);
         \fill[white] ($(3.5,0) + (92:3)$)--(3.5,0)-- ++(88:3);
         \fill[white] ($(4.5,0) + (92:3)$)--(4.5,0)-- ++(88:3);
         \draw[decoration={markings, mark=at position 0.6 with {\arrow{latex}}}, postaction={decorate}] (1,0) -- ++(92:3);
         \draw[decoration={markings, mark=at position 0.6 with {\arrow{latex}}}, postaction={decorate}] (1,0) -- ++(88:3);
         \draw[decoration={markings, mark=at position 0.6 with {\arrow{latex}}}, postaction={decorate}] (2,0) -- ++(92:3);
         \draw[decoration={markings, mark=at position 0.6 with {\arrow{latex}}}, postaction={decorate}] (2,0) -- ++(88:3);
         \draw[decoration={markings, mark=at position 0.6 with {\arrow{latex}}}, postaction={decorate}] (3.5,0) -- ++(92:3);
         \draw[decoration={markings, mark=at position 0.6 with {\arrow{latex}}}, postaction={decorate}] (3.5,0) -- ++(88:3);
         \draw[decoration={markings, mark=at position 0.6 with {\arrow{latex}}}, postaction={decorate}] (4.5,0) -- ++(92:3);
         \draw[decoration={markings, mark=at position 0.6 with {\arrow{latex}}}, postaction={decorate}] (4.5,0) -- ++(88:3);
        \draw[decoration={markings, mark=at position 0.6 with {\arrow{latex}}}, postaction={decorate}] (10.5,0) -- (10.5,-4);
        \draw[decoration={markings, mark=at position 0.6 with {\arrow{latex}}}, postaction={decorate}] (10.5,0) -- (10.5,3);

         \draw[<->] (1.1,3.1) -- (1.9,3.1);
         \draw[<->] (0,3.1) -- (0.9,3.1);
         \draw[<->] (4.6,3.1) -- (10.5,3.1);
         \draw[decoration={markings, mark=at position 0.6 with {\arrow{latex}}}, postaction={decorate}] (0,-1) -- (0,3);

         \fill (0,0) circle (1pt);
         \fill (1,0) circle (1pt);
         \fill (2,0) circle (1pt);
         \fill (3.5,0) circle (1pt);
         \fill (4.5,0) circle (1pt);
         \fill (10.5,0) circle (1pt);
         \node[above] at (0.5, 3.1) {$\lambda_{1}$};
         \node[above] at (1.5, 3.1) {$\lambda_{2}$};
         \node[above] at (4, 3.1) {$\lambda_{t-1}$};
         \node[above] at (7.3, 3.1) {$\sum_{i=t}^k \lambda_{i}$};
         
          \begin{scope}[shift = {(0,-1)}]
         \fill[white] ($(1,0) + (-92:3)$)--(1,0)-- ++(-88:3);
         \fill[white] ($(2.5,0) + (-92:3)$)--(2.5,0)-- ++(-88:3);
         \fill[white] ($(5,0) + (-92:3)$)--(5,0)-- ++(-88:3);
         \fill[white] ($(8,0) + (-92:3)$)--(8,0)-- ++(-88:3);
         \draw[decoration={markings, mark=at position 0.6 with {\arrow{latex}}}, postaction={decorate}] (1,0) -- ++(-92:3);
         \draw[decoration={markings, mark=at position 0.6 with {\arrow{latex}}}, postaction={decorate}] (1,0) -- ++(-88:3);
         \draw[decoration={markings, mark=at position 0.6 with {\arrow{latex}}}, postaction={decorate}] (2.5,0) -- ++(-92:3);
         \draw[decoration={markings, mark=at position 0.6 with {\arrow{latex}}}, postaction={decorate}] (2.5,0) -- ++(-88:3);
         \draw[decoration={markings, mark=at position 0.6 with {\arrow{latex}}}, postaction={decorate}] (5,0) -- ++(-92:3);
         \draw[decoration={markings, mark=at position 0.6 with {\arrow{latex}}}, postaction={decorate}] (5,0) -- ++(-88:3);
         \draw[decoration={markings, mark=at position 0.6 with {\arrow{latex}}}, postaction={decorate}] (8,0) -- ++(-92:3);
         \draw[decoration={markings, mark=at position 0.6 with {\arrow{latex}}}, postaction={decorate}] (8,0) -- ++(-88:3);
         \draw[decoration={markings, mark=at position 0.6 with {\arrow{latex}}}, postaction={decorate}] (0,0) -- (0,-3);

         \draw[<->] (1.1,-3.1) -- (2.4,-3.1);
         \draw[<->] (0,-3.1) -- (0.9,-3.1);
         \draw[<->] (5.1,-3.1) -- (7.9,-3.1);
         \draw[<->] (8.1,-3.1) -- (10.5,-3.1);

         \fill (0,0) circle (1pt);
         \fill (1,0) circle (1pt);
         \fill (2.5,0) circle (1pt);
         \fill (5,0) circle (1pt);
         \fill (8,0) circle (1pt);
         \fill (10.5,0) circle (1pt);
         \node[below] at (0.5, -3.1) {$\mu_{1}$};
         \node[below] at (1.7, -3.1) {$\mu_{2}$};
         \node[below] at (6.5, -3.1) {$\mu_{l-1}$};
         \node[below] at (9.2, -3.1) {$\mu_{l}$};
         \end{scope}
         
         \begin{scope}[shift = {(12.5,1)}]
                  \fill[pattern=north west lines, pattern color=pallido] (0,3) -- (0,-3) -- (6,-3) -- (6, 3);
         \fill[white] (3,3.1) -- (3,-3.1) -- (4,-3.1) -- (4, 3.1);
         \node at (3.5,0) {$\ldots$};

         \fill[white] ($(1,0) + (92:3)$)--(1,0)-- ++(88:3);
         \fill[white] ($(2,0) + (92:3)$)--(2,0)-- ++(88:3);
         \fill[white] ($(4.5,0) + (92:3)$)--(4.5,0)-- ++(88:3);
         \draw[decoration={markings, mark=at position 0.6 with {\arrow{latex}}}, postaction={decorate}] (1,0) -- ++(92:3);
         \draw[decoration={markings, mark=at position 0.6 with {\arrow{latex}}}, postaction={decorate}] (1,0) -- ++(88:3);
         \draw[decoration={markings, mark=at position 0.6 with {\arrow{latex}}}, postaction={decorate}] (2,0) -- ++(92:3);
         \draw[decoration={markings, mark=at position 0.6 with {\arrow{latex}}}, postaction={decorate}] (2,0) -- ++(88:3);
         \draw[decoration={markings, mark=at position 0.6 with {\arrow{latex}}}, postaction={decorate}] (4.5,0) -- ++(92:3);
         \draw[decoration={markings, mark=at position 0.6 with {\arrow{latex}}}, postaction={decorate}] (4.5,0) -- ++(88:3);
        \draw[decoration={markings, mark=at position 0.6 with {\arrow{latex}}}, postaction={decorate}] (6,0) -- (6,-3);
        \draw[decoration={markings, mark=at position 0.6 with {\arrow{latex}}}, postaction={decorate}] (6,0) -- (6,3);
        \fill[color = blue, opacity = 0.1] (0,-1) -- (0,-3) -- (6,-3) -- (6, -1);

         \draw[<->] (1.1,3.1) -- (1.9,3.1);
         \draw[<->] (0,3.1) -- (0.9,3.1);
         \draw[<->] (4.6,3.1) -- (6,3.1);
         \draw[decoration={markings, mark=at position 0.6 with {\arrow{latex}}}, postaction={decorate}] (0,0) -- (0,3);
         \draw[decoration={markings, mark=at position 0.6 with {\arrow{latex}}}, postaction={decorate}] (0,0) -- (0,-3);

         \fill (0,0) circle (1pt);
         \fill (1,0) circle (1pt);
         \fill (2,0) circle (1pt);
         \fill (4.5,0) circle (1pt);
         \fill (6,0) circle (1pt);
         \node[above] at (0.5, 3.1) {$\lambda_{t}$};
         \node[above] at (1.5, 3.1) {$\lambda_{t+1}$};
         \node[above] at (5.3, 3.1) {$\lambda_{k}$};
         \end{scope}

\end{tikzpicture}
\caption{The surfaces $(X^{\prime}_1,\omega_1^\prime)$ (left) and $(X^{\prime}_2,\omega_2^\prime)$ (right) as described in Case-2. The shaded parts are truncated to glue $(X^{\prime}_1,\omega_1^\prime)$ and $(X^{\prime}_2,\omega_2^\prime)$.}
\end{figure}

\smallskip 

\noindent The Proposition just proved is a re-statement of Theorem \ref{main:thme} of the Introduction.

\subsection{Relation with the Hurwitz Existence Problem} We can use Proposition \ref{prop:bc}  to relate the problem of  realizing a given integral holonomy $\chi$ to solving certain cases of Hurwitz Existence Problem, which we  had described in the Introduction.

\noindent A special case of this problem is  when the target surface $\Sigma = \Bbb S^2$, and is to determine if there is a branched covering $f:S_g \to \Bbb S^2$ with prescribed branching data. Recall that the prescribed branching data comprises the desired degree $\deg(f)=d \geq 2$, the number $n\geq 1$ of branch values on $\Bbb S^2$, and a collection $\mathcal{B} = \{B_1, B_2,\ldots , B_n\}$ where each $B_i$ is an integer partition of $d$ (with at least one integer greater than $1$) that indicates the local degrees at the preimages of the $i$-th branch-value. Moreover, the prescribed branching data satisfy the Riemann-Hurwitz formula \eqref{rhcon}. 
Recall that such a prescribed branching data is said to be \textit{realizable} if there is a positive solution to the Hurwitz existence problem above.

\smallskip
\noindent The following corollary  is immediate from  Proposition \ref{prop:bc} : 

\begin{cor}\label{cor:cond3}  The following condition is equivalent to (1) and (2) of Proposition \ref{prop:bc}:
\begin{itemize}
    \item[(3)]  There exists realizable branching data for a branched covering $f:S_g \to \Bbb S^2$ which is a collection $\mathcal{B}$  of partitions of 
    $d = \sum\limits_{i=1}^k \lambda_i$ satisfying the following:
    \begin{itemize}
        \item $\lambda$ is part of the collection unless $\lambda = (1,1,\ldots, 1)$, 
        \item $\mu$ is part of the collection unless $\mu = (1,1,\ldots, 1)$,  and
        \item in all other partitions, the only integers that are different from $1$ are exactly $\{d_1+1,\ldots, d_r +1\}$.
    \end{itemize}
\end{itemize}

\end{cor}

\noindent Corollary \ref{cor:cond3} and Theorem \ref{main:thme} then readily imply the statement of Corollary \ref{main:corf} from the Introduction. We remark that Corollary \ref{main:corf} is consistent with Corollary \ref{consthmb2} (Corollary \ref{main:cord} of the Introduction); an instance is given by the following example.

\begin{ex}\label{lastex}
Let $\mathcal{D} = (\mathbb{S}^2, \mathbb{S}^2, 3,6, \mathcal{B})$ be abstract branch datum where $\mathcal{B}$ comprises the three partitions $\lambda = (5,1),  \mu = (3,1,1,1)$ and $(5,1)$. Then $\mathcal{D}$ is realizable by Corollary \ref{consthmb2}.  The requirement \eqref{gencomp} of Theorem \ref{main:thme} is satisfied since $\sum_{i} \lambda_i=\sum_j \mu_j = 6 > 4 = d_1$. This agrees with Corollary \ref{main:corf}. 
\end{ex}

\bigskip

\appendix

\section{Periods of meromorphic differentials with prescribed poles}\label{pmdpp}

\noindent Throughout this appendix, as usual, let $S_{g,n}$ be a $n$-punctured surface of genus $g$ and let $\nu=(p_1,p_2\dots,p_n)$ be a tuple of positive integers. Let $\mathcal{H}(-;\nu)$ denotes the stratum of $\Omega\mathcal{M}_{g,n}$ of meromorphic abelian differentials with poles of degrees $p_1,p_2\dots,p_n$. We consider here the problem of determining the image of $\textsf{Per}_{|\mathcal{H}(-;\nu)}$, that is the subset of $\textsf{Hom}(\Gamma_{g,n},\C)$ of those representations that appear as the period of some meromorphic differential having a pole of order $p_i$ at the $i$-th puncture. 

\begin{thm}\label{thm:mdpp}
Let $S_{g,n}$ be a surface of genus $g$ and $n\geq 1$ punctures. Let $\nu=(p_1,p_2,\ldots p_n)$ be positive integers assigned to each puncture and let $\mathcal{I} = \{ i \ \vert\ p_i = 1\}$. Finally, let $\gamma_i$ denotes a loop around the $i$-th puncture. The following holds.
\begin{itemize}
    \item Suppose $p_i\ge2$ for every $i=1,\dots,n$ (that is $\mathcal{I}=\emptyset$). Then $\textsf{\emph{Im}}\big(\textsf{\emph{Per}}_{|\mathcal{H}(-;\nu)} \big)$ contains all non-trivial representations. Moreover, the period mapping is surjective (that is it contains also the trivial representation) if and only if $n\ge2$ or $p\ge3$ if $n=1$,
    \smallskip
    \item if $1<|\mathcal{I}|<n$, then $\textsf{\emph{Im}}\big(\textsf{\emph{Per}}_{|\mathcal{H}(-;\nu)} \big)$ is the complement of $\displaystyle\bigcup_{i\in\mathcal{I}} \{\chi:\Gamma\to\mathbb{C}\,|\, \chi(\gamma_i)=0 \}$,
    \item if $|\mathcal{I}|=n>1$ then $\textsf{\emph{Im}}\big(\textsf{\emph{Per}}_{|\mathcal{H}(-;\nu)} \big)$ is contained in the complement of $\displaystyle\bigcup_{i=1}^n \{\chi:\Gamma\to\mathbb{C}\,|\, \chi(\gamma_i)=0 \}$. More precisely:
    \begin{itemize}
        \item if $n\ge3$, then $\textsf{\emph{Im}}\big(\textsf{\emph{Per}}_{|\mathcal{H}(-;\nu)} \big)$ is equal to the complement of $\displaystyle\bigcup_{i=1}^n \{\chi:\Gamma\to\mathbb{C}\,|\, \chi(\gamma_i)=0 \}$
        \item if $n=2$ and $g=0$ then $\textsf{\emph{Im}}\big(\textsf{\emph{Per}}_{|\mathcal{H}(-;\nu)} \big)=\textsf{\emph{Hom}}(\Gamma,\mathbb{C}^*)$,
        \item if $n=2$ and $g\ge1$ then $\textsf{\emph{Im}}\big(\textsf{\emph{Per}}_{|\mathcal{H}(-;\nu)} \big)$ is the complement of $\{\chi:\Gamma\to\mathbb{C}\,|\,\chi_2-\text{part is trivial }\}\cup\mathbb{C}\cdot\{\chi \text{ integral }\,|\, \chi(\gamma)=1\}$ where $\gamma$ is a curve enclosing one of the two punctures. 
    \end{itemize}
\end{itemize}
\end{thm}
\medskip
\noindent Notice that, when $n\ge3$, the statement above simply reduces to Corollary \ref{cor:mdpp}. Let us move on the proof.

\begin{proof}[Proof of Theorem \ref{thm:mdpp}]
Let $\nu=(p_1,p_2,\dots,p_n)$ be a tuple of positive integers. Let us begin by noticing that the space $\mathcal{H}(-,\nu)$ admits a natural stratification by the strata $\mathcal{H}(\mu;\nu)$ enumerated by unordered partitions $\mu$ of $2g-2+p_1+p_2+\cdots+p_n$. Thus, a representation $\chi$ appears as the period of a meromorphic abelian differential with prescribed orders of the poles at punctures if and only if it appears in the image of the period mapping as in \eqref{permap} restricted to some stratum $\mathcal{H}(\mu;\nu)\subset\mathcal{H}(-;\nu)$. \\

\noindent Suppose $p_i\ge2$ for every $i=1,2,\dots,n$. Let $\chi:\Gamma_{g,n}\longrightarrow\C$ be a non-trivial representation and consider any tuple $\mu=(d_1,d_2,\dots,d_k)$ of positive integers such that the equation \eqref{zpec} holds, that is
\[ \sum_{j=1}^k d_j - \sum_{i=1}^n p_i=2g-2.
\] If $g=0$, then Proposition \ref{hgc:mainprop} applies and $\chi$ appears as the holonomy of some translation surface with poles determined by a meromorphic differential on $\cp$ having zeros of orders $d_1,d_2,\dots,d_k$ and poles of order $p_1,p_2,\dots,p_n$ at the punctures. In a similar fashion, if $g\ge1$, then Theorem \ref{main:thmc} (or Theorem \ref{thm:nthpgs}) applies and $\chi$ appears as the holonomy of some translation surface with poles determined by a meromorphic differential on $S_{g,n}$ having zeros of orders $d_1,d_2,\dots,d_k$ and poles of order $p_1,p_2,\dots,p_n$ at the punctures. In both cases, $\chi$ is realized as the period of a meromorphic differential in $\mathcal{H}(-;\nu)$ as desired. Let us consider the trivial representation and pick $\mu=(1,1,\dots,1)$, where $1$ repeats $2g-2+p_1+p_2+\cdots+p_n$.
If we assume $n\ge2$, then the inequality 
\[1\le n-1=2n-(n+1)\le \sum_{i=1}^n p_i -(n+1)
\] always holds. All the conditions of Theorem \ref{main:thmb} are satisfied and hence we can realize the trivial representation as the holonomy of some translation surface with all simple zeros and poles of orders $p_1,p_2,\dots,p_n$ as desired. If $n=1$, then $1\le p-n-1=p-2$ if and only if $p\ge3$. Then we can proceed in the same fashion.\\

\noindent Suppose $1<|\mathcal{I}|\le n$. In the first place we can notice that the holonomy of the curve $\gamma_i$ cannot be trivial because the $i$-th is prescribed as a simple pole for any $i\in\mathcal{I}$. As a consequence every representation in the space
\begin{equation}\label{compl}
    \bigcup_{i\in\mathcal{I}} \{\chi:\Gamma\to\mathbb{C}\,|\, \chi(\gamma_i)=0 \}
\end{equation} 
 can never be realised as the period of any meromorphic differential on $S_{g,n}$. In particular, 
\begin{equation}\label{perineq}
    \textsf{Im}\big(\textsf{Per}_{|\mathcal{H}(-;\nu)}\big)\subseteq \textsf{Hom}\big(\Gamma_{g,n},\C\big)\setminus \displaystyle\bigcup_{i\in\mathcal{I}} \{\chi:\Gamma\to\mathbb{C}\,|\, \chi(\gamma_i)=0 \}.
\end{equation} 
\noindent In the case $1<|\mathcal{I}|< n$ then at least one of the puncture is prescribed as a pole of order at least two. Let us consider $\mu=(1,1,\dots,1)$ and let $\chi$ be any representation in the complement of \eqref{compl}. Then $\chi(\gamma_i)\neq0$ for any $i\in\mathcal{I}$ and, in particular, $p_i\ge2$ whenever $\chi(\gamma_i)=0$.  Even in this case, we can apply Proposition \ref{hgc:mainprop} if $g=0$ or Theorem \ref{main:thmc} if $g\ge1$. In both cases, $\chi$ appears as the holonomy of some translation surface with poles determined by a meromorphic differential on $S_{g,n}$ having simple zeros and poles of order $p_1,p_2,\dots,p_n$ as desired. As a consequence, the equation in \eqref{perineq} turns out an equality.\\

\noindent Let us finally suppose $|\mathcal{I}|=n$. Whenever $\chi$ is a non-rational representation taken in the complement of \eqref{compl} it is straightforward to check that the discussion above applies \emph{mutatis mutandis}. Therefore, 
\begin{equation}\label{perineq2}
    \{\chi:\Gamma\to\mathbb{C}\,|\, \chi \text{ not rational}\}\subseteq
    \textsf{Im}\big(\textsf{Per}_{|\mathcal{H}(-;\nu)}\big)\subseteq \textsf{Hom}\big(\Gamma_{g,n},\C\big)\setminus \displaystyle\bigcup_{i\in\mathcal{I}} \{\chi:\Gamma\to\mathbb{C}\,|\, \chi(\gamma_i)=0 \}.
\end{equation} 

\noindent Recall that a representation $\chi$ is said to be rational if it non-trivial and its image is contained in the $\Q$-span of some complex number $c\in\C^*$. Up to rescaling by a proper factor, we can suppose the representation to be integral, that is we can assume without loss of generality that $\textsf{Im}(\chi)=\Z$. Recall also that a representation $\chi$ is realizable in a certain stratum if and only if the rescaled representations are all realizable in the same stratum.\\
\noindent Our Theorem \ref{main:thmc} does not directly apply to integral representations if the poles are all required to be simple. In fact, the realizability is subject to the further necessary condition \eqref{gcomb0}. This leads to discuss integral representations as follows.\\

\noindent Suppose $n\ge3$ and let $\chi$ be an integral representation in the complement of \eqref{compl}. Let $\mu=(1,1,\dots,1)$, where $1$ repeats $2g-2+n$ times. Then $\chi$ can be realized as the holonomy of some translation structure with all simple zero and simple poles. In fact, if $d_1=d_2=\cdots=d_r=1$, then it is easy to check that the necessary and sufficient condition \eqref{gcomb0} is satisfied and so our Theorem \ref{main:thme} applies. Even in the case the equation \eqref{perineq} turns out an equality.\\

\noindent Let us consider the case $n=2$. A representation $\chi:\Gamma_{g,2}\longrightarrow \C$ belongs to the complement of \eqref{compl} if and only if the $\chi_2$-part of $\chi$ is non-trivial. If $g=0$,  then we can easily note that 
\[ \textsf{Hom}\big(\Gamma_{0,2},\C\big)\setminus\{\chi:\Gamma\to\mathbb{C}\,|\, \chi(\gamma)=0 \}=\textsf{Hom}\big(\Z,\C^*\big)
\] where $\gamma$ is one of the two curves enclosing one puncture. Any representation $\textsf{Hom}\big(\Z,\C^*\big)$ is realised as the holonomy of an Euclidean cylinder and hence as the period of a meromorphic abelian differential in the stratum $\mathcal{H}(\emptyset;1,1)$ by Remark \ref{psrmk}, where $\mathcal{H}(\emptyset;1,1)$ denotes the stratum of Euclidean structures on a cylinder without zeros (all points are regular!).\\

\noindent Suppose $g\ge1$ and $n=2$. Let $\chi:\Gamma_{g,2}\longrightarrow \Z<\C$ be an integral representation such that $\chi(\gamma)=\lambda>0$, where $\gamma$ is a curve enclosing one of the two punctures. In this case, the necessary condition \eqref{gcomb0} simplifies to $\lambda=\chi(\gamma)>\max(d_1,d_2,\dots,d_r)$. If $\lambda=1$, it is an easy matter to check that \eqref{gcomb0} never holds for every partition $\mu$ of $2g$. As a consequence, the whole set $\C\cdot\{\chi \text{ integral }\,|\,\chi(\gamma)=1\}$ does not belong to the image of the period mapping when restricted to the stratum $\mathcal{H}(-;1,1)$. On the other hand, whenever $\lambda\ge2$ the condition \eqref{gcomb0} always holds for the partition $\mu=(1,1,\dots,1)$ and hence the integral representation can be realized as the holonomy of some translation structure with simple zeros and simple poles at the punctures. Thus the desired conclusion holds.
\end{proof}

\bigskip

\section{Proof of Lemma \ref{lem:irrmult}}\label{app:irrmultproof}

\noindent For convenience, we assume $\displaystyle\sum_{i=1}^n s_i = \sum_{j=1}^m t_j = 1$. We consider horizontal two strips of some small width and length 1 and place separators in these strips so that the resulting blocks have lengths $s_i$ and $t_j$ in the respective strips as shown in Figure \ref{fig:app1}. The setting can also be viewed as two boxes of length 1 with blocks of length $s_i$ and $t_j$ placed in them. We shall denote the blocks of length $s_i$ (resp. $t_j$) as the $s$-series (resp. $t$-series). 

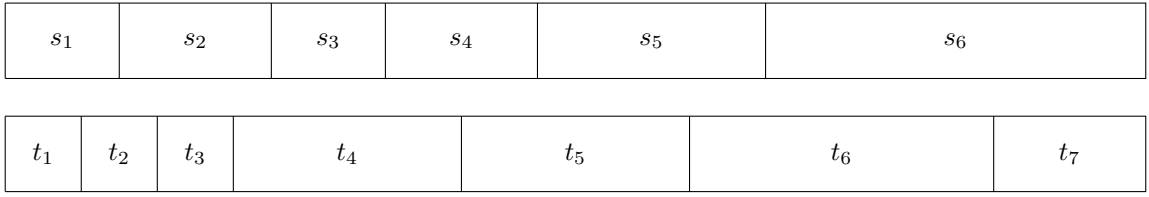
\begin{figure}[!h]
    \centering
    \begin{tikzpicture}
    \draw (0,0) -- (15,0) -- (15,1) -- (0,1);
    \draw (0,1.5) -- (15,1.5) -- (15,2.5) -- (0,2.5);
    \foreach \x / \y [count=\c] in {0/0.5, 1/1.5, 2/2.5, 3/4.5, 6/7.5, 9/11, 13/14}
    {
    \draw (\x, 0) -- (\x, 1);
    \node at (\y, 0.5) {$t_{\c}$};
    }
    \foreach \x / \y [count=\c] in {0/0.75, 1.5/2.5, 3.5/4.25, 5/6, 7/8.5, 10/12.5}
    {
    \draw (\x, 1.5) -- (\x, 2.5);
    \node at (\y, 2) {$s_{\c}$};
    }
    \end{tikzpicture}
    \caption{Partitions of boxes determined the numbers $s_i$ and $t_j$}
    \label{fig:app1}
\end{figure}

\noindent We need to show that there exists an ordering of $s_i$'s and $t_j$'s so that no two separators have the same $x$-coordinate. We do this by first describing an algorithm to place the blocks. Without loss of generality, we may assume that $s_1$ is the smallest of the numbers $\{ s_1, \ldots, s_n, t_1, \ldots, t_m \}$. We shall place the blocks of lengths $\{ s_1, \ldots, s_n, t_1, \ldots, t_m \}$ in their respective boxes from left to right and the order of placing the blocks will determine the order required by the Lemma. \\


\noindent Step 1 - Place the block of length $s_1$ in it's box.\\

\noindent Step 2 - If there is only one unplaced block of length $t_j$, place the block. If there are no more blocks to be placed in either of the series, end the algorithm. Else, place the smallest block among the unplaced blocks of length $t_j$ such that, after placing, the right end of the block does not coincide with any of the right ends of the placed blocks of length $s_i$. If no such block exists, replace the rightmost block in the $s$-series with the next smallest unplaced block and repeat step 2.\\

\noindent Step 3 - If the right end of the series of blocks of length $t_j$ that have already been placed is to the right of the right end of the series of blocks of length $s_i$ that have already been placed, go to step 4. Else, go to step 2.\\

\noindent Step 4 - If there is only one unplaced block of length $s_i$, place it. If there are no more blocks to be placed in either of the series, end the algorithm. Else, place the smallest block among the unplaced blocks of length $s_i$ such that, after placing, the right end of the block does not coincide with any of the right ends of the placed blocks of length $t_j$. If no such block exists, replace the rightmost block in the $t$-series with the next smallest unplaced block and repeat step 4.\\

\noindent Step 5 - If the right end of the series of blocks of length $s_i$ that have already been placed is to the right of the right end of the series of blocks of length $t_j$ that have already been placed, go to step 2. Else, go to step 4.\\

\noindent When this algorithm runs to completion, we obtain the required ordering. The only possible case where the algorithm does not run to completion is when all unplaced blocks in the $s$-series have the same length and all unplaced blocks in the $t$-series have the same length. At the point where the algorithm gets stuck, the arrangement looks as in Figure \ref{fig:app2}. Since the distance from the blue dashed line to the right end of the box is an integer multiple of both $s$ and $t$, $s$ and $t$ are rational multiples of each other. Let $s = p \lambda$ and $t = q \lambda$. We fill in the remaining blocks without concerning ourselves with the coinciding separators for the time being. 

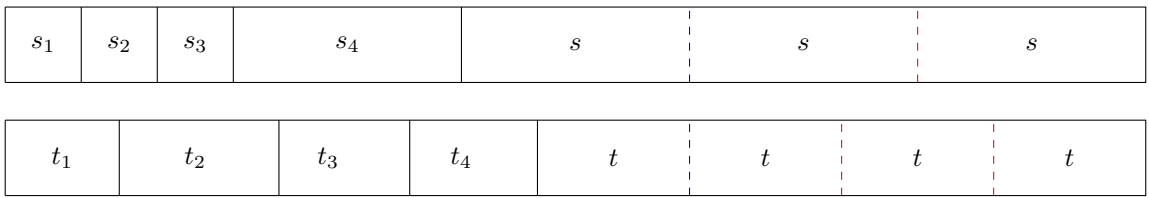
\begin{figure}[!h]
    \centering
    \begin{tikzpicture}
    \draw (0,0) -- (15,0) -- (15,1) -- (0,1);
    \draw (0,1.5) -- (15,1.5) -- (15,2.5) -- (0,2.5);
    
    \foreach \x / \y [count=\c] in {0/0.75, 1.5/2.5, 3.6/4.25, 5.32/6}
    {
    \draw (\x, 0) -- (\x, 1);
    \node at (\y, 0.5) {$t_{\c}$};
    }
    \draw (7, 0) -- (7, 1);
    \draw[blue, dashed] (9,0) -- (9,1);
    \draw[red, dashed] (11,0) -- (11,1);
    \draw[red, dashed] (13,0) -- (13,1);
    \foreach \x in {8, 10, 12, 14}
    {
    \node at (\x, 0.5) {$t$};
    }
    
    \foreach \x / \y [count=\c] in {0/0.5, 1/1.5, 2/2.5, 3/4.5}
    {
    \draw (\x, 1.5) -- (\x, 2.5);
    \node at (\y, 2) {$s_{\c}$};
    }
    \draw (6,1.5) -- (6,2.5);
    \draw[blue, dashed] (9,1.5) -- (9,2.5);
    \draw[red, dashed] (12,1.5) -- (12,2.5);
    \foreach \x in {7.5, 10.5, 13.5}
    {
    \node at (\x, 2) {$s$};
    }
   
    \end{tikzpicture}
    \caption{An example for when the algorithm fails. The values are $s_1 = s_2 = s_3 = \frac{1}{15}$, $s_4 = s = \frac{1}{5}$, $t_1 = \frac{1}{10}$, $t_2 = \frac{7}{50}$ $t_3 = \frac{\sqrt{3}}{15}$ $t_4 = \frac{3.4 - \sqrt{3}}{15}$ and $t=\frac{2}{15}$. The point of failure is when all possible blocks to be placed after $s_4$ and $t_4$ end at the blue dashed lines.}
    \label{fig:app2}
\end{figure}

\noindent Now consider the rightmost separator in either of the series whose distance from the right side of the box is an irrational multiple of $\lambda$. Such a separator exists by our initial assumption and we label it $S$. Moving the block immediately to the right of this $S$ to  the end of its series, and translating the remaining blocks towards left as shown in Figure \ref{fig:app3}, achieves the required placement of the blocks. This is because the distances of all the separators to the right of $S$ from the right end of the box are originally rational multiples of $\lambda$ and this movement moves one set of separators by an irrational multiple of $\lambda$.\\

\vspace{.1in} 

\begin{figure}[!h]
    \centering
    \begin{tikzpicture}
    \draw (0,0) -- (15,0) -- (15,1) -- (0,1);

    \foreach \x / \y [count=\c] in {0/0.75, 1.5/2.5, 3.6/4.4}
    {
    \draw (\x, 0) -- (\x, 1);
    \node at (\y, 0.5) {$t_{\c}$};
    }
    \foreach \x / \y in {5.32/6.32, 7.32/8.32, 9.32/10.32, 11.32/12.32}
    {
    \draw (\x, 0) -- (\x, 1);
    \node at (\y, 0.5) {$t$};
    }
    \draw (13.32, 0) -- (13.32, 1);
    \node at (14.2, 0.5) {$t_4$};
    \draw[red, ->] (5.32, -0.25) -- (5.32, 0);
    \node[below] at (5.32, -0.25) {$S$};
    
    \draw (0,1.5) -- (15,1.5) -- (15,2.5) -- (0,2.5);
    \foreach \x / \y [count=\c] in {0/0.5, 1/1.5, 2/2.5, 3/4.5, 6/7.5, 9/10.5, 12/ 13.5}
    {
    \draw (\x, 1.5) -- (\x, 2.5);
    \node at (\y, 2) {$s_{\c}$};
    }
   
    \end{tikzpicture}
    \caption{The shift in the case of our example.}
    \label{fig:app3}
\end{figure}
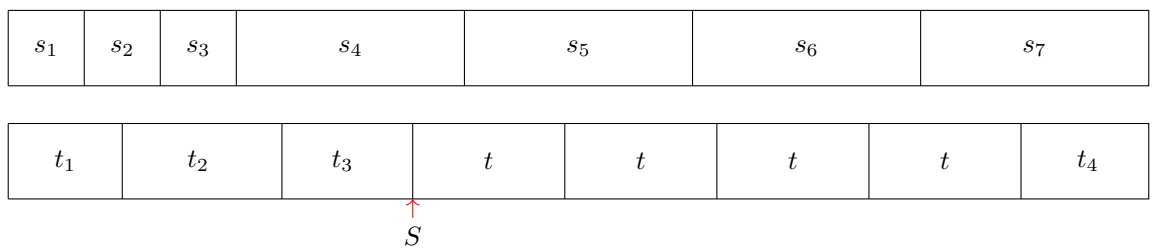

\bibliographystyle{amsalpha}
\bibliography{biblio}

\end{document}